\documentclass[11pt]{amsart}

\usepackage[utf8]{inputenc} 

\usepackage{amsfonts}
\usepackage{amssymb}
\usepackage{mathrsfs}
\usepackage{dsfont}
\usepackage{bbm}
\usepackage{latexsym}
\usepackage{bm}

\usepackage{amsmath}
\usepackage{amsthm}
\usepackage{mathtools}
\mathtoolsset{showonlyrefs} 
\allowdisplaybreaks
\usepackage{graphicx}
\usepackage{epsfig}
\usepackage{subfig}

\usepackage{tikz}
\usepgflibrary{patterns}
\usepgflibrary[patterns]
\usetikzlibrary{patterns}
\usetikzlibrary[patterns]

\usepackage{a4wide}
\usepackage{geometry}
\usepackage{color}
\usepackage{xcolor}
\usepackage{microtype}
\usepackage{caption}
\usepackage{url}
\usepackage{varioref}
\usepackage{verbatim}

\usepackage{hyperref}

\usepackage[font=small,format=hang,labelfont={sf,bf}]{caption}

\theoremstyle{plain}
\newtheorem{theorem}{Theorem}[section]
\newtheorem{lemma}[theorem]{Lemma}
\newtheorem{proposition}[theorem]{Proposition}

\newtheorem{definition}[theorem]{Definition}
\newtheorem{remark}[theorem]{Remark}

\newtheorem{coroll}[theorem]{Corollary}

\numberwithin{equation}{section}

\definecolor{ddorange}{rgb}{1,0.5,0}
\definecolor{ddcyan}{rgb}{0,0.2,1.0}

\renewcommand{\div}{\mathrm{div}}
\newcommand{\pa}{\partial}

\newcommand{\mrestr}{%
	\,\raisebox{-.127ex}{\reflectbox{\rotatebox[origin=br]{-90}{$\lnot$}}}\,%
}

\newcommand{\numberset}{\mathbb}
\newcommand{\R}{\numberset{R}}

\newcommand{\N}{\numberset{N}}

\newcommand{\medint}{-\kern -,40cm\int}
\newcommand{\medintinrigo}{-\kern -,34cm\int}

\newcommand{\fr}{\penalty-20\null\hfill\(\blacksquare\)}

\title[The anisotropic surface diffusion with elasticity in 3D]{The anisotropic surface diffusion with elasticity \\ in three dimensions via the Cahn-Taylor \\ minimizing movement scheme }

\author[A. Kubin]
{Andrea kubin}
\address[Andrea Kubin]{
	Department of Mathematics and Statistics,
P.O.\ Box 35 (MaD), FI-40014 University of Jyvaskyla}
\email{andrea.a.kubin@jyu.fi}

\author[A. Kubin]
{Anna kubin}
\address[Anna Kubin]{
TU Wien, Institute of Analysis and Scientific Computing,
Wiedner Hauptstraße 8-10, 1040, Vienna, Austria	
}
\email[Anna Kubin]{anna.kubin@asc.tuwien.ac.at}

\begin{document}

	\begin{abstract}
    In this paper, we introduce a suitable notion of flat solutions for the anisotropic surface diffusion equation with elasticity in three-dimensions, based on a minimizing movement scheme inspired by that introduced by Cahn and Taylor. Using this scheme, we prove the existence of classical solutions to the equation without using a curvature regularization term. Moreover we establish the consistency of the approximation method. 
    \noindent 
\vskip5pt
\noindent
\textsc{Keywords}: Geometric evolutions; variational methods; minimizing movement scheme.  
\vskip5pt
\noindent
\textsc{AMS subject classifications:}  
53E10; 53E40;
49Q20; 37E35; 74G65.  
	\end{abstract}
	
	\maketitle
	\tableofcontents
\section{Introduction}

In recent years, considerable attention has been devoted in both the physical and mathematical literature to the study of energy functionals involving the competition between interfacial surface energy and elastic energy. From a mathematical perspective, this problem can be formulated either as the analysis of local or global minimizers, or as the study of gradient flows associated with energies given by the sum of an elastic contribution and a surface tension term, typically modeled through isotropic or anisotropic perimeter functionals.
The corresponding static problem has been extensively investigated in the physical and numerical literature. Concerning the mathematical analysis, several works have addressed this topic:~\cite{BGZ2015, Bonacini2013,BonCham2002, FoFuLeMo2007, FM2012, GZ2014} establish existence, regularity, and stability results for variational models describing equilibrium configurations in two dimensions, while~\cite{Bonacini2015, CS2007} provide analogous results in the three-dimensional setting.
The dynamical problem can be regarded as a particular instance of the Einstein-Nernst equation. This equation describes the evolution of an interface driven by surface mass transport under the action of a chemical potential $\mu$. 
Since bulk mass transport may be neglected due to the much shorter timescale on which it occurs, see~\cite{Mullins1963}, the evolution is governed by the volume-preserving law
	\begin{equation}\label{EQEINNER}
		V_t= \Delta_{\pa E_t} \mu_t \text{ on } \pa E_t,
	\end{equation}
	where  $V_t$  is the normal velocity and $\Delta_{\pa E_t}$  is the Laplace-Beltrami operator on $\pa E_t$.  
    
We consider the free-energy functional defined by
\begin{equation}\label{Fintroeng}
		J(F)= \int_{\pa F} \varphi(\nu_F)\, d\mathcal{H}^2+ \frac{1}{2} \int_{\Omega \setminus F} Q(E(u_F))\, dx,
	\end{equation}
    where $\Omega \subset \mathbb{R}^3$ denotes the region occupied by the elastic body, while $F \subset \Omega$ represents the void formed inside the body; see~\cite{SMV2004}. In~\eqref{Fintroeng}, the function $u_F$ denotes the elastic equilibrium in $\Omega \setminus F$ subject to the Dirichlet boundary condition $u_F=w_0$ on $\partial \Omega$ and with Neumann condition on $\pa F$, namely,
\begin{equation}\label{eqelasticoliberointro}
		u_{F} \in \arg \!\min \left\{ \int_{\Omega \setminus F } Q(E(u))\, dx : u \in H^1(\Omega \setminus F,\R^3) , \, u |_{\pa \Omega}=w_0    \right\}.
	\end{equation}
The function $Q$ is the quadratic form defined by $Q(A):= \frac{1}{2} \mathbb{C}A: A$ for all $3 \times 3$-symmetric matrices $A$, where $\mathbb{C}$ is the elasticity tensor. Moreover, $E(u_F)$ stands for the symmetric part of the gradient $\nabla u_F$, namely $ E(u_{F})=\frac{1}{2} \big( \nabla u_{F}+ (\nabla u_{F})^t\big)$. Finally, $\varphi(\nu_F)$ denotes the anisotropic surface energy density evaluated at the outer unit normal $\nu_F$ to $\pa F$. 
We recall that the existence and regularity of volume-constrained minimizers of the functional $F \mapsto J(F)$ have been extensively investigated in the literature; see, for instance,~\cite{CJP2013, FFLMi2011} for the two-dimensional case. 

Under the assumption that $\mu_t$ coincides with the first variation of the free energy functional~\eqref{Fintroeng}, equation~\eqref{EQEINNER} takes the form
\begin{equation}\label{MAINEQ}
			\left\{
			\begin{aligned}
				& V_t= \Delta_{\pa E_t} \big(   H^\varphi_{E_t}-Q (E(u_{E_t}))\big), \text{ on } \pa E_t  \\
				&u_{E_t}\in \arg \!\min \left\{ \int_{\Omega \setminus E_t } Q(E(u))\, dx : u \in H^1(\Omega \setminus E_t, \R^3) , \, u |_{\pa \Omega}=w_0    \right\}.
			\end{aligned}
			\right.
		\end{equation} 
        where $H^\varphi_{E_t}$ denotes the anisotropic curvature of $\pa E_t$.
        We remark that anisotropic surface diffusion with elasticity may be regarded as a nonlocal perturbation of the anisotropic surface diffusion flow. While the latter, in the isotropic case, has been extensively investigated in the literature (see, e.g., \cite{DGDKK, DFM2023, EMS, GarckeGosswein, LeCroneSimonett, MaySim}), only a limited number of results are available for the general evolution \eqref{MAINEQ}.
        We first recall results on  the existence of solutions in the two-dimensional case. In~\cite{FFLM2012}, the authors investigate existence via a minimizing movements scheme with an additional curvature regularization term in the free energy, as well as asymptotic stability in $\mathbb{R}^2$. Subsequently,~\cite{FJM2018} establish the existence of classical solutions and the asymptotic stability of strictly stable stationary sets in the plane. Finally, in~\cite{Kub2025}, the author proves the existence of classical solutions by means of a minimizing movements-type algorithm, in the spirit of~\cite{CaTa}.    
In the three-dimensional setting, we recall that the authors of~\cite{FJM2020} investigate the existence and asymptotic stability of solutions to the isotropic surface diffusion equation with elasticity. In the three-dimensional and anisotropic setting,~\cite{FFLM2015} establish the existence of solutions to anisotropic surface diffusion in $\mathbb{R}^3$, incorporating a curvature regularization and employing a minimizing-movements type algorithm. We remark that adding a curvature regularization term into functionals of the form~\eqref{Fintroeng} is a well-established approach in the literature, see~\cite{AG1989,BHSV2007,DCGPG1992,GJ2002,H1951,RRV2006,SMV2004}. 
More precisely, in~\cite{FFLM2015}, the authors consider the $H^{-1}$  gradient flow of the functional
	\begin{equation}
		F \mapsto  \int_{\pa F} \varphi(\nu_F)\, d\mathcal{H}^2 +\frac{1}{2} \int_{\Omega \setminus F} Q(E(u_F))\, dx + \frac{\varepsilon}{p}\int_{\pa F}\vert H_{F}\vert^p \, d \mathcal{H}^2,
	\end{equation}
	where $\varepsilon>0$  and $p>2$. Their analysis focuses on periodic graph models describing the evolution of epitaxially strained elastic films in three dimensions. 
    Unlike our method, their approach employs a minimizing-movements-type algorithm restricted to sets with boundaries that can be represented as graphs of functions. In contrast, as will be explained shortly, our algorithm applies to general sets of finite perimeter. Furthermore, their approach crucially relies on the curvature regularization term. In particular, all the estimates derived therein are $\varepsilon$-dependent and degenerate as $\varepsilon \rightarrow 0^+$.
    
The algorithm we propose for modeling our equation (see Section~\ref{settingdelproblema180526} for the precise definition) is a modified version of that introduced by Cahn and Taylor in~\cite{CaTa}, based on the minimizing movements scheme applied to the energy~\eqref{Fintroeng}, with dissipation given by
    \begin{equation} \label{eq:def-distance-intro}
		d_{H^{-1}}(F,E):=\sup_{\|\nabla_{\pa E} f\|_{L^2(\pa E)}\leq1}\int_{\R^3}   f( \pi_{\pa E}(x))(\chi_F(x)-\chi_E(x))\,dx.
	\end{equation}
    Above, $\nabla_{\pa E}$ denotes the tangential gradient, $\chi_E$ the characteristic function of $E$ and $\pi_{\pa E}$ the projection on the boundary $\pa E$. 
    This algorithm allows to introduce a particular notion of flat flow solutions for the anisotropic surface diffusion with elasticity in $\mathbb{R}^3$.
    We recall that the consistency of the Cahn-Taylor algorithm was only recently established in~\cite{CFJKsd} in $\mathbb{R}^3$, in the case of isotropic perimeter and in the absence of an elastic term.

    Our main result, contained in Theorem~\ref{1THMMAIN1}, establishes the existence and uniqueness of classical solutions to equation~\eqref{MAINEQ} by means of the Cahn-Taylor algorithm, starting from a sufficiently regular initial datum $E_0$. This procedure also provides a constructive method for generating solutions to this differential equation. Furthermore, this result generalizes the ones in~\cite{CFJKsd}. This extension is nontrivial because, beyond the addition of the elasticity term, the regularity estimates developed in~\cite{CFJKsd} to ensure the convergence of the algorithm are insufficient in our setting.
    The main motivation for this problem arises from the structure of the linearization of~\eqref{MAINEQ}. Indeed, if we parametrize~\eqref{MAINEQ} with respect to the initial datum and consider its linearization, rigorously derived using the expansion of the anisotropic curvature of Lemma~\ref{10Expcurvvarphi}, we obtain
  \begin{equation}
        \pa_t u=- \Delta_{\pa E_0} L u \quad \text{ on } \pa E_0,
    \end{equation}
    where $L$ is an elliptic operator defined on $\pa E_0$. However, the operator $L$ does not commute with the Laplace–Beltrami operator, namely
    $$ \Delta_{\pa E_0} L \neq L \Delta_{\pa E_0} .$$
    In particular, this lack of commutativity creates greater difficulties than in~\cite{CFJKsd} when establishing, in Section~\ref{SEZIONEDELLESTIME}, the $H^k$  estimates required to prove the convergence of the scheme.


    Additionally, in Section~\ref{sezioneconvglobaleee}, we establish the consistency of the algorithm, namely we prove that the minimizing movements scheme converges to the solution of problem~\eqref{MAINEQ} over the entire interval of existence.

    The paper is organized as follows. In Section~\ref{sezioneprel}, we introduce the notation used throughout the paper. In Section~\ref{sez318052026}, we recall several preliminary lemmas that will be used later on, and we derive an expansion formula for the anisotropic curvature.
    In Section~\ref{settingdelproblema180526}, we introduce the function $d_{H^{-1}}(F,E)$, discuss several of its properties, we present the elasticity problems, and define the minimizing movement scheme used in the analysis. In Section~\ref{sezioniquasiminimi}, we recall some properties of almost minimizers of the anisotropic perimeter. In Section~\ref{serzquaisminH^-1}, we introduce a notion of almost minimizers of the anisotropic perimeter with respect to the function $d_{H^{-1}}$. In Section~\ref{SEZIONEDELLESTIME}, we establish various regularity estimates that will be used to prove the existence of solutions to equation~\eqref{MAINEQ}. Section~\ref{iterazionesezione} is devoted to the proof of the iterative argument, which is one of the main tools used to prove Theorem~\ref{1THMMAIN1} in Section~\ref{sezionefinale}. Finally, in Section~\ref{sezioneconvglobaleee}, we prove the convergence of the scheme to a global solution of~\eqref{MAINEQ}.

\section{Notation and definitions}\label{sezioneprel}
	We work in the three-dimensional Euclidean space $\R^3$. We denote by $ \{e_1,\, e_2,\, e_3\}$ the canonical basis of $\R^3$, by $\vert \cdot \vert $ the Euclidean norm, and by $\cdot$ the inner product in $\R^3$. If $V \subset \R^3$ is a linear subspace, we denote by $V^{\perp}:= \{ x \in \R^3 : x\cdot v=0 \text{ for all } v\in V\}$. We denote by $\nabla$ the gradient operator in $\R^3$ and by $\nabla'$ the gradient operator in $\R^2$.  Throughout the paper, $\varphi$ denotes a regular strictly convex norm, namely, $\varphi \in C^{\infty}(\R^3 \setminus \{0 \})$ and  \begin{equation}\label{unifell}
		\exists J>0 \text{ such that } \nabla^2 \varphi (\nu) \xi \cdot \xi \geq J \vert \xi \vert^2 \text{ for every } \nu ,\,\xi \in \R^3 \text{ with } \vert \nu \vert=1 \text{ and } \nu \cdot \xi=0.
	\end{equation}
   We refer to $\varphi$ as the anisotropy.
	The dual norm $\varphi^0$ is defined by $\varphi^0(\xi)= \sup_{\eta \in \R^3 \setminus \{0\}} \frac{\xi \cdot \eta}{\varphi(\eta)}$.
    Given $a,b \in \R^3$, we define the linear map $a \otimes b : \R^3 \rightarrow \R^3$ by $ a\otimes b (x):=(x \cdot b)a$.  For $r>0$, we set $B_r(x)= \{ y \in \R^3 : \vert x- y \vert < r \}$. When $x=0$, we simply write $B_r:=B_r(0)$. We also use the notation $B'_r(x') \subset\R^2$ to denote the ball of radius $r$ centered at $x' \in \R^2$ and we write $B'_r:=B'_r(0)$. For every set $E\subset \R^3$, we denote by $\mathrm{cl}(E)$ and $\mathrm{int}(E)$ its topological closure and topological interior, respectively, with respect to the Euclidean topology. The Lebesgue measure of a Borel set $E\subset \R^3$ is denoted by $\vert E \vert$. We denote by $\mathcal{H}^2$ the $2$-dimensional Hausdorff measure. Given $A, B \subset \R^3$, we denote by $d_{\mathcal{H}}(A,B)$ the Hausdorff distance between $A$ and $B$. Given $E\subset\mathbb{R}^3$ and $x\in\mathbb{R}^3$, we denote by $\mathrm{dist}_E(x)$ the distance from $x$ to $E$, namely, $\mathrm{dist}_E(x):= \inf_{y \in E} \vert x-y \vert$. We define the signed distance function to $E$ by
	\begin{equation}\label{distsd}
		d_E(x):=
		\left\{
		\begin{aligned}
			& \mathrm{dist}_E(x) &  x \in \R^3 \setminus E \, , \\
			& -\mathrm{dist}_{\R^3 \setminus E}(x) & x \in E \, . 
		\end{aligned}
		\right.
	\end{equation}
Observe that $\vert d_{E}(x)\vert = {\rm dist}_{\pa E}(x)$ for every $x \in \R^3$. Let $x\in\mathbb{R}^3$ be such that there exists a unique point $y\in\partial E$ realizing $\mathrm{dist}_E(x)$, namely, ${\rm dist}_{E}(x)= \vert x-y \vert$. We call $y$ the projection of $x$ onto $\partial E$ and denote it by $\pi_{\pa E}(x)=y$. Given $\delta>0$, we define the tubular neighborhood of $E\subset\mathbb{R}^3$ by \begin{equation}\label{intornotubolare}
		\mathcal{I}_{\delta}( E):=\{x\in\R^3:\,{\rm dist}_{\pa E}(x)<\delta\}.  
	\end{equation}
    Given $i,j \in \N$, we denote by $\delta_{ij}$ the Kronecker delta, namely, $\delta_{ij}=1$ if $i=j$ and $\delta_{ij}=0$ if $i \neq j$. Throughout the paper, $C(*,\cdots,*)$ denotes a generic positive constant, which may vary from line to line and depends only on the parameters $*,\cdots,*$.
    \subsection{Regular sets}\label{26052026regularsets}
    In this subsection, we introduce regular sets, starting with the definition of uniform ball condition and $C^{k,\alpha}$-regular set.
    \begin{definition}
		We say that a set $E \subset \mathbb{R}^3$ satisfies the \emph{uniform ball condition} (\textsc{UBC}) with radius $r_0>0$ if, for every $x \in \partial E$, there exist balls $B_{r_0}(x_{+})$ and $B_{r_0}(x_{-})$ such that
		\begin{equation}
			B_{r_0}(x_{+}) \subset \R^3 \setminus E, \quad B_{r_0}(x_{-}) \subset E \text{ and } x \in \pa B_{r_0}(x_{+}) \cap \pa B_{r_0}(x_{-}).
		\end{equation}

        A set $E \subset \R^3$ is said to be $C^{k,\alpha}$-regular, for some integer $k \geq 1$ and $ \alpha \in [0,1]$, if its boundary is a $C^{k,\alpha}$-regular surface. More precisely, for every $x \in \pa E$, there exists $r>0$ and a function $f \in C^{k,\alpha}(\R^2)$ such that, up to an isometry of $\R^3$, 
    \begin{equation}
        B_r(x) \cap  E = B_r(x) \cap \left\{ (y,t) \in \R^3 : \, t < f(y),\, y \in \R^2  \right\}.
    \end{equation}
	\end{definition}
    If $E$ satisfies the \textsc{UBC} with radius $r_0$, then, up to isometry of $\R^3$, one has
    \begin{equation}\label{24112025form1}
        B_{\frac{r_0}{2}}(x) \cap E = B_{\frac{r_0}{2}}(x) \cap \left\{ (y,t) \in \R^3 : \, t < f(y),\, y \in \R^2  \right\}, 
    \end{equation}
    where $f \in C^{1,1}(\R^2)$ satisfies $ \| f \|_{C^{1,1}(B'_{r_0/2}) } \leq \frac{10}{r_0} $ (see, for instance,~\cite[Theorem 2.6]{Da2018},~\cite[Proposition 2.7]{JN},~\cite[Proposition 2.1]{MM2000}). We now recall the definition of uniform $C^{k,\alpha}$-regularity.
    \begin{definition}
        Let $k \geq 1$, $ \alpha \in [0,1]$. A set $E \subset \R^3$ is said to be uniformly $C^{k,\alpha}$-regular with constants $r_0>0,\,C_0>0$ if it satisfies the \textsc{UBC} with radius $r_0$ and the function $f $ appearing in~\eqref{24112025form1} satisfies $ \| f \|_{C^{k,\alpha}(B'_{\frac{r_0}{2}}) } \leq C_0$.
    \end{definition}
   We recall that if a set $E \subset \mathbb{R}^3$ is $C^1$-regular, then for every $x \in \partial E$ there exists an outward unit normal vector $\nu_E(x)$. We define the matrix-valued map
   \begin{equation}\label{LAPROIEZmatrix}
     P_{\pa E}: \pa E \rightarrow \R^{3 \times 3}, \quad  P_{\pa E}(x):= I- \nu_{E}(x) \otimes \nu_E(x).
   \end{equation}
    This map is the orthogonal projection onto the tangent plane $T_{x} \pa E := \langle \nu_E(x) \rangle^{\bot}$. Let $n \geq 1$ and $X \in C^1(\R^3, \R^n)$,  we define the tangential differential of $X$ as the function $ \nabla_{\pa E} X : \pa E \rightarrow \R^{n \times 3}$ defined by
   \begin{equation}
       \nabla_{\pa E} X= \nabla X P_{\pa E}= \nabla X- \nabla X \nu_E \otimes \nu_E.
   \end{equation} 
   We remark that, if $E$ is $C^2$-regular, the tangential differential of a map $X:\partial E\to \mathbb{R}^n$ can be defined by extending $X$ to a tubular neighborhood of $\partial E$ via $\widetilde{X}:= X \circ \pi_{\pa E}$ and then setting $ \nabla_{\pa E} X:= \nabla_{\pa E} \widetilde{X}$. In the case $n=3$, we define the tangential divergence of $X$ by $\div_{\pa E} X= {\rm tr} (\nabla_{\pa E} X)$. For every $x\in \partial E$, we define the natural inclusion of $T_x \pa E$ into $\R^3$ by the linear map $I_x : T_x \pa E \rightarrow \R^3$, $ I_x \tau := \tau$. The tangential Jacobian of $X$ at $x\in \partial E$ is then defined as
   \begin{equation}
       J_{\pa E} X(x):= \sqrt{ \det \left( \left( \nabla_{\pa E} X(x)I_x \right)^T \left( \nabla_{\pa E} X(x)I_x \right)  \right)  }.
   \end{equation}
   Given a $C^2$-regular set $E\subset \mathbb{R}^3$, we define the symmetric matrix associated with the second fundamental form by $B_E:= \nabla_{\pa E} \nu_E$. We emphasize that this is not the standard definition of the second fundamental form. Indeed, $B_E\in \mathbb{R}^{3\times 3}$ always has one zero eigenvalue; however, this formulation is the most convenient for our purposes.
For every $x \in \pa E$, we have
$$B_E(x)= \kappa_1 \tau_1 \otimes \tau_1 + \kappa_2 \tau_2 \otimes \tau_2,$$
where $\kappa_1,\, \kappa_2$ are the principal curvatures and $\tau_1,\, \tau_2$ are the associated principal directions in the tangent plane $T_x \pa E$. We define the mean curvature and the Gaussian curvature of $\pa E$ at $x$ by
\begin{equation}\label{mammaroma6}
    H_E(x):= {\rm tr} \left( B_E(x) \right), \quad K_E(x):= \det \left(  I_x^{T} B_E(x) I_x \right).
\end{equation}
Observe that $ H_E= \kappa_1+\kappa_2$ and $K_E= \kappa_1 \kappa_2$. We also recall that if $E$ satisfies the \textsc{UBC} with radius $r_0$, then for every $x \in \pa E$, $ \vert B_E (x) \vert \leq 2/r_0$. For every vector field $ X \in C^1(\R^3,\R^3)$, the divergence theorem on $\partial E$ reads
\begin{equation}
    \int_{\pa E}\div_{\pa E} X \, d \mathcal{H}^2= \int_{\pa E} H_E \left( X \cdot \nu_E \right) \, d \mathcal{H}^2.
\end{equation}
Let $f \in C^2(\pa E)$, we define the Laplace-Beltrami operator by
\begin{equation}\label{LaplaceBeltraimiop}
    \Delta_{\pa E} f:= \div_{\pa E} \nabla_{\pa E}f .
\end{equation}
If a bounded set $E$ satisfies the \textsc{UBC} with radius $r_0$, then the projection onto $\pa E$ is well defined in $\mathcal{I}_{r_0}(\pa E)$ and 
\begin{equation}
    x=\pi_{\pa E}(x)+ d_E(x)\nabla d_E(x) \text{ for every }x \in \mathcal{I}_{r_0}(\pa E).
\end{equation}
If, in addition $E $ is $C^2$-regular, then $d_E \in C^2( \mathcal{I}_{r_0}(\pa E))$ and
\begin{equation}\label{24112025primacrossfit}
    \nabla^2 d_E(x)= B_E \left( \pi_{\pa E}(x)  \right) \left( I+d_E(x)B_E(\pi_{\pa E}(x) )  \right)^{-1} \text{ for every }x\in \mathcal{I}_{r_0}(\pa E),
\end{equation}
see~\cite[formula 2.33]{JN} for a proof. We note that formula~\eqref{24112025primacrossfit} is well defined since $B_E= \nabla_{\pa E} \nu_E$ is an $\R^{3 \times 3}$ matrix. Moreover, $\pi_{\pa E} \in C^1(\mathcal{I}_{r_0}(\pa E))$ and for $x= \pi_{\pa E}(y)$ we have
\begin{equation}\label{canzone1}
    \nabla \pi_{\pa E}(y)= I-  \nu_E (x) \otimes  \nu_E (x)-d_E(y)  B_E(x)  \left( I + d_E(y) B_E (x) \right)^{-1}.
\end{equation}
Given $f \in C^{1}(\pa E)$, we define its extension $f \circ \pi_{\pa E} : \mathcal{I}_r(\pa E) \rightarrow \R$. Since $f \circ \pi_{\pa E}= f \circ \pi_{\pa E} \circ \pi_{\pa E}$, it follows that, for every $y \in \mathcal{I}_{r_0}(\pa E)$,
\begin{equation}\label{1012pcrossfit}
    \nabla \big( f \circ \pi_{\pa E} \big)(y)= \big( \nabla \pi_{\pa E}(y) \big)^t \nabla \big( f \circ \pi_{\pa E} \big) \circ \pi_{\pa E}(y)= \big( \nabla \pi_{\pa E}(y) \big)^t\nabla_{\pa E} f (\pi_{\pa E}(y)),
\end{equation}
where we have used that $\nabla (f \circ \pi_{\pa E})(z)= \nabla_{\pa E} f (z)$ for $z \in \pa E$.
\subsection{Riemannian geometry notation}
In this subsection, we recall standard notation from Riemannian geometry. For a complete treatment of the topic, we refer to~\cite{DoCarmobook,  LeeBook,LeeBook2018, KobayashiNomizu1963}. Given $E \subset \R^3$ be a bounded $C^\infty$-regular set, we set $\Sigma= \pa E$. We have that $\Sigma$ is an embedded manifold of $\R^3$ and carries a natural Riemannian structure $g$ induced by the Euclidean metric of $\mathbb{R}^3$. We will always identify the tangent plane $T_p \pa E \subset \R^3$ with the tangent space $T_p \Sigma$ as defined in differential Geometry. We denote by $\{E_1,E_2\}$ a base of $T_p \Sigma$ for $p \in \Sigma$. Given two vector $X_p, Y_p \in T_p \Sigma$ we denote their inner product in $T_p \Sigma$ by $ g( X_p, Y_p )$. In local coordinates, we have $ g( X_p, Y_p ) = \sum_{i,j=1}^2g_{ij}(p)X_p^iY_p^j$. Throughout the paper, we also use the notation $\langle \cdot, \cdot \rangle_p$ if there isn't ambiguity. We denote by $g^{ij}(p)$ the inverse of the matrix $g_{ij}(p)$, i.e., $ \sum_{j=1}^2g_{ij}(p) g^{jk}(p)= \delta_{ik}$. Since $\pa E$ is a $C^\infty$-regular surface, for every $x \in \partial E$ there exist a neighborhood $U \subset \mathbb{R}^3$ of $x$ and $r > 0$ such that $\partial E \cap U$ coincides, up to rigid motions of $\mathbb{R}^3$, with the graph of a smooth function, i.e. $ \{ (y,f(y)): y \in B'_r(x')\}$. In these coordinates, the metric coefficients are given by 
$$ g_{ij}((\cdot,f(\cdot)))= \delta_{ij}+ \frac{\pa f }{\pa y_i}  \frac{\pa f }{\pa y_j}, \quad g^{ij}((\cdot,f(\cdot)))= \delta_{ij}- \frac{1}{1+ \vert \nabla' f(\cdot) \vert^2}\frac{\pa f }{\pa y_i}  \frac{\pa f }{\pa y_j}.$$
  We denote by $\mathfrak{X}(\Sigma)$ the space of smooth vector fields on $\Sigma$ and by $\mathfrak{X}^*(\Sigma)$ the space of smooth $1$-forms. Given $h, k \in \N$, we denote by $T^{(h,k)}(T\Sigma)$ the space of $(h,k)$-tensor fields, i.e., $F \in T^{(h,k)}(T\Sigma)$ is a multilinear map
 \begin{equation}
     F :\underbrace{\mathfrak{X}^*(\Sigma) \times \cdots \times \mathfrak{X}^*(\Sigma)}_{h \text{ factors}} \times \underbrace{\mathfrak{X}(\Sigma) \times \cdots \times \mathfrak{X}(\Sigma)}_{k \text{ factors}} \rightarrow C^\infty(\Sigma).
 \end{equation}
 We extend the inner product in a natural way for tensors.  We denote the Levi-Civita connection on $\Sigma$ by $\overline{\nabla}$. It is a map $\overline{\nabla}: \mathfrak{X}(\Sigma) \times \mathfrak{X}(\Sigma) \rightarrow \mathfrak{X}(\Sigma)$, defined by $\overline{\nabla}_X Y:= P_{\pa E} \nabla_{\tilde{X}} \tilde{Y} = P_{\pa E} \sum_{k=1}^3 \sum_{i=1}^3 \tilde{X}_i \frac{\pa }{\pa x_i}(\tilde{Y}_k)e_k$ where $\tilde{X}= (\tilde{X}_1, \tilde{X}_2,\tilde{X}_3) $ and $\tilde{Y}= (\tilde{Y}_1, \tilde{Y}_2,\tilde{Y}_3) $ are extensions of $X$ and $ Y$ to $\R^3$ and $P_{\pa E}$ is the  projection onto $T_p \pa E$ as defined in~\eqref{LAPROIEZmatrix}. We recall that for $p \in \Sigma$ and $X,Y \in \mathfrak{X}(\Sigma)$, the value $\overline{\nabla}_X Y(p)$ depends only on $X(p)$ and on the restriction of $Y$ along any curve
   $c:(-\varepsilon,\varepsilon) \rightarrow \Sigma$ such that $c(0)=p$ and $ \frac{d}{d t} c(0)= X (p)$. In local coordinates near $p \in \Sigma$, the Levi-Civita connection can be expressed as
   \begin{equation}
       \overline{\nabla}_X Y= \big( X^i E_i(Y^k)+ X^i Y^j \Gamma_{ij}^k  \big) E_k,
   \end{equation}
   where $\Gamma^{i}_{jk}$ are the Christoffel symbols, which in a given local chart and frame can be computed from $g_{ij}, g^{ij}, E_k(g_{ij})$ and by the coefficients of the Lie brackets $[E_i,E_j]$ in the frame.  Moreover, for all $X \in \mathfrak{X}(\Sigma)$, we set
 $
     X^k_{\,;j}:= E_j(X^k)+\Gamma^k_{lj} X^l,
 $
 and hence obtain
$
     \overline{\nabla}_{E_j} X= X^k_{\,;j} E_k
 $.
Given a function $f \in C^\infty(\Sigma)$, we define 
\begin{equation}
    \overline{\nabla}f (X):= \overline{\nabla}_X f=  X (f) \text{ for every } X \in \mathfrak{X}(\Sigma),
\end{equation}
that is, $\overline{\nabla}f \in \mathfrak{X}^*(\Sigma)$. We denote the components of the covariant derivative of $f$  in local chart  by $ \overline{\nabla}_{i} f:= \overline{\nabla}_{E_i} f$.
Given $f \in C^\infty(\Sigma)$, we denote by ${\rm grad}\, f$ the intrinsic gradient on the manifold. In local coordinates it is given by
 \begin{equation}
     {\rm grad\,} f= \sum_{j,i=1}^2(g^{ij} \overline{\nabla}_{E_j} f )E_i.
 \end{equation} 
We denote its components by
 \begin{equation}
     \nabla^i f:= \sum_{j=1}^2g^{ij} \overline{\nabla}_{j} f \text{ for } i=1,2.
 \end{equation}
 Thanks to the identification of $T_p \Sigma$ with $ T_p \pa E \subset \R^3 $, it is natural to identify ${\rm grad} \, f$ with a vector of $\R^3$. Hence,
 \begin{equation}\label{nabladeE=grad}
    \nabla_{ \pa E } f = {\rm grad \,} f  . 
 \end{equation}
Let $X \in \mathfrak{X}(\Sigma)$, we define the divergence on $\Sigma$ by 
\begin{equation}\label{Sigmadiv07}
    \div_\Sigma X:= \sum_{i=1}^2 X^i_{\, ;i}= {\rm tr} (\overline{\nabla} X).
\end{equation}
\begin{remark}\label{08remakino}
{\rm
    Given $X \in \mathfrak{X}(\Sigma)$, there exists a unique function $\hat{X}: \Sigma \rightarrow \R^3$ such that $\hat{X}(p)\cdot\nu_E(p)=0$ for all $p \in \Sigma$, i.e. $ \hat{X}(p) \in T_p \pa E$ and
    \begin{equation}\label{05032026form1}
        \overline{\nabla}_X f = \nabla_{\pa E} f \cdot \hat{X} = g({\rm grad }f, X) \quad \forall f \in C^\infty(\R^3).
    \end{equation} \fr}
\end{remark} 
\begin{remark}
{\rm 
Let $f \in C^\infty(\R^3)$ and let $X=(X_1,X_2,X_3)$ be a smooth vector field on $\R^3$. Recall that the Euclidean gradient corresponds to the covariant derivative in $\R^3$ and satisfies
\begin{equation}
    \nabla f(X):= X(f)= \sum_{i=1}^3 X_i \frac{\pa }{\pa x_i} f =\nabla f \cdot X.
\end{equation}
Let $\nu_E$ be the unit normal vector of $\pa E$. We define $ N_E:= \nu_E \circ \pi_{\pa E}$, which provides a local extension of $\nu_E$. Since $\vert N_E \vert=1$, we obtain
   \begin{equation}
       0= \nabla \langle N_E, N_E \rangle (X)= 2 \langle \nabla_X N_E, N_E \rangle \text{ for every vector field }X \text{ of }\R^3,
   \end{equation}
   and hence $$ \nabla_X N_E=  P_{\pa E} \nabla_X N_E= B_E X.$$ Therefore, with a slight abuse of notation, we set
   \begin{equation}
       \overline{\nabla} _X\nu_E := \nabla_{\hat{X}} N_E= B_E \hat{X} \text{ for every }X \in \mathfrak{X}(\Sigma).
   \end{equation} \fr}
   \end{remark}
\begin{remark}{\rm 
    Let $f \in C^\infty(\R^3)$, $X \in \mathfrak{X}(\Sigma)$, and let $\hat{X}$ be the extension introduced in Remark~\ref{08remakino}. Using the divergence theorem and~\eqref{05032026form1}, we obtain
    \begin{equation}
        \int_{\pa E} f \div_{\pa E} \hat{X}\, d \mathcal{H}^2= -\int_{\pa E} \nabla_{\pa E}f \cdot \hat{X}\, d \mathcal{H}^2= -\int_{\pa E} g({\rm grad}\, f, X)\, d \mathcal{H}^2= \int_{\pa E} f \div_\Sigma X \, d \mathcal{H}^2.
    \end{equation}
   Consequently, $\div_{\pa E} \hat{X}= \div_\Sigma X$.  Moreover, the Laplace–Beltrami operator defined in~\eqref{LaplaceBeltraimiop} satisfies \begin{equation}
        \Delta_{\pa E} f= \div_{\pa E} \nabla_{\pa E} f = \div_{\Sigma} \, {\rm grad} \,f .
    \end{equation}
    This allows us to adopt a unified notation for the gradient (and, respectively, the divergence) on $\Sigma$, namely $\nabla_{\pa E}$ (respectively $\div_{\pa E}$). \fr}
\end{remark}   
  We denote the covariant derivative of a smooth tensor field $T \in T^{(h,k)}(T\Sigma)$ by $\overline{\nabla}T$, where $\overline{\nabla}T \in T^{(h,k+1)}(T\Sigma) $; for the definition, see~\cite{LeeBook}. We also recall that the $m$-th covariant derivative of a tensor field $T \in  T^{(h,k)}(T\Sigma)$ is $ \overline{\nabla}^{m} T  \in T^{(h,k+m)}(T\Sigma) $. Since $\Sigma$ is embedded, we denote its second fundamental form by $B_\Sigma$. Note that, for every $x \in \Sigma$, $B_\Sigma(x)$ is a bilinear form on $T_x \Sigma \times T_x \Sigma$. We recall that the Riemann curvature tensor of $(\Sigma,g)$ is a $(1,3)$-tensor field defined by
\begin{equation}\label{ILRiemann}
\begin{split}
   & R: \mathfrak{X}(\Sigma) \times \mathfrak{X}(\Sigma) \times \mathfrak{X}(\Sigma) \rightarrow \mathfrak{X}(\Sigma) \\
   & R(X,Y)Z= \overline{\nabla}_X \overline{\nabla}_Y Z - \overline{\nabla}_Y \overline{\nabla}_X Z - \overline{\nabla}_{[X,
Y]}Z.
    \end{split}
\end{equation}
In local coordinates, the Riemann curvature tensor is given by
\begin{equation}\label{ILGRANDERIEMANN}
    \begin{split}
        & R(E_i,E_j)E_k = R_{ijk}^l E_l\\
        & R_{ijk}^l = \overline{\nabla}_{E_i} \Gamma_{jk}^l -\overline{\nabla}_{E_j} \Gamma_{ik}^l + \Gamma_{jk}^m\Gamma_{im}^l- \Gamma_{ik}^m \Gamma_{jm}^l,
    \end{split}
\end{equation}
where  we have used the Einstein summation convention. We recall that the Riemann curvature tensor can be expressed in terms of the second fundamental form; see~\cite[Section 2]{Mantegazza2002} or~\cite{MMT2013, Mantegazzabook}.
 Consequently, an $L^\infty$ bound on $B_\Sigma$ or on $\overline{\nabla}^k B_\Sigma$, yields an $L^\infty$ bound on $R$ and on $\overline{\nabla}^k R$, respectively, and vice-versa.
\subsection{Space of functions} 
We define  the Sobolev space $W^{l,p}(\Sigma)$, for $p \in [1,\infty]$ and $k \in \N$, as in the book~\cite{AubinBook2}. Given $f \in W^{l,p}(\Sigma)$, we define  for $p < \infty$
\begin{equation}
    \| f\|_{W^{l,p}(\Sigma)}^p := \sum_{k = 0}^l \int_\Sigma |\overline{ \nabla}^k f|^p\, d \mathcal{H}^{n-1}
\end{equation}
and for $p = \infty$
\[
\| f\|_{W^{l,\infty}(\Sigma)} := \sum_{k = 0}^l \text{ess}\sup_{\!\!\!\!\!\!\!\!\!\!\!x \in \Sigma} |\overline{ \nabla}^k f|.
\]
We set $H^l(\Sigma) := W^{l,2}(\Sigma)$.
The above definition extends naturally to tensor fields. We set $\| u\|_{C^{m}(\Sigma)} := \| u\|_{W^{m,\infty}(\Sigma)}$ for $m \geq 0$. We remark that, in order to define the space $W^{k,p}(\Sigma)$ for $k \geq 2$, it is sufficient that $\Sigma$ is $C^{k-1,1}$-regular. Moreover, we note that the first-order Sobolev space for real-valued functions can be equivalently defined using tangential derivatives; we adopt this point of view throughout the paper. However, for higher-order Sobolev spaces and Sobolev spaces of tensor fields, it is necessary to work with covariant derivatives. We define the H\"older norm of a continuous function $u : \Sigma \to \R$ by 
\[
\| u\|_{C^\alpha(\Sigma)} =  \sup_{\substack{x \neq y \\ x,y \in \Sigma}} \frac{|u(y) - u(x)|}{d(y,x)^\alpha} + \|u\|_{L^\infty(\Sigma)},
\]
where $d(y,x)$ denotes the geodesic distance between $x$ and $y$ on $\Sigma$.  
Finally, following~\cite{JL2024}, we define the H\"older norm of a covariant tensor field $T$ of order $k$ by
\[
\| T\|_{C^\alpha(\Sigma)} = \sup \{ \| T(X_1, \dots, X_k) \|_{C^\alpha(\Sigma)} : X_i \in \mathfrak{X}(\Sigma) \, \text{ with } \, \|X_i \|_{C^1(\Sigma)} \leq 1\}.
\]
\begin{definition}\label{Omegaevrietadiriferimento}
 We denote by $\Omega \Subset \R^3$ an open and connected set with $C^\infty$-regular boundary. 
   Moreover we denote by $ E_0 \Subset \Omega$ an open and connected set such that $\pa E_0$ is uniformly $C^{6,1}$-regular. We shall use $E_0$ as the initial datum.    
\end{definition}
	Let $\sigma_0$ be a positive constant such that
	\begin{equation}\label{Laconstantesigmazerofin}
		\sigma_0 < \frac{1}{2}\min\{1/\| B_{E_0} \|_{L^\infty(\pa E_0)}, d_{\mathcal{H}}(\pa E_0, \pa \Omega)\}.
	\end{equation}
	Given $k \in \N$, $ \alpha \in [0,1]$,   $\sigma< \sigma_0$ and $K>0$  we define
	\begin{align}\label{22082024pom2}
		\mathfrak{C}^{k,\alpha}_{K,\sigma}( E_0):=
		\big\{E\subset\R^3:& \,\, \pa E  \subset {\rm cl} \big( \mathcal{I}_{\sigma}(\pa E_0) \big),\,\pa E=\{y+\varphi_E(y)\nu_{ E_0}(y)\!:\,y\in \pa E_0 \},\cr
		& \,\,\|\varphi_E\|_{L^\infty(\pa E_0)}\leq \sigma,\,\|\varphi_E\|_{C^{k,\alpha}(\pa E_0)}\leq K\big\}.
	\end{align}
	For every $k \in \{1,\cdots,6\}$, we define the set $\mathfrak{H}^{k}_{K,\sigma}(E_0)$ analogously to $ \mathfrak{C}_{K,\sigma}^{k,\alpha}(E_0)$ by replacing the norm $\|\varphi_E\|_{C^{k,\alpha}(\pa E_0)}$ with  $\|\varphi_E\|_{H^k(\pa E_0)}$. Let $\{ E_n \}_{n \in \N}$ and $ E$ be such that $ E_n \in \mathfrak{C}_{K,\sigma}^{k,\alpha}(E_0)$ (respectively $\mathfrak{H}^{k}_{K,\sigma}(E_0)$) for all $n \in \N$. We say that $E_n \rightarrow E $ in $\mathfrak{C}_{K,\sigma_0}^{k,\alpha}(E_0)$ (respectively in $\mathfrak{H}_{K,\sigma_0}^{k}(E_0)$) if $\varphi_{E_n}$ is uniformly bounded by $\sigma$ in $L^{\infty}(\pa E_0)$ and it is a Cauchy sequence in $C^{k,\alpha}(\pa E_0)$ (respectively $\mathfrak{H}^{k}(\pa E_0)$).\\
	Let $K \in (0, + \infty)$, $m \in \N$, $\alpha \in [0,1]$, we define
\begin{equation}
    \mathfrak{C}^{m,\alpha}_K (\Omega):= \left\{ u: \Omega \rightarrow \R :  u \in C^{m,\alpha}(\Omega) \text{ and } \| u \|_{C^{m,\alpha}(\Omega)} \leq K   \right\}.
\end{equation}
	\subsection{Sets of finite anisotropic perimeter and anisotropic curvature}  For $E \subset \R^3$ a Borel set, we define the anisotropic perimeter of $E$ as 
	\begin{equation}
		P_\varphi(E)= \sup \left\{ \int_{E} \div T \, dx : T \in C^1_c(\R^3,\R^3),\, \sup_{x \in \R^3} \varphi^0(T(x)) \leq 1   \right\}.
	\end{equation}
	In the particular case $\varphi(\cdot)= \vert \cdot \vert$, i.e. the Euclidean norm, we simply write $P(E)$ instead of $P_{\vert \,\, \vert }(\cdot)$. We say that a Borel set $E \subset \R^3$ has finite perimeter if $P(E) < + \infty$. Under the assumption \eqref{unifell}, it is straightforward to check that $$P_{\varphi}(E)< + \infty \iff P(E)< +\infty. $$  
	For every set $E  \subset \R^3$ with finite perimeter, we denote by $\partial^* E \subset \R^3$ its reduced boundary and by $\nu_{E}: \partial^* E \rightarrow \mathbb{R}^3$ the measure-theoretic outer unit normal vector field (see, for instance,~\cite[Definition 3.53]{AmbFuscPall}). By De Giorgi's structure theorem (see, for instance,~\cite[Theorem 3.59]{AmbFuscPall},~\cite[Theorem 15.19]{Maggibook}), if $E \subset \R^3$ has finite perimeter, then $$P(E)=\mathcal{H}^{2}(\partial^*E) \quad \text{ and }\quad P_\varphi(E):= \int_{\pa^* E} \varphi(\nu_E) \, d \mathcal{H}^2.$$ 
    
	Let $E \subset \R^3$ be a $C^2$-regular set. For any vector field $X \in C^1_c(\R^3,\R^3)$, let $(\Phi (t,\cdot))_{t \in (-\varepsilon,\varepsilon)}$ be the unique solution of the Cauchy problem
	\begin{equation}
		\begin{cases}
			\frac{\pa }{\pa t} \Phi (t,x)= X\circ \Phi(t,x) \quad \forall x \in \R^3\\
			\Phi(0,x)=x \quad \forall x \in \R^3. 
		\end{cases}
	\end{equation}
	Then, we have 
	\begin{equation}\label{FeulPANI}
		\frac{d}{d t} \bigg|_{t=0} \int_{\pa \Phi(t,E)} \varphi (\nu_{\Phi(t,E)})\, d \mathcal{H}^2 = \int_{\pa E} H_E^{\varphi} X \cdot \nu_E \, d \mathcal{H}^2
	\end{equation}
	where the anisotropic curvature $H^{\varphi}_E$ of $ \pa E$ is defined by $$H_E^{\varphi}:= \div_{\pa E} (\nabla \varphi (\nu_E)).$$ 
\section{Preliminary results}\label{sez318052026}
We start by recalling the interpolation inequalities for Sobolev norms on embedded surfaces. We use the result from~\cite[Proposition 6.5]{Mantegazza2002} (see also~\cite[Proposition 4.3]{DFM2023}), which states that, under a curvature bound, the standard interpolation inequalities hold with a uniform constant.
\begin{proposition}
\label{prop:interpolation}
Let $C_0>0$ and let $\Sigma$ be a $C^{m}$-regular manifold with $m \geq 2$. Assume $\|B_{\Sigma}\|_{L^{\infty}}, \mathcal{H}^{2}(\Sigma)\leq C_0$.  Then for $0\leq k < l \leq m$, $p \in [1,\infty)$, and  $q,r \in [1,\infty]$   there exists $\theta \in [k/l,1]$ such that for every $C^l$-regular  tensor field $T$ on $\Sigma$ it holds
\[
\|\overline{\nabla}^k T\|_{L^p(\Sigma)} \leq C \| T\|_{W^{l,q}(\Sigma)}^\theta \| T\|_{L^{r}(\Sigma)}^{1-\theta}
\]
for a constant $C=C(k,l,p,q,r,\theta,C_0)>0$,  provided that the following  condition is satisfied
\[
\frac{1}{p} = \frac{k}{2} + \theta \left( \frac{1}{q} - \frac{l}{2} \right) + \frac{1}{r}(1 - \theta).
\]
Moreover, if $f: \Sigma \to \R$ is a smooth function with $\int_{\Sigma} f\,d\mathcal{H}^{2}=0$, the above inequality can be written as 
\[
\|\overline{\nabla}^k f\|_{L^p(\Sigma)} \leq C \| \overline{\nabla}^l f\|_{L^{q}(\Sigma)}^\theta \| f\|_{L^{r}(\Sigma)}^{1-\theta}.
\]
\end{proposition} 
Let $E \subset \R^3$ be a bounded and uniformly $C^{k,\alpha}$-regular set for some $k \in \N$, $ \alpha \in [0,1]$, with constants $C_0,r_0$. We set $\Sigma:= \pa E$. Then the classical interpolation inequality in H\"older norms holds:
 for $ 0 < \beta < \alpha \leq 1 $ and $0 \leq l \leq m \leq k$ we have
\begin{equation}\label{interHOLDER}
    \| f \|_{C^{l,\beta}(\Sigma)} \leq C \| f \|_{C^{m,\alpha}(\Sigma)}^{\theta} \|   f \|_{C^0(\Sigma)}^{1-\theta}, \quad \theta= \frac{l+ \beta}{m+ \alpha},
\end{equation}
where $C$ depends on $C_0,r_0,l,m,\alpha,\beta$. This result follows from the Euclidean case; see for example~\cite[Example 1.9]{Lunardi2018}. 

The interpolation inequality in Proposition~\ref{prop:interpolation} yields the following useful estimate. The proof is standard and we refer the reader to~\cite[Proposition 2.3]{JN}.   Before stating the result, we fix the following notation: given an index vector $\alpha \in \N^l$, we denote the sum of its components by $\vert \alpha \vert= \alpha_1+ \cdots+ \alpha_l.$

\begin{lemma}
\label{lem:Leibniz}
Let $C_0>0$, let $\Sigma$ be a $C^{m}$-regular manifold with $m \geq 2$, and let $T_1, \dots, T_l$ be $C^m$-regular covariant tensor fields. Assume $\|B_{\Sigma}\|_{L^{\infty}}, \mathcal{H}^{2}(\Sigma)\leq C_0$. Then for an index vector $\alpha \in \N^l$ with  $|\alpha|  \leq  k \leq m$ it holds 
\[
\| |\overline{\nabla}^{\alpha_1} T_1| \cdots  |\overline{\nabla}^{\alpha_l} T_l| \|_{L^2(\Sigma)} \leq C \sum_{\sigma \in S_l} \|T_{\sigma(1)}\|_{L^\infty(\Sigma)} \cdots  \|T_{\sigma(l-1)}\|_{L^\infty(\Sigma)}  \,    \|T_{\sigma(l)}\|_{H^k(\Sigma)}, 
\]
where the sum is over the permutation group $S_l$.  In particular, 
\[
\| |\overline{\nabla}^{k} \langle T_1 , T_2\rangle  | \|_{L^2(\Sigma)} \leq C \|T_1\|_{L^\infty(\Sigma)}   \|T_2\|_{H^k(\Sigma)}    + C \|T_2\|_{L^\infty(\Sigma)}   \|T_1\|_{H^k(\Sigma)}  . 
\] 
\end{lemma}
We recall the following lemma; the proof can be found in~\cite[Lemma~2.5 and Proposition~2.6]{JN}.
\begin{lemma}
\label{lem:hilbert-norm}
Let $C>0$ and let $E \Subset \R^3 $ be such that  $ \pa E $ is $C^{2k+1}$-regular with $ \mathcal{H}^{2}(\pa E), \|B_{ E}\|_{L^\infty} \leq C$. We set $\Sigma= \pa E$. Then for all $f \in C^{2k}(\Sigma)$ it holds
\[
\begin{split}
&\|f\|_{H^{2k}(\Sigma)} \leq C_k\big( \|\Delta^k_{\pa E} f\|_{L^{2}(\pa E)} +  (1+  \| \nabla_{\pa E} \Delta^{k-1}_{\pa E}  H_{ E} \|_{L^{2}(\pa E)})  \|f\|_{L^{\infty}(\pa E)}\big),\\
&\|f\|_{H^{2k-1}(\Sigma)} \leq C_k\big( \| \nabla_{\pa E} \Delta^{k-1}_{\pa E} f\|_{L^{2}(\pa E)} +  (1+ \| \Delta^{k-1}_{\pa E} H_{ E} \|_{L^{2}(\pa E)}  )  \|f\|_{L^{\infty}(\pa E)}\big).
\end{split}
\]

Moreover, it holds 
\[
\begin{split}
&\|B_\Sigma \|_{H^{2k-2}(\Sigma)} \leq C_k (1 + \| \Delta^{k-1}_{\pa E} H_{ E} \|_{L^{2}(\pa E)} ), \\
&\|B_\Sigma \|_{H^{2k-1}(\Sigma)} \leq C_k (1 + \| \nabla_{\pa E} \Delta^{k-1}_{\pa E}  H_{ E} \|_{L^{2}(\pa E)} ).
\end{split}
\]
Here the constant $C_k$ depends only on $k$ and $C$.
\end{lemma}
The estimates for $\|f\|_{H^{2k}(\Sigma)}$ in Lemma~\ref{lem:hilbert-norm} are sharp with respect to the norm on the curvature. 

We now state a technical lemma; the proof relies on classical arguments and is obtained by combining~\cite[Lemma 4.8]{LeeBook} and~\cite[Lemma 2.4]{JN}.
   \begin{lemma}\label{pensobene}
      Let $E \subset \R^3$ be a bounded $C^{k+1}$-regular set. Let $A$ be a $l$-covariant tensor field on $\R^3$, i.e., for every $p \in \R^3$ we have $A(p)$ is a $l$-covariant tensor on $\R^3$ and $p \mapsto A(p)$ is $C^\infty$. Let $\Sigma= \pa E$ and  $p \in \Sigma \mapsto A(\nu_E(p)) $, then $A(\nu_E(\cdot))$ is a $l$-covariant tensor field on $\Sigma$ of class $C^{k}$ and for every $1 \leq m \leq k$
      \begin{equation}
          \vert \overline{\nabla}^m A(\nu_E) \vert \leq C_m \sum_{ 1 \leq \vert \alpha \vert \leq m} \big(  \vert \overline{\nabla}^{\alpha_1} B_\Sigma \vert + \dots + \vert \overline{\nabla}^{\alpha_{m-1}} B_\Sigma \vert  \big)  \vert \nabla^{\alpha_{m}} A \vert(\nu_E).
      \end{equation}
   \end{lemma}
   
We now recall the definition of a normal graph.
	\begin{definition}\label{ngrafico}
		Let $E \subset \R^3$ be a bounded $C^2$-regular set that satisfies the \textsc{UBC} with radius $r$.  Let $F \subset \R^3$ be an open bounded set. We say that $\pa F$ is a normal graph over $\pa E$ if there exists a function $\psi: \pa E \rightarrow (-r,r)$, called the height function, such that 
		\begin{equation}\label{carloiolego}
			\pa F= \{ x + \psi(x)\nu_{E}(x) : x \in \pa E\} \text{ and } E \Delta F \subset {\rm cl} \big( \mathcal{I}_r(\pa E) \big).
		\end{equation}
	\end{definition}
  The following lemma gives an expansion of the anisotropic mean curvature with explicit dependence on the reference manifold. Before stating the lemma, we recall the following remark, which will be used in its proof.
   \begin{remark}{\rm 
      Let $E \subset \R^3$ be a $C^3$-regular set that satisfies the \textsc{UBC} with radius $r$ and let $F \subset \R^3$ be a $C^2$-regular set such that $\pa F$ is the normal graph over $\pa E$ with height function $f \in C^1(\pa E)$. Since the function $f \circ \pi_{\pa E} : \mathcal{I}_r(\pa E) \to \R$ is well-defined on $\pa F$, we can compute the tangential gradient of $f \circ \pi_{\pa E} : \pa F \to \R$   at any point $y \in \pa F$  using formula~\eqref{1012pcrossfit} and obtain
\begin{equation}
\label{regdecat12}
\nabla_{\pa F} (f \circ \pi_{\pa E})(y) =  (\nabla_{\pa F} \pi_{\pa E}(y))^T \nabla_{\pa E}f(\pi_{\pa E}(y)).
\end{equation} \fr}
   \end{remark}
\begin{lemma}\label{10Expcurvvarphi}
   Let $E \subset \mathbb{R}^3$ be a $C^3$-regular set satisfying the \textsc{UBC} with radius $r$, and let $F \subset \mathbb{R}^3$ be a $C^2$-regular set. Assume that $\partial F$ is the normal graph over $\partial E$ with height function $\psi \in C^2(\partial E)$ such that $\|\psi\|_{L^\infty(\partial E)} \leq \frac{r}{4}$.  Then, for all $x \in \Sigma= \pa E$, it holds that
    \begin{equation}\label{02expphicurv1}
        H_{F}^\varphi(x+ \psi(x)\nu_E(x))= -{\rm tr}\big((\nabla^2 \varphi )(\nu_E) \nabla_{\pa E}^2 \psi\big)(x)+ H_{E}^\varphi(x)+\mathcal{R}_0(x). 
    \end{equation}
    The error term is given by
    \begin{equation}\label{battiato1}
        \mathcal{R}_0= \langle \mathcal{A}_1(\psi B_\Sigma, \overline{\nabla} \psi, \nu_E), \overline{\nabla}^2 \psi \rangle+ \langle \mathcal{A}_2(\psi B_\Sigma, \overline{ \nabla} \psi, \nu_E), \overline{\nabla}(\psi B_\Sigma) \rangle+a_0(\psi, \overline{\nabla} \psi, B_\Sigma, \nu_E), 
    \end{equation}
    where $\mathcal{A}_1, \, \mathcal{A}_2$ are smooth tensor fields satisfying $\mathcal{A}_1(0,0,\cdot)=0$, $\mathcal{A}_2(0,0,\cdot)=0$, and $a_0$ is a smooth function satisfying  $a_0(0,0,\cdot,\cdot)=0$. 
\end{lemma}
\begin{proof}
Let $\Psi: \pa E \rightarrow \pa F$ be defined by $ \Psi(x):= x+ \psi(x)\nu_E(x)$. We observe that $ \Psi^{-1}= \pi_{\pa E} \big|_{\pa F}: \pa F \rightarrow \pa E $.   Let $\tau \in T_x \pa E$, then we have $ \nabla_{\pa E} \Psi(x) \tau \in T_{\Psi(x)} \pa F$ and
\begin{equation}\label{09022026primacorss2}
    \nabla_{\pa E} \Psi (x)\tau = \big( I+\psi(x)B_E(x) \big) \tau+ \big( \nabla_{\pa E} \psi(x)  \cdot \tau \big) \nu_E(x).
\end{equation}
We define $ N: \pa E \rightarrow \R^3$ by
\begin{equation}\label{funzN}
    N(x):= -(I+ \psi(x)B_E(x))^{-1} \nabla_{\pa E} \psi(x)+ \nu_E(x).
\end{equation}
It holds that
\begin{equation}
    \big(  \nabla_{\pa E} \Psi(x)\tau \big) \cdot N(x)=0 \quad \forall \tau \in T_x \pa E.
\end{equation}
Using this and the fact that $\pa F= \{ x + \psi (x)\nu_E(x): x \in \pa E\}$, we obtain 
\begin{equation}\label{funznu_F}
     \nu_F(y)= \frac{ N(\pi_{\pa E} (y))}{\vert N(\pi_{\pa E} (y)) \vert } \quad \forall y \in \pa F.
\end{equation}
We observe that
\begin{equation}\label{vertNsvil}
    \vert N (x)\vert= 1+ a_1(\psi B_E, \nabla_{\pa E} \psi)
\end{equation}
where $a_1$ is a smooth function satisfying $a_1(0,0)=0$.
Recalling \eqref{LAPROIEZmatrix}, using formulas~\eqref{funzN},~\eqref{funznu_F} and~\eqref{vertNsvil}, for $x= \pi_{\pa E}(y)$, we obtain
\begin{equation}\label{02122025form1}
    P_{\pa F}(y)- P_{\pa E}(x)= -\hat{A}_1 \otimes \hat{A}_1+ \hat{A}_2 \otimes \nu_E(x)+ \nu_E(x) \otimes \hat{A}_2 + a_2 \nu_E(x) \otimes \nu_E(x)
\end{equation}
where 
\begin{equation}\label{06052026form2}
\hat{A}_1 = A_1(\psi B_E, \nabla_{\pa E} \psi) \circ \pi_{\pa E},\quad \hat{A}_2 = A_2(\psi B_E, \nabla_{\pa E} \psi) \circ \pi_{\pa E}    
\end{equation}
with
$A_1(\cdot, \cdot),\,A_2(\cdot, \cdot)$ smooth tangent vector fields satisfying $A_1(0,0)=0,\, A_2(0,0)=0$, and $a_2(\cdot, \cdot)$ is a smooth function such that $a_2(0,0)=0$.

We recall that $\nabla \varphi$ is a $0$-homogeneous function. Hence,
we get  
$$H_F^\varphi(y)=\div_{\pa F}\big( \nabla \varphi (\nu_F)\big)(y)= \div_{\pa F} \big( \nabla \varphi (N) \circ \pi_{\pa E} \big)(y)\text{ for all } y \in \pa F.$$
We set $T_1(y):= \nabla \varphi(\nu_E) \circ \pi_{\pa E}(y)$ and $T_2(y):= \nabla \varphi (N)\circ \pi_{\pa E}(y)- \nabla \varphi(\nu_E) \circ \pi_{\pa E}(y) $ for all $ y \in \pa F$, then
\begin{equation}
    H_F^\varphi(y)= \div_{\pa F} T_1(y)+ \div_{\pa F}T_2(y) \text{ for all } y \in \pa F.
\end{equation}
\textit{Claim 1:} It holds
\begin{equation}\label{0212tesic1}
 \div_{\pa F} T_1(y)= H_E^\varphi \circ \pi_{\pa E}(y)+ a_3 (\psi, \nabla_{\pa E} \psi, B_E, \nu_E) \circ \pi_{\pa E}(y) \text{ for all } y \in \pa F,
\end{equation}
where $a_3(\cdot, \cdot, \cdot, \cdot)$ is a smooth function satisfying $a_3(0,0,\cdot,\cdot)=0$.

Let $x= \pi_{\pa E}(y)$ for $y \in \pa F$. By the very definition of $T_1$, we have
\begin{equation}\label{flongclaim1}
\begin{split}
    \div_{\pa F} T_1(y)&= \div_{\pa F} \big( \nabla \varphi (\nu_E) \circ \pi_{\pa E}(y)  \big)= \div_{\pa F} \big( \nabla \varphi (\nabla d_E)(y)  \big) \\
    &= {\rm tr} \big(\nabla_{\pa F} \nabla \varphi ( \nabla d_E) (y)  \big) = { \rm tr} \big( \nabla \big( \nabla \varphi (\nabla d_E) \big)(y) P_{\pa F}(y)  \big) \\
    &= {\rm tr} \big( \nabla \big( \nabla \varphi (\nu_E(x)) \big) P_{\pa E}(x)  \big)+ {\rm tr} \big(  \nabla \big( \nabla \varphi (\nabla d_E) \big)(y) (P_{\pa F}(y)- P_{\pa E}(x) )  \big)\\
    & = H_E^\varphi(x)+{\rm tr} \big(  \nabla \big( \nabla \varphi (\nabla d_E) \big)(y) (P_{\pa F}(y)- P_{\pa E}(x) )  \big) 
    \end{split}
\end{equation}
where we have used $ \nabla d_E(y)= \nu_E(x)$ and the definition of the tangential divergence of a vector field. Now, using formulas~\eqref{24112025primacrossfit},~\eqref{02122025form1} and the identity $\nabla \big(  \nabla \varphi (\nabla d_E) \big)= \nabla^2 \varphi (\nabla d_E) \nabla^2 d_E$, we obtain, for $y \in \pa F$,
\begin{equation}\label{0212025form2}
    {\rm tr} \big(  \nabla \big( \nabla \varphi (\nabla d_E) \big)(y) (P_{\pa F}(y)- P_{\pa E}(\pi_{\pa E}(y)) )  \big) = a_3 (\psi, \nabla_{\pa E} \psi, B_E, \nu_E) \circ \pi_{\pa E}(y),
\end{equation}
where $a_3$ is a smooth function satisfying $ a_3(0,0,\cdot,\cdot)=0$. Therefore, combining~\eqref{flongclaim1} and~\eqref{0212025form2}, we obtain~\eqref{0212tesic1}.\\
\textit{Claim 2:} The following identity holds for $y \in \partial F$:
\begin{equation}\label{LaforperT_2}
\begin{split}
    \div_{\pa F} T_2(y)=& -{\rm tr}\big((\nabla^2 \varphi )(\nu_E) \nabla_{\pa E}^2 \psi\big)\circ \pi_{\pa F}(y) + \langle B_1(\psi B_E, \nabla_{\pa E} \psi, \nu_E), \nabla_{\pa E}^2 \psi \rangle\circ \pi_{\pa F}(y)\\
    &   + \langle B_2(\psi B_E, \nabla_{\pa E} \psi, \nu_E), \nabla_{\pa E} (\psi B_E) \rangle \circ \pi_{\pa F}(y)\\
    &+a_4(\psi, \nabla_{\pa E} \psi, B_E, \nu_E)\circ \pi_{\pa F}(y),
    \end{split}
\end{equation}
where  $B_1(\cdot, \cdot, \cdot),\, B_2 (\cdot, \cdot, \cdot) $ are smooth tensors and  $a_4(\cdot, \cdot, \cdot, \cdot)$ is a smooth function such that $B_1(0,0,\cdot,\cdot)= B_2(0,0,\cdot,\cdot)=0$ and $a_4(0,0,\cdot,\cdot)=0$.

Fix $x= \pi_{\pa E}(y)$, where $y \in \pa F$. By the definition of $T_2$ and of $N$, we have 
\begin{equation}
\begin{split}
    T_2(y)&= \nabla \varphi (N)\circ \pi_{\pa E}(y)- \nabla \varphi(\nu_E) \circ \pi_{\pa E}(y)\\
    &= \nabla \varphi \big( -(I+ \psi B_E)^{-1} \nabla_{\pa E} \psi + \nu_E  \big) (x)-\nabla \varphi (\nu_E) (x).
    \end{split}
\end{equation}
Using formula~\eqref{regdecat12}, we obtain 
\begin{equation}\label{pepedicaienna}
    \nabla_{\pa F} T_2(y) = \nabla_{\pa E} \big[ \nabla \varphi \big( -(I+ \psi B_E)^{-1} \nabla_{\pa E} \psi + \nu_E  \big) - \nabla \varphi(\nu_E) \big] (x)\nabla_{\pa F} \pi_{\pa E}(y).
\end{equation}
By the chain rule, we have
\begin{equation}
\begin{split}
    &\nabla_{\pa E} \big[ \nabla \varphi \big( -(I+ \psi B_E)^{-1} \nabla_{\pa E} \psi + \nu_E  \big) - \nabla \varphi(\nu_E) \big]\\
    &\,\,= \nabla^2 \varphi \big(   -(I+\psi B_E)^{-1} \nabla_{\pa E} \psi+ \nu_E\big) \big[- \nabla_{\pa E} \big( (I+ \psi B_E)^{-1} \nabla_{\pa E} \psi \big)+ \nabla_{\pa E} \nu_E \big] \\
    & \quad  - \nabla^2 \varphi(\nu_E) \nabla_{\pa E} \nu_E\\
    & \,\, = \big[ \nabla^2 \varphi \big(   -(I+\psi B_E)^{-1} \nabla_{\pa E} \psi+ \nu_E\big) - \nabla^2 \varphi (\nu_E)  \big] \big[  \nabla_{\pa E} \nu_E - \nabla_{\pa E} \big( (I+ \psi B_E)^{-1} \nabla_{\pa E} \psi \big) \big]\\
    & \quad+ \nabla^2 \varphi(\nu_E) \nabla_{\pa E} \big(-(I+\psi B_E)^{-1}\nabla_{\pa E} \psi  \big).
    \end{split}
\end{equation}
Combining the above formula with~\eqref{pepedicaienna}, and recalling that $\nabla_{\pa F} \pi_{\pa E}= \nabla \pi_{\pa E} P_{\pa F}$, we deduce 
\begin{equation}
    \begin{split}
        &\div_{\pa F}T_2(y)={\rm tr} \big(\nabla_{\pa F} T_2(y) \big) \\
        &=  {\rm tr} \bigg\{\bigg[ \big[ \nabla^2 \varphi \big(   -(I+\psi B_E)^{-1} \nabla_{\pa E} \psi+ \nu_E\big) - \nabla^2 \varphi (\nu_E)  \big] (x)\nabla_{\pa E} \nu_E(x)  \\
        &\quad -\big[ \nabla^2 \varphi \big(   -(I+\psi B_E)^{-1} \nabla_{\pa E} \psi+ \nu_E\big) - \nabla^2 \varphi (\nu_E)  \big] (x) \nabla_{\pa E} \big( (I+ \psi B_E)^{-1} \nabla_{\pa E} \psi \big) (x) \\
        & \quad  + \nabla^2 \varphi(\nu_E) \nabla_{\pa E} \big(-(I+\psi B_E)^{-1}\nabla_{\pa E} \psi  \big) (x) \bigg] \nabla \pi_{\pa E}(y) P_{\pa F}(y) \bigg\}.
    \end{split}
\end{equation}
From the above expression, we are led to analyze the trace of three terms. We start by considering the trace of 
 $$ Z_1:= \big[ \nabla^2 \varphi \big(   -(I+\psi B_E)^{-1} \nabla_{\pa E} \psi+ \nu_E\big) - \nabla^2 \varphi (\nu_E)  \big]  \nabla_{\pa E} \nu_E.$$ 
Using formula~\eqref{02122025form1}, it is straightforward to verify that   
\begin{equation}\label{01022026ppale1}
\begin{split}
   {\rm tr} \left(  Z_1(x)   \nabla \pi_{\pa E}(y) P_{\pa F}(y)\right) &= {\rm tr} \left( Z_1(x) \nabla\pi_{\pa E}(y)(P_{\pa F}(y)-P_{\pa E}(x)+ P_{\pa E}(x))   \right)\\
   &= a_4(\psi, \nabla_{\pa E} \psi, B_{E}, \nu_E),
   \end{split}
\end{equation}
where $a_4$ is a smooth function satisfying $a_4(0,0,\cdot,\cdot)=0$. 

We compute the trace of the second term $$Z_2:=\big[ \nabla^2 \varphi \big(   -(I+\psi B_E)^{-1} \nabla_{\pa E} \psi+ \nu_E\big) - \nabla^2 \varphi (\nu_E)  \big]  \nabla_{\pa E} \big( (I+\psi B_E)^{-1} \nabla_{\pa E} \psi \big) .$$
Recalling that $(I+\psi B_E)^{-1}= \sum_{k=0}^\infty (-1)^k\big( \psi B_E\big)^k$, we obtain
\begin{equation}\label{01022026fomz}
    \nabla_{\pa E} \big( (I+\psi B_E)^{-1} \nabla_{\pa E} \psi \big)=  W(\psi B_E ,\nabla_{\pa E} \psi) \tilde{\star} \nabla_{\pa E}(\psi B_E)   + (I+ \psi B_E)^{-1} \nabla_{\pa E}^2 \psi  ,
\end{equation}
where $W$ is a smooth matrix-valued function satisfying $W(\cdot,0)=0$ and $\tilde{\star}$ denotes bilinear operation between a matrix and a vector producing a matrix.
Arguing as in the proof of~\eqref{01022026ppale1} and using~\eqref{01022026fomz}, we conclude that
\begin{equation}
\begin{split}
    {\rm tr} \left( Z_2(x) \nabla \pi_{\pa E}(y) P_{\pa F}(y)  \right)= & \langle \tilde{A}_1(\psi B_E, \nabla_{\pa E} \psi, \nu_E), \nabla_{\pa E}^2 \psi \rangle(x) \\
    & + \langle \tilde{A}_2(\psi B_E, \nabla_{\pa E} \psi, \nu_E), \nabla_{\pa E}(\psi B_E) \rangle(x),
    \end{split}
\end{equation}
where $\tilde{A}_1,\tilde{A}_2$ are smooth tensor-valued functions such that $\tilde{A}_1(0,0,\cdot)=0$ and $\tilde{A}_2(0,0,\cdot)=0$.

We consider the trace of the last term
\begin{equation}
    Z_3:=\nabla^2 \varphi(\nu_E) \nabla_{\pa E} \big(-(I+\psi B_E)^{-1}\nabla_{\pa E} \psi  \big) .
\end{equation}
Using similar arguments as before, we obtain
\begin{equation}
    \begin{split}
  {\rm tr} \left( Z_3(x) \nabla \pi_{\pa E}(y)P_{\pa F}(y) \right)  =& -{\rm tr}\big((\nabla^2 \varphi )(\nu_E) \nabla_{\pa E}^2 \psi\big)(x) +\langle \bar{A}_1(\psi B_E, \nabla_{\pa E} \psi, \nu_E), \nabla_{\pa E}^2 \psi \rangle(x) \\
    &  + \langle \bar{A}_2(\psi B_E, \nabla_{\pa E} \psi, \nu_E), \nabla_{\pa E}(\psi B_E) \rangle(x),
    \end{split}
\end{equation}
where $\bar{A}_1,\bar{A}_2$ are smooth tensor fields satisfying $\bar{A}_1(0,0,\cdot)=0$ and $\bar{A}_2(0,0,\cdot)=0$. 

Therefore, combining the above calculations we obtain \eqref{LaforperT_2}.

Finally, by Claim 1 and Claim 2, and using the identifications $\overline{\nabla} \equiv \nabla_{\pa E}$ and $B_E \equiv B_\Sigma$, we deduce~\eqref{02expphicurv1}. 
\end{proof}
The following two remarks yield an alternative representation of the anisotropic curvature $F$, which will be useful for our purposes.
\begin{remark}\label{battiato2}
{\rm
    Under the assumptions of the above lemma, we can rewrite  $ {\rm tr} (\nabla^2 \varphi(\nu_E) \nabla_{\pa E}^2 \psi)$  as a divergence operator plus an error term. We set
    \begin{equation}\label{derivsecvarphi}
        A(\nu_E):= \nabla^2 \varphi (\nu_E).
    \end{equation}
    Recalling that $ \nu_E \circ \pi_{\pa E} = \nabla d_E$ and that $\nabla_{\pa E} f(x)= \nabla f \circ \pi_{\pa E}(x)$ for all $x \in \pa E$, we obtain on $\pa E$  \begin{equation} 
    \begin{split}
        \div_{\pa E} \big( A(\nu_E) \nabla_{\pa E} \psi \big) &= {\rm tr} \big( \nabla_{\pa E} (A(\nu_E) \nabla_{\pa E} \psi) \big)= {\rm tr} \big( \nabla( A(\nabla d_E) \nabla_{\pa E} \psi \circ \pi_{\pa E} )\big)\\
        &={\rm tr } \big(  \nabla ( \sum_{i,j=1}^3 a_{ij}(\nabla d_E)(\nabla_{\pa E} \psi \circ \pi_{\pa E} \cdot e_j)e_i) \big)\\
        &= {\rm tr} \big( A(\nu_E) \nabla^2_{\pa E} \psi \big)+{\rm tr} \big( \sum_{i,j=1}^3 (\nabla_{\pa E} \psi \circ \pi_{\pa E} \cdot e_j) e_i \otimes \nabla^2 d_E\nabla a_{ij}(\nabla d_E)    \big)  \\
        & = {\rm tr} \big( A(\nu_E) \nabla^2_{\pa E} \psi \big)+ \langle\nabla_{\pa E} \psi, A_1(\nu_E, B_E) \rangle,
    \end{split}
    \end{equation}
    where $A_1$ is a smooth tensor field, and in the last line we have used formula~\eqref{24112025primacrossfit} on $\pa E$.
     Hence, using the identification of $\nabla_{\pa E} \psi$ with $\overline{\nabla}\psi$, we can write
     \begin{equation}\label{trindivform}
       {\rm tr} \big( A(\nu_E) \nabla^2_{\pa E} \psi\big)=   \div_{\pa E} \big( A(\nu_E) \nabla_{\pa E} \psi \big) - \langle\overline{\nabla} \psi, A_1(\nu_E, B_E) \rangle.
     \end{equation} \fr}
\end{remark}
\begin{remark}\label{battiato3}
{\rm
    We observe that the function $A(\nu_E)$ defined in~\eqref{derivsecvarphi} cannot be identified with a $(1,1)$-tensor filed. Indeed, it maps tangent vectors of $T_x \pa E$ into vectors in $\R^3$.
    We decompose $A(\nu_E)$ in order to apply classical tensor calculus. Specifically, we write
    \begin{equation}
    \begin{split}
        A(\nu_E)= (I- \nu_E \otimes \nu_E) A(\nu_E)+ \nu_E \otimes \nu_E A(\nu_E)
      \end{split}  
    \end{equation}
    where \begin{equation}
        (I- \nu_E \otimes \nu_E)A(\nu_E)(x)\big|_{T_x \pa E}: T_x \pa E \rightarrow T_x \pa E, \quad \nu_E \otimes \nu_E A(\nu_E)(x)\big|_{T_x \pa E}: T_x \pa E \rightarrow (T_x\pa E)^{\perp}.
    \end{equation}
    We set 
    \begin{equation}\label{07032026posdef}
        \mathcal{A}(\nu_E)(x):= (I- \nu_E \otimes \nu_E)A(\nu_E)(x) \text{ for all } x \in \pa E.
    \end{equation}
    Therefore, we have
    \begin{equation}\label{mammaroma5}
        \div_{\pa E} \big( A(\nu_E) \nabla_{\pa E} \psi  \big)= \div_{\pa E} \big( \mathcal{A}(\nu_E) \nabla_{\pa E} \psi \big)+ \div_{\pa E} \big( \nu_E \otimes \nu_E A(\nu_E) \nabla_{\pa E} \psi \big).
    \end{equation}

Using the symmetry of $A(\nu_E)$, we obtain for all $ v \in \R^3$
\begin{equation}
\begin{split}
    (\nu_E \otimes \nu_E ) A(\nu_E) v&= \nu_E \otimes \nu_E (A(\nu_E)v)= \nu_E (\nu_E \cdot A(\nu_E)v)\\
    &= \nu_E (A(\nu_E) \nu_E \cdot v)= (\nu_E \otimes A(\nu_E)\nu_E) v  .
    \end{split}
\end{equation}
Recalling that $\nabla_{\pa E}^2 \psi(x) : T_x \pa E \rightarrow T_x \pa E$ is symmetric, we obtain
\begin{equation}\label{mammaroma1}
    \begin{split}
        \nabla_{\pa E}& \big( (\nu_E \otimes A(\nu_E) \nu_E) \nabla_{\pa E} \psi   \big) = \nabla_{\pa E} \big( (A(\nu_E) \nu_E \cdot \nabla_{\pa E} \psi ) \nu_E    \big)\\
        &=  \nu_E \otimes \nabla_{\pa E} \big( A(\nu_E)\nu_E \cdot \nabla_{\pa E} \psi  \big)+ \big( A(\nu_E)\nu_E \cdot \nabla_{\pa E} \psi \big)\nabla_{\pa E} \nu_E          \\
        &= \nu_E \otimes \big[ \nabla_{\pa E} (A(\nu_E)\nu_E) \nabla_{\pa E} \psi + \nabla^2_{\pa E} \psi A(\nu_E)\nu_E   \big]+ \big( (A(\nu_E)\nu_E) \cdot \nabla_{\pa E} \psi \big) \nabla_{\pa E} \nu_E .
    \end{split}
\end{equation}
Moreover, we have
\begin{equation}\label{mammaroma2}
    {\rm tr} \big( \nu_E \otimes \nabla^2_{\pa E} \psi A(\nu_E)\nu_E  \big)=  \nu_E \cdot \nabla^2_{\pa E} \psi A(\nu_E) \nu_E =0
\end{equation} and 
\begin{equation}\label{mammaroma3}
\begin{split}
    {\rm tr} \big( \nu_E \otimes \nabla_{\pa E} (A(\nu_E) \nu_E)\nabla_{\pa E } \psi  \big)&=  \nu_E \cdot \nabla_{\pa E} (A(\nu_E) \nu_E) \nabla_{\pa E} \psi \\
    &=  \nu_E \cdot \nabla \big(  A(\nabla d_E) \nabla d_E \big) \nabla_{\pa E} \psi \\
    &=  \nu_E \cdot \nabla \big(  \sum_{i,j=1}^3 a_{ij}(\nabla d_E) (\nabla d_E \cdot e_j)e_i\big) \nabla_{\pa E} \psi  \\
    &=  \langle A_2(\nu_E, B_E) ,  \nabla_{\pa E} \psi \rangle .
    \end{split}
\end{equation} 
Using that $B_E= \nabla_{\pa E} \nu_E$ and~\eqref{mammaroma6}, we obtain
\begin{equation}\label{mammaroma4}
 {\rm tr} \left( \big( (A(\nu_E)\nu_E) \cdot \nabla_{\pa E} \psi \big) \nabla_{\pa E} \nu_E \right) = \big( (A(\nu_E)\nu_E) \cdot \nabla_{\pa E} \psi \big) H_E.
\end{equation}
Combining~\eqref{mammaroma1}-\eqref{mammaroma4}, we conclude that
\begin{equation}\label{16052026form1}
\begin{split}
    \div_{\pa E} \big( \nu_E \otimes \nu_E A(\nu_E) \nabla_{\pa E} \psi \big)
    = \big( (A(\nu_E)\nu_E) \cdot \nabla_{\pa E} \psi \big) H_E+ \langle A_2(\nu_E, B_E) ,  \nabla_{\pa E} \psi \rangle.
    \end{split}
\end{equation}

Finally, from~\eqref{mammaroma5} and~\eqref{16052026form1}, we deduce
\begin{equation}\label{battiato4}
\begin{split}
   \div_{\pa E}& \big( A(\nu_E) \nabla_{\pa E} \psi  \big)\\
   &= \div_{\pa E} \big( \mathcal{A}(\nu_E) \nabla_{\pa E} \psi \big)+ \langle A_2(\nu_E, B_E) ,  \nabla_{\pa E} \psi \rangle  + \big( (A(\nu_E)\nu_E) \cdot \nabla_{\pa E} \psi \big) H_E \\
   & =\div_{\pa E} \big( \mathcal{A}(\nu_E) \nabla_{\pa E} \psi \big)+ \langle A_3(\nu_E, B_E),  \overline{\nabla} \psi \rangle,
   \end{split}
\end{equation}
where $A_3$ is a smooth tensor, and we have used the identification of $\nabla_{\partial E}\psi$ with $\overline{\nabla}\psi$. \fr}
\end{remark}

In the following corollary, we present the expansion of the anisotropic curvature which will be used in the subsequent analysis.

\begin{coroll}\label{zonzo}
   Let $E, F \subset \R^3$ be as in Lemma~\ref{10Expcurvvarphi}. Then, for every $x \in \Sigma= \pa E$, the following identity holds:
   \begin{equation}\label{expcurvcondiv}
       H_F^\varphi(x+ \psi(x)\nu_E(x))= - \div_{\pa E} \big( \mathcal{A}(\nu_E) \nabla_{\pa E} \psi(x) \big) + H^\varphi_E(x)+ \widetilde{\mathcal R}_0(x).
   \end{equation}
   Here, $\mathcal{A} \in T^{(1,1)}(T\Sigma)$ is positive definite and the error term $\widetilde{\mathcal R}_0$ is given by
   \begin{equation}\label{nuovoresto}
       \widetilde{\mathcal R}_0 =\langle \widetilde{\mathcal{A}}_1(\psi B_\Sigma, \overline{\nabla} \psi, \nu_E), \overline{\nabla}^2 \psi \rangle+ \langle \widetilde{\mathcal{A}}_2(\psi B_\Sigma, \overline{ \nabla} \psi, \nu_E), \overline{\nabla}(\psi B_\Sigma) \rangle+\tilde{a}_0(\psi, \overline{\nabla} \psi, B_\Sigma, \nu_E) ,
    \end{equation}
    where $\widetilde{\mathcal{A}}_1, \, \widetilde{\mathcal{A}}_2$ are smooth tensor fields satisfying $\widetilde{\mathcal{A}}_1(0,0,\cdot)=0$, $\widetilde{\mathcal{A}}_2(0,0,\cdot)=0$ and $a_0$ is a smooth function such that $\tilde{a}_0(0,0,\cdot,\cdot)=0$. Moreover it  holds
    \begin{equation}\label{26052026Gaeta}
    \begin{split}
        &\vert \overline{\nabla}\widetilde{\mathcal{A}}_i (\psi B_\Sigma, \overline{\nabla} \psi, \nu_E) \vert \leq C \big(  \vert \psi \vert + \vert \overline{\nabla} \psi \vert + \vert \overline{\nabla}^2 \psi \vert \big), \\
        &\vert \overline{\nabla}\tilde{a}_0 (\psi, \overline{\nabla}\psi, B_\Sigma, \nu_E) \vert \leq C \big( \vert \psi \vert+ \vert \overline{\nabla} \psi \vert + \vert \overline{\nabla}^2 \psi \vert   \big),
        \end{split}
    \end{equation}
    where $C=C(\| B_\Sigma \|_{L^\infty})$.
\end{coroll}
\begin{proof}
   By Remarks~\ref{battiato2} and~\ref{battiato3},  in particular, by formulas~\eqref{trindivform} and~\eqref{battiato4}, we obtain
    \begin{equation}
        {\rm tr} \big( (\nabla^2 \varphi)(\nu_E) \nabla^2_{\pa E} \psi  \big)= \div_{\pa E} \big( \mathcal{A}(\nu_E) \nabla_{\pa E} \psi \big)+ \langle C(\nu_E, B_E), \overline{\nabla} \psi \rangle,
    \end{equation}
    where $\mathcal{A}$ is the $(1,1)$-tensor field defined in~\eqref{07032026posdef}, and $C$ is a smooth tensor. Therefore, combining this identity with Lemma~\ref{10Expcurvvarphi}, one readily obtains~\eqref{expcurvcondiv} and~\eqref{nuovoresto}. Finally, by  straightforward computations, one can derive~\eqref{26052026Gaeta}.
\end{proof}
\section{Setting of the problem}\label{settingdelproblema180526}
\subsection{Pseudo-pseudo-$H^{-1}$ distance}
In this subsection, we recall the definition and some basic properties of the pseudo-pseudo-$H^{-1}$ distance introduced in~\cite{CaTa} to model surface diffusion.
	\begin{definition}[Pseudo-pseudo-$H^{-1}$ metric]
		Let $ E \subset \R^3$ be a set of finite perimeter and let $ F \subset \R^3$ be a measurable set. We define the function $d_{H^{-1}}(F,E)$ as
		\begin{equation}\label{d_{H^{-1}}f}
			d_{H^{-1}}(F,E):= \sup_{ \| \nabla_{\pa E} f \|_{L^2(\pa E)} \leq 1} \int_{\R^3} f \circ \pi_{\pa E}(x) (\chi_F(x)-\chi_E(x))\, dx.
		\end{equation}
	\end{definition}
We observe that the projection $\pi_{\pa E}$  is well defined almost everywhere in $\R^3$, since it is unique at points where the signed distance function $d_E$ is differentiable. We also note that $d_{H^{-1}}$ is not symmetric and therefore does not define a metric, as it is constructed using the projection onto $\pa E$. 
If the set $E$ satisfies the \textsc{UBC} with radius $r_0$ and the set $F$ is such that $ F \Delta E \subset \mathcal{I}_{r_0}(\pa E)$, then it is possible to compute $d_{H^{-1}}(F,E)$ explicitly. This computation is contained in the following lemma, the proof of which can be found in~\cite[Lemma 2.4]{CFJKsd} and~\cite[Lemma 3.3]{Kub2025}.
\begin{lemma}\label{lemmaH^-1}
    Let $E \Subset \R^3$ be a connected $C^2$-regular set satisfying the \textsc{UBC} with radius $r_0$. Let $F \subset \R^3$ be a measurable set such that $F \Delta E \subset \mathcal{I}_{r_0}(\pa E)$ and we set
    \begin{equation}\label{funzxi_FE}
        \xi_{F,E}(x):= \int_{-r_0}^{r_0} \big(\chi_{F}(x+t \nu_E(x))-\chi_{E}(x+ t \nu_E(x)) \big)(1+ H_E(x)t+ K_E(x)t^2) \, dt.
    \end{equation}
    Then $d_{H^{-1}}(F,E) < + \infty $ if and only if $\int_{\pa E} \xi_{F,E} \, d \mathcal{H}^2 =0$, which is also equivalent to $\vert F \vert= \vert E \vert$. Moreover, it holds 	\begin{equation}\label{normH^-1funzv}
		d_{H^{-1}}^2(F,E)= \int_{\pa E} \vert \nabla_{\pa E} v_{F,E} \vert^2 \, d \mathcal{H}^{2},
	\end{equation}
	where $v_{F,E} $ is the unique solution to the equation
	\begin{equation}\label{eqv}
		\left\{
		\begin{aligned}
			& -\Delta_{\pa E} v_{F,E}= \xi_{F,E} & \text{ on } \pa E , \\
			& \int_{\pa E} v_{F,E} \, d \mathcal{H}^2=0.
		\end{aligned}
		\right.
	\end{equation} 
\end{lemma}
A consequence of the above lemma is that the function $f$ that attains the supremum in~\eqref{d_{H^{-1}}f} is given by
\begin{equation}\label{lafunzcherelsupH^-1}
    f=\frac{v_{F,E}}{d_{H^{-1}}(F,E)}.
\end{equation}
If, in addition to the assumptions of Lemma~\ref{lemmaH^-1}, we assume that $\partial F$ is a normal graph of a $C^1$ function over $\partial E$, i.e., $\pa F= \{ x + \psi(x )\nu_E(x): x \in \pa E\}$ with $\psi \in C^1(\pa E)$, then the function in~\eqref{funzxi_FE} can be written explicitly as
\begin{equation}\label{xiFEfunz}
    \xi_{F,E}= \psi + \frac{H_E}{2} \psi^2 + \frac{K_E}{3}\psi^3.
\end{equation}
Moreover, we observe that if $E$ is $C^{2,\alpha}$-regular, then the function $f$ is of class $C^{2,\alpha}$; this follows from standard elliptic regularity estimates.

We define the standard $H^{-1}$-norm for a function $u$ on $\pa E$ as
\begin{equation}
    \| u \|_{H^{-1}(\pa E)}= \sup_{\| \nabla_{\pa E} g \|_{L^2(\pa E)} \leq 1} \int_{\pa E} g u \, d \mathcal{H}^2.
\end{equation}
We recall the following interpolation inequality:
\begin{equation}\label{gucciniintepol}
    \| u \|_{L^2(\pa E)} \leq \| \nabla_{\pa E} u \|_{L^2(\pa E)}^{\frac{1}{2}} \| u \|_{H^{-1}}^{\frac{1}{2}}.
\end{equation}
Under the assumptions of Lemma~\ref{lemmaH^-1}, if $\vert E \vert= \vert F\vert $, it holds
\begin{equation}
    \| \xi_{F,E} \|_{H^{-1}}= d_{H^{-1}}(F,E).
\end{equation}

The next lemma contains the first variation of the functional $F \mapsto d_{H^{-1}}(F,E)$. While its proof is provided for $\R^2$ in~\cite[Proposition 3.4]{Kub2025}, the same argument applies to $\R^3$.
\begin{proposition}\label{propELDH-1}
		Let $ E \Subset \R^3$ be a $C^2$-regular set that satisfies the UBC with radius $r_0$, and let $F \Subset \R^3$ be a set of class $C^1$ such that $ F \Delta E \subset \mathcal{I}_{r_0}(\pa E)$. Given $ X \in C_c^1(\R^3,\R^3)$  such that $ \div X=0$, we set $\Psi: (-\varepsilon,\varepsilon) \times \R^3 \rightarrow \R^3$ to be the solution of the Cauchy problem
		\begin{equation}
			\begin{cases}
				\frac{\pa }{\pa t} \Psi (t,x)= X\circ \Psi (t,x) \quad \forall x \in \R^3\\
				\Psi(0,x)=x \quad \forall x \in \R^3. 
			\end{cases}
		\end{equation}
		If $ f \in H^1(\partial E)$ with $ \int_{\pa F} f \,d \mathcal{H}^{2}=0$ denotes the function that realizes the supremum in the definition of $d_{H^{-1}}(F,E)$,
		then 
		\begin{equation}
			\frac{d }{d t} d_{H^{-1}}(\Psi(t,F),E) |_{t=0}= \int_{\partial F} f(\pi_{\pa E}(x)) X(x) \cdot \nu_{F}(x) d \mathcal{H}^{2}_x.
		\end{equation}
	\end{proposition}

\subsection{Elastic term}
Let $\Omega$ be as in Definition~\ref{Omegaevrietadiriferimento}, $F \Subset \Omega$, and let $u : \Omega \setminus F \rightarrow \R^3$ be an  elastic displacement. We define the symme $E(u)$, the symmetric part of $\nabla u$, as 
\begin{equation}\label{02042026Eu}
    E(u):=\frac{\nabla u+ (\nabla u)^{T}}{2}.
\end{equation} 
    Throughout this work, we consider
	$\mathbb{C}: \R^{3\times 3}_{sym} \to \R^{3\times 3}_{sym}$ an elasticity tensor satisfying the coercivity condition $$\mathbb{C}A : A >0 \text{ for all } A \in \R_{sym}^{3 \times 3} \setminus \{ 0\} .$$ We define the elastic energy density by $$Q(A):= \frac{1}{2} \mathbb{C}A : A.$$  
     \subsubsection{Elastic energy}\label{subs:elastic energy}
    Let $w_0 \in C^{\infty}(\pa \Omega)$ be a prescribed boundary displacement. We define the elastic problem as
	\begin{equation}\label{minelast}
		u_F \in \arg \!\min \left\{ \int_{\Omega \setminus F} Q(E(u))\, dx : u \in H^1(\Omega \setminus F, \R^3) ,\, u |_{\pa \Omega}= w_0     \right\}
	\end{equation}
	and we define the corresponding energy as 
	\begin{equation*}
		\mathcal{E}(E(u_F)):= \int_{\Omega \setminus F} Q(E(u_F))\, dx.
	\end{equation*}
	In particular, $u_F$ is the unique solution in $H^1(\Omega \setminus F, \R^3)$ to
	\begin{equation}\label{eqelliel}
		\left\{
		\begin{aligned}
			& \div \mathbb{C}E(u_F)=0 & \text{ in } \Omega \setminus F  , \\
			& \mathbb{C}E(u_F)[\nu_F]=0 & \text{ on }  \pa F , \\
			& u_F= w_0 & \text{ on } \pa \Omega.
		\end{aligned}
		\right.
	\end{equation}  
\subsubsection{Constrained elastic energy}
	Let $w_0 \in C^{\infty}(\pa \Omega)$ be a prescribed boundary displacement as above. Fix $K_{el}>0$ such that $ \| w_0 \|_{C^{4}(\pa \Omega)} \leq K_{el} / 2$.  Let $F \Subset \Omega$ be a set of finite perimeter. 
    \begin{remark}
        {\rm 
   We recall that $\mu:= \mathcal{H}^2 \mrestr \pa^* F$ is a finite Radon measure on $\R^3$, and that $(\R^3, \vert \cdot \vert, \mu)$ is a metric measure space. Moreover, one can define the Sobolev space $H^{1,2}(\R^3, \vert \cdot \vert, \mu)$, which is a separable Hilbert space; see, for instance,~\cite{DiMarLucPas2020,GigliPasqualetto2017}.  We denote the norm in $H^{1,2}(\R^3, \vert \cdot \vert, \mu) $ by $ \| \cdot \|_{H^{1,2}( \mu)}$.  If $f \in C^1(\R^3)$, then we have $ \| f \|_{H^{1,2}(\mu )}= \| \nabla_{\pa F}f \|_{L^2(\R^3, \mu)}+  \| f \|_{L^2(\R^3, \mu)}$.  In particular, if $\pa F$ is smooth, then the space $ H^{1,2}(\R^3, \vert \cdot \vert, \mu)$ coincides with $ H^1(\pa F)$.\fr}
    \end{remark}
    We define the minimization problem \begin{equation}\label{minelvinc}
    \begin{split}
		u_F^{K_{el},h} \in \mathrm{argmin} \bigg\{ &\int_{\Omega \setminus F} Q(E(u))\, dx :  u \in \mathfrak{C}_{K_{el}}^{3,\frac{1}{4}}(\Omega ,\R^3) , \,\, u |_{\pa \Omega}= w_0, \\  
        &\| \pa^\alpha u \|_{H^{1,2}(\mathcal{H}^2 \mrestr \pa^* F)} \leq K_{el},\\
        & \int_{\pa^* F} \nabla_{\pa F} \pa^\alpha (u * \rho_\varepsilon )\cdot \nabla_{\pa F} \xi \, d \mathcal{H}^2  \leq K_{el} h^{-\frac{1}{4}}    \,\, \forall \xi \in \mathcal{D}, \, \varepsilon \in D, \, \vert \alpha \vert =3
          \bigg\},
        \end{split}
	\end{equation}
where $\mathcal{D} \subset C^\infty(\R^3)$ denotes a countable dense subset of the unit ball of $L^2(\R^3, \mathcal{H}^2 \mrestr \pa^* F)$, $D $ denotes a countable dense subset of $[0,1]$, and $\rho_\varepsilon$ is a standard Friedrichs mollifier. 

\begin{remark}\label{19042026remak}
    {\rm We observe that if $\pa F$ is sufficiently smooth, then the condition
    \begin{equation}
       \int_{\pa^* F} \nabla_{\pa F} \pa^\alpha (u * \rho_\varepsilon )\cdot \nabla_{\pa F} \xi \, d \mathcal{H}^2  \leq K_{el} h^{-\frac{1}{4}} \,\, \forall \xi \in \mathcal{D}, \, \varepsilon \in D , \, \vert \alpha \vert=3
    \end{equation}
    is equivalent to requiring that $\| \Delta_{\pa F}  \pa^\alpha u \|_{L^2(\pa F)} \leq K_{el} h^{-\frac{1}{4}} $ for $\vert \alpha \vert=3$. \fr}
\end{remark}
We define the constrained elastic energy as follows
	\begin{equation*}
		\mathcal{E}(E(u_F^{K_{el},h})):= \int_{\Omega \setminus F} Q(E(u_F^{K_{el},h}))\, dx.
	\end{equation*}

    \begin{remark}\label{02052026falchi}
   { \rm  Let $E_0$ be as in Definition~\ref{Omegaevrietadiriferimento}. By Schauder estimates, we can choose a constant $K_{el}$ such that $u_{E_0}$, the minimizer of~\eqref{minelast}, coincides with the minimizer $u_{E_0}^{K_{el}}$ of~\eqref{minelvinc} for every $h \leq 1$.  \fr }
\end{remark}

\begin{lemma}
    Let $F \Subset \Omega$ be  a set of finite perimeter. Then problem~\eqref{minelvinc} admits a solution.
\end{lemma}
\begin{proof}
Let $\mathcal{P}$ denote the collection of properties in ~\eqref{minelvinc}; the set of functions $u: \Omega \rightarrow \R^3 $ such that $\mathcal{P}$ holds is nonempty.
Indeed,  for any $r \leq \min\{ 1/(2\| B_{\pa \Omega} \|_{L^\infty(\pa \Omega)}), d_\mathcal{H}(F, \pa \Omega) \}$, the function  $u = \lambda w_0 ( \pi_{\pa \Omega} )$ satisfies  $\mathcal{P}$. Here, $\lambda \in C_c^\infty(\mathcal{I}_r(\pa \Omega))$ is a cutoff function satisfying $\lambda=1$ in $\mathcal{I}_{r/2}(\pa \Omega)$, whose $C^5$ norm is sufficiently small.

Now, we show existence of a minimizer of~\eqref{minelvinc}. 
Let $\{ u_n\}_{n \in \N} $ be a minimizing sequence. Then, by the assumption $ u_n \in \mathfrak{C}_{K_{el}}^{3,\frac{1}{4}}(\Omega ,\R^3)$ and Arzelà--Ascoli theorem, we obtain that $ u_n \rightarrow u$ in $\mathfrak{C}_{K_{el}}^{3}(\Omega ,\R^3)$ and that $u \in \mathfrak{C}_{K_{el}}^{3,\frac{1}{4}}(\Omega ,\R^3)$. Using that $H^{1,2}(\R^3, \vert \cdot \vert, \mathcal{H}^2 \mrestr \pa^* F)$ is a separable Hilbert space, by the Banach-Alaoglu-Bourbaki theorem and the lower semicontinuity of the norm, we obtain, up to  a subsequence, that
 \begin{equation}
     \| \pa^\alpha u \|_{H^{1,2}(\mathcal{H}^2 \mrestr \pa^* F)} \leq \liminf_{n \rightarrow + \infty} \| \pa^\alpha u_n \|_{H^{1,2}(\mathcal{H}^2 \mrestr \pa^* F)} \leq K_{el}.
 \end{equation}
 Finally, let $\varepsilon \in D$ and $\xi \in \mathcal{D}$ be fixed, then
 \begin{equation}
     \begin{split}
         \int_{\pa^* F} \nabla_{\pa F} \pa^\alpha (u * \rho_\varepsilon )\cdot  \nabla_{\pa F} \xi  =& \lim_{n \rightarrow + \infty}\int_{\pa^* F} \nabla_{\pa F} \pa^\alpha (u * \rho_\varepsilon - u_{n} * \rho_\varepsilon )\cdot \nabla_{\pa F} \xi \\
          &+\lim_{n \rightarrow + \infty} \int_{\pa^* F} \nabla_{\pa F} \pa^\alpha (u_n * \rho_\varepsilon )\cdot \nabla_{\pa F} \xi \\
         \leq&  \lim_{n \rightarrow + \infty} \int_{\pa^* F} \nabla_{\pa F} \pa^\alpha (u * \rho_\varepsilon - u_{n} * \rho_\varepsilon )\cdot \nabla_{\pa F} \xi + K_{el}h^{-\frac{1}{4}}\\
         \leq&  K_{el}h^{-\frac{1}{4}}.
     \end{split}
 \end{equation}
\end{proof}
	For brevity, we omit the explicit dependence of $u_{F}^{K_{el},h}$ on $h$ and  write $u_{F}^{K_{el}}$.

	In the next proposition, we derive the first variation of the elastic energy. The proof, established for $\R^2$ in~\cite[Proposition 3.6]{Kub2025}, extends verbatim to $\R^3$.
	\begin{proposition}\label{PropelELASt}
		Let $F \Subset \Omega$ be a $C^1$-regular set. Given $X \in C^1_c(\Omega,\R^3)$, let $( \Phi(t,\cdot))_{t \in (-\varepsilon,\varepsilon)}$ denote the unique solution of the Cauchy problem
		\begin{equation}
			\left\{
			\begin{aligned}
				& \frac{\pa }{\pa t} \Phi(t,x)= X \circ \Phi(t,x) & \forall x \in \R^3 ,\\
				& \Phi(0,x)=x & \forall x \in \R^3.
			\end{aligned}
			\right.
		\end{equation}
		Setting $F_t:= \Phi(t,F)$, the first variation of the functional $\mathcal{E}$ is given by
		\begin{equation}\label{FVELAST}
			\frac{d}{d t} \bigg|_{t=0} \mathcal{E}(u_{F_t}^{K_{el}})= -\int_{\pa F} Q(E(u_F^{K_{el}})) X \cdot \nu_F \, d \mathcal{H}^2.
		\end{equation}
	\end{proposition}

\subsection{Minimizing movement scheme and flat solution}
\label{subsec:flatsol}
Fix  a time discretization step $h>0$, and let $ K_{el},\beta>0$ be fixed constants. Let $E \Subset \R^3 $ be a set of finite perimeter. 
For every measurable set $F \subset \R^3$ such that $F \Delta E \subset \mathcal{I}_{\beta}(\pa E)$, we define the functional
\begin{equation}\label{energia}
	\mathcal{F}_h(F,E):= \mathcal{G}(F)+ \frac{1}{2h}d_{H^{-1}}^2(F,E),
\end{equation}
where 
\begin{equation}\label{defG}
	\mathcal{G}(F)=P_{\varphi}(F)+\mathcal{E}(E(u_F^{K_{el}})).
\end{equation} 
\begin{definition}[Constrained discrete flat flow] \label{12092023def1}
	Let $K_{el},\,\beta, \, h>0$. Given   a set of finite perimeter $E_0 \Subset \R^3$, we define   $\{ E_{hk}^{h,\beta}\}_{k \in \mathbb{N}}$ iteratively by setting $E_0^{h,\beta}:=E_0$ and for $k\geq 1$ 
	$$   E_{ hk}^{h,\beta} \in \arg \min \left\{\mathcal{F}_{h}(F, E_{h(k-1)}^{h,\beta}) :  F \Delta E_{h(k-1)}^{h,\beta} \subset {\rm cl}( \mathcal{I}_{\beta}(\pa E_{h(k-1)}^{h,\beta})) \right\} .$$
	Moreover, we set $E_{t}^{h,\beta}:= E_{hk }^{h,\beta}$ for any  $t \in [k h,\, (k+1)h).$
	The family $\{E_t^{h,\beta}\}_{t \geq 0}$ is called a constrained discrete flat flow with initial datum $E_0$ and time step $h$. 
\end{definition}
We define a \emph{constrained flat flow solution of the anisotropic surface diffusion with elasticity} with initial datum $E_0$ any family of sets $\{E^\beta_t\}_{t \geq 0}$ such that $ P(E^\beta_t) < +\infty$ for every $t \geq 0$ and
\begin{equation}
    \lim_{n \rightarrow + \infty} \vert E^{h_n , \beta}_t \Delta E^\beta_t \vert =0 \text{ locally uniformly in $t$}.
\end{equation}

\section{Almost minimizers of the anisotropic perimeter}\label{sezioniquasiminimi}
In this section, we discuss properties of almost-minimizers of the anisotropic perimeter beginning with the definition.
\begin{definition}\label{15052026deflambdaalpha}
    Let $\Lambda \geq 0$ and $\alpha \in [0,\frac{1}{3})$. A set of finite perimeter $E \subset \R^{3}$ is called a $(\Lambda,\alpha)$-minimizer of the anisotropic perimeter if, for every set of finite perimeter $G \subset \R^3$, it holds that
    \begin{equation}
        P_{\varphi}(E) \leq P_{\varphi}(G)+ \Lambda \vert E \Delta G \vert^{1-\alpha}.
    \end{equation}
    When $\alpha=0$, we refer to $E$ as a $\Lambda$-minimizer of the anisotropic perimeter.
\end{definition}
It is known that if $E \subset \mathbb{R}^3$ is 
$(\Lambda,\alpha)$-minimizer of the anisotropic perimeter, then $\partial E$ is of class $C^{1,\eta}$ for every $ \eta \in [0 , \frac{\beta}{2})$; see ~\cite{AlmgrenSimonSchoen1977},~\cite{Bombieri1982} and~\cite{DePMa2015}. 
In the next lemma, we prove that if a bounded set is $C^2$-regular and satisfies the \textsc{UBC}, then it is a $\Lambda$-minimizer of the anisotropic perimeter.
\begin{lemma}\label{lamdaminC2set}
    Let $E \subset \R^3$ be a $C^2$-regular bounded set satisfying the \textsc{UBC} with radius $r_0$. Then $E$ is a $\Lambda$-minimizer of the anisotropic perimeter; that is, for every set $G \subset \R^3$ of finite perimeter, it holds
    \begin{equation}\label{anisoLambdmin}
        P_\varphi(E) \leq P_\varphi(G) + \Lambda \vert E \Delta G \vert,
    \end{equation}
    where $\Lambda= \Lambda(r_0,\varphi)$.
\end{lemma}
\begin{proof}
The proof is based on a standard calibration argument, where the calibration is defined as
\begin{equation}\label{020226primacrossfit1}
      X : \R^3 \rightarrow \R^3,\,\,  X:= \nabla \varphi (\nu_E \circ \pi_{\pa E}) \xi,
    \end{equation}
    where $\xi \in C^\infty_c(\mathcal{I}_{r_0}(\pa E))$ and $\xi=1$ in $\mathcal{I}_{\frac{r_0}{2}}(\pa E)$. Details of the proof can be found in~\cite[Lemma~4.6]{Kub2025}.
\end{proof}
\begin{remark}\label{REMingrassato}
    {\rm Let $E \subset \R^3$ be as in Lemma~\ref{lamdaminC2set}, and let $ r \in (-\frac{r_0}{2}, \frac{r_0}{2})$. Then the set
    \begin{equation}\label{ingrassato}
        E_r := \left\{ x \in \R^3 : d_E(x) < r   \right\}
    \end{equation}
    is $C^2$-regular and satisfies the \textsc{UBC} with radius $r_0/2$, where we recall $d_E$ denotes the signed distance to $E$. Hence, applying Lemma~\ref{lamdaminC2set}, we obtain that for every set of finite perimeter $G \subset \R^3$
    \begin{equation}\label{060226form1}
        P_\varphi(E_r) \leq P_\varphi(G)+ \Lambda \vert E_r \Delta G \vert,
    \end{equation}
    where $\Lambda= \Lambda(r_0,\varphi)$. Moreover, by testing~\eqref{060226form1} with $G \cup E_r$ and $G \cap E_r$ we deduce
    \begin{equation}\label{060226formutil1}
        P_\varphi(G \cap E_r) \leq P_\varphi(G)+ \Lambda \vert G \setminus E_r \vert \text{ and } P_\varphi(G \cup E_r) \leq P_\varphi(G)+ \Lambda \vert E_r \setminus G \vert.
    \end{equation}
     \fr}
\end{remark}
The following lemma states some useful regularity properties for  $(\Lambda,\alpha)$-minimizers of the anisotropic perimeter. The proof is similar to those in~\cite[Lemma 2.8]{CFJKsd},~\cite[Lemma 4.3]{Kub2025} and~\cite[Lemma 2.12]{KLMP2025}, therefore we omit it here.
\begin{lemma}\label{epsilonregolarita}
    Let $E \Subset \Omega$ be a uniformly $C^{2,\beta}$-regular set with constants $r_0,C_0$ satisfying the \textsc{UBC} with radius $r_0$. Let $F$ be a $(\Lambda,\alpha)$-minimizer of the anisotropic perimeter. Given $\gamma < \min\{ \beta, \frac{1}{2}- \frac{3 \alpha}{2} \}$, there exists $\delta_0 \in (0, \frac{r_0}{2})$, depending on $r_0,C_0,\Lambda, \alpha,\gamma$, such that if 
    \begin{equation}
        F \Delta E \subset \mathcal{I}_{\delta_0}(\pa E),
    \end{equation}
    then there exists $\psi \in C^{1,\gamma}(\pa E)$ such that
    \begin{equation}
        \pa F = \left\{ x+ \psi(x)\nu_E(x) : x \in \pa E  \right\}.
    \end{equation}
    Moreover, for every $\varepsilon>0$ there exists $\delta_0= \delta_0(\varepsilon)$ such that $ \| \psi \|_{C^{1,\gamma}(\pa E)} \leq \varepsilon$.
\end{lemma}
In the next definition we introduce the notion of constrained $(\Lambda,\alpha)$-minimizer of the anisotropic perimeter.
\begin{definition}\label{DEF:lambdamincon}
    Let $\delta>0$, $\Lambda \geq 0$, $\alpha \in [0,\frac{1}{3})$, $E \subset \R^3$, and $F$ a set of finite perimeter. We say that $F \subset \R^3$ is a $(\delta,E)$-constrained $(\Lambda,\alpha)$-minimizer of the anisotropic perimeter if, for every set of finite perimeter $G \subset \R^3$ such that
    \begin{equation}
        G \Delta F \subset {\rm cl} \big( \mathcal{I}_{\delta}(\pa E) \big) \text{ and } \vert G \vert = \vert F \vert,
    \end{equation}
    it holds
    \begin{equation}
        P_\varphi(F) \leq P_\varphi(G)+ \Lambda \vert F \Delta G \vert^{1-\alpha}.
    \end{equation}
\end{definition}
    
	We recall that $\Omega\subset \R^3$ is a bounded open set with smooth boundary. For every $E \Subset \Omega$ we define 
	\begin{equation}\label{cost:simga_E}
		\sigma_E:= \frac{1}{2}{\rm dist}(\pa E, \pa \Omega).
	\end{equation} 

The next lemma shows that, under sufficient regularity assumptions on $E$, a $(\delta,E)$-constrained $(\Lambda,\alpha)$-minimizer of the anisotropic perimeter is a $(\Lambda,\alpha)$-minimizer of the anisotropic perimeter. For the proof we refer the reader to~\cite[Lemma~2.10]{CFJKsd}.
\begin{lemma}\label{lambdaminCONstrain}
    Let $E\Subset \Omega$ be a bounded $C^2$-regular set satisfying the \textsc{UBC} with radius $r_0$. Let $F \subset \R^3$ be a $(\delta,E)$-constrained $(\Lambda,\alpha)$-minimizer of the anisotropic perimeter for $\delta <  \min\{\frac{r_0}{2},\sigma_E\}$ such that
    \begin{equation}
        F \Delta E \subset {\rm cl} \big( \mathcal{I}_\delta(\pa E) \big) \text{ and } \vert F \vert = \vert E \vert.
    \end{equation}
    Then $F$ is a $(\Lambda_1,\alpha)$-minimizer of the anisotropic perimeter, where $\Lambda_1$ is a constant depending only on $\Lambda,\alpha,r_0,\sigma_E,\varphi, \vert \Omega \vert$.
\end{lemma}

\section{A notion of almost minimality respect the pseudo-pseudo-${H^{-1}}$ distance}\label{serzquaisminH^-1}
The aim of this section is to introduce a notion of almost minimality with respect to $d_{H^{-1}}$. By assuming higher regularity of the boundary of $E$, we will derive an analogue of~\eqref{anisoLambdmin}, in which the measure of the symmetric difference is replaced by the $d_{H^{-1}}$ and the anisotropic perimeter $P_\varphi$ is replaced by $\mathcal{G}$.

We begin with a few remarks that will be useful in the sequel.
\begin{remark}
{\rm 
    We observe that if $E \Subset \Omega$ satisfies the \textsc{UBC}, then $\partial E$ is of class $C^{1,1}$, and we can define the Sobolev space $H^2(\Sigma)$ with $\Sigma = \partial E$. \fr}
\end{remark}
\begin{remark}{\rm 
    Recalling that $\varphi$ is a smooth strictly convex norm, the following inequality holds:
\begin{equation}\label{cost:Jvarphi}
    \exists J_\varphi>0\,: \, \nabla^2 \varphi(\nu)\xi \cdot \xi \geq J_\varphi \vert \xi \vert^2 \,\, \forall \nu \in {\rm cl} \big( \mathcal{I}_{\frac{1}{4}}(\pa B_1)  \big) \text{ and } \xi \in \R^3 \text{ s.t. }\xi \cdot\nu=0.
\end{equation} \fr}
\end{remark}
\begin{proposition}\label{ordinaryprop}
		Let $E \Subset \Omega$ be a $C^2$-regular set satisfying the \textsc{UBC} with radius $r_0$. Set $\Sigma:= \pa E$ and assume that $ \| H_E \|_{H^2(\Sigma)} \leq K_0$ for some $K_0 >0$.  Then there exists $\delta_1 \in \left(0, \min \{ \frac{r_0}{2},\sigma_E\} \right)$ (recall \eqref{cost:simga_E}) depending on $r_0,K_0,\sigma_E,K_{el}$, such that, for every set $F$ of finite perimeter satisfying  $F \Delta E \subset \mathcal{I}_{\delta_1}(\pa E)$ and $\vert F\vert = \vert E \vert$, it holds that
		\begin{equation}\label{formulaimpostante}
			\mathcal{G}(E) \leq \mathcal{G}(F)+\Lambda_2 d_{H^{-1}}(F,E),
		\end{equation}
		where $\mathcal{G}$ is defined in~\eqref{defG} and $\Lambda_2= \Lambda_2(r_0,K_0,\sigma_E,K_{el})$. Moreover, if $\pa F$ is a normal graph with $\pa F= \{ x+ \psi(x)\nu_E(x): x \in \pa E\}$, $\psi \in C^1(\pa E)$ and $ \| \psi \|_{C^1(\pa E)} \leq \delta_1$, then
		\begin{equation}\label{H-1lambminperfunz}
			\frac{J_\varphi}{4} \| \nabla_{\pa E} \psi \|^2_{L^2(\pa E)}+ \mathcal{G}(E) \leq \mathcal{G}(F)+\Lambda_2 d_{H^{-1}}(F,E),
		\end{equation}
		where $J_\varphi$ is the constant given by~\eqref{cost:Jvarphi}.
	\end{proposition}

\begin{proof} We divide the proof in two steps.\\
\emph{Step 1}: We prove inequality~\eqref{H-1lambminperfunz}.\\ 
\emph{Claim 1}: The following estimates hold
     \begin{equation}\label{08022026form1}
          \| B_\Sigma \|_{H^2(\Sigma)}\leq C(K_0) \text{ and } \| H_E^\varphi \|_{H^2(\Sigma)} \leq C(K_0).
     \end{equation}
     
     The first inequality in~\eqref{08022026form1} follows immediately from Lemma~\ref{lem:hilbert-norm}. 
     Recall that the anisotropic curvature is not invariant under isometries of $\R^3$, hence, in order to compute $H^\varphi_E$ using the local chart $$\pa E \cap B_{\frac{r_0}{2}}(x)=\Big \{ (y',f(y'): y' \in B'_{\frac{r_0}{2}}(x')\Big\} \cap B_{\frac{r_0}{2}}(x)  $$
   we must accordingly modify the anisotropy. We consider the modification of $\varphi$ defined as $\tilde{\varphi}(x):= \varphi(Rx)$, where $R$ is an isometry of $\R^3$ chosen so as to represent the surface in this local chart. Therefore, using Lemma~\ref{10Expcurvvarphi} and in particular formula~\eqref{02expphicurv1} (with the choice $\pa F=\{ (y',f(y'): y' \in B'_{\frac{r_0}{2}}(x')\} $ and $ \pa E=B'_{\frac{r_0}{2}}(x')$),  we deduce that $f$ satisfies an equation of the form:
    \begin{equation}\label{100226ppranzo1}
        H_E^\varphi(\cdot,f(\cdot))=- {\rm tr} \big( A \nabla'^2 f  \big)+ H_{B'_{\frac{r_0}{2}}}^{\tilde \varphi} + \langle \mathcal{A}(\nabla' f), \nabla'^2 f \rangle +a_0(f, \nabla' f,\cdot) \text{ in } B'_{\frac{r_0}{2}}
    \end{equation}
   where $A \in \R^{2 \times 2}$ is positive definite and $\mathcal{A}$ is a smooth matrix-valued function taking values in $\R^{2 \times 2}$ and $a_0$ is a smooth real value function such that $\mathcal{A}(0)=0$ and $a_0(0,0,\cdot)=0$. Therefore, by~\eqref{100226ppranzo1} and using the $H^4$-bound of $f$, it follows that $ H_E^\varphi \in H^2(\Sigma)$ and that the second inequality in~\eqref{08022026form1} holds.

    By formula~\eqref{08022026form1} and the Sobolev embedding theorem, we obtain that $E$ is uniformly $C^{2,\alpha}$-regular for every $\alpha \in (0,1)$ and it holds
    \begin{equation}\label{linfB_Eealtro}
        \forall \, \,p \in (1,+\infty)\, \, \exists \, \,C_p(r_0,K_0,p)>0 \,\,: \,\, \| B_\Sigma \|_{L^\infty(\Sigma)}+ \| \overline{\nabla}   B_\Sigma \|_{L^p(\Sigma)} \leq C_p.
    \end{equation}
   Recall that $ \xi_{F,E}= \psi + \frac{H_E}{2}\psi^2+ \frac{K_E}{3}\psi^3$ (see Lemma~\ref{lemmaH^-1} and formula~\eqref{xiFEfunz}).\\
    \textit{Claim 2:} There exists $\delta_1 = \delta_1(r_0, K_0, \sigma_E)>0$ such that if $\| \psi \|_{C^1(\partial E)} \leq \delta_1$, then
    \begin{align}
        & \frac{1}{2}\vert \psi (x) \vert \leq \vert \xi_{F,E}(x) \vert \leq 2 \vert \psi (x)\vert \text{ for every } x \in \pa E ,\label{090226ppranzf1} \\ 
        & \| \nabla_{\pa E} \xi_{F,E} \|_{L^2(\pa E)} \leq  2 \| \nabla_{\pa E} \psi \|_{L^2(\pa E)}+ \sqrt{\delta_1} \| \psi \|_{L^2(\pa E)}, \label{090226ppranzf2}\\
        & \frac{1}{C}\| \nabla_{\pa E} \psi \|_{L^2(\pa E)} \leq \| \nabla_{\pa E} \xi_{F,E} \|_{L^2(\pa E)} \leq C \| \nabla_{\pa E} \psi \|_{L^2(\pa E)}, \label{10022026pprinazof2}
    \end{align}
    for some constant $C=C(r_0,K_0,\sigma_E)>0$.
    
   By~\eqref{linfB_Eealtro}, inequalities~\eqref{090226ppranzf1} follow immediately, provided that $\delta_1$ is sufficiently small. We compute the tangential gradient of $\xi_{F,E}$ and find
   \begin{equation}\label{ladefivatadixiFE}
       \nabla_{\pa E} \xi_{F,E}= \nabla_{\pa E} \psi \big( 1+ H_E \psi + K_E \psi^2  \big)+ \frac{\psi^2}{2} \nabla_{\pa E} H_E + \frac{\psi^3}{3} \nabla_{\pa E} K_E .
   \end{equation}
   Therefore, from formula~\eqref{linfB_Eealtro}, the Cauchy–Schwarz inequality, and the interpolation inequality of Proposition~\ref{prop:interpolation}, we deduce, for $\delta_1$ sufficiently small, that
   \begin{align*}
       \| \nabla_{\pa E} \xi_{F,E} \|_{L^2(\pa E)} &\leq 2 \| \nabla_{\pa E} \psi \|_{L^2(\pa E)}+C \left(  \int_{\pa E} \psi^4 \vert \overline{\nabla}B_{\Sigma}\vert^2 \, d \mathcal{H}^2\right)^{\frac{1}{2}}\notag\\
       & \leq 2 \| \nabla_{\pa E} \psi \|_{L^2(\pa E)}+ C \| \psi \|^2_{L^8(\pa E)} \| \overline{\nabla} B_E \|_{L^4(\pa E)}\notag\\
       & \leq 2 \| \nabla_{\pa E} \psi \|_{L^2(\pa E)}+ C \| \psi \|_{L^2(\pa E)} \| \psi \|_{W^{1,4}(\pa E)}\notag  \\
       & \leq 2 \| \nabla_{\pa E} \psi \|_{L^2(\pa E)}+ C\delta_1 \| \psi \|_{L^2(\pa E)} \notag  \\
       & \leq   2 \| \nabla_{\pa E} \psi \|_{L^2(\pa E)}+ \sqrt{\delta_1} \| \psi \|_{L^2(\pa E)},
   \end{align*}
   which proves~\eqref{090226ppranzf2}. 
   To derive the second inequality in~\eqref{10022026pprinazof2}, we combine~\eqref{090226ppranzf1}, the above estimates, and the Poincaré inequality, to obtain that
 \begin{equation}
     \begin{split}
         \| \nabla_{\pa E} \xi_{F,E}\|_{L^2(\pa E)}& \leq 2 \| \nabla_{\pa E} \psi \|_{L^2(\pa E)}+ \sqrt{\delta_1} \| \psi \|_{L^2(\pa E)}\\
         & \leq 2 \| \nabla_{\pa E} \psi \|_{L^2(\pa E)}+ 2 \sqrt{ \delta_1} \| \xi_{F,E} \|_{L^2(\pa E)}\\
         & \leq 2 \| \nabla_{\pa E} \psi \|_{L^2(\pa E)}+ 2\sqrt{\delta_1} C\| \nabla_{\pa E} \xi_{F,E}\|_{L^2(\pa E)}\\
         & \leq 2 \| \nabla_{\pa E} \psi \|_{L^2(\pa E)}+ \frac{1}{2}\| \nabla_{\pa E} \xi_{F,E} \|_{L^2(\pa E)}.
     \end{split}
 \end{equation}
Whereas, to obtain the first inequality in~\eqref{10022026pprinazof2}, we use formulas~\eqref{090226ppranzf1}, ~\eqref{ladefivatadixiFE}, and the Poincaré inequality to deduce that
\begin{equation}
\begin{split}
    \| \nabla_{\pa E} \psi \|_{L^2(\pa E)}& \leq \| \nabla_{\pa E} \xi_{F,E} \|_{L^2(\pa E)}+ C \delta_1 \| \nabla_{\pa E} \psi \|_{L^2(\pa E)}+ C \bigg( \int_{\pa E} \psi^4 \vert \overline{\nabla} B_\Sigma \vert^2 \, d \mathcal{H}^2 \bigg)^{\frac{1}{2}}\\
    & \leq \| \nabla_{\pa E} \xi_{F,E} \|_{L^2(\pa E)}+ C \delta_1 \| \nabla_{\pa E} \psi \|_{L^2(\pa E)}+ C \sqrt{\delta_1} \| \psi \|_{L^2(\pa E)}\\
    & \leq \| \nabla_{\pa E} \xi_{F,E} \|_{L^2(\pa E)}+ C \delta_1 \| \nabla_{\pa E} \psi \|_{L^2(\pa E)}+ C \sqrt{\delta_1} \| \xi_{F,E} \|_{L^2(\pa E)} \\
    & \leq C \| \nabla_{\pa E} \xi_{F,E} \|_{L^2(\pa E)}+ C \delta_1 \| \nabla_{\pa E} \psi \|_{L^2(\pa E)}.
\end{split}
\end{equation}
\textit{Claim 3:} There exists a constant $C= C(r_0,K_0)>0$ such that
   \begin{equation}\label{120226tesi1}
       \frac{3 J_\varphi}{8} \| \nabla_{\pa E} \psi \|^2_{L^2(\pa E)}+ P_\varphi(E) \leq P_\varphi(F)+ C d_{H^{-1}}(F,E).
   \end{equation}
   
Define $\Psi: \pa E \rightarrow \pa F$ by $ \Psi(x):= x+ \psi(x) \nu_E(x)$. We remark that, since $E$ is uniformly $C^{2,\alpha}$-regular, the function $\Psi$ is a $C^{1,\alpha}$ diffeomorphism. Fix $x \in \pa E$, let $\{ \tau_1,\, \tau_2 \}$ be the principal directions of $T_x \pa E$ and $\kappa_1(x),  \kappa_2(x)$ be the principal curvatures at $x$, i.e., $\kappa_i(x) \tau_i = B_E(x) \tau_i$ for $i =1,2$. Recall formula~\eqref{09022026primacorss2},  the tangential Jacobian is given by
\begin{equation}\label{evvrailtempo1}
    J_{\pa E} \Psi= \big( (1+ \psi \kappa_1)^2(1+ \psi \kappa_2)^2 +(1+ \psi \kappa_2)^2 \vert \overline{\nabla}_{\tau_1} \psi \vert^2+(1+ \psi \kappa_1)^2 \vert \overline{\nabla}_{\tau_2} \psi \vert^2\big)^{\frac{1}{2}}.
\end{equation}
Let $N: \pa E \rightarrow \R^3$ be the function defined in~\eqref{funzN}. By formula~\eqref{funznu_F}, we have  $\nu_F(\Psi(x))= N(x)/ \vert N(x)\vert$ for $x \in \pa E$. Let $x \in \pa E$, a straightforward computation yields
\begin{equation}
    \vert N(x) \vert = \frac{J_{\pa E} \Psi(x)}{(1+\psi(x)\kappa_1(x))(1+ \psi(x)\kappa_2(x))},
\end{equation}
hence, it holds
\begin{equation}
    \nu_F \circ \Psi(x) = N(x) \frac{ (1+ \psi(x)\kappa_1)(1+ \psi(x)\kappa_2)}{J_{\pa E} \Psi(x)}.
\end{equation}
Using this identity and the area formula, we obtain
\begin{align}
    P_\varphi(F)&= \int_{\pa F} \varphi(\nu_F)\, d\mathcal{H}^2= \int_{\pa E} \varphi(\nu_F \circ \Psi) J_{\pa E} \Psi \, d \mathcal{H}^2\notag\\
    & = \int_{\pa E} \varphi(N )(1+ \psi \kappa_1)(1+ \psi \kappa_2)\, d \mathcal{H}^2 \notag\\
    & = \int_{\pa E} \varphi(-(1+\psi \kappa_2) \overline{\nabla}_{\tau_1} \psi -(1+\psi \kappa_1) \overline{\nabla}_{\tau_2} \psi + (1+ \psi \kappa_1)(1+\psi \kappa_2)\nu_E )\, d\mathcal{H}^2.\label{giordano}
\end{align}
By the convexity of $\varphi$,~\eqref{cost:Jvarphi}, and possibly taking $\delta_1$  smaller, we obtain
\begin{align}
        &\varphi(-(1+\psi \kappa_2) \overline{\nabla}_{\tau_1} \psi  -(1+\psi \kappa_1) \overline{\nabla}_{\tau_2} \psi + (1+ \psi \kappa_1)(1+\psi \kappa_2)\nu_E)\notag\\
        &\geq \varphi((1+\psi \kappa_1)(1+\psi \kappa_2)\nu_E)-\nabla \varphi(\nu_E) \cdot \nabla_{\pa E} \psi \notag\\
        & \quad -\nabla \varphi(\nu_E) \cdot (  \psi \kappa_2\overline{\nabla}_{\tau_1}\psi +  \psi \kappa_1\overline{\nabla}_{\tau_2}\psi  )+\frac{J_\varphi}{2} \vert \nabla_{\pa E} \psi \vert^2 \notag\\
        & \geq \varphi(\nu_E)+ \big( \psi H_E+ \frac{\psi^2}{2}K_E \big) \nabla \varphi (\nu_E) \cdot \nu_E -\nabla \varphi(\nu_E) \cdot \nabla_{\pa E} \psi \notag\\
        & \quad -\nabla \varphi(\nu_E) \cdot (  \psi \kappa_2\overline{\nabla}_{\tau_1}\psi +  \psi \kappa_1\overline{\nabla}_{\tau_2}\psi  )+\frac{J_\varphi}{2} \vert \nabla_{\pa E} \psi \vert^2
        \notag\\
        & = \varphi(\nu_E)+ \big( \psi H_E+ \frac{\psi^2}{2}K_E \big) \varphi (\nu_E)  -\nabla \varphi(\nu_E) \cdot \nabla_{\pa E} \psi \label{slablig}\\
        & \quad -\nabla \varphi(\nu_E) \cdot (  \psi \kappa_2\overline{\nabla}_{\tau_1}\psi +  \psi \kappa_1\overline{\nabla}_{\tau_2}\psi  )+\frac{J_\varphi}{2} \vert \nabla_{\pa E} \psi \vert^2,\notag
\end{align}
where in the last step we used that $\nabla \varphi(x)\cdot x = \varphi(x)$. 

Let $\varepsilon > 0$ to be chosen later. Using formulas~\eqref{linfB_Eealtro},~\eqref{090226ppranzf1} and~\eqref{10022026pprinazof2}, we obtain
\begin{align}
        &\int_{\pa E} \Big (\psi H_E + \frac{\psi^2}{2}K_E \Big ) \varphi(\nu_E)\, d \mathcal{H}^2 \notag\\
        &\leq \int_{\pa E} (\psi + \frac{\psi^2}{2}H_E+ \frac{\psi^3}{3}K_E)H_E \varphi(\nu_E)\, d \mathcal{H}^2+ C \int_{\pa E} \psi^2 \, d \mathcal{H}^2 \notag\\
        &\leq\| H_E \varphi(\nu_E) \|_{H^1(\Sigma)}\| \psi +\frac{\psi^2}{2}H_E+ \frac{\psi^3}{3}K_E \|_{H^{-1}(\pa E)}  + C \int_{\pa E} \psi^2 \, d \mathcal{H}^2 \\
        & \leq C \| \xi_{F,E} \|_{H^{-1}(\pa E)}+ C \| \xi_{F,E} \|^2_{L^2(\pa E)} \notag\\
        & \leq C\| \xi_{F,E} \|_{H^{-1}(\pa E)}+ C \| \xi_{F,E}\|_{H^{-1}(\pa E)} \| \nabla_{\pa E} \xi_{F,E}\|_{L^2(\pa E)}\notag\\
        & \leq \frac{C}{\varepsilon} \| \xi_{F,E}\|_{H^{-1}(\pa E)}+ \varepsilon \| \nabla_{\pa E} \xi_{F,E} \|_{L^2(\pa E)}^2 \notag\\
        & \leq \frac{C}{\varepsilon} \| \xi_{F,E}\|_{H^{-1}(\pa E)}+ \varepsilon \| \nabla_{\pa E} \psi  \|_{L^2(\pa E)}^2.\label{schifani}
\end{align}
Arguing similarly as above and additionally using the divergence theorem and~\eqref{08022026form1}, we infer that
\begin{align}
        &\int_{\pa E}  \nabla \varphi(\nu_E) \cdot \nabla_{\pa E} \psi \, d \mathcal{H}^2 = \int_{\pa E} \div_{\pa E}(\nabla \varphi(\nu_E)) \psi \, d \mathcal{H}^2 \notag\\
        & \leq \int_{\pa E} H^\varphi_E ( \psi + \frac{\psi^2}{2}H_E+\frac{\psi^3}{3}K_E ) \, d \mathcal{H}^2 + C \int_{\pa E} \psi^2 \, d \mathcal{H}^2\notag \\
        & \leq \| H^\varphi_E \|_{H^1(\pa E)} \| \xi_{F,E}\|_{H^{-1}(\pa E)}+ C \| \xi_{F,E}\|^2_{L^2(\pa E)} \notag\\
        & \leq C \| \xi_{F,E}\|_{H^{-1}(\pa E)}+ C\| \xi_{F,E} \|_{H^{-1}(\pa E)}\| \nabla_{\pa E} \xi_{F,E}\|_{L^2(\pa E)}\notag\\
        &  \leq \frac{C}{\varepsilon} \| \xi_{F,E}\|_{H^{-1}(\pa E)}+ \varepsilon \| \nabla_{\pa E} \psi \|_{L^2(\pa E)}.\label{corchia}
\end{align}
Analogously,  we have
\begin{align}
        &\int_{\pa E}-\nabla \varphi(\nu_E) \cdot (  \psi \kappa_2\overline{\nabla}_{\tau_1}\psi +  \psi \kappa_1\overline{\nabla}_{\tau_2}\psi  ) \, d \mathcal{H}^2 \leq \frac{C}{\varepsilon} \| \psi \|_{L^2(\pa E)}^2+\varepsilon \| \nabla_{\pa E} \psi \|_{L^2(\pa E)}^2 \notag\\
        & \leq \frac{C}{\varepsilon} \| \xi_{F,E} \|_{L^2(\pa E)}^2+\varepsilon \| \nabla_{\pa E} \psi \|_{L^2(\pa E)}^2  \leq \frac{C}{\varepsilon^2} \| \xi_{F,E} \|_{H^{-1}(\pa E)}+ 2\varepsilon \| \nabla_{\pa E} \psi \|^2_{L^2(\pa E)}.\label{nador}
\end{align}
	Therefore, combining~\eqref{giordano},~\eqref{slablig},~\eqref{schifani},~\eqref{corchia} and~\eqref{nador}, and choosing $\varepsilon>0$ sufficiently small, we conclude the claim, recalling also that $d_{H^{-1}}(F,E)= \| \xi_{F,E} \|_{H^{-1}(\partial E)}$ (see formulas~\eqref{normH^-1funzv} and~\eqref{eqv}).\\
		\textit{Claim 4:} There exists $C=C(r_0,K_0,K_{el})>0$ such that
		\begin{equation}\label{120226tesi2}
			\mathcal{E}(E(u_E^{K_{el}})) \leq \mathcal{E}(E(u_F^{K_{el}}))+C d_{H^{-1}}(F,E)+ \frac{J_\varphi}{8} \| \nabla_{\pa E} \psi \|^2_{L^2(\pa E)}.
		\end{equation}
        
		By the minimality of $u^{K_{el}}_E$ and the  definition
		of $d_{H^{-1}}(F,E)$, we obtain that
		\begin{align}
				\mathcal{E}(E(u_F^{K_{el}}))&= \int_{\Omega \setminus F} Q(E(u_F^{K_{el}}))\, dx = \int_{\R^3}   Q(E(u_F^{K_{el}})) (\chi_\Omega- \chi_F)\, dx\notag\\
				& = \int_{\Omega \setminus E} Q(E(u_F^{K_{el}}))\, dx + \int_{\R^3} Q(E(u_F^{K_{el}}))(\chi_E- \chi_F)\, dx \notag\\
				& \geq \mathcal{E}(E(u_E^{K_{el}}))- \int_{\R^3}Q(E(u_F^{K_{el}}))  \circ \pi_{\pa E} (\chi_E-\chi_F)\, d x\notag\\
				& \quad +\int_{\R^3} \big(    Q(E(u_F^{K_{el}}))- Q(E(u_F^{K_{el}})) \circ \pi_{\pa E}\big) ( \chi_E-\chi_F) \, d x\notag\\
				& \geq \mathcal{E}(E(u_E^{K_{el}}))-CK_{el} d_{H^{-1}}(F,E)\label{120226formz1}\\
				& \quad +\int_{\R^3} \big(    Q(E(u_F^{K_{el}}))- Q(E(u_F^{K_{el}})) \circ \pi_{\pa E}\big) ( \chi_E-\chi_F) \, d x.\notag
		\end{align}
		We observe that
		\begin{equation}
			\vert Q(E(u_F^{K_{el}}))(x)- Q(E(u_F^{K_{el}})) \circ \pi_{\pa E}(x) \vert \leq C
			K_{el} \vert x-\pi_{\pa E}(x)\vert ,
		\end{equation}
		where $C>0$ is a universal constant.
		Using this inequality together with the coarea formula and arguing as in the previous claim, we obtain
		\begin{align}
				& \int_{\R^3}(Q ( E (u_F^{K_{el}}))-Q ( E (u_F^{K_{el}}))\circ \pi_{\pa E} )(\chi_E- \chi_F)\,dx \notag\\
				&\leq  CK_{el}\int_{\R^3} \vert x- \pi_{\pa E}(x)\vert \chi_{E \Delta F}(x)\, dx \notag\\
				& \leq C \int_{\pa E} \vert \psi(x)\vert  \int_{0}^{\vert \psi (x)\vert } (1+ t H_E(x)+t^2 K_E(x))\, dt \, d \mathcal{H}^2_x \\
				& \leq C  \| \xi_{F,E}\|^2_{L^2(\pa E)}  \leq \frac{C}{\eta} \| \xi_{F,E}\|_{H^{-1}(\pa E)}^2+ \eta C \| \nabla_{\pa E} \psi \|_{L^2(\pa E)}^2.\label{z18032025form2}
		\end{align}
		 Combining~\eqref{120226formz1} and~\eqref{z18032025form2}, and choosing $\eta$ suitably, we deduce~\eqref{120226tesi2}.
		
		Finally,  from~\eqref{120226tesi1} and~\eqref{120226tesi2} we conclude the proof of~\eqref{H-1lambminperfunz}.\\ \emph{Step 2}: We proceed to show the almost-minimality estimate~\eqref{formulaimpostante}. 
        
		We define the functional 
		\begin{equation}
			\mathfrak{G}(F):= \mathcal{G}(F)+ \Lambda_2 d_{H^{-1}}(F,E), \quad F \Delta E \subset {\rm cl} \big(\mathcal{I}_{\delta}(\pa E)\big),
		\end{equation}
		where $\Lambda_2$ is the constant appearing in~\eqref{H-1lambminperfunz} and $\delta \in (0,\min \{\frac{r_0}{4},\delta_1\})$. By the direct method of the calculus of variations, the functional $\mathfrak{G}$ admits a minimizer, which we denote by $F_{min}$. Moreover, by Lemma~\ref{lemmaH^-1}, we have $\vert F_{min} \vert=\vert E \vert$. We observe that proving~\eqref{formulaimpostante} is equivalent to proving that there exists $\delta$ such that $F_{min}=E$.\\
		\textit{Claim 5:} The set $F_{min}$ is an $(\Lambda,\alpha)$-minimizer of the anisotropic perimeter for any $\alpha \in (0,\frac{1}{3})$, with $\Lambda=\Lambda(r_0,K_0,\sigma_E,K_{el},\alpha)$.
		
		According to Lemma~\ref{lambdaminCONstrain}, the claim follows once it is shown that $F_{min}$ is a $(\delta,E)$-constrained $(\Lambda,\alpha)$-minimizer of the anisotropic perimeter. Let $G\subset \R^3$ be a set of finite perimeter such that \begin{equation}\label{12022026vincolo}
			G \Delta F_{min} \subset {\rm cl} \big( \mathcal{I}_{\delta}(\pa E) \big) \text{ and } \vert G \vert = \vert F_{min} \vert.
		\end{equation}
		By Lemma~\ref{lemmaH^-1}, we have that $d_{H^{-1}}(G,E) < +\infty$, since $\vert G \vert=\vert F_{min}\vert= \vert E \vert$. Let $f_G \in H^{1}(\pa E)$ be the function realizing the supremum in the definition of $d_{H^{-1}}(G,E)$. Then Lemma~\ref{lemmaH^-1} also implies that $f_G= v_{G,E} /d_{H^{-1}}(G,E) $ and $\int_{\pa E} f_G \, d \mathcal{H}^2=0$. Therefore, by the Sobolev embedding theorem, we have for $p >2$
		\begin{equation}
			\| f_G \|_{L^p(\pa E)} \leq C_p \| \nabla_{\pa E} f_G \|_{L^2(\pa E)} \leq C_p.
		\end{equation}
		Using this estimate and the coarea formula, we obtain
		\begin{align}
				\| f_G \circ \pi_{\pa E}\|_{L^p(\mathcal{I}_\delta(\pa E))}^p= \int_{\pa E} \vert f_G \vert^p \int_{-\delta}^{\delta} \big(1+ H_E t +K_E t^2 \big) \, d t d \mathcal{H}^2 \leq C_p.\label{12022026dopo1}
		\end{align}
		Recalling the definitions of $\mathfrak{G}$ and $\mathcal{G}$, by the minimality of $F_{min}$, we obtain 
		\begin{multline}\label{12022026pcrossfit-1}
			P_{\varphi}(F_{min})-P_\varphi(G) \leq \\ \Lambda_2 \big(  d_{H^{-1}}(G,E)-d_{H^{-1}}(F_{min},E) \big) + \mathcal{E}(E(u_G^{K_{el}}))-\mathcal{E}(E(u_{F_{min}}^{K_{el}})) .
		\end{multline}
		We first consider  the difference of the elastic terms
		\begin{align}
				\mathcal{E}(E(u_{F_{min}}^{K_{el}}))-\mathcal{E}(E(u_G^{K_{el}}))&=  \int_{\Omega \setminus F_{min}} Q(E(u_{F_{min}}^{K_{el}})) \,dx - \int_{\Omega \setminus G} Q(E(u_G^{K_{el}})) \,dx\notag\\
				& \leq \int_{\Omega \setminus F_{min} } Q(E(u_G^{K_{el}})) \,dx - \int_{\Omega \setminus G} Q(E(u_G^{K_{el}})) \,dx \notag\\
				& \leq  K_{el} \vert G \Delta F_{min} \vert, \label{26032025form2zz}
		\end{align}
		where the first inequality follows from the minimality of  $u_{F_{min}}^{K_{el}}$. We now estimate the difference 
		\begin{align}
				&d_{H^{-1}}(G,E)-d_{H^{-1}}(F_{min},E) \notag\\
				& \leq \int_{\R^3} f_G \circ \pi_{\pa E} (\chi_G-\chi_E)\, d x- \int_{\R^3} f_G \circ \pi_{\pa E} ( \chi_{F_{min}}-\chi_E ) \, d x \notag\\
				& = \int_{\R^3} f_G \circ \pi_{\pa E} (\chi_G-\chi_{F_{min}})\, d x \notag\\
				& \leq \| f_G \circ \pi_{\pa E} \|_{L^p(\mathcal{I}_\delta(\pa E))} \| \chi_G -\chi_{F_{min}} \|_{L^{\frac{p}{p-1}}(\R^3)}\notag \\
				& \leq C_p \vert G \Delta F_{min} \vert^{1-\frac{1}{p}},\label{12022026primacorss1}
		\end{align}
		where in the last inequality we used~\eqref{12022026dopo1}. Combining the above inequalities, we obtain
		\begin{equation}
			P_\varphi(F_{min})- P_\varphi(G) \leq C \vert G \Delta F_{min} \vert^{1-\frac{1}{p}}
		\end{equation}
		for every set of finite perimeter $G\subset \R^3$ satisfying~\eqref{12022026vincolo}, which concludes the proof of the claim.\\
		\textit{Claim 6:} There exists a sufficiently small constant $\delta>0$ such that $F_{min}=E$.
		
		Using Lemma~\ref{epsilonregolarita}, we obtain, provided $\delta>0$ is chosen sufficiently small, that $\partial F_{min}$ is the normal graph over $\partial E$ with height function $\psi \in C^{1}(\pa E)$ satisfying $\| \psi \|_{C^1(\pa E)} \leq \delta$. Therefore, by \emph{Step 1} and the minimality of $F_{min}$, we obtain
		\begin{equation}
			\frac{J_\varphi}{4} \| \nabla_{\pa E} \psi \|^2_{L^2(\pa E)}+ \mathcal{G}(E) \leq \mathcal{G}(F_{min})+\Lambda_2 d_{H^{-1}}(F_{min},E)= \mathfrak{G}(F_{min}) \leq \mathfrak{G}(E)=\mathcal{G}(E).
		\end{equation}
		Hence, $\psi$ is a constant function, but recalling that $\vert F_{min} \vert = \vert E \vert$, we conclude that $\psi \equiv 0$, i.e. $E=F_{min}$.
\end{proof}

\section{Preliminary estimates}\label{SEZIONEDELLESTIME}
The aim of this section is to establish regularity estimates for the set $F$ minimizing the incremental problem
\begin{equation}\label{ProblemaINc}
    \min \left\{  \mathcal{F}_{h}(A,E) : A \text{ measurable set, }  A \Delta E \subset \mathcal{I}_\eta(\pa E) \right\}
\end{equation}
under suitable assumptions on $ E \subset \Omega$ and on $ \eta>0$.  We recall that $\mathcal{I}_\eta$ was defined in~\eqref{intornotubolare}. 
\subsection{First regularity estimate} Throughout this subsection, we assume that the set $E \Subset \Omega$ is $C^5$-regular and 
\begin{equation}\label{13022026pcrossfit1}
  E \text{ satisfies the \textsc{UBC} with radius } r_0  \text{ and }\| H_E \|_{H^2(\pa \Sigma)} \leq K_0,
\end{equation}
for some $K_0>0$. Observe that these assumptions imply that $E$ is uniformly $C^{2,\gamma}$-regular for every $\gamma \in (0,1)$.

\begin{proposition}\label{CaneBernard06}
		Let $E \Subset \Omega$ be a $C^5$-regular set such that~\eqref{13022026pcrossfit1} holds. Let $F \Subset \Omega$ be a minimizer of~\eqref{ProblemaINc} for $\eta< \delta_1$, where $\delta_1$ is the constant given by Proposition~\ref{ordinaryprop}. Then there exist $C_0=C_0(r_0,\sigma_E,K_0,K_{el})>0$ and $h_0=h_0(r_0,\sigma_E,K_0,K_{el}, \eta)>0$ such that 
		\begin{equation}\label{dH-1hhh}
			d_{H^{-1}}(F,E) \leq C_0 h
		\end{equation}
		for every $h \leq h_0$. Moreover, $F$ is a $(\Lambda,\alpha)$-minimizer of the anisotropic perimeter for every $\alpha \in (0,\frac{1}{3})$ with $\Lambda= \Lambda(r_0,\sigma_E,K_0, \alpha,K_{el})$. In addition, if $\eta< \delta_2= \delta_2(r_0,\sigma_E)$, then $F$ satisfies $ \| \nabla_{\pa F} H_F^\varphi \|_{L^2(\pa F)} \leq C_0$, and $\pa F$ is the normal graph over $\pa E$ with height function $\psi \in C^1(\pa E)$ such that, for every $h \leq h_0$, 
		\begin{equation}\label{1803pons}
			\begin{split}
				&\|\psi\|_{L^2(\partial E)}\le C_0 h^{\frac{3}{4}}, \quad \| \psi \|_{H^1(\pa E)} \leq C_0 h^{\frac{1}{2}}, \quad \| \psi \|_{C^1(\pa E)} \leq C_0 h^{\frac{1}{9}}, \\
				&\text{for every }\varepsilon>0 \, \text{ exits } C_\varepsilon>0 : \| \psi \|_{L^\infty(\pa E)} \leq C_\varepsilon h^{\frac{1}{2}-\varepsilon}.
			\end{split}
		\end{equation}
	\end{proposition}
	\begin{proof}
    By~\eqref{formulaimpostante} and the minimality of $F$, we have 
		\begin{equation}\label{dH-2hhh}
			\mathcal{G}(F)+ \frac{1}{2h}d_{H^{-1}}(F,E)^2 \leq \mathcal{G}(E) \leq \mathcal{G}(F)+ \Lambda_2 d_{H^{-1}}(F,E).
		\end{equation}
		Therefore, inequality~\eqref{dH-1hhh} follows with $C_0 = 2\Lambda_2$. We divide the proof into two steps.\\
		\textit{Step 1:} We prove that $F$ is a $(\Lambda,\alpha)$-minimizer of the anisotropic perimeter for all $\alpha \in (0, \frac{1}{3})$ and with $\Lambda= \Lambda(r_0,\sigma_E,K_0, \alpha,K_{el})$, i.e.
		\begin{equation}
			P_\varphi(F) \leq P_\varphi(G)+ \Lambda \vert F \Delta G \vert^{1-\alpha} \text{ for all } G \subset \R^3.
		\end{equation}
		
        We claim that, for every $ G \subset \R^3 $ such that $ G \Delta E \subset \mathcal{I}_\eta (\pa E)$ and $\vert G \vert =1$, the following inequality holds:
		\begin{equation}\label{22022026primpale}
			\mathcal{G}(F) \leq \mathcal{G}(G)+ 3 \Lambda_2 \big( d_{H^{-1}}(G,E)- d_{H^{-1}}(F,E)  \big).
		\end{equation}
		We first assume that $d_{H^{-1}}(G,E) > 4 \Lambda_2 h$. 
		Then, by~\eqref{dH-2hhh}, it holds that $2 d_{H^{-1}}(F,E) \leq 4\Lambda_2 h < d_{H^{-1}}(G,E)$,
		which implies
		\begin{equation*}
			\frac{1}{2}d_{H^{-1}}(G,E) \leq d_{H^{-1}}(G,E)- d_{H^{-1}}(F,E).
		\end{equation*}
		Therefore, using the minimality of $F$ together with~\eqref{formulaimpostante}, we deduce
		\begin{align*}
				\mathcal{G}(F) \leq  &  \mathcal{G}(E) 
				\leq \mathcal{G}(G)+  \Lambda_2 d_{H^{-1}}(G,E) \\
				\leq & \mathcal{G}(G)+ 2 \Lambda_2 \big( d_{H^{-1}}(G,E)-d_{H^{-1}}(F,E)  \big).
		\end{align*}
		Now, we assume $d_{H^{-1}}(G,E) \leq 4 \Lambda_2 h$. 
		Using the minimality of $F$ and~\eqref{dH-2hhh}, we obtain
		\begin{equation*}
			\begin{split}
				\mathcal{G}(F)-\mathcal{G}(G) \leq & \frac{1}{2h} \big(  d_{H^{-1}}(G,E)+d_{H^{-1}}(F,E)\big) \big( d_{H^{-1}}(G,E)-d_{H^{-1}}(F,E) \big)\\
				\leq & 3 \Lambda_3 \big(d_{H^{-1}}(G,E)-d_{H^{-1}}(F,E) \big).
			\end{split}
		\end{equation*}
		This establishes~\eqref{22022026primpale}.

		Thanks to Lemma~\ref{lambdaminCONstrain}, it is enough to prove that $F$ is a $(\eta,E)$-constrained $(\Lambda,\alpha)$-minimizer of the anisotropic perimeter. Let $G\subset \R^3$ be a set of finite perimeter such that
			$G \Delta F \subset {\rm cl} \big( \mathcal{I}_{\eta}(\pa E) \big) \text{ and } \vert G \vert = \vert F \vert.$
		Using the definition of $\mathcal{G}$ (see~\eqref{defG}) and~\eqref{22022026primpale}, we obtain
		\begin{equation}
			P_\varphi(F)-P_\varphi(G) \leq \mathcal{E}(E(u_F^{K_{el}}))- \mathcal{E}(E(u_G^{K_{el}}))+ 3\Lambda_2 \big(d_{H^{-1}}(G,E)-d_{H^{-1}}(F,E) \big).
		\end{equation}
		We conclude by applying the same arguments as in~\eqref{26032025form2zz} to estimate $\mathcal{E}(E(u_F^{K_{el}})) - \mathcal{E}(E(u_G^{K_{el}}))$, and those in~\eqref{12022026primacorss1} to control $d_{H^{-1}}(G,E) - d_{H^{-1}}(F,E)$. \\
		\textit{Step 2:} We prove the remaining assertions of the statement.

        Let $\delta_2 := \delta_0(\delta_1)$ be the constant provided by Lemma~\ref{epsilonregolarita} for $\varepsilon = \delta_1$, where $\delta_1$ is given in Proposition~\ref{ordinaryprop}.
        If $\eta<\delta_2$, we can apply Lemma~\ref{epsilonregolarita}, which yields that $\pa F$ is the normal graph over $\pa E$  with height function $\psi \in C^{1,\frac{1}{3}}(\pa E)$ and $ \| \psi \|_{C^{1,\frac{1}{3}}(\pa E)} \leq \delta_1$. 
		Then, combining formula~\eqref{H-1lambminperfunz}, the minimality of $F$, and~\eqref{dH-1hhh}, we obtain
		\begin{equation}
			\frac{J_\varphi}{4} \| \nabla_{\pa E} \psi \|^2_{L^2(\pa E)}+ \mathcal{G}(E) \leq \mathcal{G}(F)+\Lambda_2 d_{H^{-1}}(F,E) \leq \mathcal{G}(E)+ C h,
		\end{equation}
		which implies $\| \nabla_{\pa E} \psi \|_{L^2(\pa E)} \leq C \sqrt{h}$.
		Using this estimate together with~\eqref{090226ppranzf1} and~\eqref{10022026pprinazof2}, we deduce
		\begin{align*}
				\| \psi \|_{H^1(\pa E)} &= \| \psi \|_{L^2(\pa E)}+ \| \nabla_{\pa E} \psi \|_{L^2(\pa E)} \leq 2 \| \xi_{F,E} \|_{L^2(\pa E)}+ \| \nabla_{\pa E} \psi \|_{L^2(\pa E)} \\
				&\leq C \| \nabla_{\pa E} \xi_{F,E} \|_{L^2(\pa E)}+ \| \nabla_{\pa E} \psi \|_{L^2(\pa E)} \\
				&\leq  C\| \nabla_{\pa E} \psi \|_{L^2(\pa E)} \leq C\sqrt{h},
		\end{align*}
		where $\xi_{F,E}= \psi + \frac{H_E}{2}\psi^2+ \frac{K_E}{3}\psi^3$, and in the second inequality we applied Poincaré’s inequality, noting that $\int_{\pa E} \xi_{F,E} \, d \mathcal{H}^2=0$. Fix $\varepsilon>0$. By  interpolation and the $C^1$-bound on $\psi$, we obtain  
		\begin{equation}
			\| \psi \|_{L^\infty(\pa E)} \leq C_\varepsilon \| \psi \|_{C^1(\pa E)}^{2 \varepsilon} \| \psi \|_{H^1(\pa E)}^{1-2 \varepsilon} \leq C_\varepsilon h^{\frac{1}{2}-\varepsilon},
		\end{equation}
        and similarly $\| \psi \|_{C^1(\pa E)} \leq C h^{\frac{1}{9}}$ follows.
        Moreover, the bound $ \| \psi \|_{L^2(\pa E)} \leq C h^{\frac{3}{4}} $
        follows from~\eqref{gucciniintepol},~\eqref{090226ppranzf1} and~\eqref{dH-1hhh}.

	    For $h$ sufficiently small we deduce that $F \triangle E \Subset \mathcal{I}_\eta(\pa E)$. Hence, we can compute the Euler-Lagrange equation of the functional $ F \mapsto \mathcal{F}_h(F,E)$. Applying~\eqref{FeulPANI}, Proposition~\ref{propELDH-1}, and Proposition~\ref{PropelELASt}, we obtain
		\begin{equation}\label{guccinone1}
			H^\varphi_F(y)-Q(E(u_F^{K_{el}}))(y)+ \frac{d_{H^{-1}}(F,E)}{h}f \circ \pi_{\pa E}(y)=L \quad \text{ for all } y \in \pa F,
		\end{equation}
		where $f \in H^1(\pa E)$ is the function that realizes the supremum in formula~\eqref{d_{H^{-1}}f} and $L$ is the Lagrange multiplier. Recalling that $Q(E(u_F^{K_{el}})) \in C^{2,\frac{1}{4}}(\Omega)$ (see~\eqref{minelast}) and that $f \in C^{3,\gamma}(\pa E)$, since $f$ solves ~\eqref{eqv} (see also~\eqref{lafunzcherelsupH^-1}), we deduce that $F$ is $C^{4,\gamma}$-regular for any $\gamma \in (0,1)$. We emphasize, however, that the $C^4$-regularity depends on $h$. Now, we prove that $ \| \nabla_{\pa F}(f \circ \pi_{\pa E}) \|_{L^2(\pa F)} \leq C(K_0)$. Indeed, by~\eqref{regdecat12},~\eqref{canzone1}, and Poincar\'e's inequality we infer
		\begin{align*}
				\int_{\pa F} \vert \nabla_{\pa F} (f \circ \pi_{\pa E}) \vert^2 \, d \mathcal{H}^2 & \leq \int_{\pa F} \vert \nabla_{\pa F} \pi_{\pa E} \vert^2 \vert \nabla_{\pa E} f \vert^2 \circ \pi_{\pa E} \, d \mathcal{H}^2 \\
				& \leq \int_{\pa F} \vert \nabla \pi_{\pa E} P_{\pa F} \vert^2 \vert \nabla_{\pa E} f \vert^2 \circ \pi_{\pa E} \, \mathcal{H}^2\\
				& \leq C(K_0) \int_{\pa E} \vert \nabla_{\pa E} f \vert^2 J_{\pa E} \Psi \, d \mathcal{H}^2 \\
				& \leq C(K_0) \int_{\pa E} \vert f \vert^2 \, d \mathcal{H}^2 \leq C(K_0).
		\end{align*}
		Here we denoted by $\Psi : \partial E \to \partial F$ the map $ \Psi(x):=x + \psi(x) \nu_E(x)$, while $J_{\pa E} \Psi$ is defined in~\eqref{evvrailtempo1}, with $ \| J_{\pa E} \Psi \|_{L^\infty(\pa E)} \leq C(K_0)$. 
		Therefore, estimate $ \| \nabla_{\pa F} H_F^\varphi \|_{L^2(\pa F)} \leq C$ follows using also~\eqref{guccinone1}, $d_{H^{-1}}(F,E) \leq C h$, and the regularity $Q(E(u_F^{K_{el}})) \in C^{2,\frac{1}{4}}(\Omega) $. 
	\end{proof}

	An important consequence of Proposition~\ref{CaneBernard06} is that any minimizer of the incremental problem~\eqref{ProblemaINc} satisfies $F \triangle E \Subset \mathcal{I}_\eta(\pa E)$ for every $h \leq h_0$. This allows us to compute the Euler-Lagrange equation for the minimizer of $F$. Combining~\eqref{eqv},~\eqref{lafunzcherelsupH^-1}, and~\eqref{guccinone1}, we obtain
	\begin{equation}\label{Comeprima}
		\left\{
		\begin{aligned}
			& H_F^\varphi- Q(E(u_F^{K_{el}}))+ \frac{d_{H^{-1}}(F,E)}{h} f \circ \pi_{\pa E}= L \quad \text{ on } \pa F \\
			& - \Delta_{\pa E} f = \frac{\xi_{F,E}}{d_{H^{-1}}(F,E)}  \quad \text{ on } \pa E
		\end{aligned}
		\right.
	\end{equation}
	where $\xi_{F,E}$ is the function defined in~\eqref{xiFEfunz}. 
	\begin{remark}\label{szzz}
    {\rm 
		If $E$ is $C^{5,\gamma}$-regular for some $\gamma \in (0,1)$, then, recalling that $u_{F}^{K_{el}} \in C^{3, \frac{1}{4}}(\Omega)$ (see formula~\eqref{minelvinc}), by elliptic regularity, we can conclude that $F$ is $C^{4,\min\{\gamma, \frac{1}{4} \}}$-regular. Observe that these estimates are not uniform with respect to $h$.   \fr}
	\end{remark} 
	By Corollary~\ref{zonzo}, we may combine the two equations in~\eqref{Comeprima} to obtain
	\begin{equation}\label{PeerGynt}
		\begin{split}
			\frac{1}{h} \bigg( \psi+ \frac{H_E}{2}\psi^2+ \frac{K_E}{3} \psi^3 \bigg)(x)=&- \Delta_{\pa E} \big(\div_{\pa E} \big( \mathcal{A}(\nu_E) \nabla_{\pa E} \psi \big) \big)(x)+ \Delta_{\pa E} H_E^\varphi(x)\\
			&- \Delta_{\pa E} \big(Q(E(u_F^{K_{el}}))(x+ \psi(x)\nu_E) \big)+ \Delta_{\pa E} \widetilde{\mathcal{R}}_0(x),
		\end{split}
	\end{equation}
	where $\mathcal{A}$ is a positive definite $(1,1)$-tensor, see formula~\eqref{07032026posdef}, and $\widetilde{\mathcal{R}}_0$ is defined in~\eqref{nuovoresto}.

\subsection{Second regularity estimate}
Throughout this subsection, we assume that the set $E \Subset \Omega$ is $C^6$-regular, and we set $\Sigma= \pa E$. Moreover, we assume that 
\begin{equation}\label{AnitraD}
\begin{split}
      &E \text{ satisfies the \textsc{UBC} with radius } r_0  \text{ and }\| \nabla_{\pa E}\Delta_{\pa E}H_E \|_{L^2(\pa \Sigma)} \leq K_0, \\
    &  E \in \mathfrak{H}^{5}_{K_0, \sigma_1}(E_0),
      \end{split}
\end{equation}
for some $K_0>0$, where $\sigma_1 \leq \frac{\sigma_0}{2}$ and $\sigma_0$ is defined in~\eqref{Laconstantesigmazerofin}. 

In what follows, we adopt the notation $S_1 \star S_2$ from~\cite{Hamilton1982, Mantegazza2002} to denote a tensor obtained by contracting some indices of the tensors $S_1$ and $S_2$ using the coefficients of the metric tensor $g_{ij}$. This notation is convenient since it implies $$ \vert S_1 \star S_2 \vert \leq  C \vert S_1 \vert \vert S_2 \vert, $$
where $C$ is a constant depending on the tensor ranks of $S_i \in T^{(h_i,k_i)}(T \Sigma)$, with $h_i,k_i \in \N$.

\subsubsection{Estimate of the $H^3$-norm for the heightfunction}
The aim of this subsubsection is to obtain a uniform $H^3(\Sigma)$-estimate for the function $\psi$ that solves~\eqref{PeerGynt}. 

In the next lemma, we compute the commutator between the Laplace–Beltrami operator and the tangential divergence, as well as the commutator between the Laplace–Beltrami operator and the tangential gradient. This will allow us to rewrite equation~\eqref{PeerGynt}, which will be useful in the main proof of this subsubsection.
\begin{lemma}\label{Lemmacommutatoredifficile}
   Let $E \Subset \Omega$ be a $C^6$-regular set satisfying~\eqref{AnitraD}. Then the following identities hold:
   \begin{itemize}
       \item Let $f \in C^3(\Sigma)$, then 
       \begin{equation}\label{Eq:lemma:swapLapGrad}
           \Delta_{\pa E} \nabla_{\pa E}f = \nabla_{\pa E} \Delta_{\pa E} f+ a(\overline{\nabla}f \star B_\Sigma,\overline{\nabla}^2f),
       \end{equation}
       where $ a$ is a smooth vector field satisfying  $ a(\cdots)=0$ if $f=0$ and $f \mapsto a(\overline{\nabla}f \star B_\Sigma,\overline{\nabla}^2f) $ is linear.
       \item Let $X$ be a vector field on $\pa E$ of class $C^3$, then
       \begin{equation}\label{Eq:lemma:swapLapDiv}
            \Delta_{\pa E} \div_{\pa E} X= \div_{\pa E} \Delta_{\pa E} X+r_1(B_\Sigma \star \overline{\nabla}X, \overline{\nabla} B_\Sigma \star X,\overline{\nabla}^2X),
       \end{equation}
       where $r_1$ is a smooth function satisfying  $r_1(\cdots)=0$ if $X=0$ and $X \mapsto r_1(B_\Sigma \star \overline{\nabla}X, \overline{\nabla} B_\Sigma \star X,\overline{\nabla}^2X) $ is a linear function.
       \item  Let $X$ be a vector field of class $C^2$ and $A$ a $(1,1)$-tensor field of class $C^2$, then
    \begin{equation}\label{Eq:lemma:LapAX}
        \Delta_{\pa E}  \big(AX \big)= (\Delta_{\pa E}A)X+ A \Delta_{\pa E}X+ Z( A \star X \star B_\Sigma, \overline{\nabla }A \star \overline{\nabla} X),
    \end{equation}
    where $Z$ is a vector field satisfying  $Z(\cdots)=0$ if $X=0$ and $X \mapsto Z( A \star X \star B_\Sigma, \overline{\nabla }A \star \overline{\nabla} X)$ is a linear function.
   \end{itemize}
\end{lemma}
\begin{proof}
    The proof is provided in Appendix~\ref{appendix}, where each formula is established in a separate lemma.
\end{proof}
\begin{lemma}\label{nonstaropiu21}
    Let $E \Subset \Omega$ be a $C^6$-regular set such that~\eqref{AnitraD} holds. Let $\mathcal{A}$ denote the $(1,1)$-tensor field defined in~\eqref{07032026posdef}, and let $f \in C^5(\pa E)$. Then the following identity holds: 
    \begin{equation}\label{LapdivAnunabf}
    \begin{split}
        \Delta_{\pa E} \div_{\pa E} \big( \mathcal{A}(\nu_E) \nabla_{\pa E} f  \big)&= \div_{\pa E} \big( \mathcal{A}(\nu_E) \nabla_{\pa E} \Delta_{\pa E} f \big)+ \div_{\pa E} \big( \Delta_{\pa E} \mathcal{A}(\nu_E)  \nabla_{\pa E} f \big)\\
        & \quad+ \widehat{\mathcal{R}}(\overline{\nabla} B_\Sigma \star \overline{\nabla}^2 f, B_\Sigma \star\overline{\nabla}^3 f  ),
        \end{split}
    \end{equation}
    where $\widehat{\mathcal{R}}$ is a smooth function satisfying $ \widehat{\mathcal{R}}(0,0)=0$.
\end{lemma}
\begin{proof}
    Using formula~\eqref{Eq:lemma:swapLapDiv} together with the Leibniz rule, we obtain
    \begin{align*}
            \Delta_{\pa E} \div_{\pa E} \big(   \mathcal{A}(\nu_E) \nabla_{\pa E} f\big)&= \div_{\pa E} \big( \Delta_{\pa E} \big(\mathcal{A}(\nu_E) \nabla_{\pa E} f  \big)  \big)\\
            & \quad+{r}_1( B_\Sigma \star \overline{\nabla}^2 f,\overline{\nabla}B_\Sigma \star \overline{\nabla} f, \mathcal{A}(\nu_E) \star \overline{\nabla}^3 f).
    \end{align*}
    We note that lower-order terms should also be included in $r_1$; however, we omit them here to avoid overloading the notation. Now, ~\eqref{Eq:lemma:swapLapGrad},~\eqref{Eq:lemma:LapAX},  the chain rule and Leibniz rule imply
    \begin{equation*}
    \begin{split}
        \div_{\pa E} \big( \Delta_{\pa E} \big(\mathcal{A}(\nu_E) \nabla_{\pa E} f  \big)  \big)&=  \div_{\pa E} \big( \Delta_{\pa E} \mathcal{A}(\nu_E)\nabla_{\pa E} f + \mathcal{A}(\nu_E)  \Delta_{\pa E} \nabla_{\pa E} f   \big)  + \div_{\pa E}   Z  \\
        &=  \div_{\pa E} \big( \mathcal{A}(\nu_E) \nabla_{\pa E} \Delta_{\pa E} f \big)+ \div_{\pa E} \big( \mathcal{A}(\nu_E) a(\overline{\nabla}f \star B_\Sigma,\overline{\nabla}^2f) \big) \\
        & \quad + \div_{\pa E} \big(   \Delta_{\pa E} \mathcal{A}(\nu_E)  \nabla_{\pa E} f  \big)+ \div_{\pa E}  Z \\
        &= \div_{\pa E} \big( \mathcal{A}(\nu_E) \nabla_{\pa E} \Delta_{\pa E} f \big)+ \div_{\pa E} \big(   \Delta_{\pa E} \mathcal{A}(\nu_E)  \nabla_{\pa E} f  \big)\\
        & \quad + r_2(\overline{\nabla} B_\Sigma \star \overline{\nabla}^2 f, B_\Sigma \star\overline{\nabla}^3 f  ),
        \end{split}
    \end{equation*}
    where $r_2$ is a smooth function satisfying $r_2(0,0)=0$. We note that lower-order terms should also be included in $r_2$; however, they are omitted here to simplify the notation. Finally, combining the above equalities, we obtain~\eqref{LapdivAnunabf}.
\end{proof}
\begin{lemma}\label{H3mcurvimpH3acurv}
Let $E \Subset \Omega$ be a $C^6$-regular set such that~\eqref{AnitraD} holds, then 
\begin{equation}
    \|  \nabla_{\pa E} \Delta_{\pa E} H^\varphi_E \|_{L^2(\pa E)} \leq C(K_0,r_0), \quad \| H_E^\varphi \|_{H^3(\pa E)} \leq C(K_0,r_0).
\end{equation}
\end{lemma}
\begin{proof}
   The proof is analogous to that of \emph{Claim 1} in Proposition~\ref{ordinaryprop}.
\end{proof}
\begin{lemma}\label{Paperplane}
    Let $E \Subset \Omega$ be a $C^6$-regular set such that~\eqref{AnitraD} holds. Let $F \subset \Omega$ be a minimizer of~\eqref{ProblemaINc} for $\eta < \delta_2$, where $\delta_2$ is the constant provided by Proposition~\eqref{CaneBernard06}.  Then there exists $h_1 \leq h_0$, where  $h_0$ is the constant provided by~\eqref{CaneBernard06}, such that the solution of~\eqref{PeerGynt} satisfies
    \begin{equation}\label{stimeinH^3}
     \| \psi \|_{H^2(\pa E)} \leq C_0 h^{\frac{1}{4}}, \qquad  \| \psi \|_{H^3(\pa E)} \leq C_0
    \end{equation}
    for every $h \leq h_1$ and where $C_0=C_0(r_0,\sigma_E,K_0,K_{el},\delta_2,P(E))$.
\end{lemma}
\begin{proof}
In order to simplify notation, we identify $B_E$ with $B_\Sigma$, and we omit the dependence of $C$  on $K_0$ and $r_0$. 

We observe that it suffices to show that $\|\psi\|_{H^3(\partial E)}\le C_0$,  then the estimate for $ \| \psi \|_{H^2(\pa E)}$ will follow from Proposition~\ref{prop:interpolation}.
By the assumptions on $E$ and Lemma~\ref{lem:hilbert-norm}, we have
\begin{equation}
    \| B_E \|_{H^3(\pa E)} \leq C (1 + \| \nabla_{\pa E} \Delta_{\pa E}  H_{\pa E} \|_{L^{2}(\pa E)} ).
\end{equation}
Hence, by the Sobolev embedding theorem, we obtain
\begin{equation}\label{NormC^1alphB_E}
    \| B_E \|_{C^{1,\alpha}(\pa E)} \leq C  \text{ for every } \alpha \in (0,1) .
\end{equation}
Combining~\eqref{PeerGynt} and~\eqref{LapdivAnunabf} with $f = \psi$, we obtain
\begin{align}
   \frac{1}{h} \bigg( \psi+ H_E \frac{\psi^2}{2} + K_E \frac{\psi^3}{3} \bigg)(x)=& -\div_{\pa E} \big( \mathcal{A}(\nu_E) \nabla_{\pa E} \Delta_{\pa E} \psi   \big)(x)+ \Delta_{\pa E} H^\varphi_E (x) \label{NewPeerGynt}\\
   & -\div_{\pa E} \big( \big(\Delta_{\pa E} \mathcal{A}(\nu_E) \big) \nabla_{\pa E} \psi \big)(x)\notag\\
   &- \widehat{\mathcal{R}}(\overline{\nabla} B_E \star \overline{\nabla} \psi, B_E \star \overline{\nabla}^3 \psi)(x)\notag\\
   &- \Delta_{\pa E} \big(Q(E(u_F^{K_{el}}))(x+ \psi(x)\nu_E)\big) + \Delta_{\pa E} \widetilde{\mathcal{R}}_0(x).\notag
\end{align}
Multiplying  by $-\Delta_{\partial E} \psi$, integrating by parts on $\partial E$, and rearranging terms, we deduce
\begin{equation}\label{eq:sistint}
    \begin{split}
        \frac{1}{h} \int_{\pa E} \vert \nabla_{\pa E} \psi \vert^2 + \int_{\pa E} \mathcal{A}(\nu_E) \nabla_{\pa E}& \Delta_{\pa E} \psi \cdot \nabla_{\pa E} \Delta_{\pa E} \psi = \int_{\pa E} \nabla_{\pa E} \Delta_{\pa E} H^{\varphi}_E \cdot \nabla_{\pa E} \psi \\
        & -\frac{1}{h} \int_{\pa E} \nabla_{\pa E} \Big(  H_E \frac{\psi^2}{2}+ K_E \frac{\psi^3}{3}\Big) \cdot \nabla_{\pa E} \psi \\
        &- \int_{\pa E} \big(  \Delta_{\pa E} \mathcal{A}(\nu_E)\big) \nabla_{\pa E} \psi \cdot \nabla_{\pa E} \Delta_{\pa E} \psi \\
        & + \int_{\pa E} \widehat{\mathcal R}(\overline{\nabla} B_E \star \overline{\nabla} \psi, B_E \star \overline{\nabla}^3 \psi) \Delta_{\pa E} \psi \\
        & - \int_{\pa E} \nabla_{\pa E} \big(Q(E(u_F^{K_{el}}))(x+ \psi(x)\nu_E(x))  \big)  \cdot \nabla_{\pa E} \Delta_{\pa E} \psi \\
        & + \int_{\pa E} \nabla_{\pa E} \widetilde{\mathcal{R}}_0 \cdot \nabla_{\pa E} \Delta_{\pa E } \psi .
    \end{split}
\end{equation}
We now proceed to estimate the right-hand side of~\eqref{eq:sistint}. Let $\varepsilon>0$ to
be chosen later. \\
\textit{Estimate of $ \int_{\pa E} \nabla_{\pa E} \Delta_{\pa E} H^{\varphi}_E \cdot \nabla_{\pa E} \psi $.}\\
By Lemma~\ref{H3mcurvimpH3acurv} and Young’s inequality, we obtain
\begin{align*}
   \int_{\pa E} \nabla_{\pa E} \Delta_{\pa E} H^{\varphi}_E \cdot \nabla_{\pa E} \psi &\leq C(\varepsilon) \| \nabla_{\pa E} \Delta_{\pa E} H^{\varphi}_E \|_{L^2(\pa E)}^2 + \varepsilon \| \nabla_{\pa E } \psi \|^2_{L^2(\pa E)} \notag\\
   &\leq C + \varepsilon \| \nabla_{\pa E } \psi \|^2_{L^2(\pa E)}.
\end{align*}
\textit{Estimate of $ -\frac{1}{h} \int_{\pa E} \nabla_{\pa E} \Big(  H_E \frac{\psi^2}{2}+ K_E \frac{\psi^3}{3}\Big) \cdot \nabla_{\pa E} \psi $.}\\
We have that 
\begin{equation*}
    \nabla_{\pa E} \Big( H_E \frac{\psi^2}{2}+ K_E \frac{\psi^3}{3}\Big)= \psi \nabla_{\pa E} \psi \big( H_E+ K_E \psi  \big)+ \psi^2 \Big( \frac{\nabla_{\pa E} H_E}{2}
    + \psi\frac{\nabla_{\pa E} K_E}{3}\Big).
\end{equation*}
Combining the above expression with~\eqref{1803pons} and~\eqref{NormC^1alphB_E}, we obtain
\begin{equation}\label{22piovani1}
    \left\|  \nabla_{\pa E}  \Big( H_E \frac{\psi^2}{2}+ K_E \frac{\psi^3}{3}\Big)\right\|_{L^2(\pa E)} \leq C h^{\frac{1}{2}}.
\end{equation}
Hence, using~\eqref{1803pons} and the Cauchy-Schwarz inequality, we conclude that
\begin{equation*}
    -\frac{1}{h} \int_{\pa E} \nabla_{\pa E} \Big(  H_E \frac{\psi^2}{2}+ K_E \frac{\psi^3}{3}\Big) \cdot \nabla_{\pa E} \psi \leq  C.
\end{equation*}
\textit{Estimate of $ - \int_{\pa E} \big(  \Delta_{\pa E} \mathcal{A}(\nu_E)\big) \nabla_{\pa E} \psi \cdot \nabla_{\pa E} \Delta_{\pa E} \psi$.}\\
Using Young's inequality, Lemma~\ref{pensobene},~\eqref{1803pons} and~\eqref{NormC^1alphB_E}, we obtain
\begin{align*}
   -\int_{\pa E} \big(  \Delta_{\pa E} \mathcal{A}(\nu_E)\big) \nabla_{\pa E} \psi \cdot \nabla_{\pa E} \Delta_{\pa E} \psi &\leq C(\varepsilon) \| \psi \|^2_{C^1(\pa E)} \| \mathcal{A}(\nu_E) \|_{H^2(\pa E)}^2+\varepsilon \|  \nabla_{\pa E} \Delta_{\pa E} \psi \|_{L^2(\pa E)}^2\\
   & \leq C(\varepsilon) \| B_E \|_{H^1(\pa E)}^2+ \varepsilon\|  \nabla_{\pa E} \Delta_{\pa E} \psi \|_{L^2(\pa E)}^2 \\
   & \leq C+ \varepsilon\|  \nabla_{\pa E} \Delta_{\pa E} \psi \|_{L^2(\pa E)}^2.
\end{align*}
\textit{Estimate of $\int_{\pa E} \widehat{\mathcal R}(\overline{\nabla} B_E \star \overline{\nabla} \psi, B_E \star \overline{\nabla}^3 \psi) \Delta_{\pa E} \psi $.}\\
Since $\widehat{R}$ is smooth, by~\eqref{NormC^1alphB_E}, we obtain
\begin{align*}
   |\widehat{\mathcal R}(\overline{\nabla} B_E \star \overline{\nabla} \psi, B_E \star \overline{\nabla}^3 \psi) \vert &\leq C\big(  \| B_E \|_{C^1(\pa E)}  \big) \big(   1+ \vert \overline{\nabla} \psi  \vert + \vert \overline{\nabla}^2 \psi  \vert+ \vert \overline{\nabla}^3 \psi  \vert \big) \\
   &  \leq C\big(   1+ \vert \overline{\nabla} \psi  \vert + \vert \overline{\nabla}^2 \psi  \vert+ \vert \overline{\nabla}^3 \psi  \vert \big).
\end{align*}
Therefore, by Young's inequality, we deduce
\begin{align*}
    \int_{\pa E} \widehat{\mathcal R}(\overline{\nabla} B_E \star \overline{\nabla} \psi, B_E \star \overline{\nabla}^3 \psi) \Delta_{\pa E} \psi \leq&\frac{\varepsilon}{2} \| \psi \|_{H^3(\pa E)}^2+ C(\varepsilon)\| \Delta_{\pa E} \psi \|^2_{L^2(\pa E)} \\
    \leq &   \varepsilon  \| \psi \|_{H^3(\pa E)}^2+  C(\varepsilon) \| \psi \|_{H^1(\pa E)}^2                  \\
    \leq & \varepsilon \| \nabla_{\pa E} \Delta_{\pa E} \psi \|_{L^2(\pa E)}^2+ C, 
\end{align*}
where in the second inequality we used Proposition~\ref{prop:interpolation}, and in last inequality we used Lemma~\ref{lem:hilbert-norm} combined with~\eqref{1803pons} and~\eqref{AnitraD}.\\
\textit{Estimate of $-\int_{\pa E} \nabla_{\pa E} \big(Q(E(u_F^{K_{el}}))(x+ \psi(x)\nu_E(x))  \big)  \cdot \nabla_{\pa E} \Delta_{\pa E} \psi$.}\\
We have that
\begin{align*}
        &\nabla_{\pa E} \big(Q(E(u_F^{K_{el}}))(x+ \psi(x)\nu_E(x))  \big) = \nabla Q(E(u_F^{K_{el}})) (x+ \psi(x)\nu_E(x)) \nabla_{\pa E} \big( x+ \psi(x)\nu_E(x) \big)\\
        &=   \nabla Q(E(u_F^{K_{el}})) (x+ \psi(x)\nu_E)  \big( P_{\pa E}(x)+ \nu_E(x)\otimes\nabla_{\pa E} \psi(x)+ \psi(x)B_E(x)  \big).
\end{align*}
 We recall that $u_{F}^{K_{el}} \in \mathfrak{C}_{K_{el}}^{3,\frac{1}{4}}(\Omega, \R^3)$, see~\eqref{minelvinc}. Hence, using also~\eqref{1803pons}, we obtain
 \begin{equation*}
    \|  \nabla_{\pa E} \big(Q(E(u_F^{K_{el}}))(\cdot+ \psi(\cdot)\nu_E)  \big) \|^2_{L^2(\pa E)}  \leq C(K_{el}). 
 \end{equation*}
 Next, by Young's inequality and the previous estimate, we obtain
 \begin{align*}
  -\int_{\pa E} &\nabla_{\pa E} \big(Q(E(u_F^{K_{el}}))(x+ \psi(x)\nu_E)  \big)  \cdot \nabla_{\pa E} \Delta_{\pa E} \psi   \\
  & \leq  C(\varepsilon)\|  \nabla_{\pa E} \big(Q(E(u_F^{K_{el}}))(\cdot+ \psi(\cdot)\nu_E)  \big) \|^2_{L^2(\pa E)} + \varepsilon \| \nabla_{\pa E} \Delta_{\pa E} \psi \|_{L^2(\pa E)}^2 \\
  & \leq C+ \varepsilon \| \nabla_{\pa E} \Delta_{\pa E} \psi \|_{L^2(\pa E)}^2.
 \end{align*}
 \textit{Estimate of $\int_{\pa E} \nabla_{\pa E} \widetilde{\mathcal{R}}_0 \cdot \nabla_{\pa E} \Delta_{\pa E } \psi$ and existence of $h_1$.}\\
 We recall that $\widetilde{R}_0$ is composed of three addends, see~\eqref{nuovoresto}. Using Young’s inequality, we obtain
 \begin{equation*}
     \int_{\pa E} \nabla_{\pa E} \widetilde{\mathcal{R}}_0 \cdot \nabla_{\pa E} \Delta_{\pa E } \psi \leq C(\varepsilon)\| \nabla_{\pa E} \widetilde{\mathcal{R}}_0 \|^2_{L^2(\pa E)}+ \varepsilon \| \nabla_{\pa E} \Delta_{\pa E } \psi \|_{L^2(\pa E)}^2.
 \end{equation*}
  By the Leibniz rule, we obtain
 \begin{align*}
        \| \nabla_{\pa E}  \widetilde{\mathcal{R}}_0 \|_{L^2(\pa E)} &\leq  C \sum_{j+k=1} \|    \vert \overline{\nabla}^j \big( \widetilde{\mathcal{A}}_1(\psi B_E, \overline{\nabla} \psi, \nu_E)  \big) \vert \vert \overline{\nabla}^{2+k} \psi \vert \|_{L^2(\pa E)}\\
        & \quad+C \sum_{j+k=1} \|    \vert \overline{\nabla}^j \big( \widetilde{\mathcal{A}}_2(\psi B_E, \overline{\nabla} \psi, \nu_E)  \big) \vert \vert \overline{\nabla}^{1+k} \big(\psi B_E \big) \vert \|_{L^2(\pa E)}\\
        &\quad+ \| \tilde{a}_0 (\psi, \overline{\nabla}\psi, B_E, \nu_E) \|_{H^1(\pa E)}.
 \end{align*}
 For $j,k$ such that $j+k=1$, we apply Lemma~\ref{lem:Leibniz} with $T_1= \widetilde{\mathcal{A}}_1(\psi B_E, \overline{\nabla} \psi, \nu_E)  $ and $T_2= \overline{\nabla} \psi$ to estimate
 \begin{equation*}
     \begin{split}
         \| \vert  \overline{\nabla}^j \big(  \widetilde{\mathcal{A}}_1(\psi B_E, \overline{\nabla} \psi, \nu_E) \big) \vert \vert \overline{\nabla}^{2+k} \psi \vert \|_{L^2(\pa E)} &\leq C \| \widetilde{\mathcal{A}}_1(\psi B_E, \overline{\nabla} \psi, \nu_E) \|_{L^\infty(\pa E)} \|  \psi \|_{H^3(\pa E)}\\
         & \quad +C \| \psi \|_{C^1(\pa E)} \|  \widetilde{\mathcal{A}}_1(\psi B_E, \overline{\nabla} \psi, \nu_E) \|_{H^2(\pa E)} .
     \end{split}
 \end{equation*}
 Similarly, taking $T_1= \widetilde{\mathcal{A}}_2(\psi B_E, \overline{\nabla} \psi, \nu_E)$ and $T_2= \psi B_E$, we obtain
 \begin{equation*}
     \begin{split}
         &\|    \vert \overline{\nabla}^j \big( \widetilde{\mathcal{A}}_2(\psi B_E, \overline{\nabla} \psi, \nu_E)  \big) \vert \vert \overline{\nabla}^{1+k} \big(\psi B_E \big) \vert \|_{L^2(\pa E)} \\
         &\leq  C \| \widetilde{\mathcal{A}}_2(\psi B_E, \overline{\nabla} \psi, \nu_E) \|_{L^\infty(\pa E)} \|  \psi B_E \|_{H^2(\pa E)}\\
         & \quad+ C \| \psi B_E \|_{L^\infty(\pa E)} \| \widetilde{\mathcal{A}}_2(\psi B_E, \overline{\nabla} \psi, \nu_E) \|_{H^2(\pa E)} .
     \end{split}
 \end{equation*}
 Since $\widetilde{\mathcal{A}}_i$  is smooth and satisfies $ \widetilde{\mathcal{A}}_i(0,0,\cdot)=0$, we obtain
 \begin{equation}\label{filmporcorosso}
     \| \widetilde{\mathcal{A}}_i (\psi B_E, \overline{\nabla} \psi, \nu_E) \|_{L^\infty(\pa E)} \leq C \| \psi \|_{C^1(\pa E)} \le C h^{\frac{1}{9}},
 \end{equation}
 where the last inequality holds for $h \leq h_1 \leq h_0$, by~\eqref{1803pons}.
 Moreover,  we have
 \begin{equation*}
     \begin{split}
     &\vert \overline{\nabla}^2 \widetilde{\mathcal{A}}_i(\psi B_E, \overline{\nabla} \psi,\nu_E) \vert \\ 
     &\leq C \sum_{\vert \alpha \vert
     \leq 2} \big( 1+ \vert \overline{\nabla}^{\alpha_1}( \psi B_E) \vert  \big) \big( 1+ \vert \overline{\nabla}^{\alpha_2} (\psi B_E) \vert \big) \big( 1+ \vert \overline{\nabla}^{1+\alpha_3} \psi \vert \big) \big( 1+ \vert \overline{\nabla}^{1+\alpha_4} \psi \vert \big), 
     \end{split}
 \end{equation*}
 where the constant $C=C\big( \| B_E \|_{C^1(\pa E)} \big)$. Therefore, applying Lemma~\ref{lem:Leibniz} with $T_i=  \psi B_E$ for $i=1,2$ and $T_i= \overline{\nabla}\psi$ for $i=3,4$, we get
 \begin{align*}
     &\| \widetilde{\mathcal{A}}_i(\psi B_E, \overline{\nabla} \psi,\nu_E) \|_{H^2(\pa E)} \\
     &\leq C \big(  (1+ \| \psi B_E \|_{L^\infty(\pa E)})(1+ \| \psi \|_{H^3(\pa E)}) +  (1+ \| \psi \|_{C^1(\pa E)})(1+ \| \psi B_E \|_{H^2(\pa E)})\big)\\
     & \leq C \big( 1+ \| \psi B_E \|_{H^2(\pa E)}+ \| \psi \|_{H^3(\pa E)} \big) \\
     & \leq C \big(    1+ \| B_E \|_{H^2(\pa E)}+ \| \psi \|_{H^3(\pa E)}+ C \big) \\
     & \leq  C \big(1+ \| \nabla_{\pa E} \Delta_{\pa E} \psi \|_{L^2(\pa E)}  \big),
 \end{align*}
 where  in the last inequality we used Lemma~\ref{lem:hilbert-norm} and~\eqref{AnitraD}.
 
 Finally, we estimate $ \tilde{a}_0$. Using  the chain rule, the regularity of $\tilde{a}_0$, equation~\eqref{NormC^1alphB_E}, Lemma~\ref{lem:Leibniz}, and the interpolation inequalities of Proposition~\ref{prop:interpolation}, we obtain
 \begin{equation*}
     \begin{split}
         \| \tilde{a}_0 (\psi, \overline{\nabla} \psi, B_E, \nu_E) \|_{H^1(\pa E)} \leq& C \| (1+| \overline{\nabla}\psi \vert )(1+ \vert \overline{\nabla}^2 \psi \vert )(1+ \vert B_E \vert)(1+\vert \overline{\nabla}B_E \vert)    \|_{L^2(\pa E)}\\
        \leq & C \| \psi \|_{H^2(\pa E)}+C \leq \varepsilon \| \psi \|_{H^3(\pa E)}+ C(\varepsilon) \| \psi \|_{H^1(\pa E)}+ C\\
        \leq & \varepsilon \| \nabla_{\pa E} \Delta_{\pa 
        E} \psi \|_{L^2(\pa E)}+ C .
     \end{split}
 \end{equation*}
 Up to taking $h_1$ smaller, by~\eqref{filmporcorosso},, and the above inequalities, we conclude that
 \begin{equation*}
     \int_{\pa E} \nabla_{\pa E} \widetilde{\mathcal{R}}_0 \cdot \nabla_{\pa E} \Delta_{\pa E } \psi \leq C+ C\varepsilon \| \nabla_{\pa E} \Delta_{\pa E } \psi \|_{L^2(\pa E)}^2.
 \end{equation*}

Recall that, by~\eqref{unifell} and~\eqref{07032026posdef}, we have
 \begin{equation*}
     C\| \nabla_{\pa E} \Delta_{\pa E} \psi \|_{L^2(\pa E)}^2 \leq \int_{\pa E} \mathcal{A}(\nu_E) \nabla_{\pa E}\Delta_{\pa E} \psi \cdot \nabla_{\pa E} \Delta_{\pa E} \psi.
 \end{equation*}
Finally, by combing all the previously derived estimates with equation~\eqref{eq:sistint}, we infer
 \begin{equation}
     \frac{1}{h} \| \nabla_{\pa E} \psi \|_{L^2(\pa E)}^2+ C\| \nabla_{\pa E} \Delta_{\pa E} \psi \|_{L^2(\pa E)}^2 \leq C,
 \end{equation}
and Lemma~\ref{lem:hilbert-norm} yields $\| \psi \|_{H^3(\pa E)} \leq C.$
 
\end{proof}

\subsubsection{Estimates in $H^5$}
The aim of this subsubsection is to obtain an $H^5(\pa E)$-bound, with a constant independent of $h$, for the function $\psi$ solving equation~\eqref{PeerGynt}.
\begin{lemma}\label{Lemma:stimefinali}
    Let $E \Subset \Omega$ be a $C^6$-regular set such that~\eqref{AnitraD} holds and that satisfies $\| \Delta_{\pa E}^2 H_E \|_{L^2(\pa E)} \leq  K_0 h^{-\frac{1}{4}}$. Let $F \subset \Omega$ be a minimizer of~\eqref{ProblemaINc} for $\eta < \delta_2$, where $\delta_2$ is the constant given in Proposition~\ref{CaneBernard06}.  Then there exists $h_2 \leq h_1$, where $h_1$ is the constant provided by Lemma~\ref{Paperplane}, such that the solution of~\eqref{PeerGynt} satisfies, for every $h \leq h_2$,
    \begin{equation}\label{LESTIMEFINALI}
    \begin{split}
      &\| \psi \|_{H^1(\pa E)} \leq C_1 h, \,\| \psi \|_{H^2(\pa E)} \leq C_1 h^{\frac{3}{4}},\, \| \psi \|_{H^3(\pa E)} \leq C_1h^{\frac{1}{2}}, \\
      &\| \psi \|_{H^4(\pa E)} \leq C_1h^{\frac{1}{4}}, \,  \| \psi \|_{H^5(\pa E)} \leq C_1, \, \| \psi \|_{C^{3,\alpha}(\pa E)} \leq \tilde{C}_{1,\alpha} \\
      & \| \psi \|_{C^{0,\alpha}(\pa E)} \leq \tilde{C}_{1,\alpha} h^{\frac{3}{4}} \,, \| \psi \|_{C^{1,\alpha}(\pa E)} \leq \tilde{C}_{1,\alpha} h^{\frac{1}{2}},\, \| \psi \|_{C^{2,\alpha}(\pa E)} \leq \tilde{C}_{1,\alpha} h^{\frac{1}{4}}  \quad \text{ for } \alpha \in (0,1), 
      \end{split}
    \end{equation}
    where $C_1=C_1(r_0,\sigma_E,K_0,K_{el},\delta_2,h_1,P(E))$ and $\tilde{C}_{1,\alpha}=\tilde{C}_1(\alpha, C_1)$.
\end{lemma}
\begin{proof}
In order to simplify the notation, we identify $B_E$ with $B_\Sigma$. Notice that, by the assumptions on $E$, using Lemma~\ref{lem:hilbert-norm} and the Sobolev embedding theorem, we have
    \begin{equation}\label{2203stimecurvthm}
        \| B_E \|_{H^3(\pa E)} \leq C, \quad \| B_E \|_{C^{1,\alpha}(\pa E)} \leq C, \quad \| B_E \|_{H^4} \leq C h^{-\frac{1}{4}},
    \end{equation}
    where $C= C(K_0,r_0)$. For the reader’s convenience, and to avoid overburdening the notation, we will omit the explicit dependence on $K_0,\, r_0, P(E)$ and $\Omega, K_{el}$. 
    
    We compute  the Laplace-Beltrami operator applied to equation~\eqref{NewPeerGynt} (recall that~\eqref{PeerGynt} and~\eqref{NewPeerGynt} are equivalent).
     Using Lemma~\ref{nonstaropiu21} with $ f= \Delta_{\pa E} \psi$, we obtain
    \begin{equation*}
        \begin{split}
        \Delta_{\pa E} \div_{\pa E} \big( \mathcal{A}(\nu_E) \nabla_{\pa E} \Delta_{\pa E} \psi   \big)=& \div_{\pa E} \big( \mathcal{A}(\nu_E) \nabla_{\pa E} \Delta_{\pa E}^2  \psi \big)\\
        & + \div_{\pa E} \big( \big(\Delta_{\pa E} \mathcal{A}(\nu_E) \big) \nabla_{\pa E} \Delta_{\pa E} \psi \big)\\
        & + \widehat{\mathcal{R}}(\overline{\nabla} B_E \star \overline{\nabla} \Delta_{\pa E}\psi, B_E \star \overline{\nabla}^3 \Delta_{\pa E}\psi).
        \end{split}
    \end{equation*}
Therefore, by~\eqref{NewPeerGynt} we have
\begin{equation}\label{UltimaPeerGynt}
    \begin{split}
   \frac{1}{h} \left( \Delta_{\pa E}\psi+  \Delta_{\pa E}\Big(H_E \frac{\psi^2}{2} \Big) + \Delta_{\pa E} \Big(K_E \frac{\psi^3}{3}  \Big) \right)(x)=
    &-\div_{\pa E} \big( \mathcal{A}(\nu_E) \nabla_{\pa E} \Delta_{\pa E}^2 \psi   \big)(x)\\
   &- \ \div_{\pa E} \big( \big(\Delta_{\pa E} \mathcal{A}(\nu_E) \big) \nabla_{\pa E} \Delta_{\pa E} \psi \big)(x) \\
        & - \widehat{\mathcal{R}}(\overline{\nabla} B_E \star \overline{\nabla} \Delta_{\pa E}\psi, B_E \star \overline{\nabla}^3 \Delta_{\pa E}\psi)(x)\\
   &+ \Delta_{\pa E}^2 H^\varphi_E (x)\\
   & -\Delta_{\pa E}\div_{\pa E} \big( \big(\Delta_{\pa E} \mathcal{A}(\nu_E) \big) \nabla_{\pa E} \psi \big)(x)\\
   &- \Delta_{\pa E}\widehat{\mathcal{R}}(\overline{\nabla} B_E \star \overline{\nabla} \psi, B_E \star \overline{\nabla}^3 \psi)(x)\\
   &- \Delta_{\pa E}^2 \big(Q(E(u_F^{K_{el}}))(x+ \psi(x)\nu_E)  \big)\\
   &+ \Delta_{\pa E}^2 \widetilde{\mathcal{R}}_0(x).
   \end{split}
\end{equation}
\textit{Step 1:} We prove the estimate for  $ \| \psi \|_{H^5(\pa E)}$ and the existence of $h_2$.\\
Multiplying~\eqref{UltimaPeerGynt} by $-\Delta_{\partial E}^2 \psi$ and integrating by parts, we obtain
\begin{align*}
        \frac{1}{h} \int_{\pa E} \vert \nabla_{\pa E} \Delta_{\pa E} \psi \vert^2 + \int_{\pa E} \mathcal{A}&(\nu_E) \nabla_{\pa E} \Delta_{\pa E}^2 \psi \cdot \nabla_{\pa E} \Delta_{\pa E}^2 \psi = \int_{\pa E} \nabla_{\pa E} \Delta_{\pa E} H^{\varphi}_E \cdot \nabla_{\pa E} \Delta_{\pa E}^2 \psi \\
        & -\frac{1}{h} \int_{\pa E} \nabla_{\pa E} \Big(  H_E \frac{\psi^2}{2}+ K_E \frac{\psi^3}{3}\Big) \cdot \nabla_{\pa E} \Delta_{\pa E}^2 \psi \\
        &- \int_{\pa E} \nabla_{\pa E}\div_{\pa E} \Big(\big(  \Delta_{\pa E} \mathcal{A}(\nu_E)\big) \nabla_{\pa E} \psi \Big) \cdot \nabla_{\pa E} \Delta_{\pa E}^2 \psi \\
        & - \int_{\pa E} \nabla_{\pa E} \widehat{\mathcal{R}}(\overline{\nabla} B_E \star \overline{\nabla} \psi, B_E \star \overline{\nabla}^3 \psi)\cdot\nabla_{\pa E}\Delta_{\pa E}^2 \psi \\
        & - \int_{\pa E} \nabla_{\pa E} \Delta_{\pa E} \big(Q(E(u_F^{K_{el}}))(x+ \psi(x)\nu_E(x))  \big)  \cdot \nabla_{\pa E} \Delta_{\pa E}^2 \psi \\
        & - \int_{\pa E} \big(\Delta_{\pa E} \mathcal{A}(\nu_E) \big)\nabla_{\pa E} \Delta_{\pa E} \psi \cdot \nabla_{\pa E} \Delta^2_{\pa E} \psi                     \\
        & + \int_{\pa E} \widehat{\mathcal{R}}(\overline{\nabla} B_E \star \overline{\nabla} \Delta_{\pa E}\psi, B_E \star \overline{\nabla}^3 \Delta_{\pa E} \psi) \Delta^2_{\pa E} \psi              \\
        & + \int_{\pa E} \nabla_{\pa E} \Delta_{\pa E}\widetilde{\mathcal{R}}_0 \cdot \nabla_{\pa E} \Delta_{\pa E }^2 \psi .
\end{align*}
We now proceed to estimate the right-hand side of the above identity. Let $\varepsilon > 0$ be fixed and to be chosen later. \\
\textit{Estimate of $\int_{\pa E} \nabla_{\pa E} \Delta_{\pa E} H^{\varphi}_E \cdot \nabla_{\pa E} \Delta_{\pa E}^2 \psi$. }\\
By Lemma~\ref{H3mcurvimpH3acurv} and Young's inequality, we obtain
\begin{align}
   \int_{\pa E} \nabla_{\pa E} \Delta_{\pa E} H^{\varphi}_E \cdot \nabla_{\pa E} \Delta^2_{\pa E} \psi \leq& C(\varepsilon) \| \nabla_{\pa E} \Delta_{\pa E} H^{\varphi}_E \|_{L^2(\pa E)}^2 + \varepsilon \| \nabla_{\pa E } \Delta^2_{\pa E} \psi \|^2_{L^2(\pa E)} \notag\\
   \leq & C+ \varepsilon \| \nabla_{\pa E } \Delta^2_{\pa E} \psi \|^2_{L^2(\pa E)}.\label{22Rivoluzione1}
\end{align}
\textit{Estimate of $-\frac{1}{h} \int_{\pa E} \nabla_{\pa E} \left(  H_E \frac{\psi^2}{2}+ K_E \frac{\psi^3}{3}\right) \cdot \nabla_{\pa E} \Delta_{\pa E}^2 \psi$.}\\
Using Young's inequality together with formula~\eqref{22piovani1}, we deduce
\begin{equation}\label{22Rivoluzione2}
  -\frac{1}{h} \int_{\pa E} \nabla_{\pa E} \Big(  H_E \frac{\psi^2}{2}+ K_E \frac{\psi^3}{3}\Big) \cdot \nabla_{\pa E} \Delta_{\pa E}^2 \psi \leq C+ \varepsilon \| \nabla_{\pa E } \Delta^2_{\pa E} \psi \|^2_{L^2(\pa E)}.  
\end{equation}
\textit{Estimate of $- \int_{\pa E} \nabla_{\pa E}\div_{\pa E} \left(\big(  \Delta_{\pa E} \mathcal{A}(\nu_E)\big) \nabla_{\pa E} \psi \right) \cdot \nabla_{\pa E} \Delta_{\pa E}^2 \psi$.}\\
Using Young's inequality, we obtain
\begin{align*}
    - \int_{\pa E}& \nabla_{\pa E}\div_{\pa E} \Big(\big(  \Delta_{\pa E} \mathcal{A}(\nu_E)\big) \nabla_{\pa E} \psi \Big) \cdot \nabla_{\pa E} \Delta_{\pa E}^2 \psi \\ 
    &\leq C(\varepsilon) \| \nabla_{\pa E}\div_{\pa E} \Big(\big(  \Delta_{\pa E} \mathcal{A}(\nu_E)\big) \nabla_{\pa E} \psi \Big) \|^2_{L^2(\pa E)}+ \varepsilon \| \nabla_{\pa E} \Delta_{\pa E}^2 \psi \|^2_{L^2(\pa E)}.
\end{align*}
A straightforward computation yields
\begin{equation}
    \vert \nabla_{\pa E}\div_{\pa E} \Big(\big(  \Delta_{\pa E} \mathcal{A}(\nu_E)\big) \nabla_{\pa E} \psi \Big) \vert \leq C \big(  \vert \overline{\nabla}^3 B_E \vert \vert \overline{\nabla} \psi \vert+ \vert \overline{\nabla}^2 B_E \vert \vert \overline{\nabla}^2 \psi \vert + \vert \overline{\nabla} B_E \vert \vert \overline{\nabla}^3 \psi \vert \big).
\end{equation}
Therefore, by Young's inequality and Proposition~\ref{prop:interpolation}, we obtain
\begin{align*}
        \int_{\pa E }  \left \vert \nabla_{\pa E}\div_{\pa E} \Big(\big(  \Delta_{\pa E} \mathcal{A}(\nu_E)\big) \nabla_{\pa E} \psi \Big) \right\vert^2 \leq& \frac{1}{2} \| B_E \|^4_{W^{2,4}(\pa E)}+ \frac{1}{2} \| \psi \|^4_{W^{2,4}(\pa E)}\\
        &+ \| B_E \|^2_{H^3(\pa E)} \| \overline{\nabla} \psi \|^2_{L^\infty(\pa E)}\\
        &+ \| B_E \|_{C^1(\pa E)}^2 \| \psi \|^2_{H^3(\pa E)} \\
        \leq & C \| B_E \|_{H^3(\pa E)}^{\frac{5}{6}}+ C \| \psi \|^{\frac{5}{6}}_{H^3(\pa E)} \\
        &+ \| B_E \|^2_{H^3(\pa E)} \| \overline{\nabla} \psi \|^2_{L^\infty(\pa E)}\\
        &+ \| B_E \|_{C^1(\pa E)}^2 \| \psi \|^2_{H^3(\pa E)} \leq C
\end{align*}
where in the last inequality we used~\eqref{stimeinH^3} and~\eqref{2203stimecurvthm}. Combining the above estimates, we conclude that
\begin{equation}\label{22Rivoluzione3}
     - \int_{\pa E} \nabla_{\pa E}\div_{\pa E} \bigg(\big(  \Delta_{\pa E} \mathcal{A}(\nu_E)\big) \nabla_{\pa E} \psi \bigg) \cdot \nabla_{\pa E} \Delta_{\pa E}^2 \psi  \leq C + \varepsilon \| \nabla_{\pa E} \Delta_{\pa E}^2 \psi \|^2_{L^2(\pa E)}.
\end{equation}
\textit{Estimate of $- \int_{\pa E} \nabla_{\pa E} \widehat{\mathcal{R}}(\overline{\nabla} B_E \star \overline{\nabla} \psi, B_E \star \overline{\nabla}^3 \psi) \cdot \nabla_{\pa E}\Delta_{\pa E}^2 \psi$.}\\
By Young's inequality, we have
\begin{multline}
- \int_{\pa E} \nabla_{\pa E} \widehat{\mathcal{R}}(\overline{\nabla} B_E \star \overline{\nabla} \psi, B_E \star \overline{\nabla}^3 \psi) \cdot \nabla_{\pa E}\Delta_{\pa E}^2 \psi \\
\leq C(\varepsilon) \|  \widehat{\mathcal{R}}(\overline{\nabla} B_E \star \overline{\nabla} \psi, B_E \star \overline{\nabla}^3 \psi)  \|_{H^1(\pa E)}^2+ \varepsilon \| \nabla_{\pa E} \Delta_{\pa E}^2\psi \|^2_{L^2(\pa E)}.
\end{multline}
Using the regularity of $\widehat{\mathcal R}$, we get
\begin{equation*}
\begin{split}
   \|  \widehat{\mathcal{R}}(\overline{\nabla} B_E \star \overline{\nabla} \psi, B_E \star \overline{\nabla}^3 \psi) \|_{H^1(\pa E)}^2 &\leq C \| \psi \|_{H^4(\pa E)}^2 + \| B_E \|_{H^2(\pa E)}^2\\
   & \leq C + \varepsilon \| \nabla_{\pa E} \Delta^2_{\pa E} \psi \|_{L^2(\pa E)} .
   \end{split}
\end{equation*}
Here we used~\eqref{2203stimecurvthm}, Proposition~\ref{prop:interpolation}, Lemma~\ref{lem:hilbert-norm}, and~\eqref{stimeinH^3}, which, via the Sobolev embedding theorem, implies that $ \| \psi \|_{L^\infty(\pa E)} \leq C h^{\frac{1}{4}}$. Therefore, we conclude that
\begin{equation}\label{22Rivolutione4}
    - \int_{\pa E} \nabla_{\pa E} \widehat{\mathcal{R}}(\overline{\nabla} B_E \star \overline{\nabla} \psi, B_E \star \overline{\nabla}^3 \psi) \cdot \nabla_{\pa E}\Delta_{\pa E}^2 \psi \leq C+ 2 \varepsilon \| \nabla_{\pa E} \Delta^2_{\pa E} \psi \|^2_{L^2(\pa E)}.
\end{equation}
\textit{Estimate of $ - \int_{\pa E} \nabla_{\pa E} \Delta_{\pa E} \big(Q(E(u_F^{K_{el}}))(x+ \psi(x)\nu_E(x))  \big)  \cdot \nabla_{\pa E} \Delta_{\pa E}^2 \psi$.}\\
Applying  Young's inequality, we obtain 
\begin{equation}\label{24032026sera1}
\begin{split}
  - \int_{\pa E}& \nabla_{\pa E} \Delta_{\pa E} \big(Q(E(u_F^{K_{el}}))(x+ \psi(x)\nu_E(x))  \big)  \cdot \nabla_{\pa E} \Delta_{\pa E}^2 \psi \\
  &\leq C(\varepsilon) \| \nabla_{\pa E} \Delta_{\pa E} \big(Q(E(u_F^{K_{el}}))(x+ \psi(x)\nu_E(x))  \big) \|^2_{L^2(\pa E)}  + \varepsilon \| \nabla_{\pa E} \Delta_{\pa E}^2 \psi \|^2_{L^2( \pa E)}.
\end{split}
\end{equation}
We aim to bound the first term on the right-hand side of the above inequality.

Set $F:= Q \big( E(u_F^{K_{el}})   \big)$, we have that
\begin{equation}
\begin{split}
    \nabla_{\pa E} F(x+ \psi(x)&\nu_E(x))= \nabla F (x+ \psi(x)\nu_E(x)) \nabla_{\pa E} \big(  x+ \psi(x)\nu_E (x)\big) \\
    &= \nabla F (x+ \psi(x)\nu_E(x)) \big( P_{\pa E}(x)+ \nu_E(x)\otimes\nabla_{\pa E} \psi(x)+ \psi(x)B_E(x) \big).
    \end{split}
\end{equation}
We now compute $ \nabla_{\pa E}^2  F(x+ \psi(x)\nu_E(x)) $, obtaining 
\begin{align}
      \nabla_{\pa E}^2  F(x+ \psi(x)\nu_E(x)) =& \nabla^2 F (x+\psi(x)\nu_E(x)) \big( P_{\pa E}+ \nu_E \otimes \nabla_{\pa E} \psi+ \psi B_E  \big)^2(x) \label{derivativeofFqe}\\
      & + Z_1(\nu_E, B_E, \overline{\nabla} B_E, \overline{\nabla} \psi, \overline{\nabla}^2 \psi), \notag
\end{align}
where $Z_1$ is a matrix-valued function whose $C^1$-norm is controlled by a constant depending on $K_{el}$, since its coefficients depend on derivatives of $F$, namely $ \| Z_{1}\|_{C^1} \leq K_{el} $.
Moreover, $Z_1$ satisfies
\begin{equation}\label{24032026form1}
\begin{split}
   &\vert  \nabla_{\pa E} Z_1 (\nu_E, B_E, \overline{\nabla} B_E, \overline{\nabla} \psi, \overline{\nabla}^2 \psi) \vert \\
    &\leq C(K_{el}) \big( \vert \overline{\nabla}^3 \psi \vert+ \vert  B_E \vert \vert \overline{\nabla}^2 \psi \vert+ \vert \overline{\nabla} B_E \vert (1+ \vert \overline{\nabla}\psi \vert)+ \psi \vert \overline{\nabla}^2 B_E \vert    \big).
\end{split}
\end{equation}
We have 
\begin{equation}
\begin{split}
   &{\rm tr} \bigg( \nabla^2 F (x+\psi(x)\nu_E(x)) \big( P_{\pa E}+ \nu_E \otimes \nabla_{\pa E} \psi+ \psi B_E  \big)^2(x)  \bigg) \\
   &= \sum_{i,j=1}^3 a_{ij}(x+ \psi(x)\nu_E(x)) b_{ij}(\nu_E^{I}, \nabla_{\pa E} \psi^I, \psi, B_{E}^I)(x),
   \end{split}
\end{equation}
where $a_{ij}$ are the coefficients of $\nabla^2 F$, and $b_{ij}$ are smooth function depending on $\nu_E^{I},\, \nabla_{\pa E} \psi^I$ and $ \psi, B_{E}^I$. Here, the notation $ A^{I}$, where $A $ is a vector or a matrix, denotes a selected component of the vector or matrix $A$. Therefore, by ~\eqref{derivativeofFqe}, we conclude that
\begin{align*}
     \Delta_{\pa E} F(x+ \psi(x)\nu_E(x))&=  {\rm tr} \big( \nabla_{\pa E}^2  F(x+ \psi(x)\nu_E(x)) \big)\\
     &=  \sum_{i,j=1}^3 a_{ij}(x+ \psi(x)\nu_E(x)) b_{ij}(\nu_E^{I}, \nabla_{\pa E} \psi^I, \psi, B_{E}^I)(x)\\
     & \quad +{\rm tr}  \big( Z_1(\nu_E, B_E, \overline{\nabla} B_E, \overline{\nabla} \psi, \overline{\nabla}^2 \psi)(x)  \big).
\end{align*}
Hence, we deduce that
\begin{align*}
    \nabla_{\pa E} \Delta_{\pa E} F(x+ \psi(x)\nu_E(x)) = & \sum_{i,j=1}^3 \nabla_{\pa E} a_{ij}(x+ \psi(x)\nu_E(x))  b_{ij}(\nu_E^{I}, \nabla_{\pa E} \psi^I, \psi, B_{E}^I)(x)\\
    & + \sum_{i,j=1}^3  a_{ij}(x+ \psi(x)\nu_E(x)) \nabla_{\pa E} b_{ij}(\nu_E^{I}, \nabla_{\pa E} \psi^I, \psi, B_{E}^I)(x)\\
    & + \nabla_{\pa E}{\rm tr}  \big( Z_1(\nu_E, B_E, \overline{\nabla} B_E, \overline{\nabla} \psi, \overline{\nabla}^2 \psi)(x)  \big).
\end{align*}
By~\eqref{24032026form1}, we can estimate the last term as follows
\begin{equation*}
\begin{split}
     &\nabla_{\pa E}{\rm tr}  \big( Z_1(\nu_E, B_E, \overline{\nabla} B_E, \overline{\nabla} \psi, \overline{\nabla}^2 \psi)(x)  \big) \\ 
     &\leq C(K_{el}) \big( \vert \overline{\nabla}^3 \psi \vert+ \vert  B_E \vert \vert \overline{\nabla}^2 \psi \vert+ \vert \overline{\nabla} B_E \vert (1+ \vert \overline{\nabla}\psi \vert)+ \psi \vert \overline{\nabla}^2 B_E \vert    \big).
\end{split}
\end{equation*}
By the regularity of the functions $b_{ij}$ and formulas~\eqref{stimeinH^3} and~\eqref{2203stimecurvthm}, we obtain
\begin{equation*}
\begin{split}
    \vert \nabla_{\pa E} b_{ij}(\nu_E^{I}, \nabla_{\pa E} \psi^I, \psi, B_{E}^I) \vert &\leq C(K_{el}) \| \nabla b_{ij} \|_{L^\infty} \big( 1+ \vert \overline{\nabla}^2 \psi \vert + \vert \overline{\nabla} B_E \vert ) \\
    & \leq C(K_{el})(1+ \vert \overline{\nabla}^2 \psi \vert  ),
    \end{split}
\end{equation*}
moreover, we have
\begin{equation*}
     \vert b_{ij}(\nu_E^{I}, \nabla_{\pa E} \psi^I, \psi, B_{E}^I) \vert \leq \| b_{ij} \|_{L^\infty} \leq C.
\end{equation*}
Recalling ~\eqref{minelvinc} (in particular that $ u^{K_{el}}_F \in \mathfrak{C}^{3, \frac{1}{4}}_{K_{el}}(\Omega; \R^3)$),
we get
\begin{equation*}
    \vert a_{ij}(x+ \psi(x)\nu_E(x)) \vert \leq C(K_{el}).
\end{equation*}
Now, combing the previous inequalities with ~\eqref{stimeinH^3} and~\eqref{2203stimecurvthm}, we deduce
\begin{align*}
\int_{\pa E} &\vert \nabla_{\pa E} \Delta_{\pa E} \big(Q(E(u_F^{K_{el}}))(x+ \psi(x)\nu_E(x))  \big) \vert^2 \\
&\leq C\sum_{i,j=1}^3 \int_{\pa E} \vert \nabla_{\pa E} a_{ij}(x+ \psi(x)\nu_E(x)) \vert^2 
  +C\big(1+ \| \psi \|_{H^3(\pa E)}^2+ \| \psi \|_{L^\infty}^2 \| \overline{ \nabla}^2 B_E \|_{L^2(\pa E)}  \big) \\
& \leq C \sum_{i,j=1}^3 \int_{\pa E} \vert \nabla_{\pa F}a_{ij} \vert^2 (x+ \psi(x)\nu_E(x))+ C \\
& \leq C \sum_{i,j=1}^3 \int_{\pa F} \vert \nabla_{\pa F} a_{ij } \vert^2 +C \leq C.
\end{align*}
 Therefore, using the above estimate in~\eqref{24032026sera1}, we obtain
 \begin{equation}\label{22Rivolutione5}
      - \int_{\pa E} \nabla_{\pa E} \Delta_{\pa E} \big(Q(E(u_F^{K_{el}}))(x+ \psi(x)\nu_E(x))  \big)  \cdot \nabla_{\pa E} \Delta_{\pa E}^2 \psi 
  \leq C  + \varepsilon \| \nabla_{\pa E} \Delta_{\pa E}^2 \psi \|^2_{L^2( \pa E)}.
 \end{equation}
 \textit{Estimate of $-\int_{\pa E} \big(\Delta_{\pa E} \mathcal{A}(\nu_E) \big)\nabla_{\pa E} \Delta_{\pa E} \psi \cdot \nabla_{\pa E} \Delta^2_{\pa E} \psi$.}\\
 Using Young's inequality,~\eqref{2203stimecurvthm} and~\eqref{stimeinH^3}, we obtain 
 \begin{align}
       -\int_{\partial E}& \big(\Delta_{\pa E} \mathcal{A}(\nu_E) \big)\nabla_{\pa E} \Delta_{\pa E} \psi \cdot  \nabla_{\pa E} \Delta^2_{\pa E} \psi  \notag\\
       &\leq C(\varepsilon) \| \mathcal{A}(\nu_E) \nabla_{\pa E} \Delta_{\pa E} \psi \|^2_{L^2(\pa E)}+ \varepsilon \| \nabla_{\pa E} \Delta_{\pa E} \psi \|^2_{L^2(\pa E)}  \notag\\
       & \leq C(\varepsilon)\| \psi\|_{H^3(\pa E)}^2+ \varepsilon \|\nabla_{\pa E} \Delta_{\pa E}^2 \psi \|_{L^2(\pa E)}\notag\\
       & \leq C(\varepsilon)+ \varepsilon \|\nabla_{\pa E} \Delta_{\pa E}^2 \psi \|_{L^2(\pa E)}.\label{22Rivolutione6}
 \end{align}
 \textit{Estimate of $ \int_{\pa E} \widehat{\mathcal{R}}(\overline{\nabla} B_\Sigma \star \overline{\nabla} \Delta_{\pa E}\psi, B_\Sigma \star \overline{\nabla}^3 \Delta_{\pa E} \psi)\Delta^2_{\pa E} \psi$.}\\
 Since $\widehat{\mathcal{R}}$ is smooth (see Lemma~\ref{nonstaropiu21}), we obtain
\begin{align*}
   &|\widehat{\mathcal{R}}(\overline{\nabla} B_E \star \overline{\nabla} \Delta_{\pa E}\psi, B_E \star \overline{\nabla}^3 \Delta_{\pa E} \psi) \vert \\
   &\leq C\big(  \| B_E \|_{C^1(\pa E)}  \big) \big(   1+ \vert \overline{\nabla} \Delta_{\pa E} \psi  \vert + \vert \overline{\nabla}^2  \Delta_{\pa E} \psi  \vert+ \vert \overline{\nabla}^3 \Delta_{\pa E} \psi  \vert \big) \\
   &  \leq C\big(   1+ \vert \overline{\nabla} \Delta_{\pa E} \psi  \vert + \vert \overline{\nabla}^2  \Delta_{\pa E}\psi  \vert+ \vert \overline{\nabla}^3 \Delta_{\pa E} \psi  \vert \big). 
\end{align*}
Hence, by Young's inequality and the interpolation inequalities, we obtain
\begin{align}
      \int_{\pa E}  \widehat{\mathcal{R}}(\overline{\nabla} B_E \star \overline{\nabla} \Delta_{\pa E}&\psi, B_E \star \overline{\nabla}^3 \Delta_{\pa E} \psi) \Delta^2_{\pa E} \psi \notag\\
      \leq &C(\varepsilon)\| \psi \|_{H^4(\pa E)}^2  + \varepsilon \| \psi \|_{H^5(\pa E)}^2\notag\\
       \leq& C(\varepsilon)\| \psi \|^2_{H^3(\pa E)}+2\varepsilon \| \psi \|^2_{H^5(\pa E)} \notag\\
       \leq & \varepsilon C \| \nabla_{\pa E} \Delta_{\pa E}^2 \psi \|_{L^2(\pa E)}^2 +C(\varepsilon) \big( 1+ \| \Delta^2_{\pa E} H_E \|^2_{L^2(\pa E)} \| \psi \|^2_{L^\infty(\pa E)}  \big)\notag\\
       \leq &\varepsilon C \| \nabla_{\pa E} \Delta_{\pa E}^2 \psi \|_{L^2(\pa E)}^2 +C(\varepsilon),\label{22Rivolutione7}
  \end{align}
  where in the third estimate we applied Lemma~\ref{lem:hilbert-norm} together with~\eqref{stimeinH^3}; and in the last estimate we used the inequality $\|\psi\|_{L^\infty(\pa E)}\le Ch^{1/4}$ together with~\eqref{2203stimecurvthm} to deduce 
  \begin{equation}\label{Ilfantasmaformaggino}
      \| \Delta^2_{\pa E} H_E \|_{L^2(\pa E)} \| \psi \|_{L^\infty(\pa E)} \leq C.
  \end{equation} 
  \textit{Estimate of $\int_{\pa E} \nabla_{\pa E} \Delta_{\pa E}\widetilde{\mathcal{R}}_0 \cdot \nabla_{\pa E} \Delta_{\pa E }^2 \psi$ and existence of $h_2$.}\\
  We recall that $\widetilde{R}_0$ consists of three terms, see~\eqref{nuovoresto}. Using Young's inequality, we get
 \begin{equation}\label{25032026form1}
     \int_{\pa E} \nabla_{\pa E} \Delta_{\pa E}\widetilde{\mathcal{R}}_0 \cdot \nabla_{\pa E} \Delta_{\pa E }^2 \psi \leq C(\varepsilon)\| \nabla_{\pa E} \Delta_{\pa E} \widetilde{\mathcal{R}}_0 \|^2_{L^2(\pa E)}+ \varepsilon \| \nabla_{\pa E} \Delta_{\pa E }^2 \psi \|_{L^2(\pa E)}^2.
 \end{equation}
 Using the Leibniz rule, we obtain
 \begin{align*}
        \| \nabla_{\pa E} \Delta_{\pa E} \widetilde{\mathcal{R}}_0 \|^2_{L^2(\pa E)} \leq  C \sum_{j+k=3} &\|    \vert \overline{\nabla}^j \big( \widetilde{\mathcal{A}}_1(\psi B_E, \overline{\nabla} \psi, \nu_E)  \big) \vert \vert \overline{\nabla}^{2+k} \psi \vert \|_{L^2(\pa E)}\\
        & +C \sum_{j+k=3} \|    \vert \overline{\nabla}^j \big( \widetilde{\mathcal{A}}_2(\psi B_E, \overline{\nabla} \psi, \nu_E)  \big) \vert \vert \overline{\nabla}^{1+k} \big(\psi B_E \big) \vert \|_{L^2(\pa E)}\\
        &+ \| \tilde{a}_0 (\psi, \overline{\nabla}\psi, B_E, \nu_E) \|_{H^3(\pa E)}.
 \end{align*}
 For indices $j,k$ satisfying $j+k=3$, we apply Lemma~\ref{lem:Leibniz} with $T_1= \widetilde{\mathcal{A}}_1(\psi B_E, \overline{\nabla} \psi, \nu_E)  \big)$ and $T_2= \overline{\nabla} \psi$ to estimate  \begin{align*}
         &\| \vert  \overline{\nabla}^j \big(  \widetilde{\mathcal{A}}_1(\psi B_E, \overline{\nabla} \psi, \nu_E) \big) \vert \vert \overline{\nabla}^{2+k} \psi \vert \|_{L^2(\pa E)} \\
         &\leq C \| \widetilde{\mathcal{A}}_1(\psi B_E, \overline{\nabla} \psi, \nu_E) \|_{L^\infty(\pa E)} \|  \psi \|_{H^5(\pa E)}\\
         & \quad  +C \| \psi \|_{C^1(\pa E)} \|  \widetilde{\mathcal{A}}_1(\psi B_E, \overline{\nabla} \psi, \nu_E) \|_{H^4(\pa E)} \\
         &\leq  C \| \psi \|_{C^1(\pa E)}  \big( \| \psi \|_{H^5(\pa E)} +  \|  \widetilde{\mathcal{A}}_1(\psi B_E, \overline{\nabla} \psi, \nu_E) \|_{H^4(\pa E)}  \big) \\
         &\leq  C h^{\frac{1}{8}} \big( \| \psi \|_{H^5(\pa E)} +  \|  \widetilde{\mathcal{A}}_1(\psi B_E, \overline{\nabla} \psi, \nu_E) \|_{H^4(\pa E)}  \big),
 \end{align*}
 where in the last inequality we used that, by the Sobolev embedding, it holds $ \| \psi \|_{C^1(\pa E)} \leq  Ch^{\frac{1}{8}}$.
 Similarly, with $T_1= \widetilde{\mathcal{A}}_2(\psi B_E, \overline{\nabla} \psi, \nu_E)$ and $T_2= \psi B_E$, we get 
 \begin{align*}
         &\|    \vert \overline{\nabla}^j \big( \widetilde{\mathcal{A}}_2(\psi B_E, \overline{\nabla} \psi, \nu_E)  \big) \vert \vert \overline{\nabla}^{1+k} \big(\psi B_E \big) \vert \|_{L^2(\pa E)} \\
         &\leq  C \| \widetilde{\mathcal{A}}_2(\psi B_E, \overline{\nabla} \psi, \nu_E) \|_{L^\infty(\pa E)} \|  \psi B_E \|_{H^4(\pa E)}\\
         & \quad+ C \| \psi B_E \|_{L^\infty(\pa E)} \| \widetilde{\mathcal{A}}_2(\psi B_E, \overline{\nabla} \psi, \nu_E) \|_{H^4(\pa E)} \\
         & \leq C h^{\frac{1}{8}} \| \psi B_E \|_{H^4(\pa E)}+C h^{\frac{1}{4}}\| \widetilde{\mathcal{A}}_2(\psi B_E, \overline{\nabla} \psi, \nu_E) \|_{H^4(\pa E)}.
 \end{align*}
  Thanks to the smoothness of $ \widetilde{\mathcal{A}}_i$ and by the chain rule, we obtain
 \begin{equation}
     \begin{split}
         \vert \overline{\nabla}^4 \widetilde{\mathcal{A}}_i (\psi B_E, \overline{\nabla} \psi, \nu_E )\vert \leq & C \sum_{\vert \alpha \vert \leq 4} \big( 1+ \vert \overline{\nabla}^{\alpha_1} (\psi B_E) \vert  \big) \cdots \big( 1+ \vert \overline{\nabla}^{\alpha_4} (\psi B_E) \vert  \big) \cdots\\
         & \qquad \qquad \cdots \big( 1+ \vert \overline{\nabla}^{1+ \alpha_5} \psi  \vert  \big) \cdots \big( 1+ \vert \overline{\nabla}^{1+ \alpha_8} \psi  \vert \big) \cdots\\
         & \qquad  \qquad\cdots \big( 1+ \vert \overline{\nabla}^{\alpha_9} B_E \vert  \big)  \cdots\big( 1+ \vert \overline{\nabla}^{\alpha_{11}} B_E \vert  \big).
     \end{split}
 \end{equation}
 Therefore, by Lemma~\ref{lem:Leibniz} with $T_i= \psi B_E$ for $i=1,2,3,4$, $T_i= \overline{\nabla} \psi$ for $ i=5,6,7,8$, $T_i= B_E$ for $i=9,10,11$, we conclude that
 \begin{align*}
         &\| \overline{\nabla}^4 \widetilde{\mathcal{A}}_i(\psi B_E, \overline{\nabla} \psi, \nu_E) \|_{L^2(\pa E)}\\
         &\leq C \big( 1+ \| \psi B_E \|  \big) \big( 1+ \| \psi \|_{H^5(\pa E)}\big)  +C \big( 1+ \| \psi \|_{C^1(\pa E)} \big)  \big( 1+ \| \psi B_E \|_{H^4(\pa E)}\big)\\
         & \quad +C \big( 1+ \| \psi \|_{C^1(\pa E)} \big) \big( 1+ \| \psi B_E \|_{L^\infty(\pa E)}\big) \big( 1+ \| B_E \|_{H^3(\pa E)}\big) \\
         &\leq  C \big(1+ \| \psi B_E \|_{H^4(\pa E)} + \| \psi \|_{H^5(\pa E)} \big) 
         \\
         &\leq  C \big( 1+ \| \psi \|_{L^\infty(\pa E)} \| B_E \|_{H^4(\pa E)}+ \| \psi \|_{H^5(\pa E)} \big) \\
          &\leq C \big( 1+ \| \nabla_{\pa E} \Delta^2_{\pa E} \psi \|_{L^2(\pa E)} \big),
 \end{align*}
 where we also used~\eqref{stimeinH^3},~\eqref{2203stimecurvthm},  as well as Lemma~\ref{lem:hilbert-norm}. We  observe that  we also have
 \begin{equation*}
     \| \psi B_E \|_{H^4(\pa E)} \leq C \big( 1+ \| \nabla_{\pa E} \Delta^2_{\pa E} \psi \|_{L^2(\pa E)} \big) .
 \end{equation*}
 Similarly, we estimate $\tilde a_0$ as follows
  \begin{align*}
          &\| \tilde{a}_0 (\psi, \overline{\nabla} \psi,B_E, \nu_E ) \|_{H^3(\pa E)} \\
          & \leq C \sum_{ \vert \alpha \vert \leq 3} \| \big( 1+ \vert \overline{\nabla}^{1+ \alpha_1} \psi \vert  \big) \cdots \big(1+ \vert \overline{\nabla}^{1+ \alpha_3} \psi \vert  \big) \big( 1+ \vert \overline{\nabla}^{\alpha_4} B_E \vert  \big) \cdots \big( 1+ \vert \overline{\nabla}^{ \alpha_6} B_E \vert \big)     \|_{L^2(\pa E)} \\
          & \leq C \big( 1+ \| \psi \|_{H^4(\pa E)}+ \| B_E \|_{H^3(\pa E)} \big) \leq C \big( 1+ \varepsilon \| \nabla_{\pa E} \Delta^2_{\pa E} \psi \|_{L^2(\pa E)} + C(\varepsilon)  \| \psi \|_{H^3(\pa E)}\big) \\
           & \leq C(\varepsilon)+  \varepsilon \| \nabla_{\pa E} \Delta^2_{\pa E} \psi \|_{L^2(\pa E)}. 
  \end{align*}
  Therefore, combining  the previous estimates and choosing $h_2$, sufficiently small with respect to $\varepsilon$, we conclude
  \begin{equation}
      \| \nabla_{\pa E} \Delta_{\pa E} \widetilde{\mathcal{R}}_0 \|_{L^2(\pa E)} \leq \varepsilon \| \nabla_{\pa E} \Delta^2_{\pa E} \psi \|_{L^2(\pa E)}+ C(\varepsilon).
  \end{equation}
 Therefore, by~\eqref{25032026form1}, we deduce
  \begin{equation}\label{22Rivolutione8}
     \int_{\pa E} \nabla_{\pa E} \Delta_{\pa E}\widetilde{\mathcal{R}}_0 \cdot \nabla_{\pa E} \Delta_{\pa E }^2 \psi \leq C+ \varepsilon C\| \nabla_{\pa E} \Delta_{\pa E }^2 \psi \|_{L^2(\pa E)}^2 
  \end{equation}

  Recall that, by~\eqref{unifell} and~\eqref{07032026posdef}, we have
 \begin{equation}
     C\| \nabla_{\pa E} \Delta_{\pa E}^2 \psi \|_{L^2(\pa E)}^2 \leq \int_{\pa E} \mathcal{A}(\nu_E) \nabla_{\pa E}\Delta_{\pa E}^2 \psi \cdot \nabla_{\pa E} \Delta_{\pa E}^2 \psi.
 \end{equation}
Combining this observation with~\eqref{22Rivoluzione1},~\eqref{22Rivoluzione2},~\eqref{22Rivoluzione3},~\eqref{22Rivolutione4},~\eqref{22Rivolutione5},~\eqref{22Rivolutione6},~\eqref{22Rivolutione7},~\eqref{22Rivolutione8}
 we conclude that 
 \begin{equation}
     \frac{1}{h} \| \nabla_{\pa E} \Delta_{\pa E}\psi \|_{L^2(\pa E)}^2+ C\| \nabla_{\pa E} \Delta_{\pa E}^2 \psi \|_{L^2(\pa E)}^2 \leq C. 
 \end{equation}
 Finally, Lemma~\ref{lem:hilbert-norm} and~\eqref{Ilfantasmaformaggino} yield
 \begin{equation}
     \| \psi \|_{H^5(\pa E)} \leq C\big( \| \nabla_{\pa E} \Delta_{\pa E}^2 \psi \|_{L^2(\pa E)}+ \| \psi \|_{L^\infty(\pa E)} \| \Delta^2_{\pa E} H_E \|_{L^2(\pa E)} \big) \leq C.
 \end{equation}
 \textit{Step 2:} We prove that $ \| \psi \|_{H^1(\pa E)} \leq C h$ and we establish the other estimates.\\
 Multiplying~\eqref{UltimaPeerGynt} by $-\psi$, integrating by parts on $\partial E$, applying the Cauchy–Schwarz inequality, and reorganizing the resulting terms, we obtain
 \begin{align}
         \frac{1}{h} \| \nabla_{\pa E} \psi \|^2_{L^2(\pa E)} \leq & \| \nabla_{\pa E} \psi \|_{L^2(\pa E)} \bigg[ \| \nabla_{\pa E} \big( H_E \frac{\psi^2}{2} + K_E \frac{\psi^3}{3}\big)  \|_{L^2(\pa E)}\label{Kafka1}\\
         &+ \| \mathcal{A}(\nu_E) \nabla_{\pa E} \Delta^2_{\pa E} \psi  \|_{L^2(\pa E)} +\| \nabla_{\pa E} \Delta_{\pa E} H^\varphi_E \|_{L^2(\pa E)} \notag\\
         & + \| \nabla_{\pa E} \div_{\pa E} \big(\big( \Delta_{\pa E} \mathcal{A}(\nu_E) \big) \nabla_{\pa E} \psi   \big) \|_{L^2(\pa E)}\notag\\
         & + \| \nabla_{\pa E} \widehat{\mathcal{R}}(\overline{\nabla} B_E \star \overline{\nabla} \psi, B_E \star \overline{\nabla}^3 \psi) \|_{L^2(\pa E)} \notag\\
         & + \| \nabla_{\pa E} \Delta_{\pa E} \big(Q(E(u_F^{K_{el}}))(\cdot+ \psi(\cdot)\nu_E(\cdot))  \big) \|_{L^2(\pa E)} \notag\\
         & + \|  \big(\Delta_{\pa E} \mathcal{A}(\nu_E) \big)\nabla_{\pa E} \Delta_{\pa E} \psi  \|_{L^2(\pa E)}    + \| \nabla_{\pa E} \Delta_{\pa E}\widetilde{\mathcal{R}}_0  \|_{L^2(\pa E)}\bigg]\notag\\
         & +\| \psi \|_{L^2(\pa E)} \| \widehat{\mathcal{R}}(\overline{\nabla} B_E \star \overline{\nabla} \Delta_{\pa E}\psi, B_E \star \overline{\nabla}^3 \Delta_{\pa E}\psi) \|_{L^2(\pa E)}.\notag
 \end{align}
 Using~\eqref{090226ppranzf1},~\eqref{10022026pprinazof2}, and the Poincaré inequality, we obtain
 \begin{equation}\label{Kafka2}
     \| \psi \|_{L^2(\pa E)} \leq \| \xi_{F,E} \|_{L^2(\pa E)} \leq C \| \nabla_{\pa E} \xi_{F,E} \|_{L^2(\pa E)} \leq C \| \nabla_{\pa E} \psi \|_{L^2(\pa E)},
 \end{equation}
 where $\xi_{F,E}$ was defined in~\eqref{xiFEfunz}. We observe that, as a byproduct of Step 2, the term inside the square brackets in~\eqref{Kafka1} is bounded, and that $\| \widehat{\mathcal{R}}(\overline{\nabla} B_E \star \overline{\nabla} \Delta_{\pa E}\psi, B_E \star \overline{\nabla}^3 \Delta_{\pa E}\psi) \|_{L^2(\pa E)}$ is also bounded. Hence, combining these bounds with~\eqref{Kafka1} and~\eqref{Kafka2}, we deduce
     $\| \psi \|_{H^1(\pa E)} \leq C h.$
 Finally, applying Proposition~\ref{prop:interpolation}, we obtain the estimates  $$\| \psi \|_{H^2(\pa E)} \leq C h^{\frac{3}{4}}, \quad\| \psi \|_{H^3(\pa E)} \leq C h^{\frac{1}{2}}, \quad \| \psi \|_{H^4(\pa E)} \leq C h^{\frac{1}{4}}.$$
 The corresponding H\"older estimates then follow from the Sobolev embedding theorem.
\end{proof}
\begin{lemma}\label{Lispetttoregadget}
    Let $E \Subset \Omega$ be a $C^6$-regular set such that~\eqref{AnitraD} holds. Let $F \Subset \Omega$ be a minimizer of~\eqref{ProblemaINc} for $\eta < \delta_2$, where $\delta_2$ is the constant provided by Proposition~\ref{CaneBernard06}.    Then, for $u \in C^4(\mathcal{I}_{r_0}(\pa E))$, the following estimates hold on $\pa F$
    \begin{align}
       &\vert  \nabla_{\pa F} \Delta_{\pa F} u \vert \leq     C \big(  \vert \nabla^3 u \vert+ \vert B_F \vert \vert \nabla^2 u \vert + \vert \nabla u \vert \big(  \vert \overline{\nabla} B_F \vert+ \vert B_F \vert  \big)        \big)       \label{31032026form1}\\
       & \vert  \Delta_{\pa F}^2 u \vert  \leq C\big( \vert \nabla^4 u \vert+ \vert \nabla^3 u \vert \vert B_F \vert+ \vert \nabla^2 u \vert \vert \overline{\nabla} B_F \vert + \vert \nabla u \vert \vert \overline{\nabla}^2 B_F \vert   \big) \label{30032026form1}
 \end{align}
 where the constant $C= C(\| B_F \|_{L^\infty(\pa F)})$.
\end{lemma}
\begin{proof}
We recall that the gradient operator in $\R^3$ coincides with the Levi–Civita connection on $\R^3$. Hence, we may invoke classical properties of the Levi–Civita connection; see~\cite{LeeBook, LeeBook2018, KobayashiNomizu1963}.  \\
\textit{Step 1:} Given $u \in C^2(\mathcal{I}_{r_0}(\pa E))$, we  show that
\begin{equation*}
\Delta_{\R^3} u-  \nabla^2 u \nabla d_F \cdot\nabla d_F   - {\rm tr} \big( \nabla^2 d_F \big) \nabla u \cdot \nabla d_F
\end{equation*} 
provides a natural extension of $ \Delta_{\pa F} u$ in $\mathcal{I}_{\frac{r_0}{2}}(\pa E)$. 

We recall that $\nabla_{\pa F} u= \nabla u - \nabla u \cdot \nu_F  \nu_F$ on $\pa F$. Therefore, if we define
\begin{equation*}
    Z_u:= \nabla u- \nabla u \cdot \nabla d_{F} \nabla d_F  \text{ in } \mathcal{I}_{\frac{r_0}{2}}(\pa E),
\end{equation*}
then $Z_u= \nabla_{\pa F} u$ on $\pa F$. We also introduce $\widetilde{P}_{\pa F}:= I- \nabla d_F \otimes \nabla d_F$. Recalling that $ \nabla_{\pa F} X:= \nabla X P_{\pa F}$ for a vector field $X$, we define $\widetilde{\nabla }_{\pa F}X:= \nabla X \widetilde{P}_{\pa F}$.
With this notation, we obtain
\begin{align*}
   \widetilde{\nabla}_{\pa F} Z_u=& \nabla \big( \nabla u- \nabla u \cdot \nabla d_F \nabla d_F  \big) \widetilde{P}_{\pa F} \\
   =& \nabla^2 u \widetilde{P}_{\pa F} - \nabla u \cdot \nabla d_F \nabla^2 d_F  \widetilde{P}_{\pa F}-\nabla^2 d_F \nabla u \otimes \nabla d_F \widetilde{P}_{\pa F}\\
   =&\nabla^2 u \widetilde{P}_{\pa F} - \nabla u \cdot \nabla d_F \nabla^2 d_F  \widetilde{P}_{\pa F},
\end{align*}
where we have used that $\widetilde{P}_{\pa F}^2 = \widetilde{P}_{\pa F}$ and $ \nabla u \otimes \nabla d_F \widetilde{P}_{\pa F}=0$.
Taking the trace, we obtain
\begin{equation*}
    {\rm tr} \big( \widetilde{\nabla}_{\pa F} Z_u\big) = \Delta_{\R^3} u-  \nabla^2 u \nabla d_F \cdot\nabla d_F   - {\rm tr} \big( \nabla^2 d_F \big) \nabla u \cdot \nabla d_F 
\end{equation*}
and we have ${\rm tr} \big( \widetilde{\nabla}_{\pa F} Z_u\big) = \Delta_{\pa F} u $.
An immediate consequence is that  on $\partial F$ it holds
    \begin{equation}\label{qullochenonguc}
    \Delta_{\pa F}  u = \Delta_{\R^3}  u -\nabla^2  u \nu_F \cdot \nu_F -H_F \nabla  u \cdot \nu_F.
    \end{equation}
\textit{Step 2:} We compute $\nabla_{\pa F} \Delta_{\pa F} u$ and prove formula~\eqref{31032026form1}.\\
By \emph{Step 1} applied to $f=u$, we obtain on $\partial F$:
\begin{align*}
      \nabla_{\pa F} \Delta_{\pa F} u=& P_{\pa F} \nabla \big( \Delta_{\R^3} u-  \nabla^2 u \nabla d_F \cdot\nabla d_F   - {\rm tr} \big( \nabla^2 d_F \big) \nabla u \cdot \nabla d_F  \big) \\
      =& P_{\pa F}\nabla \Delta_{\R^3} u- H_F P_{\pa F} \nabla^2 u \nu_F -H_F P_{\pa F} B_F \nabla u \\
      &  - \sum_{i=1}^3 \bigg(  {\rm tr} \big(  {\rm tr} \big(       {\rm tr}\big( \nabla^3 u \otimes e_i \big)   \otimes\nu_F\big)   \otimes \nu_F\big) + {\rm tr} \big( {\rm tr} \big( \nabla^2 u \otimes \nabla^2 d_F e_i \big)    \otimes \nu_F \big) \bigg)e_i\\
      & + {\rm tr} \big( {\rm tr}\big( {\rm tr}\big( \nabla^3 u \otimes \nu_F \big)\otimes \nu_F\big)  \otimes  \nu_F\big)\\
      & - \nabla u \cdot \nu_F \sum_{i=1}^3 {\rm tr} \big({\rm tr}\big( \nabla^3 d_F \otimes e_i \big)  \big) e_i+ {\rm tr} \big({\rm tr}\big( \nabla^3 d_F \otimes \nu_F \big)  \big) \nu_F,
\end{align*}
where, in the second equality, we used that $\nabla_X F= {\rm tr} \big( \nabla F \otimes X \big) $ and $ \nabla_{X} {\rm tr} \big( F \big)= {\rm tr} \big( \nabla_{X} F \big)$ and $ \nabla_{X,Y}^2 F= {\rm tr} \big( {\rm tr}\big( \nabla^2 F \otimes X \big) \otimes Y\big)$ (see e.g.~\cite[Propositions 4.15 and 4.21, Example 4.22]{LeeBook2018}), together with the fact that $B_F \nu_F=0$. Using~\eqref{24112025primacrossfit} and~\eqref{canzone1}, we obtain on $\partial F$:
\begin{equation}\label{Gladiatori31}
    \vert \nabla^3 d_F \vert \leq C \big( \vert \nabla_{\pa F} B_F \vert+ \vert B_F \vert \big) \leq C \big( \vert \overline{\nabla} B_F \vert + \vert B_F \vert  \big).
\end{equation}
Finally, combining the above equations, we conclude~\eqref{31032026form1}.\\
\textit{Step 3:} We prove formula~\eqref{30032026form1}.

Define $g:= \Delta_{\R^3} u-  \nabla^2 u \nabla d_F \cdot\nabla d_F   - {\rm tr} \big( \nabla^2 d_F \big ) \nabla u \cdot \nabla d_F$,
we need to estimate
\begin{equation}
    \Delta_{\pa F}^2 u= \Delta_{\pa F} g= \Delta_{\R^3}  g -\nabla^2  g \nu_F \cdot \nu_F -H_F \nabla  g \cdot \nu_F.
\end{equation}
Note that this computation can be carried out in a fully rigorous manner, as in \emph{Step 2}; however, here we present only the two main estimates. \\
\textit{Estimate of $ \vert \Delta_{\R^3} g\vert $ and $ \vert \nabla^2 g \vert$.}\\
We observe that it is sufficient to obtain an estimate for $ \vert \nabla^2 g \vert$. We have on $\pa F$
\begin{equation*}
\begin{split}
    \vert \nabla^2 g \vert &\leq C \big(  \vert \nabla^4 u \vert + \vert \nabla^3 u \vert \vert B_F \vert + \vert \nabla^2 u \vert \vert \nabla^3 d_F \vert + \vert \nabla u \vert \vert \nabla^4 d_F \vert    \big)\\
    & \leq C \big( \vert \nabla^4 u \vert + \vert \nabla^3 u \vert \vert B_F \vert + \vert \nabla^2 u \vert \vert \overline{\nabla} B_F \vert + \vert \nabla u \vert \vert \overline{\nabla}^2 B_F \vert    \big),
    \end{split}
\end{equation*}
where we  used~\eqref{Gladiatori31} and
$\vert \nabla^4 d_F \vert \leq C \big( \vert \overline{\nabla} B_F \vert + \vert \overline{\nabla}^2 B_F \vert \big) \text{ on }\pa F$,
which can be shown by~\eqref{24112025primacrossfit} and~\eqref{canzone1}.\\
\textit{Estimate of $|H_F \nabla g \cdot \nu_F|$.}\\
By a straightforward computation,~\eqref{24112025primacrossfit} and~\eqref{canzone1}, we obtain
\begin{equation*}
\begin{split}
 \vert H_F \nabla g \cdot \nu_F \vert &\leq C \vert B_F \vert \big(  \vert \nabla^3 u \vert + \vert B_F \vert \vert \nabla^2 u \vert + \vert \nabla^3 d_F\vert \vert \nabla u \vert \big)\\
 & \leq C \big(  \vert \nabla^3 u \vert + \vert B_F \vert \vert \nabla^2 u \vert + \vert \overline{\nabla} B_F \vert \vert \nabla u \vert  \big).
 \end{split}
\end{equation*}

Finally, combining the above estimates, we conclude~\eqref{30032026form1}.
\end{proof}

In the next lemma, we show that the regularity established in Lemma~\ref{Lemma:stimefinali} implies regularity of the set $F$ minimizer of~\eqref{ProblemaINc}.
\begin{lemma}\label{lemmahnegativa}
    Let $E \Subset \Omega$ be a $C^6$-regular set such that~\eqref{AnitraD} holds and that satisfies $\| \Delta_{\pa E}^2 H_E \|_{L^2(\pa E)} \leq  K_0 h^{-\frac{1}{4}}$. Let $F \subset \Omega$ be a minimizer of~\eqref{ProblemaINc} for $\eta < \delta_2$ and $h \leq h_2$, where $\delta_2$ is the constant provided by Proposition~\ref{CaneBernard06} and $h_2$ is the constant provided by Lemma~\ref{Lemma:stimefinali}. Then there exists a positive constant $C_2=C_2(C_1,r_0,K_0,K_{el})$, where $C_1$ is the constant provided in Lemma~\ref{Lemma:stimefinali}, such that
    \begin{equation}\label{deriv^3H_Federiv^4H_F}
        \| \nabla_{\pa F} \Delta_{\pa F} H_F \|_{L^2(\pa F)} \leq C_2, \quad \| \Delta^2_{\pa F} H_F \|_{L^2(\pa F)} \leq C_2 h^{-\frac{1}{4}}.
    \end{equation}
Moreover, the set $F$ satisfies the \textsc{UBC} with radius $r_0/2$.
\end{lemma}
\begin{proof}
    In what follows, we denote by $C$ a generic constant depending on $ C_1, r_0, K_0, K_{el}$, by $G \subset \R^3$  a smooth set, and  by $\overline{\nabla}_{\pa G}$ the Levi-Civita connection on $\pa G$.

   Recall~\eqref{LESTIMEFINALI}, in particular that $ \| \psi \|_{C^{2,\alpha}(\pa E)} \leq \tilde{C}_{1,\alpha} h^{1/4}$. Using~\cite[Theorem 2.6]{Da2018} and~\cite[Remark 3.4.6]{Antonia}, we obtain that $F$ satisfies the \textsc{UBC} with radius $r_0/2$, and in particular
    \begin{equation}\label{arremba01042026}
        \| B_F \|_{L^\infty(\pa F)} \leq \frac{2}{r_0}.
    \end{equation}

    Let $f \in H^1(\pa E)$ be the function attaining the supremum in~\eqref{d_{H^{-1}}f}, and define $\tilde{f}:= \frac{d_{H^{-1}}(F,E)}{h}f$. Hence, we can rewrite the Euler–Lagrange equations~\eqref{Comeprima} as
    \begin{equation}\label{Cosac'e'}
    \left\{
    \begin{aligned}
        & H_F^\varphi- Q(E(u_F^{K_{el}}))+ \tilde{f} \circ \pi_{\pa E}= L \quad \text{ on } \pa F, \\
        & - \Delta_{\pa E} \tilde{f} = \frac{\xi_{F,E}}{h}  \quad \text{ on } \pa E,
    \end{aligned}
    \right.
\end{equation}
where we recall $\xi_{F,E}$ is the function defined in~\eqref{xiFEfunz}. We observe that, using~\eqref{090226ppranzf1},~\eqref{10022026pprinazof2} and~\eqref{LESTIMEFINALI}, we obtain
\begin{equation}\label{xistainH^1}
    \| \xi_{F,E}\|_{H^1(\pa E)} \leq C h.
\end{equation}
We divide the proof into several steps.\\
 \textit{Step 1:} We show the following estimates
 \begin{equation}\label{03042026form1}
    \| \tilde{f} \|_{L^\infty(\pa E)} \leq C, \,\, \| \tilde{f} \|_{H^3(\pa E)} \leq C, \,\, \| \tilde{f} \|_{H^4(\pa E)} \leq C h^{-\frac{1}{4}}.
 \end{equation}

We start by proving the first estimate in~\eqref{03042026form1}. Using~\eqref{xistainH^1} and the Sobolev embedding theorem, we obtain $ \|\xi_{F,E}\|_{  L^p(\pa E)} \leq C_p h
$ for every $p\geq 1$. Moreover, by standard elliptic regularity theory, we deduce that $ \tilde{f} \in L^p(\pa E)$ for every $p \geq 1$. Fixing $p>2$ and arguing as in the proof of~\cite[Theorem 8.17]{GilbargTrudinger1977}, one can show that for all $x \in \pa E$
\begin{equation}
    \| \tilde{f} \|_{L^\infty(\pa E \cap B_{r_0/2 }(x))} \leq C(r_0) \big( \| \tilde{f} \|_{L^2(\pa E \cap B_{r_0}(x))}+ \frac{1}{h}\| \xi_{F,E} \|_{L^2(\pa E \cap B_{r_0}(x))} \big) 
\end{equation}
 and  the desired inequality follows.

We now prove the other two estimates in \eqref{03042026form1}.
Since~\eqref{arremba01042026} holds, we can apply Lemma~\ref{lem:hilbert-norm}, this combined with the second equation in~\eqref{Cosac'e'} imply
\begin{equation}\label{01042026primpranz1}
    \begin{split}
        \| \tilde{f} \|_{H^3(\pa E)} &\leq C \big( \| \nabla_{\pa E}\Delta_{\pa E} \tilde{f} \|_{L^2(\pa E)}+\| \tilde{f} \|_{L^\infty(\pa E)} (1+ \| \Delta_{\pa E} H_E \|_{L^2(\pa E)})\\
        & \leq C \frac{1}{h} \| \xi_{F,E} \|_{H^1(\pa E)} +C \leq C.
    \end{split}
\end{equation}
A straightforward computation yields
\begin{equation}
\begin{split}
    \Delta_{\pa E} \xi_{F,E}= &\Delta_{\pa E} \psi \big( 1+ \frac{\psi }{2}H_E+ \frac{\psi^2 }{3} K_E \big) +2 \nabla_{\pa E} \psi \nabla_{\pa E}\big(  \frac{\psi}{2}H_E \frac{\psi^2}{3} K_E \big)\\
    & + \psi \Delta_{\pa E} \big( \frac{\psi}{2}H_E + \frac{\psi^2}{3}K_E \big).
    \end{split}
\end{equation}
Hence, using~\eqref{LESTIMEFINALI} and Lemma~\ref{lem:hilbert-norm}, we deduce $\| \xi_{F,E} \|_{H^2(\pa E)} \leq C  h^{\frac{3}{4}}.$
Consequently, from  the second equation in~\eqref{Cosac'e'}, Lemma~\ref{lem:hilbert-norm} and ~\eqref{LESTIMEFINALI}, we obtain
\begin{align*}
    \| \tilde{f} \|_{H^4(\pa E)} &\leq C\big( \| \Delta^2_{\pa E} \tilde{f} \|_{L^2(\pa E)}+ \| \tilde{f} \|_{L^\infty(\pa E)}(1+ \| \nabla_{\pa E} \Delta_{\pa E} H_E \|_{L^2(\pa E)}) \big) \\
    & \leq C \frac{1}{h} \| \xi_{F,E} \|_{H^2(\pa E)} +C \leq C h^{-\frac{1}{4}}.
\end{align*}

To obtain estimate~\eqref{deriv^3H_Federiv^4H_F}, we need to relate higher-order derivatives of the function $ \tilde{f}\circ \pi_{\pa E}$ on $\pa F$ to higher-order derivatives of $\tilde{f} $ on $\pa E$.  To this end, we proceed as follows: first, we apply Lemma~\ref{Lispetttoregadget} to estimate derivatives on the boundary of $F$ in terms of standard derivatives in $\R^3$; subsequently, we estimate these derivatives in $\R^3$ in terms of derivatives on the boundary of $E$. For this purpose, we recall  that for all $x \in \mathcal{I}_{\frac{r_0}{2}}(\pa E)$, it holds that
\begin{equation}\label{01042026vorrei}
    | \nabla^k (\tilde{f} \circ \pi_{\pa E}) \vert  (x) \leq C_k \sum_{ \vert \alpha \vert \leq k} \prod_{j=1}^{k-1} \big( 1+ \vert \overline{\nabla}_{\pa E}^{\alpha_j} B_E \vert (\pi_{\pa E}(x)) \big)   \vert \overline{\nabla}_{\pa E}^{\alpha_k} \tilde{f}\vert (\pi_{\pa E}(x)) ,
\end{equation}
where the last index is never zero, i.e., $\alpha_k \neq 0$, see e.g.~\cite[Lemma 5.3]{JN}.\\
\textit{Step 2:} We show the following inequalities
\begin{equation}\label{Lafolliacorelli}
 \| \tilde{f}\circ\pi_{\pa E} \|_{H^3(\pa F)}, \,\, \|  H_F \|_{H^3(\pa F)} ,\,\,   \| H_F^\varphi \|_{H^3(\pa F)} \leq C.
\end{equation}

We begin by showing the bounds 
\begin{equation}\label{eq:boundssss}
    \| \tilde{f} \circ \pi_{\pa E} \|_{H^2(\pa F)}, \,\, \| H_F \|_{H^2(\pa F)}, \,\, \| H_F^\varphi \|_{H^2(\pa F)} \le C.
\end{equation}
From Lemma~\ref{lem:hilbert-norm} and the inequality $P(F) \leq C P_{\varphi}(F) \leq CP_{\varphi}(E)$, we obtain
\begin{equation}
\begin{split}
   \| \tilde{f} \circ \pi_{\pa E} \|_{H^2(\pa F)} \leq &C \big(   \|  \Delta_{\pa F} \tilde{f} \circ \pi_{\pa E} \|_{L^2(\pa F)}+ \|  \tilde{f} \circ \pi_{\pa E} \|_{L^2(\pa F)} \big)\\
   \leq&  C \big(   \|  \Delta_{\pa F} \tilde{f} \circ \pi_{\pa E} \|_{L^2(\pa F)}+ \| \tilde{f} \|_{L^\infty(\pa E)} P(F)^{\frac{1}{2}} \big)\\
   \leq & C \big(  \|  \Delta_{\pa F} \tilde{f} \circ \pi_{\pa E} \|_{L^2(\pa F)}+ C P_{\varphi}(E)^{\frac{1}{2}} \big).
   \end{split}
\end{equation} 
Next, using ~\eqref{qullochenonguc},~\eqref{arremba01042026},~\eqref{01042026primpranz1} and~\eqref{01042026vorrei}, we obtain
\begin{align}
    \| \Delta_{\pa F} \tilde{f} \circ \pi_{\pa E} \|_{L^2(\pa F)} \leq& C \big( 1+ \| \nabla^2 \tilde{f} \circ \pi_{\pa E} \|_{L^2(\pa F)} \big)\notag\\
    \leq& C \big(1+ \| \overline{\nabla}_{\pa E}^2 \tilde{f} \circ \pi_{\pa E}  \|_{L^2(\pa F)} + \| \overline{\nabla}_{\pa E}\tilde{f} \circ \pi_{\pa E} \|_{L^2(\pa F)}  \big)
    \notag\\
    \leq&   C\big( 1+ \| \tilde{f} \|_{H^2(\pa E)} \big) \leq C,\label{nunununununu01042026}
\end{align}
where in the third inequality we used the change of variables $ \Psi(x):= x+ \psi(x)\nu_E(x)$. Finally, computing the $\Delta_{\pa F}$ of the first equation in~\eqref{Cosac'e'} and using~\eqref{minelvinc} together with~\eqref{nunununununu01042026}, we obtain
\begin{equation}
    \| \Delta_{\pa F} H_F^\varphi \|_{L^2(\pa E)} \leq \|  Q(E(u_F^{K_{el}}))  \|_{H^2(\pa F)}+ \| \tilde{f} \circ \pi_{\pa E} \|_{H^2(\pa F)}  \leq C(K_{el})+  C.
\end{equation}
Combining these estimates and arguing as in Lemma~\ref{H3mcurvimpH3acurv}, we deduce~\eqref{eq:boundssss} .

We now proceed to prove the bound on $\| \tilde{f} \circ \pi_{\pa E} \|_{H^3(\pa F)}$.
From Lemma~\ref{lem:hilbert-norm} and the estimates~\eqref{eq:boundssss}, we obtain
\begin{align*}
    \| \tilde{f} \circ \pi_{\pa E} \|_{H^3(\pa F)} \leq& C\big(  \| \nabla_{\pa F} \Delta_{\pa F} ( \tilde{f} \circ \pi_{\pa E}) \|_{L^2(\pa F)}+ (1+\|  H_F\|_{H^2(\pa F)}) \| \tilde{f} \|_{L^\infty(\pa E)} \big) \\
    \leq & C \big( C+ \| \nabla_{\pa F} \Delta_{\pa F}  (\tilde{f} \circ \pi_{\pa E} )\|_{L^2(\pa F)} \big)\\
    \leq & C \big( C+  \| \nabla^3 (\tilde{f} \circ \pi_{\pa E} )\|_{L^2(\pa F)}+ \| B_F \|_{L^2(\pa F)} \| \nabla^2 (\tilde{f} \circ \pi_{\pa E} )\|_{L^2(\pa F)} \\
    &  \quad+ \| \nabla (\tilde{f} \circ \pi_{\pa E}) \|_{L^2(\pa F)} \big(  \| \overline{\nabla}_{\pa F} B_F \|_{L^2(\pa F)}+ \| B_F \|_{L^2(\pa F)} \big)\\
   \leq  & C \big(  1+\| \nabla^3 (\tilde{f} \circ \pi_{\pa E} )\|_{L^2(\pa F)}+\| \nabla^2 (\tilde{f} \circ \pi_{\pa E}) \|_{L^2(\pa F)}+ \| \nabla (\tilde{f} \circ \pi_{\pa E} )\|_{L^2(\pa F)} \big),
\end{align*}
where in the second inequality we also used the first inequality in~\eqref{03042026form1}, in the third inequality we used~\eqref{30032026form1}. We now estimate the terms $\| \nabla^3 ( \tilde{f} \circ \pi_{\pa E}) \|_{L^2(\pa F)}, \,\,\| \nabla^2( \tilde{f} \circ \pi_{\pa E} )\|_{L^2(\pa F)}, \,\, \| \nabla (\tilde{f} \circ \pi_{\pa E}) \|_{L^2(\pa F)}$. We present only the estimate for $\| \nabla^3 (\tilde{f} \circ \pi_{\pa E}) \|_{L^2(\pa F)}$, as the other cases are analogous. By the above estimate and~\eqref{01042026vorrei}, we obtain
\begin{align}
      \| \nabla^3 (\tilde{f} \circ \pi_{\pa E} )\|_{L^2(\pa F)} \leq&  C \big( \| (1+ \vert \overline{\nabla}_{\pa E} B_E \vert)^2 (1+ \vert \overline{\nabla}_{\pa E} \tilde{f} \vert ) \circ \pi_{\pa E} \|_{L^2(\pa F)}\notag\\
      & \quad +\| (1+ \vert \overline{\nabla}_{\pa E} B_E \vert)(1+ \vert \overline{\nabla}_{\pa E}^2 \tilde{f}\vert) \circ \pi_{\pa E}   \|_{L^2(\pa F)}\notag\\
      & \quad +(1+ \| B_E\|_{L^\infty(\pa E)}) \| \overline{\nabla}_{\pa E}^3 \tilde{f} \circ \pi_{\pa E}  \|_{L^2(\pa F)}\notag\\
      &\quad + \| (1+ \vert \overline{\nabla}_{\pa E}^2 B_E \vert )(1+ \vert \overline{\nabla}_{\pa E} \tilde{f}\vert ) \circ \pi_{\pa E} \|_{L^2(\pa F)}\big) \notag\\
     \leq  & C\big(1+ \|  \vert \overline{\nabla}^2 B_E \vert \circ \pi_{\pa E}  \vert \overline{\nabla}_{\pa E} \tilde{f}\vert \circ \pi_{\pa E}  \|_{L^2(\pa F)}  \big)\notag\\
     \leq & C \big(  1+ \| \overline{\nabla}_{\pa E}^2 B_E \|^2_{L^4(\pa E)} +\| \nabla_{\pa F} (\tilde{f} \circ \pi_{\pa E} ) \|_{L^4(\pa F)}^2 \big)\notag\\
     \leq & C \big(  1+ \| B_E \|_{H^3(\pa E)}^{\frac{10}{3}} \| B_E \|_{L^2(\pa E)}^{\frac{2}{3}}+ \| H_F^\varphi \|_{W^{1,4}(\pa F)}+ \| Q(E(u_{F}^{K_{el}})) \|_{L^4(\pa F)}     \big)\notag\\
     \leq & C(K_{el})     \label{AmandaLear02042026}
\end{align}
where in the second inequality we used a change of variables as in~\eqref{nunununununu01042026} and the estimates~\eqref{2203stimecurvthm}, \eqref{03042026form1} and \eqref{eq:boundssss}; in the third inequality we used that
\begin{equation}\label{02042026primacena1}
  \vert \nabla_{\pa E} \tilde{f} \vert \circ \pi_{\pa E}  \leq C \vert \nabla_{\pa F} (\tilde{f} \circ \pi_{\pa E} ) \vert  \text{ on } \pa F,
\end{equation}
see~\eqref{regdecat12}; in the fourth inequality we applied Young’s inequality and Proposition~\ref{prop:interpolation} together with the first equation in~\eqref{Cosac'e'} to obtain
 \begin{equation}
     \| \nabla_{\pa F} (\tilde{f} \circ \pi_{\pa E}) \|_{L^4(\pa E)} \leq \| H_F^\varphi \|_{W^{1,4}(\pa F)}+ \| Q(E(u_{F}^{K_{el}})) \|_{W^{1,4}(\pa F)} ;
 \end{equation}
in the last inequality we  used~\eqref{2203stimecurvthm} and~\eqref{eq:boundssss}.

 We prove the estimates of $ \| H_{F}^\varphi \|_{H^3(\pa F)} $ and of $ \| H_{F} \|_{H^3(\pa F)} $.
 Computing $\nabla_{\pa F} \Delta_{\pa F}$ applied to the first equation in~\eqref{Cosac'e'} and using~\eqref{Lafolliacorelli} we obtain
\begin{align}
     \| \nabla_{\pa F} \Delta_{\pa F} H^\varphi_{F} \|_{L^2(\pa F)} &\leq \| \tilde{f} \circ \pi_{\pa E} \|_{H^3(\pa F)}+ \|   \nabla_{\pa F} \Delta_{\pa F} Q(E(u_F^{K_{el}})) \|_{L^2(\pa F)} \notag\\
     &\leq C+ \|   \nabla_{\pa F} \Delta_{\pa F} Q(E(u_F^{K_{el}})) \|_{L^2(\pa F)}.\label{02042026ILMAESTRO1}
 \end{align}
 Therefore, we need to estimate  $\|   \nabla_{\pa F} \Delta_{\pa F} Q(E(u_F^{K_{el}})) \|_{L^2(\pa F)}.$ Using the arguments of Lemma~\ref{Lispetttoregadget}, we obtain
\begin{equation}
\begin{split}
  & \nabla_{\pa F} \Delta_{\pa F} Q(E(u_F^{K_{el}}))\\
   &= \nabla_{\pa F} \bigg( \Delta_{\R^3} Q(E(u_F^{K_{el}}))  -  \nabla^2 Q(E(u_F^{K_{el}}))  \nabla d_F \cdot\nabla d_F   - {\rm tr} \big( \nabla^2 d_F \big) \nabla Q(E(u_F^{K_{el}}))  \cdot \nabla d_F   \bigg).
   \end{split}
\end{equation}
By the above identity, together with the Leibniz rule and recalling~\eqref{02042026Eu}, we deduce
\begin{align}
  \| \nabla_{\pa F} \Delta_{\pa F} Q(E(u_F^{K_{el}})) \|_{L^2(\pa F)} \leq &   C \Big(  \|   \nabla_{\pa F} \nabla^3 u_F^{K_{el}} \|_{L^2(\pa F)}+ \| \vert \nabla^3  u_F^{K_{el}} \vert \vert B_F \vert \|_{L^2(\pa F)} \notag \\
  & \quad +\| \vert  \overline{\nabla}_{\pa F} B_F \vert \vert \nabla^2 u_{F}^{K_{el}} \vert \|_{L^2(\pa F)}+ \| \vert \nabla^2 u_F^{K_{el}} \vert \vert B_F \vert^2     \|_{L^2(\pa F)} \notag \\
  & \quad + \|  \vert H_F \vert \vert \nabla^3 u_F^{K_{el}} \vert   \|_{L^2(\pa F)}\Big)\notag\\
  \leq &  C \big(    C(K_{el})+ \| \overline{\nabla}_{\pa F} B_F \|_{L^2(\pa F)} \big) \leq C,\label{02042026ILMAESTRO2}
\end{align}
where in the second inequality we used~\eqref{minelvinc} and~\eqref{arremba01042026}, and in the last inequality we have applied \eqref{eq:boundssss} together with Lemma~\ref{lem:hilbert-norm}. Combining~\eqref{02042026ILMAESTRO1},~\eqref{02042026ILMAESTRO2}, and~\eqref{arremba01042026}, and invoking Lemma~\ref{lem:hilbert-norm}, we conclude the proof of the estimates.\\ 
 \textit{Step 3:} In this step we prove
 \begin{equation*}
      \| \Delta_{\pa F}^2 H_F \|_{L^2(\pa F)} \leq C h^{-\frac{1}{4}}. 
 \end{equation*}

 \emph{Claim:} It  holds $\| \Delta^2_{\pa F} (\tilde{f} \circ \pi_{\pa E}) \|_{L^2(\pa F)} \leq C h^{-\frac{1}{4}}.$
 Indeed, using formula~\eqref{30032026form1}, Lemma~\ref{lem:hilbert-norm} and the previous steps, we obtain
 \begin{align}
     \| \Delta^2_{\pa F}( \tilde{f} \circ \pi_{\pa E}) &\|_{L^2(\pa F)} \notag\\
     &\leq C \big( \|  \nabla^4 (\tilde{f} \circ \pi_{\pa E} )\|_{L^2(\pa F)}+ \| \vert \nabla^3 (\tilde{f} \circ \pi_{\pa E} )\vert \vert B_F \vert   \|_{L^2(\pa F)}\notag\\
     & \quad+ \| \vert \nabla^2 (\tilde{f} \circ \pi_{\pa E} )\vert \vert \overline{\nabla}_{\pa F} B_F \vert   \|_{L^2(\pa F)}+ \|\vert \nabla ( \tilde{f} \circ \pi_{\pa E} )\vert \vert \overline{\nabla}_{\pa F}^2 B_F \vert  \|_{L^2(\pa F)}    \big)\notag\\
     & \leq C \big( 1+ \|  \nabla^4 (\tilde{f} \circ \pi_{\pa E} )\|_{L^2(\pa F)}+  \|\vert \nabla ( \tilde{f} \circ \pi_{\pa E} )\vert \vert \overline{\nabla}_{\pa F}^2 B_F \vert  \|_{L^2(\pa F)}    \big)\notag\\
     & \leq C \big( 1+ \|  \nabla^4 (\tilde{f} \circ \pi_{\pa E} )\|_{L^2(\pa F)}+  \|\vert \overline{\nabla}_{\pa F} ( \tilde{f} \circ \pi_{\pa E} )\vert \vert \overline{\nabla}_{\pa F}^2 B_F \vert  \|_{L^2(\pa F)}    \big)\notag\\
     & \leq C \big(1+ \|  \nabla^4 (\tilde{f} \circ \pi_{\pa E} )\|_{L^2(\pa F)}\notag\\
     & \quad+ \| \tilde{f} \circ \pi_{\pa E} \|_{L^\infty(\pa F)} \|  B_F \|_{H^2(\pa F)}+ \| \overline{\nabla}_{\pa F} B_F \|_{L^\infty(\pa F)} \| \tilde{f} \circ \pi_{\pa F} \|_{H^1(\pa F)}    \big)\notag\\
      & \leq C \big( 1+ \|  \nabla^4 (\tilde{f} \circ \pi_{\pa E} )\|_{L^2(\pa F)}   \big),\label{02042026primacena2z}
 \end{align}
 where  in the third inequality we also used~\eqref{01042026vorrei} and~\eqref{02042026primacena1}, and 
 in the fourth inequality we applied Lemma~\ref{lem:Leibniz} with $T_1= \tilde{f}\circ \pi_{\pa E}$ and $T_2= \overline{\nabla} B_F$. Next, using~\eqref{01042026vorrei}, we obtain
 \begin{equation*}
 \begin{split}
     | \nabla^4 (\tilde{f} \circ \pi_{\pa E}) | &\leq C \big(  \vert \overline{\nabla}_{\pa E}^4 \tilde{f} \vert + \vert \overline{\nabla}_{\pa E} B_E \vert \vert \overline{\nabla}_{\pa E}^3 \tilde{f} \vert +\vert \overline{\nabla}^2_{\pa E} B_E \vert \vert \overline{\nabla}_{\pa E}^2 \tilde{f} \vert + \vert \overline{\nabla}^3_{\pa E} B_E \vert \vert \overline{\nabla}_{\pa E} \tilde{f} \vert  \\
     &\quad+ \vert \overline{\nabla}_{\pa E} B_E \vert + \vert \overline{\nabla}_{\pa E}^3 \tilde{f}\vert + \vert \overline{\nabla}_{\pa E}^2 B_E \vert + \vert \overline{\nabla}_{\pa E}^2 \tilde{f} \vert + \vert\overline{\nabla}^3_{\pa E} B_E \vert+ \vert \overline{\nabla}_{\pa E} \tilde{f} \vert  \big) \circ\pi_{\pa E}.
     \end{split}
 \end{equation*}
 From which we infer the desired estimate using also~\eqref{02042026primacena2z}.  We only present the estimates for the terms appearing in the first line of the above inequality as the remaining terms can be treated analogously. Hence, using a change of variables, the Cauchy-Schwarz inequality, formulas~\eqref{2203stimecurvthm} and~\eqref{03042026form1}, we obtain 
 \begin{equation*}
 \begin{split}
    \|\nabla^4 (\tilde{f} \circ \pi_{\pa E})\|_{L^2(\pa F)}&\leq \| \overline{\nabla}_{\pa E}^4 \tilde{f} \|_{L^2(\pa E)} + \| \overline{\nabla}_{\pa E} B_E \|_{L^\infty(\pa E)} \| \overline{\nabla}_{\pa E}^3 \tilde{f} \|_{L^2(\pa E)} \\
    & \quad +\| \overline{\nabla}^2_{\pa E} B_E \|_{L^4(\pa E)}^2 \| \overline{\nabla}_{\pa E}^2 \tilde{f} \|_{L^4(\pa E)}^2 + \| \overline{\nabla}^3_{\pa E} B_E \|_{L^2(\pa E)} \| \overline{\nabla}_{\pa E} \tilde{f} \|_{L^\infty(\pa E)}\\
    & \leq C \big(   h^{-\frac{1}{4}}+ 1+ \| \overline{\nabla}^2_{\pa E} B_E \|_{L^4(\pa E)}^2 \| \overline{\nabla}_{\pa E}^2 \tilde{f} \|_{L^4(\pa E)}^2 \big) \\
    & \leq C \big( h^{-\frac{1}{4}}+ \| B_E \|_{H^3(\pa E)}^{\frac{5}{4}}\| B_E \|_{L^2(\pa E)}^{\frac{1}{3}} \| \tilde{f}\|_{H^3(\pa E)}^{\frac{5}{3}}\| \tilde{ f} \|_{L^2(\pa E)}^{\frac{1}{3}} \big) \\
    & \leq C \big( 1+ h^{-\frac{1}{4}}  \big)
    \end{split}
 \end{equation*}
where  in the third inequality we also applied Proposition~\eqref{prop:interpolation}. Combining the above estimates, the claim follows.

 Now computing the $\Delta^2_{\pa F}$ of the first equation of~\eqref{Cosac'e'}, we obtain
 \begin{equation*}
 \begin{split}
     \| \Delta^2_{\pa F} H_F^\varphi \|_{L^2(\pa F)} \leq &\| \Delta^2_{\pa F} (\tilde{f} \circ \pi_{\pa E}) \|_{L^2(\pa F)}+ \| \Delta^2_{\pa F} Q(E(u_F^{K_{el}}))\|_{L^2(\pa F)} \leq C h^{-\frac{1}{4}},
     \end{split}
 \end{equation*}
 where we used the claim, Lemma~\ref{Lispetttoregadget} and~\eqref{minelvinc}. This concludes the proof.
\end{proof}

Recall that $E_0 \Subset \Omega$ is the initial datum of our problem; see~\ref{Omegaevrietadiriferimento}.
In the following lemma, we prove the existence of a constant $\sigma_1$, see~\eqref{AnitraD}, such that if $ \pa E $ is a normal graph over $\pa E_0$ and $ d_{\mathcal{H}}(\pa E, \pa E_0) \leq \sigma_1$, then $\pa F$ is a normal graph over $\pa E_0$, and moreover $ d_{\mathcal{H}}(\pa F, \pa E_0)\leq 2 \sigma_1$.
\begin{lemma}\label{0105nottelemma}
Let $E$ and $F$ be as in Lemma~\ref{lemmahnegativa}. Then there exist $\sigma_1 > 0$ and $h_3 \leq h_2$, where $h_2$ is the constant provided by Lemma~\ref{Lemma:stimefinali}, such that if $E \in \mathfrak{H}^5_{K_0,\sigma_1}(E_0)$, then for every $h \leq h_3$ it holds that
\begin{equation}
     F \in \mathfrak{H}^5_{ K_1, 2\sigma_1   }(E_0)
    \text{ and } F \in \mathfrak{H}^6_{h^{-1/4}K_1,2 \sigma_1}(E_0)
\end{equation}
where $K_1=K_1(K_0,K_{el},C_2)$ and $C_2$ is the constant provided by Lemma~\ref{lemmahnegativa}.
\end{lemma}
\begin{proof}
  By Proposition~\ref{CaneBernard06}, we have that $F$ is a $(\Lambda,\alpha)$-minimizer of the anisotropic perimeter. Moreover, by Lemma~\ref{Lemma:stimefinali}, we obtain that $ d_{\mathcal{H}}(\pa E, \pa F) \leq C h^{\frac{3}{4}}$.  Let $\delta_0$ be given by Lemma~\ref{epsilonregolarita} with $E = E_0$ and with the constant $\Lambda$ given by Lemma~\ref{CaneBernard06}. Now, for $\sigma_1 \leq \delta_0/2$, we have
  \begin{equation}
      d_{\mathcal{H}}(\pa F, \pa E_0 ) \leq d_{\mathcal{H}}(\pa F, \pa E)+ d_{\mathcal{H}}(\pa E, \pa E_0) \leq  C h^{3/4} +\sigma_1 \leq \delta_0,
  \end{equation}
  for $h$ sufficiently small. Hence, combining Lemma~\ref{epsilonregolarita} with Lemma~\ref{lemmahnegativa}, we conclude.
\end{proof}
\subsubsection{Estimate for the elastic energy}
In this subsubsection, we show that the minimum of the constrained  problem~\eqref{minelvinc} coincides with the minimum of ~\eqref{minelast}, under suitable assumptions on the constant $K_{el}$ appearing in~\eqref{minelvinc}. Throughout this subsubsection, we assume that $E$ and $F$ are the same sets as in Lemma~\ref{lemmahnegativa} and Lemma~\ref{0105nottelemma}. 
We proceed with two technical lemmas that will be instrumental in achieving our goal.

\begin{lemma}\label{Lemmatecnico1elsticicita}
  Let $0 < \delta < \sigma_0$, where $\sigma_0$ is defined in~\eqref{Laconstantesigmazerofin}. Let $E_f \Subset \Omega$ be such that  $\pa E_f= \{ x+ f(x)\nu_{E_0}(x): x \in \pa E_0\}$ with $ \| f \|_{C^{3,\alpha}(\pa E_0)} \leq \delta$. If $f \in H^k(\partial E_0)$ with $k \in \{5,6\}$, then
 \begin{equation}\label{26042026onepiece1}
 \| \nabla^3 u_{E_f}(\cdot + f(\cdot) \nu_{E_0}(\cdot)) \|_{H^{k-\frac{7}{2}}(\pa E_0)} \leq C \big(  1+ \| f \|_{H^{k}(\pa E_0)} \big)
 \end{equation}
 where $u_{E_f}$ denotes the minimizer of problem~\eqref{minelast}, and $C=C(k, \delta)$.
\end{lemma}
\begin{proof}
    Without loss of generality, we may assume that $ f$ is a smooth function. It follows that $u_{E_f}$ is also smooth, since it solves system~\eqref{eqelliel}. 
   
    We start by considering a diffeomorphism to locally flatten $\pa E_0$, and study how the differential equation satisfied by $u_{E_f}$ changes under this coordinate transformation.
    We define  $\Phi_f: \Omega \setminus E_0 \rightarrow \Omega \setminus E_f$
    by $ \Phi_f(x):= x+ f\circ \pi_{ \pa E_0}(x)\nu_{E_0} \circ \pi_{\pa E_0}(x)$ for all $ x \in \mathcal{I}_\delta^{+}(\pa E_0)$, where $\mathcal{I}^+_\delta( \pa E_0) := \{ x \in \Omega \setminus E_0: d_{E_0}(x) < \delta \}$. Moreover, we may assume that $\| \Phi_f - {\rm id} \|_{H^k(\Omega \setminus E_0)} + \| \Phi_f^{-1} - {\rm id} \|_{H^k(\Omega \setminus E_0)} \leq C \| f \|_{H^k(\pa E_0)}.$
    Let $x_0 \in \pa E_0$ and let $\Phi: U \rightarrow B_{2R}$ be another diffeomorphism, where $U \Subset \Omega$ is a neighborhood  of $x_0$ and $\Phi(U \setminus E_0)= B^+_{2R}$, with $B^+_{2R}:= B_{2R} \cap \{ x_3 >0 \}$. 
    We set $v:= u_{E_f} \circ \Phi_f^{-1} \circ \Phi^{-1}  \text{ and } \bar{f}:= f \circ \pi_{\pa E_0} \circ \Phi^{-1}.$
    We observe that by~\cite[Lemma 7.1]{FJM2020} it holds
    \begin{equation}\label{27042026porcorosso2}
        \| \bar{f} \|_{H^k(B^+_{2R})} \leq C \big(  1+ \| f \|_{H^k(\pa E_0)} \big).
    \end{equation}
    Moreover, since $u_{E_f}$ solves~\eqref{eqelliel}, $v $ satisfies 
    \begin{equation}\label{26042026dopopranzo1}
        \int_{B^+_{2R}} \mathbb{A}(\cdot, \bar{f}, \nabla \bar{f}) \nabla v : \nabla \xi \, d x=0,
    \end{equation}
    for every $\xi \in C^\infty(B^+_{2R}, \R^3)$ such that $\xi=0$ on $\pa B_{2R}^+ \cap \{ x_3 >0 \}$, where $\mathbb{A}$ is a smooth tensor. We observe that the above equation in divergence form reads
    \begin{equation}\label{27042026chic'easd}
        \div \big( \mathbb{A}(\cdot, \bar{f}, \nabla \bar{f}) \nabla v   \big)=0 \text{ in } B^+_{2R}.
    \end{equation}
    Furthermore, for all $\xi \in C^\infty(B^+_{2R}, \R^3)$ such that $\xi=0$ on $ \pa B_{2R}^+ \cap \{ x_3 >0 \}$ it holds
    \begin{equation}\label{26042026primapale1}
        \int_{B^+_{2R}} \mathbb{A}(\cdot, \bar{f}, \nabla \bar{f}) \nabla \xi : \nabla \xi \, dx \geq c \int_{B^+_{2R}} \vert \nabla \xi \vert^2 \, dx. 
    \end{equation}
    
    We turn to prove~\eqref{26042026onepiece1}.
    Let $\beta  \in \N^3$ be a multi-index such that $ \vert \beta \vert =k-1$ for $k \in \{5,6\}$. We differentiate equation~\eqref{26042026dopopranzo1}, more precisely, we apply  $\nabla^\beta = \frac{\pa^{\beta_1}}{\pa x_1^{\beta_1}} \frac{\pa^{\beta_2}}{\pa x_2^{\beta_2}}$ to~\eqref{26042026dopopranzo1}. 
    
    We start with the case when $\beta=(\beta_1,\beta_2,0)$. By differentiating~\eqref{26042026dopopranzo1} in the direction $\beta$, we obtain
    \begin{equation}\label{26042026dopocaffe1}
      \int_{B^+_{2R}} \nabla^\beta(\mathbb{A}(\cdot, \bar{f}, \nabla \bar{f}) \nabla v ): \nabla \xi \, d x=0  
    \end{equation}
    for every $\xi \in C^\infty(B^+_{2R}, \R^3)$ such that $\xi=0$ on $ \pa B_{2R}^+ \cap \{ x_3 >0 \}$. Let $\eta \in C^\infty_c(B_{2R})$ such that $ \eta=1$ in $B_R$ and $ 0 \leq \eta \leq 1$. We set $ \xi:= \eta^2 \nabla^\beta v$ and use this as a test function in~\eqref{26042026dopocaffe1}. By the Leibniz rule, we obtain
    \begin{equation}
    \nabla^\beta (\mathbb{A}(\cdot, \bar{f}, \nabla \bar{f}) \nabla v ) =  \mathbb{A}(\cdot, \bar{f}, \nabla \bar{f}) \nabla \nabla^\beta v+ \sum_{\substack{0 \neq \tilde{\beta}  \leq \beta , \\ \tilde{\beta}=(\tilde{\beta}_1, \tilde{\beta}_2,0)  } }c_{\tilde{\beta}}\nabla^{\tilde{\beta }} (\mathbb{A}(\cdot, \bar{f}, \nabla \bar{f}))   \nabla^{\beta-\tilde{\beta}}\nabla v    ,
    \end{equation}
    where $ \tilde{\beta}\leq \beta$ if and only if $ \tilde{\beta}_i \leq \beta_i$ for $i=1,2$.
    From this, we deduce
    \begin{equation*}
        \begin{split}
            &\int_{B_{2R}^+} \eta^2 \mathbb{A}(\cdot, \bar{f}, \nabla \bar{f}) \nabla \nabla^\beta v : \nabla \nabla^\beta v \\
            &\leq C \int_{B^+_{2R}} \vert \nabla \nabla^\beta v \vert \eta \vert \nabla \eta \vert \vert \nabla^\beta v \vert + C \sum_{j=1}^{k-1} \int_{B^+_{2R}} \vert \nabla^j \mathbb{A}(\cdot, \bar{f}, \nabla \bar{f}) \vert \vert \nabla^{k-j} v \vert\big(\eta^2 \vert \nabla \nabla^\beta v \vert + \eta \vert \nabla \eta \vert \vert \nabla^\beta v \vert  \big) .
        \end{split}
    \end{equation*}
    Next, using~\eqref{26042026primapale1} with $\xi = \eta \nabla^\beta v$, we obtain
    \begin{equation*}
        c \int_{B^+_{2R}} \vert \nabla (\eta \nabla^\beta v)\vert^2 \leq \int_{B^+_{2R}} \mathbb{A}(\cdot, \bar{f}, \nabla \bar{f}) \nabla (\eta \nabla^\beta v) : \nabla (\eta \nabla^\beta v).
    \end{equation*}
    Combining the two above inequalities and using $ \nabla(\eta \nabla^\beta v)= \eta \nabla \nabla^\beta v+ \nabla^\beta v \nabla \eta $, we obtain
    \begin{align*}
            \frac{c}{2} \int_{B^+_{2R}} \vert \nabla \nabla^\beta v \vert^2 \eta^2 \leq& C\int_{B^+_{2R}} \vert \nabla (\eta \nabla^\beta v) \vert^2+ C \int_{B^+_{2R}} \vert \nabla^\beta v \vert^2 \vert \nabla \eta \vert^2 \\
            \leq & \int_{B^+_{2R}} \mathbb{A}(\cdot, \bar{f}, \nabla \bar{f}) \nabla (\eta \nabla^\beta v) : \nabla (\eta \nabla^\beta v)+   C \int_{B^+_{2R}} \vert \nabla^\beta v \vert^2 \vert \nabla \eta \vert^2 \\
            \leq & \int_{B^+_{2R}} \eta^2 \mathbb{A}(\cdot, \bar{f}, \nabla \bar{f})\nabla \nabla^\beta v : \nabla \nabla^\beta v\\
            & + C \int_{B_{2R}^+} \vert \nabla^\beta v \vert^2 \vert \nabla \eta \vert^2+ C \int_{B^+_{2R}} \vert \nabla \nabla^\beta v \vert \eta \vert \nabla \eta \vert \vert \nabla^\beta v \vert \\
            \leq & C \int_{B_{2R}^+} \vert \nabla^\beta v \vert^2 \vert \nabla \eta \vert^2+ C \int_{B^+_{2R}} \vert \nabla \nabla^\beta v \vert \eta \vert \nabla \eta \vert \vert \nabla^\beta v \vert \\
            &  +C \sum_{j=1}^{k-1} \int_{B^+_{2R}} \vert \nabla^j \mathbb{A}(\cdot, \bar{f}, \nabla \bar{f}) \vert \vert \nabla^{k-j} v \vert\big(\eta^2 \vert \nabla \nabla^\beta v \vert + \eta \vert \nabla \eta \vert \vert \nabla^\beta v \vert  \big) \\
            \leq & C(\varepsilon)\int_{B_{2R}^+} \vert \nabla^\beta v \vert^2 \vert \nabla \eta \vert^2+ \varepsilon \int_{B^+_{2R}} \vert \nabla \nabla^\beta v \vert^2 \eta^2  \\
            & + C \sum_{j=1}^{k-1} \int_{B^+_{2R}} \vert \nabla^j \mathbb{A}(\cdot, \bar{f}, \nabla \bar{f}) \vert^2 \vert \nabla^{k-j} v \vert^2
    \end{align*}
    where in the third inequality we used that $\bar{f}$ is bounded in $C^{3,\alpha}$, which implies that $\mathbb{A}(\cdot, \bar{f}, \nabla \bar{f})$ is bounded.  Hence, choosing $\varepsilon$ sufficiently small and using that $\eta = 1$ in $B_R^+$, we obtain
\begin{equation}\label{27042026porcorosso3}
    \int_{B_R^+} \vert \nabla \nabla^\beta v \vert^2 \leq C \int_{B_{2R}^+} \vert \nabla^{k-1} v \vert^2+ C \sum_{j=1}^{k-1} \int_{B^+_{2R}} \vert \nabla^j \mathbb{A}(\cdot, \bar{f}, \nabla \bar{f}) \vert^2 \vert \nabla^{k-j} v \vert^2.
\end{equation}

We estimate the term $\sum_{j=1}^{k-1} \int_{B^+_{2R}} \vert \nabla^j \mathbb{A}(\cdot, \bar{f}, \nabla \bar{f}) \vert^2 \vert \nabla^{k-j} v \vert^2 $.
We set $w:= \nabla \bar{f}$. Combining the Leibniz rule with the Fa\`a di Bruno's formula, we obtain
\begin{align*}
        \sum_{j=1}^{k-1} \vert \nabla^j \mathbb{A}(\cdot, \bar{f}, \nabla \bar{f}) \vert^2 \vert \nabla^{k-j} v \vert^2 &\leq C \sum_{j=1}^{k-1} \vert \nabla^{k-j} v \vert^2 \\
        & \quad+ C \sum_{j=1}^{k-1} \,\, \sum_{ \substack{1 \leq i_1 \leq \cdots \leq i_m \leq j \\
        i_1+\cdots+i_m \leq j\\
        m \in \N}} \vert \nabla^{i_1} w \vert^2 \cdots \vert \nabla^{i_m} w \vert^2 \vert \nabla^{k-j} v \vert^2.
\end{align*}
Integrating over $B^+_{2R}$ and applying  H\"older's inequality, we obtain
\begin{align}
    \sum_{j=1}^{k-1}&\int_{B^+_{2R}} \vert \nabla^j \mathbb{A}(\cdot, \bar{f}, \nabla \bar{f}) \vert^2 \vert \nabla^{k-j} v \vert^2 \notag\\
    & \quad \leq C \| v \|^2_{H^{k-1}(B^+_{2R})} \notag\\
    & \quad+ C \sum_{j=1}^{k-1} \,\, \sum_{ \substack{1 \leq i_1 \leq \cdots \leq i_m \leq j \\
        i_1+\cdots+i_m \leq j\\
        m \in \N}} \|  \nabla^{i_1} w \|^2_{L^{ \frac{2(k-1)}{i_1} }(B^+_{2R})} \cdots \|  \nabla^{i_m} w \|^2_{L^{  \frac{2(k-1)}{i_m}  }(B^+_{2R})} \|  \nabla^{k-j} v \|^2_{L^{ \frac{2(k-1)}{k-1-I_m} }(B^+_{2R})}\notag\\
      &  \quad \leq C \| v \|^2_{H^{k-1}(B^+_{2R})} \label{27042026lalberoazzuro}\\
    & \quad+ C \sum_{j=1}^{k-1} \,\, \sum_{ \substack{1 \leq i_1 \leq \cdots \leq i_m \leq j \\
        i_1+\cdots+i_m \leq j\\
        m \in \N}} \|  \nabla^{i_1} w \|^2_{L^{ \frac{2(k-1)}{i_1} }(B^+_{2R})} \cdots \|  \nabla^{i_m} w \|^2_{L^{  \frac{2(k-1)}{i_m}  }(B^+_{2R})} \|  \nabla^{k-j} v \|^2_{L^{ \frac{2(k-1)}{k-1-j} }(B^+_{2R})} \notag
\end{align}
here $I_m:= i_1+ \cdots+ i_m$, and if $I_m=k-1$, we set $ \|  \nabla^{k-j} w \|^2_{L^{ \frac{2(k-1)}{k-1-I_m} }(B^+_{2R})}:= \| \nabla^{k-1} w\|_{L^\infty (B^+_{2R})} $, and in the last inequality we have used $ L^{\frac{2(k-1)}{k-1-I_m}}(B^+_{2R}) \subset L^{\frac{2(k-1)}{k-1-j}}(B_{2R}^+)$.  We also remark that the assumption $ \| f \|_{C^{3,\alpha}(\pa E_0)} \leq \delta $ together with standard Schauder estimates, implies $\|v\|_{C^3(B^+_{2R})} \leq C(\delta)$. Using interpolation inequalities, we obtain
\begin{equation*}
    \| \nabla^{i_l} w \|_{L^{  \frac{2(k-1)}{i_l} } (B^+_{2R})    } \leq C \| w \|^{ \theta(i_l)}_{ H^{k-1}(B^+_{2R}) } \| w \|^{1-\theta(i_l)}_{L^\infty(B^+_{2R})} \leq C \|   w \|^{ \theta(i_l)}_{ H^{k-1}(B^+_{2R}) }
\end{equation*}
where $\theta(i_l):= i_l / (k-1)$. Similarly,
\begin{equation*}
    \|   \nabla^{k-j} v \|_{ L^{ \frac{2(k-1)}{k-1-j} }(B^+_{2R})  } \leq  C\| v \|^\theta_{ H^k(B^+_{2R}) } \| \nabla v \|^{1-\theta}_{L^\infty(B^+_{2R})   } \leq C \| v \|^\theta_{ H^k(B^+_{2R}) }
\end{equation*}
where $\theta:= (k-1-j)/ (k-1)$. Since $ \theta(i_1)+ \cdots+ \theta(i_m) \leq  \frac{j}{k-1}$, combining these inequalities and using Young’s inequality yield
\begin{align*}
     \|  \nabla^{i_1} &w \|^2_{L^{ \frac{2(k-1)}{i_1} }(B^+_{2R})} \cdots \|  \nabla^{i_m} w \|^2_{L^{  \frac{2(k-1)}{i_m}  }(B^+_{2R})} \|  \nabla^{k-j} v \|^2_{L^{ \frac{2(k-1)}{k-1-j} }(B^+_{2R})} 
     \\
     &\leq C \| w \|_{H^{k-1}(B^+_{2R})}^{2 (\theta(i_1)+ \cdots + \theta(i_m))} \| v \|^{2 \theta}_{H^k(B^+_{2R})} \leq C \| w \|_{ H^{k-1}(B^+_{2R}) }^{2 (\theta(i_1)+ \cdots + \theta(i_m)) \frac{1}{1-\theta}}+ \varepsilon \| v \|^2_{H^{k}(B_{2R}^+)}\\
     & \leq C (1+ \| w \|_{ H^{k-1}(B^+_{2R}) } )^{2\frac{j}{k-1}\frac{1}{1-\theta}}+ \varepsilon \| v \|^2_{H^{k}(B_{2R}^+)} \\
     &\leq \varepsilon \|  \nabla^k v \|^2_{L^2(B_{2R}^+)}+ C \|  v \|^2_{H^{k-1}(B_{2R}^+)}+C\big(1+ \| w \|_{H^{k-1}(B_{2R}^+)}^2  \big).
\end{align*}
This combined with~\eqref{27042026porcorosso2} and~\eqref{27042026lalberoazzuro} imply
\begin{multline}
\sum_{j=1}^{k-1} \int_{B^+_{2R}} \vert \nabla^j \mathbb{A}(\cdot, \bar{f}, \nabla \bar{f}) \vert^2 \vert \nabla^{k-j} v \vert^2 \\ \leq \varepsilon \|  \nabla^k v \|^2_{L^2(B_{2R}^+)}+ C \|  v \|^2_{H^{k-1}(B_{2R}^+)}+C\big(1+ \| f \|_{H^{k}(\pa E_0)}^2  \big).
\end{multline}

Finally,  using also~\eqref{27042026porcorosso3} we conclude
\begin{equation}\label{27042026fineclaim1}
    \int_{B_R^+} \vert \nabla \nabla^\beta v \vert^2 \leq \varepsilon \|  \nabla^k v \|^2_{L^2(B_{2R}^+)}+ C \|  v \|^2_{H^{k-1}(B_{2R}^+)}+C\big(1+ \| f \|_{H^{k}(\pa E_0)}^2  \big).
\end{equation}

We now consider the general case and estimate the derivatives in the remaining directions.
We proceed by differentiating equation~\eqref{27042026chic'easd}. We argue by induction. Let $ \beta= (\beta_1, \beta_2,0)$ with $ \vert \beta \vert= k-2$. Then arguing as in the previous case, we obtain an estimate for $ \nabla^\beta \pa_{x_3} v$. Next, if $\beta$ satisfies $ \vert \beta \vert= k-3$, we obtain an estimate for $ \nabla^\beta \pa_{x_3}^2 v$. We repeat this procedure until $|\beta| = 0$,  which gives an estimate for $\partial_{x_3}^{k-1} v$. Consequently, we obtain
\begin{equation*}
    \int_{B_R^+} \vert \nabla \nabla \pa_{x_3}^{k-1} v \vert^2 \leq \varepsilon \|  \nabla^k v \|^2_{L^2(B_{2R}^+)}+ C \|  v \|^2_{H^{k-1}(B_{2R}^+)}+C\big(1+ \| f \|_{H^{k}(\pa E_0)}^2  \big).
\end{equation*}
By combining this estimate with~\eqref{27042026fineclaim1}, we deduce
\begin{equation}\label{28042026crossfitxharju}
    \int_{B_R^+} \vert \nabla^k v \vert^2 \leq \varepsilon \| \nabla^k v \|^2_{L^2(B^+_{2R})}+ C \| v \|^2_{H^{k-1}(B^+_{2R})}+ C \big( 1+ \| f \|^2_{H^k(\pa E_0)} \big).
\end{equation}

 To conclude the validity of~\eqref{26042026onepiece1}, we consider a covering of $ \mathcal{I}_{\frac{\delta}{2}}^+(\pa E_0)$ with a finite collection of balls of radius $\delta$. Recalling that $v= u_{E_f} \circ \Phi_f^{-1} \circ \Phi$, we set $u:= u_{E_f} \circ \Phi_f^{-1}$. By a change of variables in~\eqref{28042026crossfitxharju}, we obtain
\begin{align*}
    \int_{\mathcal{I}^+_{\frac{\delta}{2}}(\pa E_0)} \vert \nabla^k u \vert^2 \, dx & \leq \varepsilon C \int_{\mathcal{I}^+_{\delta}(\pa E_0)} \vert \nabla^k u \vert^2 \, dx+ C \|  u \|^2_{H^{k-1} (\mathcal{I}^+_{\delta}(\pa E_0)) }+ C  \big(  1+ \| f \|^2_{H^k(\pa E_0)}  \big)\\
    & \leq 2 \varepsilon C \int_{\mathcal{I}^+_{\delta}(\pa E_0)} \vert \nabla^k u \vert^2 \, dx+ C \|  u \|^2_{L^2( \mathcal{I}^+_{\delta}(\pa E_0) )}+C  \big(  1+ \| f \|^2_{H^k(\pa E_0)}  \big),
\end{align*}
where, in the second inequality, we used the standard interpolation inequality. Therefore, choosing $\varepsilon$ sufficiently small, we obtain 
\begin{align}
    \int_{\mathcal{I}^+_{\frac{\delta}{2}}(\pa E_0)} \vert \nabla^k u \vert^2 \, dx & \leq \varepsilon C \|  \nabla^k u \|^2_{L^2( \mathcal{I}^+_{\delta}(\pa E_0) \setminus \mathcal{I}^+_{\frac{\delta}{2}}(\pa E_0)  )}+ C \|  u \|^2_{L^2( \mathcal{I}^+_{\delta}(\pa E_0) )}+C \big(  1+ \| f \|^2_{H^k(\pa E_0)}  \big)\notag\\
    & \leq  C \|  u_{E_f} \|^2_{L^2( \Omega \setminus E_f )}+C \big(  1+ \| f \|^2_{H^6(\pa E_0)}  \big) \notag\\
    & \leq C \big(  1+ \| f \|^2_{H^6(\pa E_0)}  \big),\label{28042026battiato1}
\end{align}
where in the second inequality we  used standard interior regularity for elliptic equations (see, for instance,~\cite[Theorem 9.11]{GilbargTrudinger1977}), in the fourth inequality we  used the minimality of $u_{E_f}$ together with the Poincaré inequality to obtain
\begin{align*}
    \|  u_{E_f} \|_{L^2( \Omega \setminus E_f )} & \leq \|  u_{E_f}- w_0 \|_{L^2( \Omega \setminus E_f )}+ \|  w_0\|_{L^2( \Omega \setminus E_f )} \\
    &\leq C \|\nabla(  u_{E_f}- w_0 ) \|_{L^2( \Omega \setminus E_f )}+ \|  w_0\|_{L^2( \Omega \setminus E_f )} \\
    & \leq C \| w_0 \|_{H^1(\Omega \setminus E_f)}+ C \bigg( \int_{\Omega \setminus E_f} Q(E(u_{E_f}))  \bigg)^{\frac{1}{2}} \\
    & \leq C \| w_0 \|_{H^1(\Omega \setminus E_f)},
\end{align*}
where we recall $w_0$ is the prescribed boundary displacement. 
Finally, combining~\eqref{28042026battiato1} with the Sobolev trace theorem yields~\eqref{26042026onepiece1}.
\end{proof}
We now turn to the second technical lemma concerning the elastic displacement.
\begin{lemma}\label{lemmanichel}
     Let $0 < \delta < \sigma_0$, where $\sigma_0$ is defined in~\eqref{Laconstantesigmazerofin}. Let $E_{f_i} \Subset \Omega$ be such that  $\pa E_{f_i}= \{ x+ f_i(x)\nu_{E_0}(x): x \in \pa E_0\}$ with $\| f_i \|_{C^{3,\alpha}(\partial E_0)} \leq \delta$ and $\| f_i \|_{H^5(\partial E_0)} \leq \tilde C$ for $i=1,2$. Then the following estimate holds:
 \begin{equation}\label{28042026meriggiare}
     \| \nabla^3 u_{E_{f_2}}(\cdot + f_2(\cdot) \nu_{E_0}(\cdot))- \nabla^3 u_{E_{f_1}}(\cdot + f_1(\cdot) \nu_{E_0}(\cdot)) \|_{L^2(\pa E_0)} \leq C \| f_2-f_1 \|_{H^4(\pa E_0)},
 \end{equation}
 where $u_{E_{f_i}}$ denotes the solution of~\eqref{minelast} associated with $E_{f_i}$, and $C = C(\delta, \tilde C)$.
\end{lemma}
\begin{proof}
The proof of this lemma is very similar to that of the previous one; therefore, we only provide a sketch. We adopt the same notation as in Lemma~\ref{Lemmatecnico1elsticicita}. Using the $C^{3,\alpha}$ bound for $f_i$, we may assume that
    \begin{equation}\label{30042026gosino}
        \| \Phi_{f_2} -\Phi_{f_1} \|_{ C^{k,\alpha}(\Omega \setminus E_0)}+ \| \Phi_{f_2}^{-1} -\Phi_{f_1}^{-1} \|_{C^{k,\alpha} (\Omega \setminus E_0)} \leq C \| f_2-f_1 \|_{C^{k,\alpha}(\pa E_0)} \text{ for all } k \leq 3.
    \end{equation}
    Let $x_0 \in \pa E_0$ be fixed. As in Lemma~\ref{Lemmatecnico1elsticicita}, we define 
\begin{equation}
     v_i:= u_{E_{f_i}} \circ \Phi_{f_i} \circ \Phi^{-1}, \,\,\, \bar{f}_i:= f_i \circ \pi_{\pa E_0} \circ \Phi \text{ for }i=1,2,
\end{equation}
and we have 
\begin{equation}\label{28042026pallido}
    \int_{B^+_{2R}} \mathbb{A}(\cdot, \bar{f}_i, \nabla \bar{f}_i) \nabla v_i : \nabla \xi \, dx=0 \text{ for }i=1,2
\end{equation}
for every $ \xi \in C^\infty(B^+_{2R}, \R^3)$ such that $ \xi=0$ on $ \pa B^+_{2R} \cap \{ x_3>0\}$. Let $\beta \in \N^3$ be a multi-index such that $ \vert \beta \vert=3$. To obtain~\eqref{28042026meriggiare}, we differentiate the two equations in~\eqref{28042026pallido} and subtract the resulting identities. As in  Lemma~\ref{Lemmatecnico1elsticicita}, we first consider the case $\beta = (\beta_1, \beta_2, 0)$ and then treat the case $\beta_3 \neq 0$.

We begin by computing  $\nabla^\beta (\mathbb{A}(\cdot, \bar{f}_i, \nabla \bar{f}_i) \nabla v_i ) $ for $\beta=(\beta_1,\beta_2,0)$
\begin{equation}\label{28042026assorto}
    \nabla^\beta (\mathbb{A}(\cdot, \bar{f}_i, \nabla \bar{f}_i) \nabla v_i ) =  \mathbb{A}(\cdot, \bar{f}_i, \nabla \bar{f}_i) \nabla \nabla^\beta v_i+ \sum_{\substack{0 \neq  \tilde{\beta}  \leq  \beta , \\ \tilde{\beta}=(\tilde{\beta}_1, \tilde{\beta}_2,0)  } }c_{\tilde{\beta}}\nabla^{\tilde{\beta }} (\mathbb{A}(\cdot, \bar{f}_i, \nabla \bar{f}_i))   \nabla^{\beta-\tilde{\beta}}\nabla v_i,    
\end{equation}
for $i=1,2$. Subtracting the two identities in~\eqref{28042026assorto} and using~\eqref{28042026pallido}, we obtain
\begin{align*}
     \int_{B^+_{2R}} \mathbb{A}(\cdot,\bar{f}_2,&\nabla \bar{f}_2) \big(   \nabla \nabla^{\beta} (v_2- v_1) \big) : \nabla \xi \, dx \\
     & = \int_{B^+_{2R}} \big( \mathbb{A}(\cdot,\bar{f}_1,\nabla \bar{f}_1)- \mathbb{A}(\cdot,\bar{f}_2,\nabla \bar{f}_2)   \big) \nabla^{\beta} \nabla v_1 : \nabla \xi\, dx \\
     &\quad -\sum_{\substack{0 \neq  \tilde{\beta}  \leq \beta , \\ \tilde{\beta}=(\tilde{\beta}_1, \tilde{\beta}_2,0)  } }c_{\tilde{\beta}} \int_{B^+_{2R}} \nabla^{\tilde{\beta }} (\mathbb{A}(\cdot,\bar{f}_2,\nabla \bar{f}_2) )   \nabla^{\beta-\tilde{\beta}}\nabla (v_2   - v_1): \nabla \xi\, dx\\
     & \quad  - \sum_{\substack{0 \neq \tilde{\beta}  \leq  \beta , \\ \tilde{\beta}=(\tilde{\beta}_1, \tilde{\beta}_2,0)  } }c_{\tilde{\beta}} \int_{B^+_{2R}}\nabla^{\tilde{\beta}} \big( \mathbb{A}(\cdot,\bar{f}_2,\nabla \bar{f}_2)- \mathbb{A}(\cdot,\bar{f}_1,\nabla \bar{f}_1) \big) \nabla^{\beta-\tilde{\beta}} \nabla v_1 : \nabla \xi\, dx.
\end{align*}
We choose $ \xi:= \eta^2 \nabla^{\beta}(v_2-v_1)$ as a test function in the above identity and argue as in Lemma~\ref{Lemmatecnico1elsticicita} (in particular, see the proof of~\eqref{27042026porcorosso3}), this yields
 \begin{equation}\label{29042026presso}
     \begin{split}
        \int_{B^+_R} | \nabla \nabla^{\beta}(v_2-v_1) \vert^2 \, dx \leq & C\int_{B^+_{2R}} (\vert \bar{f}_2- \bar{f}_1 \vert^2+ \vert \nabla \bar{f}_2-\bar{f}_1 \vert^2) \vert  \nabla^4 v_1 \vert^2 \, dx \\
        & +C \int_{B^+_{2R}} \big(  1+ \sum_{i=1}^4 \vert \nabla^i \bar{f}_2 \vert^2 \big) \vert \nabla^3 (v_2-v_1)\vert^2\, dx\\
        & + C\int_{B^+_{2R}} \big(   \sum_{i=1}^4 \vert \nabla^i (\bar{f}_2-\bar{f}_1 )\vert^2 \big) \big(  \sum_{i=1}^3 \vert \nabla^i v_1 \vert^2 \big)\, dx,
     \end{split}
 \end{equation}
 where we used the estimates
 \begin{equation}
 \begin{split}
     \vert \mathbb{A}(\cdot, \bar{f}_2, \nabla \bar{f}_2)- \mathbb{A}(\cdot, \bar{f}_1, \nabla \bar{f}_1)  \vert &\leq C \big(  \vert \bar{f}_2-\bar{f}_1  \vert+ \vert \nabla (\bar{f}_2-\bar{f}_1 ) \vert \big),\\ 
    \sum_{\substack{0 \neq \tilde{\beta}  \leq  \beta , \\ \tilde{\beta}=(\tilde{\beta}_1, \tilde{\beta}_2,0)  } }   \vert \nabla^{\tilde \beta} \mathbb{A}(\cdot, \bar{f}_2, \nabla \bar{f}_2) \vert &\leq C \big(   1+ \sum_{i=1}^4 \vert \nabla^i \bar{f}_2  \vert\big),\\
      \sum_{\substack{0 \neq \tilde{\beta}  \leq  \beta , \\ \tilde{\beta}=(\tilde{\beta}_1, \tilde{\beta}_2,0)  } }  \vert \nabla^{\tilde{\beta}} \big( \mathbb{A}(\cdot,\bar{f}_2, \nabla \bar{f}_2)- \mathbb{A}(\cdot, \bar{f}_1, \nabla \bar{f}_1)  \big)  \vert  &\leq C \sum_{i=1}^4\vert\nabla^i (\bar{f}_2-\bar{f}_1 ) \vert .
     \end{split}
 \end{equation}
 It remains to estimate the right-hand side of~\eqref{29042026presso}. To this end, we fix $\varepsilon>0$, whose value will be specified later. \\
 \textit{Estimate of $\int_{B^+_{2R}} (\vert \bar{f}_2-\bar{f}_1 \vert^2+ \vert \nabla (\bar{f}_2-\bar{f}_1 ) \vert^2) \vert  \nabla^4 v_1 \vert^2 \, dx$.}\\
 Recalling that $E_{f_i}$ satisfies the assumptions of Lemma~\ref{Lemmatecnico1elsticicita}, we obtain
 $ \| v_i \|_{H^4(B^+_{2R})} \leq C $. Moreover, combining this with the estimate $ \| \bar{f}_2-\bar{f}_1  \|_{C^1(B^+_{2R})} \leq C \| f_2-f_1\|_{C^1(\pa E_0)}$ and applying the Sobolev embedding theorem, we deduce 
 \begin{equation*}
    \int_{B^+_{2R}} (\vert \bar{f}_2-\bar{f}_1  \vert^2+ \vert \nabla (\bar{f}_2-\bar{f}_1 ) \vert^2) \vert  \nabla^4 v_1 \vert^2 \, dx \leq C \| f_2-f_1 \|_{H^3(\pa E_0)}.
 \end{equation*}
 \textit{Estimate of $ \int_{B^+_{2R}} \big(  1+ \sum_{i=1}^4 \vert \nabla^i \bar{f}_2 \vert^2 \big) \vert \nabla^3 (v_2-v_1)\vert^2\, dx$.}\\
 Using  H\"older's inequality, the Sobolev embedding theorem and the interpolation inequalities, we obtain
 \begin{align*}
          \int_{B^+_{2R}} \hspace{-0.1cm} \big(  1+ \sum_{i=1}^4 \vert \nabla^i \bar{f}_2 \vert^2 \big) \vert \nabla^3 (v_2-v_1)\vert^2 &\leq \| v_2-v_1 \|^2_{H^3(B^+_{2R})}+ \| \bar{f}_2 \|^2_{W^{4,4}(B^+_{2R})} \| v_2-v_1 \|^2_{W^{3,4}(B^+_{2R})}\\
          & \leq \| v_2-v_1 \|^2_{H^3(B^+_{2R})}\\
          & \quad+ C \| \bar{f}_2 \|_{H^5(B^+_{2R})}^2 \| v_2-v_1 \|^{\frac{1}{2}}_{H^4(B^+_{2R})}\| v_2-v_1 \|^{\frac{3}{2}}_{H^2(B^+_{2R})}  \\
          & \leq \varepsilon \|v_2-v_1 \|^2_{H^4(B^+_{2R})}+ C(\varepsilon) \|v_2-v_1 \|^2_{H^3(B_{2R}^+)}
 \end{align*}
 where in the last inequality we applied Young’s inequality together with  formula~\eqref{27042026porcorosso2}.\\
 \textit{Estimate of $\int_{B^+_{2R}} \big(   \sum_{i=1}^4 \vert \nabla^i( \bar{f}_2-\bar{f}_1)\vert^2 \big) \big(  \sum_{i=1}^3 \vert \nabla^i v_1 \vert^2 \big)\, dx$.}\\
 Using Schauder estimates, we obtain $ \| v_1 \|_{C^{3,\alpha}(B^+_{2R})} \leq C$, and consequently, we deduce that
 \begin{equation*}
    \int_{B^+_{2R}} \big(   \sum_{i=1}^4 \vert \nabla^i (\bar{f}_2-\bar{f}_1)\vert^2 \big) \big(  \sum_{i=1}^3 \vert \nabla^i v_1 \vert^2 \big) \leq  C\| (\bar{f}_2-\bar{f}_1) \|_{H^4(B^+_{2R})}^2.
 \end{equation*}
 
 Combining~\eqref{29042026presso} with the above estimates, we infer 
 \begin{equation*}
     \begin{split}
      \int_{B^+_R} | \nabla &\nabla^{\beta}(v_2-v_1) \vert^2 \, dx \\
      &\leq \varepsilon \|v_2-v_1 \|^2_{H^4(B^+_{2R})}+ C \| \bar{f}_2-\bar{f}_1 \|^2_{H^4(B^+_{2R})}+ C \| v_2-v_1 \|^2_{H^3(B^+_{2R})}  . 
     \end{split}
 \end{equation*}

Similarly as in  Lemma~\ref{Lemmatecnico1elsticicita}, we can also estimate the derivative in the direction $e_3$. Hence, for every $ \varepsilon \in (0,1)$, we deduce 
 \begin{equation}
     \int_{B^+_R} \vert \nabla^4(v_2-v_1) \vert^2 \leq \varepsilon \int_{B^+_{2R}} \vert \nabla^4(v_2-v_1) \vert^2 + C \| f_2-f_1 \|^2_{H^4(\pa E_0)}+ C \| v_2-v_1 \|^2_{H^3(B^+_{2R})}, 
 \end{equation}
from which it can be shown  
 \begin{equation}\label{30042026ppranzo2}
    \int_{B^+_R} \vert \nabla^4(v_2-v_1) \vert^2 \leq  C \| \bar{f}_2-\bar{f}_1 \|^2_{H^4(\pa E_0)}+ C \| v_2-v_1 \|^2_{H^3(B^+_{2R})} . 
 \end{equation}
 From this, we conclude  as in the proof of~\cite[Theorem 4.1 and Remark 4.2]{FJM2020}.
\end{proof}

\begin{remark}\label{01052026remrkino}
    {\rm 
    We remark that if $E$ 
    satisfies~\eqref{AnitraD}, with $\sigma_1$ given by Lemma~\ref{0105nottelemma}, and satisfies the assumptions of Lemma~\ref{lemmahnegativa}, then $ \pa E= \{ x + f(x)\nu_{E_0}(x): x \in \pa E_0\}$ where $f \in C^{3,\alpha}(\pa E_0)$ for every $\alpha \in (0,1)$, $ \| f \|_{H^5(\pa E_0)} \leq K_0$ and $ \| f \|_{H^6(\pa E_0)} \leq K_0 h^{-1/4}$. We can thus apply Lemma~\ref{Lemmatecnico1elsticicita} and obtain
    \begin{equation}\label{02052026auriga}
        \begin{split}
            & \| \nabla^3 u_{E}(\cdot+ f(\cdot) \nu_{E_0}(\cdot)) \|_{H^{3/2}(\pa E_0)} \leq C(K_0),\\
            & \| \nabla^3 u_{E}(\cdot+ f(\cdot) \nu_{E_0}(\cdot)) \|_{H^{5/2}(\pa E_0)} \leq h^{-1/4}C(K_0).
        \end{split}
    \end{equation}
    By the regularity of $f$ together with Schauder estimates, we deduce that $u_E \in C^{3,\frac{1}{2}}(\Omega \setminus E)$, and the following holds:
    \begin{equation}\label{02052026brick}
        \| u_E \|_{C^{3,\frac{1}{4}}(\Omega \setminus E)} \leq C \big(   \| w_0 \|_{C^{3,\frac{1}{2}}(\pa \Omega)}+ \| f \|_{C^{3,\frac{1}{4}}(\pa E_0)} \big),
    \end{equation}
    where $C$ is a universal constant depending only on $\pa E_0$.
    We define a natural extension of $u_E$ in $\Omega$ as follows: 
    \begin{equation}
  \tilde{u}_E(x):=  \left\{
    \begin{aligned}
        & u_E(x)\quad \text{ for }x\in \Omega \setminus E \\
        & u_E (\pi_{\pa E_0}(x)+f(\pi_{\pa E_0}(x))\nu_{E_0}(x)   ) \eta (d_{E_0}(x)/ \sigma_0) \quad \text{ for } x \in  E,
    \end{aligned}
    \right.
\end{equation}
here $ \eta \in C^\infty_c(-2,2)$ is such that $\eta=1$ in $(-1,1)$ and $\eta\geq 0$. We also set
\begin{equation}\label{K_eldelremarkino}
    K_{el}:= 2 \max\{ C(K_0)+ \| w_0\|_{C^{3, \frac{1}{4}}(\pa \Omega)}, \| \tilde{u}_E \|_{C^{3,\frac{1}{4}}(\Omega)}   \}.
\end{equation}
Now, it is easy to check that $\tilde{u}_E$ is a minimizer of problem\eqref{minelvinc}. Hence, the solution $u_E$ of~\eqref{minelast} coincides with the solution $u_E^{K{el}}$ of~\eqref{minelvinc} for every $h \leq 1$ in $\Omega \setminus E$.
    \fr}
\end{remark}
We are now in a position to prove the main lemma of this subsubsection.
\begin{lemma}\label{02052026silvia}
    Let $E \Subset \Omega$ be a $C^6$-regular set such that~\eqref{AnitraD} holds, with $\sigma_1$ provided by Lemma~\ref{0105nottelemma}, and assume that $\| \Delta_{\pa E}^2 H_E \|_{L^2(\pa E)} \leq  K_0 h^{-\frac{1}{4}}$. Let $K_{el} > 0$ be the constant in~\eqref{K_eldelremarkino}, and let $F \subset \Omega$ be a set satisfying the assumptions of Lemma~\ref{0105nottelemma}. 
    Then, there exists $0<h_4 = h_4(K_0, K_{el}) \leq h_3$, where $h_3$ is given by Lemma~\ref{0105nottelemma}, such that the solution $u_F$ of~\eqref{minelast} coincides with the solution $u_F^{K_{el}}$ of~\eqref{minelvinc} for all $h \leq h_4$ in $\Omega \setminus F$.
\end{lemma}
\begin{proof}
  To establish the claim, we need to prove the following:
  \begin{itemize}
      \item[$a)$] $u_F$ admit a natural extension in $F$, and $ \| u_F \|_{C^{3,\frac{1}{4}}(\Omega)} < K_{el}$,
      \item[$b)$] $\| \nabla^3 u_F \|_{H^{1}(\pa F)} < K_{el}$,
      \item[$c)$] $ \| \Delta_{\pa F} \nabla^3 u_F \|_{L^2(\pa F)} < K_{el} h^{-1/4}$ for every $h \leq h_4$.
  \end{itemize}
  
   We start by proving that $a) $ holds.
  By Lemma~\ref{0105nottelemma}, we have that $F \in \mathfrak{H}^5_{K_1,2 \sigma_1}(E_0)$. Denote by $g: \pa E_0 \rightarrow \R$ the function such that $\pa F= \{ x+ g(x)\nu_{E_0}(x): x \in \pa E_0\}$ and $ g \in H^5(\pa E_0)$. By Lemma~\ref{Lemma:stimefinali}, we deduce
  \begin{equation}\label{cantoLasorpresa02052026}
      \| g-f \|_{L^1(\pa E_0)} \leq C \| \psi \|_{L^1(\pa E)} \leq C h,
  \end{equation}
  where $\psi $ is the function such that $\partial F$ is the normal graph over $\partial E$ with height function $\psi$, and $f $ is the function introduced in Remark~\ref{01052026remrkino}. Using  Proposition~\ref{prop:interpolation} together with the Sobolev embedding theorem, we deduce
  \begin{equation}
      \| g- f \|_{C^0(\pa E_0)} \leq C \| g-f \|_{H^2(\pa E_0)} \leq C \| g- f \|_{H^5(\pa E_0)}^{\theta_1} \| g-f \|_{L^1(\pa E_0)}^{1-\theta} \leq C h^{1-\theta_1}
  \end{equation}
  for some $\theta_1 \in (0,1)$. From this, and applying again the interpolation inequality and the Sobolev embedding theorem, we conclude
  \begin{equation}\label{lalocomotiva02052026}
      \| g- f\|_{C^{3,\frac{1}{2}}(\pa E_0)} \leq C \| g-f \|^{\theta_2}_{H^5(\pa E_0)} \| g-f \|^{1-\theta_2}_{C^0(\pa E_0)} \leq Ch^{\theta_3}
  \end{equation}
  for some $\theta_2 \in (0,1)$, and where $\theta_3= \theta_3(\theta_1, \theta_2)$.   Therefore, by classical Schauder estimate  we obtain $u_F \in C^{3,\frac{1}{2}}(\Omega \setminus F)$ and
  \begin{equation}
  \begin{split}
      \| u_F \|_{C^{3,\frac{1}{4}}(\Omega \setminus F)} &\leq C \big( \| w_0 \|_{C^{3,\frac{1}{2}}(\pa \Omega)}+ \| g\|_{C^{3,\frac{1}{4}}(\pa E_0)}   \big) \\
      & \leq C \big( \| w_0 \|_{C^{3,\frac{1}{4}}(\pa \Omega)} +C h^{\theta_3}+ \| f \|_{C^{3,\frac{1}{4}}(\pa E_0)} \big)<  K_{el} 
      \end{split}
  \end{equation}
  for $ h $ small enough, where in the above estimate we have used~\eqref{02052026brick}. Finally, arguing as in Remark~\ref{01052026remrkino}, we can construct a natural extension of $u_F$ such  that $ \| u_F \|_{C^{3,\frac{1}{4}}(\Omega)}< K_{el}$.

  Now, we prove $c)$; the proof of $b)$ is analogous. We define $\Psi_f: \pa E_0 \rightarrow \pa E$ and $\Psi_g: \pa E_0 \rightarrow \pa F$ by $\Psi_f(x):= x+ f(x)\nu_{E_0}(x) \text{ and } \Psi_g(x):= x+ g(x) \nu_{E_0}(x).$
  Using the fractional interpolation inequality (see, for instance,~\cite{LeoniBookfrac, VanSch2023}) and a change of variables, we obtain
  \begin{align*}
          \| \nabla^3 u_F \|_{H^2(\pa F)} &\leq C \| \nabla^3 u_F \circ \Psi_g \|_{H^2(\pa E_0)}\\
          & \leq C\| \nabla^3 u_F \circ \Psi_g- \nabla^3 u_E \circ \Psi_f \|_{H^2(\pa E_0)}+ C\| \nabla^3 u_E  \circ \Psi_f\|_{H^2(\pa E_0)}\\
          & \leq C \| \nabla^3 u_F \circ \Psi_g - \nabla^3 u_E \circ \Psi_f \|_{L^2(\pa E_0)}^{\frac{1}{5}} \| \nabla^3 u_F \circ \Psi_g - \nabla^3 u_E \circ \Psi_f \|_{H^{\frac{5}{2}}(\pa E_0)}^{\frac{4}{5}}\\
          & \quad+ C\| \nabla^3 u_E  \circ \Psi_f\|_{H^2(\pa E_0)}\\
          & \leq C h^{-1/5} \| \nabla^3 u_F \circ \Psi_g - \nabla^3 u_E \circ \Psi_f \|_{L^2(\pa E_0)}^{\frac{1}{5}} + Ch^{-1/4} \\
          & \leq C h^{-1/5} \| g-f \|^{1/5}_{H^4(\pa E_0)}+ Ch^{-1/4}\\
          & \leq  C h^{-1/5} \| g-f \|_{H^5(\pa E_0)}^{5/30} \| g-f \|_{L^1(\pa E_0)}^{1/30}+ Ch^{-1/4}\\
          & \leq C h^{-1/4} \big(   1+ h^{1/12} \big),
  \end{align*}
  here in the fourth inequality we used~\eqref{02052026auriga}, in the fifth inequality we used formula~\eqref{28042026meriggiare}, and in the last inequality we used~\eqref{cantoLasorpresa02052026}. We observe that  the change of variable in the first inequality above depends at most on the $C^2$ norm of $ g$ which is controlled by the Sobolev embedding theorem and does not grow significantly compared to the $C^2$-norm of $f$, see formula~\eqref{lalocomotiva02052026}.
  This concludes the proof of $c)$.
\end{proof}
\subsection{Main theorem of the regularity estimate}
We collect all the results obtained in this section into the following theorem.
\begin{theorem}\label{Bologna}
    Let $K_{el}>0$ be as in Remark~\ref{02052026falchi}. Let $r_0,K_0>0$ be fixed. There exist positive constants $\sigma_1=\sigma_1(r_0,K_0,K_{el})$, $ \eta_0= \eta_0(r_0,K_0,K_{el})$, $h_0=h_0(r_0,K_0,K_{el})$, $K_1=K_1(r_0,K_0,K_{el})$, and $K_2=K_2(r_0,K_0,K_{el})$ such that: if $E$ is a $C^6$-regular set  that satisfies the \textsc{UBC} with radius $r_0$, $E \in \mathfrak{H}^5_{K_0,\sigma_1}(E_0)$, $ \| \nabla_{\pa E} \Delta_{\pa E} H_E \|_{L^2(\pa E)} \leq K_0$, and $ \| \Delta^2_{\pa E} H_E \|_{L^2(\pa E)} \leq h^{-1/4} K_0$, then for every $\eta \leq \eta_0$ and $h \leq h_0$, the minimizer $F$ of~\eqref{ProblemaINc} satisfies $ \pa F \Subset \mathcal{I}_{\eta}(\pa E)$ and $\pa F$ is the normal graph over $\pa E$ with height function $\psi: \pa E \rightarrow \R$ satisfying 
    \begin{equation}\label{Quattrostracci}
    \begin{split}
      &\| \psi \|_{H^1(\pa E)} \leq K_1 h, \,\| \psi \|_{H^2(\pa E)} \leq K_1 h^{\frac{3}{4}},\, \| \psi \|_{H^3(\pa E)} \leq K_1h^{\frac{1}{2}}, \\
      &\| \psi \|_{H^4(\pa E)} \leq K_1h^{\frac{1}{4}}, \,  \| \psi \|_{H^5(\pa E)} \leq K_1, \, \| \psi \|_{C^{3,\alpha}(\pa E)} \leq \tilde{K}_{1,\alpha} \\
      & \| \psi \|_{C^{0,\alpha}(\pa E)} \leq \tilde{K}_{1,\alpha} h^{\frac{3}{4}} \,, \| \psi \|_{C^{1,\alpha}(\pa E)} \leq \tilde{K}_{1,\alpha} h^{\frac{1}{2}},\, \| \psi \|_{C^{2,\alpha}(\pa E)} \leq \tilde{K}_{1,\alpha} h^{\frac{1}{4}}  \quad \text{ for } \alpha \in (0,1),
      \end{split}
    \end{equation}
    and it holds
       \begin{equation}\label{venezia}
        \| \nabla_{\pa F} \Delta_{\pa F} H_F \|_{L^2(\pa F)} \leq K_2, \,\, \| \Delta^2_{\pa F} H_F \|_{L^2(\pa F)} \leq K_2 h^{-\frac{1}{4}}.
    \end{equation}
    Furthermore, $F \in \mathfrak{H}^5_{K_1,2 \sigma_1}(E_0)$, $F$ satisfies the \textsc{UBC} with radius $r_0/2$, and
    the minimizer $u_F$ (respectively $u_E$) of problem~\eqref{minelast} coincides with $u_F^{K_{el}} \mrestr \Omega \setminus F$ (respectively $u_E^{K_{el}} \mrestr \Omega \setminus E$), the minimizer of problem~\eqref{minelvinc}. 
\end{theorem}
\section{Iteration} \label{iterazionesezione}
In this section, we prove a crucial iteration formula. Let $E \Subset \Omega$ and $K_{el}$ be as in the hypotheses of Theorem~\ref{Bologna}.  By Theorem~\ref{Bologna}, there exist constants $ \sigma_1, \eta_0,K_1, K_2, h_0$ depending only on $K, K_{el}, r_0$, such that if $ 0 < h\leq h_0$ and
\begin{equation}
 F \in   {\rm argmin} \left\{  \mathcal{F}_{h}(A,E) :  A \Delta E \subset \mathcal{I}_{\eta_0}(\pa E) \right\},
\end{equation}
then $F \in \mathfrak{H}^5_{K_1,2 \sigma_1}(E_0)$, $F$ is uniformly $C^{3,\alpha}$-regular and satisfies the \textsc{UBC} with radius $r_0/2$. Furthermore, we have $ \pa F \Subset \mathcal{I}_{\eta_0}(\pa E)$, $\pa F= \{   x+ \psi_{F,E}(x)\nu_E(x): x \in \pa E\}$, and
\begin{equation}\label{LESTIMEpsiFE}
    \begin{split}
      &\| \psi_{F,E} \|_{H^1(\pa E)} \leq K_1 h, \,\| \psi_{F,E} \|_{H^2(\pa E)} \leq K_1 h^{\frac{3}{4}},\, \| \psi_{F,E} \|_{H^3(\pa E)} \leq K_1h^{\frac{1}{2}}, \\
      & \| \psi_{F,E} \|_{H^4(\pa E)} \leq K_1h^{\frac{1}{4}},  \, \| \psi_{F,E} \|_{H^5(\pa E)} \leq K_1,\\
      & \| \nabla_{\pa F} \Delta_{\pa F} H_F \|_{L^2(\pa F)} \leq K_2, \, \| \Delta^2_{\pa F} H_F \|_{L^2(\pa F)} \leq K_2 h^{-\frac{1}{4}}. 
      \end{split}
    \end{equation}
Applying Theorem~\ref{Bologna} with $F$ in place of $E$, we obtain new constants $ \sigma_2,\eta_1,K_3,K_4,h_1$ depending only on $r_0, K_0, K_{el}$. If $ \eta \leq \eta_1$ and $ 0 < h \leq h_2:= \min \{ h_1, h_0\}$ the set 
\begin{equation}
   G \in   {\rm argmin} \left\{  \mathcal{F}_{h}(A,F) :  A \Delta E \subset \mathcal{I}_{\eta}(\pa F) \right\} 
\end{equation}
is uniformly $C^{3,\alpha}$-regular. Moreover, $\partial G \Subset \mathcal{I}_{\sigma_2}(\partial F)$, and $ \pa G = \{ x+ \psi_{G,E}(x)\nu_F(x): x \in \pa G \}$, with the following estimates:
\begin{equation}\label{LESTIMEpsiGF}
    \begin{split}
      &\| \psi_{G,F} \|_{H^1(\pa F)} \leq K_3 h, \,\| \psi_{G,F} \|_{H^2(\pa F)} \leq K_3 h^{\frac{3}{4}},\, \| \psi_{G,F} \|_{H^3(\pa F)} \leq K_3h^{\frac{1}{2}}, \\
      &  \| \psi_{G,F} \|_{H^4(\pa F)} \leq K_3h^{\frac{1}{4}}, \, \| \psi_{G,F} \|_{H^5(\pa F)} \leq K_3, \\
      &  \| \nabla_{\pa G} \Delta_{\pa G} H_G \|_{L^2(\pa G)} \leq K_4, \, \| \Delta^2_{\pa G} H_G \|_{L^2(\pa G)} \leq K_4 h^{-\frac{1}{4}}. 
      \end{split}
    \end{equation}
Moreover,  the minimizers $u_E, u_F, u_G$ of~\eqref{minelast} coincide with the minimizers $u_{E}^{K_{el}} \mrestr \Omega \setminus E, u_{F}^{K_{el}} \mrestr \Omega \setminus F, u_{G}^{K_{el}} \mrestr \Omega \setminus G$ of~\eqref{minelvinc}.
Throughout this
section, we will use the notation just introduced.

Throughout this section, to simplify the notation, we adopt the following convention:
$$ \nabla_{\pa A}^k= \overline{\nabla}_{\pa A}^k $$
for every $k \in \N$. Moreover, when $A$ is a set with smooth boundary, we identify the second fundamental form in the sense of differential geometry with the one introduced in Section~\ref{26052026regularsets}.

Before proving the main lemmas of this section, we state a preliminary result whose proof follows that of Lemmas~\ref{Lemmatecnico1elsticicita} and~\ref{lemmanichel}. Its validity relies on the uniform $C^{3,\alpha}$ estimates established above, which allow us to obtain an analogue of Lemma~\ref{lemmanichel} without reference to the manifold $\partial E_0$.

\begin{lemma}\label{lem:stimeQ}
    Let $E,F,G$ be as above. Then the following estimates hold:
     \begin{equation}\label{03052026label}
     \begin{split}
     &\| Q(E (u_{F}) ) (\cdot+ \psi_{F,E}(\cdot)\nu_E(\cdot))- Q(E (u_E)) \|_{H^{\frac{1}{2}}(\pa E)} \leq C \| \psi_{F,E} \|_{H^2(\pa E)}\\
     &\| Q(E (u_{G}) ) (\cdot+ \psi_{G,F}(\cdot)\nu_F(\cdot))- Q(E (u_F)) \|_{H^{\frac{1}{2}}(\pa F)} \leq C \| \psi_{G,F} \|_{H^2(\pa F)}
     \end{split}
 \end{equation}
 and
    \begin{equation}\label{28042026ossi}
    \begin{split}
     &\| Q(E (u_{F}) ) (\cdot+ \psi_{F,E}(\cdot)\nu_E(\cdot))- Q(E (u_E)) \|_{H^{\frac{3}{2}}(\pa E)} \leq C \| \psi_{F,E} \|_{H^3(\pa E)}\\
     &\| Q(E (u_{G}) ) (\cdot+ \psi_{G,F}(\cdot)\nu_F(\cdot))- Q(E (u_F)) \|_{H^{\frac{3}{2}}(\pa F)} \leq C \| \psi_{G,F} \|_{H^3(\pa F)}
     \end{split}
    \, .
 \end{equation}
\end{lemma}
\begin{proof}
 We only  prove the second inequality in~\eqref{28042026ossi}, since the proofs of the other cases are analogous.

Recall that $u_E$ and $u_F$ satisfy the following systems
 \begin{equation}
     \left\{
		\begin{aligned}
			& \div \mathbb{C}E(u_E)=0 & \text{ in } \Omega \setminus E  , \\
			& \mathbb{C}E(u_E)[\nu_E]=0 & \text{ on }  \pa E , \\
			& u_E= w_0 & \text{ on } \pa \Omega,
		\end{aligned}
		\right. \quad \left\{
		\begin{aligned}
			& \div \mathbb{C}E(u_F)=0 & \text{ in } \Omega \setminus F  , \\
			& \mathbb{C}E(u_F)[\nu_F]=0 & \text{ on }  \pa F , \\
			& u_F= w_0 & \text{ on } \pa \Omega.
		\end{aligned}
		\right.
 \end{equation}
Fix a unit vector $e \in \R^3$, by differentiating the above equations in the direction $e$, we obtain
 \begin{equation}\label{04052026form1}
     \left\{
		\begin{aligned}
			& \div \mathbb{C}E(\pa_e u_E)=0 & \text{ in } \Omega \setminus E  , \\
			& \mathbb{C}E(\pa_e u_E)[\nu_E]=g_E & \text{ on }  \pa E , \\
			& \pa_e u_E= \pa_e w_0 & \text{ on } \pa \Omega,
		\end{aligned}
		\right.  \quad \left\{
		\begin{aligned}
			& \div \mathbb{C}E(\pa_e u_F)=0 & \text{ in } \Omega \setminus F  , \\
			& \mathbb{C}E(\pa_e u_F)[\nu_F]=g_F & \text{ on }  \pa F , \\
			& \pa_e u_F= \pa_e w_0 & \text{ on } \pa \Omega,
		\end{aligned}
		\right.
 \end{equation}
 where 
 \begin{equation}
     g_E= -\mathbb{C}E(u_E)B_E e, \quad g_F=- \mathbb{C}E(u_F)B_Fe.
 \end{equation}
 Let $\Phi_{F,E} : \Omega \setminus E \rightarrow \Omega \setminus F$ be defined by $\Phi_{F,E}(x):= x+ \psi_{F,E}\circ \pi_{\partial E}(x) \nu_{E}\circ\pi_{\partial E}(x)$. We have 
 \begin{align*}
    \vert g_F \circ \Phi_{F,E} -g_E \vert &\leq C \vert \nabla u_F \circ \Phi_{F,E}- \nabla u_E \vert \vert B_F  \vert \circ \Phi_{F,E}+ \vert \nabla u_E \vert \vert B_F  \circ \Phi_{F,E}- B_E \vert \\
    &\leq C \vert \nabla u_F \circ \Phi_{F,E}- \nabla u_E \vert + C ( \vert \psi_{F,E} \vert + \vert \nabla_{\pa E} \psi_{F,E} \vert + \vert \nabla_{\pa E}^2 \psi_{F,E} \vert ),
 \end{align*}
 where we used an expansion of $B_F$ similar to the expansion of the anisotropic curvature in Lemma~\ref{10Expcurvvarphi}.
Hence, we obtain
 \begin{equation}\label{04052026adessonnn}
 \begin{split}
     \| g_F \circ \Phi_{F,E} - g_E \|_{L^2(\pa E)} \leq & C \| \psi_{F,E} \|_{H^2(\pa E)}+ C \| \nabla u_F \circ \Phi_{F,E}- \nabla u_E \|_{L^2(\pa E)} \\
     \leq & C \| \psi_{F,E} \|_{H^2(\pa E)},
     \end{split}
 \end{equation}
 where the second inequality can be proved similarly as~\eqref{28042026meriggiare} (see also~\cite[Remark 4.2]{FJM2020}). 
 
 Let $x_0 \in \pa E$, and $\Phi: U \rightarrow B_{2R} $ be a diffeomorphism that flattens the boundary of $E$, where $U \Subset \Omega$ is a neighborhood of $x_0$ and $ \Phi(U \setminus E)= B^+_{2R}$, with $B_{2R}^+= B_{2R} \cap \{ x_3 >0\}$. We set 
 \begin{equation}\label{04052026form2}
     v_E := \pa_{e} u_E \circ \Phi^{-1}, \,\, v_F:= \pa_{e} u_F \circ \Phi_{F,E}^{-1} \circ \Phi^{-1},\,\, \bar{\psi}_{F,E}:= \psi_{F,E}\circ \pi_{\pa E} \circ \Phi^{-1}.
 \end{equation}
Therefore, by~\eqref{04052026form1},~\eqref{04052026form2}, we infer
\begin{equation}\label{04052026form3}
 \left\{
		\begin{aligned}
			& \div \big( \mathbb{A}(\cdot, \bar{\psi}_{F,E}, \nabla \bar{\psi}_{F,E})\nabla v_F\big)=0 & \text{ in } B^+_{2R}  , \\
			& E(v_F)[e_3]=G_F & \text{ on }  \{x_3=0\} \cap B_{2R} ,
		\end{aligned}
		\right.
\end{equation}
where the first equation is understood in the sense of distributions, i.e. by testing against functions $\xi \in C^\infty(B^+_{2R},\mathbb{R}^3)$ such that $\xi=0$ on $\partial B^+_{2R} \cap \{x_3>0\}$. Similarly, 
\begin{equation}\label{04052026form4}
 \left\{
		\begin{aligned}
			& \div \big( \mathbb{A}(\cdot,0,0)\nabla v_E \big)=0 & \text{ in } B^+_{2R}  , \\
			& E(v_E)[e_3]=G_E & \text{ on }  \{x_3=0\} \cap B_{2R} ,
		\end{aligned}
		\right.
\end{equation}
where again the first equation is interpreted in the sense of distributions with the same class of test functions. Moreover, the tensors $\mathbb{A}(\cdot,\bar{\psi}_{F,E},\nabla \bar{\psi}_{F,E})$ and $\mathbb{A}(\cdot,0,0)$ are uniformly elliptic. Finally, by ~\eqref{04052026adessonnn}, we conclude that $$ \| G_F -G_E \|_{L^2(\{ x_3=0\} \cap B_{2R} )} \leq C \|\psi_{F,E} \|_{H^2(\pa E)}.$$

Let $\xi \in C^\infty(B^+_{2R}, \R^3)$ be such that $\xi=0$ on $ \pa B^+_{2R} \cap \{x_3>0\}$. Differentiating equations~\eqref{04052026form3} and~\eqref{04052026form4}, we obtain
\begin{align*}
  \int_{B^+_{2R}} \mathbb{A}(\cdot,\bar{\psi}_{F,E}, \nabla \bar{\psi}_{F,E}) \nabla \pa v_F : \nabla \xi\, dx+ \int_{B^+_{2R}} \pa \mathbb{A}(\cdot, \bar{\psi}_{F,E}, \nabla \bar{\psi}_{F,E}) \nabla v_F : \nabla \xi\, dx&  \\
  + \int_{ \{ x_3=0  \} \cap B_{2R}^+   } \partial \xi  G_F \, d \mathcal{H}^2 =0&
\end{align*}
and 
\begin{align*}
  \int_{B^+_{2R}} \mathbb{A}(\cdot,0, 0) \nabla \pa v_E : \nabla \xi\, dx+ \int_{B^+_{2R}} \pa \mathbb{A}(\cdot, 0,0) \nabla v_E : \nabla \xi\, dx&  \\
  + \int_{ \{ x_3=0  \} \cap B_{2R}^+   } \pa \xi G_E \, d \mathcal{H}^2 =0.&
\end{align*}
Subtracting the two equations  above, we obtain
\begin{align*}
        \int_{B^+_{2R}} \mathbb{A}&(\cdot, \bar{\psi}_{F,E}, \nabla \bar{\psi}_{F,E} ) \big( \nabla \pa v_F- \nabla \pa v_E \big): \nabla \xi dx \\
        &=- \int_{B^+_{2R}} \big( \mathbb{A}(\cdot, \bar{ \psi}_{F,E}, \nabla \bar{\psi}_{F,E}) - \mathbb{A}(\cdot, 0,0)\big) \nabla \pa v_E : \nabla \xi\, dx\\
        &\quad- \int_{B^+_{2R}} \pa \mathbb{A}(\cdot, \bar{\psi}_{F,E}, \nabla \bar{\psi}_{F,E})(\nabla v_F- \nabla v_E): \nabla \xi \, dx\\
         &\quad- \int_{B^+_{2R}}  \big( \pa \mathbb{A}(\cdot, \bar{\psi}_{F,E}, \nabla \bar{\psi}_{F,E})- \pa \mathbb{A}(\cdot,0,0)  \big) \nabla v_E : \nabla \xi\, dx\\
         & \quad- \int_{\{ x_3=0 \} \cap B_{2R}  } \pa_\tau \xi \cdot (G_F-G_E) \, d\mathcal{H}^2.
\end{align*}
We first differentiate in the tangential directions $e_1,\,e_2$, and then in the normal direction $e_3$.
Testing the above equation with $ \xi:= \eta^2\pa_\tau  \big(  v_F- v_E \big)$, where $\eta \in C^\infty_{c}(B^+_{2R})$ and $\eta=1$ in $B_R$, and arguing as in Lemma~\ref{Lemmatecnico1elsticicita}, we obtain the desired estimate. The main difference from the proof of Lemma~\ref{Lemmatecnico1elsticicita} lies in the term
$$ \int_{\{ x_3=0 \} \cap B^+_{2R}   } \pa_\tau \xi \cdot (G_F-G_E),$$
so we briefly indicate how to bound it for the reader’s convenience.
Let $\varepsilon>0$ to be chosen later, we have 
\begin{align*}
        \int_{\{x_3=0\} \cap B_{2R}} \pa_{\tau} \xi \cdot  (G_F -G_E)  \leq& C(\varepsilon)\int_{ \{x_3=0\} \cap B_{2R} } \vert \xi \vert^2 + \varepsilon \int_{\{x_3=0\} \cap B_{2R} } \vert \pa _{\tau}(G_F -G_E) \vert^2\\
        \leq& \varepsilon  \| v_F -v_E \|^{2}_{H^{3/2}(\pa E)}+ C(\varepsilon) \| v_F-v_E \|^2_{L^2(\pa E)}\\
        &  + \varepsilon  \int_{\{x_3=0\} \cap B_{2R} } \vert \pa _{\tau}(G_F -G_E) \vert^2\\
        \leq& \varepsilon  \| v_F -v_E \|^{2}_{H^{3/2}(\pa E)}+ C(\varepsilon) \| \psi \|^2_{H^2(\pa E)} \\
        &+ \varepsilon C \int_{ \pa E } \vert \pa _{\tau}(g_F-g_E) \vert^2 \\
        \leq& \varepsilon \| v_F -v_E \|^{2}_{H^{3/2}(\pa E)}+ C(\varepsilon) \| \psi \|^2_{H^2(\pa E)} + C \| \psi_{F,E} \|^2_{H^3(\pa E)}  \\
        &  + \varepsilon C \| (\nabla u_F )(\cdot+ \psi_{F,E}(\cdot) \nu_E)- \nabla u_E (\cdot) \|_{H^1(\pa E)}^2,
\end{align*}
where  the second inequality follows by the definition of $\xi$ and by the interpolation, the third inequality can be proved analogously as~\cite[Remark 4.2]{FJM2020}, and the fourth inequality follows similarly as~\eqref{04052026adessonnn}.
\end{proof}

We now state a lemma that is essential for proving the main result of this section. Recall the notation introduced at the beginning of the section.

\begin{lemma}\label{propdiiterazione11}
    Let $\eta \leq \eta_1$, and let $\xi_{G,F}:= \psi_{G,F}+ \psi_{G,F}^2 \frac{H_F}{2}+ \psi^3_{G,F} \frac{K_F}{3}$. There exists a constant $h_3>0$, depending only on $K_0,r_0,K_{el}$, such that the following inequality holds:
    \begin{multline}\label{05052026dueminitui}
        \int_{\pa F}\Big(\frac{1}{h}-C\Big)\vert \nabla_{\pa F} \xi_{G,F} \vert^2 + \frac{3}{4} \mathcal{A}(\nu_F) \nabla_{\pa F} \Delta_{\pa F} \psi_{G,F} \cdot \nabla_{\pa F} \Delta_{\pa F} \psi_{G,F} \\
        \leq - \int_{\pa F} \xi_{G,F} \big(\Delta^2_{\pa F} H^\varphi_F - \Delta^2_{\pa F} Q(E(u_F)) \big)
    \end{multline}
    for all $ 0 < h \leq h_3$, where $C=C(K_0,K_{el},r_0)$.
\end{lemma}
\begin{proof}
In what follows, we denote by $C$ a generic constant depending on $K_{0}, K_{\mathrm{el}}, r_0$. By the Euler–Lagrange equation~\eqref{NewPeerGynt}, it holds
\begin{align*}
        \frac{1}{h} \Delta_{\pa F} \xi_{G,F}=& -\Delta_{\pa F} \div_{\pa F} \big( \mathcal{A}(\nu_F) \nabla_{\pa F} \Delta_{\pa F } \psi_{G,F}   \big) + \Delta_{\pa F}^2 H_F^\varphi \\
        & - \Delta_{\pa F} \div_{\pa F} \big( \big( \Delta_{\pa F} \mathcal{A}(\nu_F)  \big) \nabla_{\pa F} \psi_{G,F} \big) \\
        & - \Delta_{\pa F} \widehat{\mathcal{R}}(\nabla_{\pa F} B_F \star \nabla_{\pa F} \psi_{G,F}, B_F \star \nabla^3_{\pa F} \psi_{G,F})\\
        & - \Delta^2_{\pa F} \big( Q(E(u_G))(\cdot+ \psi_{G,F}(\cdot)\nu_F(\cdot))   \big) + \Delta^2_{\pa F} \widetilde{\mathcal{R}}_0.
\end{align*}
Multiplying by $-\xi_{G,F}$ and integrating by parts, we infer
\begin{align}
    \frac{1}{h} \int_{\pa F} \vert \nabla_{\pa F} \xi_{G,F} \vert ^2+& \int_{\pa F} \mathcal{A}(\nu_F)  \nabla_{\pa F} \Delta_{\pa F} \psi_{G,F} \cdot \nabla_{\pa F} \Delta_{\pa F}  \psi_{G,F}=- \int_{\pa F} \xi_{G,F} \Delta_{\pa F}^2 H^\varphi_F\\
    &\quad-\int_{\pa F} \nabla_{\pa F} \big( Q(E(u_G))(\cdot+ \psi_{G,F}(\cdot)\nu_F(\cdot))   \big) \cdot \nabla_{\pa F}\Delta_{\pa F} \xi_{G,F} \label{08042026pom1}\\
    & \quad - \int_{\pa F} \big( \Delta_{\pa F}\mathcal{A}(\nu_F)  \big) \nabla_{\pa F} \psi_{G,F} \cdot \nabla_{\pa F} \Delta_{\pa F} \xi_{G,F} \label{08042026pom2}\\
    &\quad- \int_{\pa F} \mathcal{A}(\nu_F) \nabla_{\pa F} \Delta_{\pa F} \psi_{G,F} \cdot \nabla_{\pa F} \Delta_{\pa F} \big( \frac{\psi_{G,F}^2}{2} H_F+ \frac{\psi_{G,F}^3}{3}K_F  \big)  \label{08042026pom3}\\
    &\quad- \int_{\pa F}   \widehat{\mathcal{R}} \cdot    \Delta_{\pa F} \xi_{G,F}\label{08042026pom5}\\
    & \quad+\int_{\pa F} \nabla_{\pa F} \widetilde{R}_0 \cdot \nabla_{\pa F} \Delta_{\pa F} \xi_{G,F}. \label{08042026pom4}
    \end{align}
By Lemma~\ref{lem:hilbert-norm}, also recall~\eqref{10022026pprinazof2} and~\eqref{LESTIMEpsiGF}, we have
\begin{equation}
        \begin{split}
        \| \nabla_{\pa F} \Delta_{\pa F} \xi_{G,F} \|_{L^2(\pa F)} & \leq \| \nabla^3_{\pa F} \big( \psi_{G,F}(1+ \psi_{G,F}\frac{H_F}{2}+ \psi_{G,F}^2\frac{K_F}{3})   \big)\|_{L^2(\pa F)}\\
        & \leq C \bigg( \| \psi_{G,F}\|_{L^\infty(\pa F)} \| 1+ \psi_{G,F} \frac{H_F}{2}+ \psi^2_{G,F}\frac{K_F}{3}\|_{H^3(\pa F)}\\
        & \quad\qquad+\| \psi \|_{H^3(\pa F)} \| 1+ \psi_{G,F} \frac{H_F}{2}+ \psi^2_{G,F}\frac{K_F}{3}\|_{L^\infty(\pa F)} \bigg)\\
        & \leq C  \| \psi_{G,F} \|_{H^3(\pa F)} \\
        & \leq C \big( \| \nabla_{\pa F} \Delta_{\pa F} \psi_{G,F} \|_{L^2(\pa F)}+  \| \psi_{G,F} \|_{L^2(\pa F)} \big) \\
        & \leq C \big( \| \nabla_{\pa F} \Delta_{\pa F} \psi_{G,F} \|_{L^2(\pa F)}+ \| \nabla_{\pa F}\xi_{G,F} \|_{L^2(\pa F)} \big).\label{03052026caffe1}
        \end{split}
\end{equation}
We now proceed to estimate~\eqref{08042026pom1},~\eqref{08042026pom2},~\eqref{08042026pom3},~\eqref{08042026pom5},~\eqref{08042026pom4}. We fix $\varepsilon$, which will be chosen later.\\
\textit{Estimate of~\eqref{08042026pom1}}.\\
By Lemma \ref{lem:stimeQ},~\eqref{LESTIMEpsiGF} and~\eqref{03052026caffe1}, we have
\begin{align}
       - \int_{\pa F}& \nabla_{\pa F} \big( Q(E(u_G))(\cdot+ \psi_{G,F}(\cdot)\nu_F(\cdot))  \big) \cdot \nabla_{\pa F} \Delta_{\pa F} \xi_{G,F} \notag\\
       &= -\int_{\pa F}  \nabla_{\pa F} Q(E(u_F))  \cdot \nabla_{\pa F} \Delta_{\pa F} \xi_{G,F} \notag\\
        &\quad -\int_{\pa F} \nabla_{\pa F} \big( Q(E(u_G))(\cdot+ \psi_{G,F}(\cdot)\nu_F(\cdot))- Q(E(u_F))  \big) \cdot \nabla_{\pa F} \Delta_{\pa F} \xi_{G,F} \notag\\
       &\leq   - \int_{\pa F}  \nabla_{\pa F} Q(E(u_F))  \cdot \nabla_{\pa F} \Delta_{\pa F} \xi_{G,F} + \varepsilon \| \nabla_{\pa F} \Delta_{\pa F} \xi_{G,F} \|^2_{L^2(\pa F)} \notag\\
       &\quad + C(\varepsilon) \| Q(E(u_G))(\cdot+ \psi_{G,F}(\cdot)\nu_F(\cdot))- Q(E(u_F)) \|^2_{H^1(\pa F)} \notag\\
       &\leq  - \int_{\pa F}  \nabla_{\pa F} Q(E(u_F))  \cdot \nabla_{\pa F} \Delta_{\pa F} \xi_{G,F} + \varepsilon C \| \nabla_{\pa F} \Delta_{\pa F} \psi_{G,F} \|^2_{L^2(\pa F)}\notag\\
       &\quad + \varepsilon C \| \nabla_{\pa F} \xi_{G,F }\|^2_{L^2(\pa F)} + C(\varepsilon) \| \psi_{G,F} \|_{H^2(\pa F )} \| \psi_{G,F} \|_{H^3(\pa F)}\notag\\
      &\leq   - \int_{\pa F}  \nabla_{\pa F} Q(E(u_F))  \cdot \nabla_{\pa F} \Delta_{\pa F} \xi_{G,F} \label{05052026cara-2} \\
      &\quad  + \varepsilon C \| \nabla_{\pa F} \Delta_{\pa F} \psi_{G,F} \|^2_{L^2(\pa F)}+ C \| \nabla_{\pa F} \xi_{G,F }\|^2_{L^2(\pa F)}.\notag
\end{align}
\textit{Estimate of~\eqref{08042026pom2}.}\\
By integrating by parts, using Young’s inequality and~\eqref{10022026pprinazof2}, we obtain
\begin{align}
  - \int_{\pa F}& \big(  \Delta_{\pa F}\mathcal{A}(\nu_F)  \big) \nabla_{\pa F} \psi_{G,F} \cdot \nabla_{\pa F} \Delta_{\pa F} \xi_{G,F} \notag \\
  &\leq C(\varepsilon)\| B_F \|_{C^1(\pa F)}^2 \| \nabla_{\pa F} \psi_{G,F} \|^2_{L^2(\pa F)}+  \varepsilon \| \nabla_{\pa F}\Delta_{\pa F} \psi_{G,F} \|^2_{L^2(\pa F)}+\varepsilon C \| \nabla_{\pa F} \xi_{G,F} \|^2_{L^2(\pa F)} \notag \\
  &\leq  \varepsilon \| \nabla_{\pa F} \Delta_{\pa F} \psi_{G,F} \|_{L^2(\pa F)}^2+ C \| \nabla_{\pa F} \xi_{G,F} \|^2_{L^2(\pa F)}.\label{05052026cara-1}
\end{align}
\textit{Estimate of~\eqref{08042026pom3}.}\\
We start by establishing the following estimate, whose proof is similar to that of~\eqref{03052026caffe1}:
\begin{equation}
\begin{split}
    \| \nabla_{\pa F} \Delta_{\pa F} \big(  \psi_{G,F}^2 ( \frac{H_F}{2}&+ \psi_{G,F} \frac{K_F}{3}) \big) \|_{L^2(\pa F)} \\
    \leq& C \| \psi_{G,F}^2 \|_{H^3(\pa F)}+ \| \psi_{G,F} \|_{L^\infty(\pa F)} \| \frac{H_F}{2}+ \psi_{G,F} \frac{K_F}{3} \|_{H^3(\pa F)} \\
    \leq & C \| \psi_{G,F} \|_{L^\infty(\pa F)} \| \psi_{G,F} \|_{H^3(\pa F)}+ C \| \psi_{G,F} \|_{H^2(\pa F)}\\
    \leq & C \big( \varepsilon + h^{3/4} \big)\| \nabla_{\pa F} \Delta_{\pa F} \psi_{G,F} \|_{L^2(\pa F)}+ C \| \nabla_{\pa F} \xi_{G,F} \|_{L^2(\pa F)} 
\end{split}
\end{equation}
where in the third inequality we used~\eqref{LESTIMEpsiGF} together with the Sobolev embedding theorem. Hence, applying Young’s inequality and the estimate above for $h^{3/2} \leq \varepsilon$, we obtain
\begin{align}
      - \int_{\pa F}& \mathcal{A}(\nu_F) \nabla_{\pa F} \Delta_{\pa F} \psi_{G,F} \cdot \nabla_{\pa F} \Delta_{\pa F} \big( \frac{\psi_{G,F}^2}{2} H_F+ \frac{\psi_{G,F}^3}{3}K_F  \big) \notag\\
      &\leq \varepsilon \| \nabla_{\pa F} \Delta_{\pa F} \psi_{G,F} \|^2_{L^2(\pa F)} + \frac{C(\varepsilon+ h^{3/4})^2}{\varepsilon} \| \nabla_{\pa F} \Delta_{\pa F} \psi_{G,F} \|^2_{L^2(\pa F)}\notag\\
      & \quad+ C(\varepsilon) \| \nabla_{\pa F} \xi_{G,F}\|^2_{L^2(\pa F)}\notag\\
      & \leq C\varepsilon \| \nabla_{\pa F} \Delta_{\pa F} \psi_{G,F} \|^2_{L^2(\pa F)}+ C \| \nabla_{\pa F} \xi_{G,F} \|^2_{L^2(\pa F)}.\label{05052026cara0}
\end{align}
\textit{Estimate of~\eqref{08042026pom5}.}\\
Recall that $\widehat{\mathcal{R}}=\widehat{\mathcal{R}}(\overline{\nabla} B_F \star \overline{\nabla} \psi_{G,F}, B_F \star \overline{\nabla}^3 \psi_{G,F})$ is a smooth function. By Young's inequality,  interpolation  and~\eqref{03052026caffe1}, we infer
\begin{align}\label{05052026cara111}
    -\int_{\pa F}   \widehat{\mathcal{R}} \cdot    \Delta_{\pa F} \xi_{G,F} &\le \varepsilon\|\widehat{\mathcal{R}}\|^2_{L^2(\Omega)} + C(\varepsilon)\|  \Delta_{\pa F} \xi_{G,F} \|^2_{L^2(\pa F)} \\
    &\le \varepsilon C \| \nabla_{\pa F} \Delta_{\pa F} \psi_{G,F} \|^2_{L^2(\pa F)}+ C \| \nabla_{\pa F} \xi_{G,F }\|^2_{L^2(\pa F)}.
\end{align}
\textit{Estimate of~\eqref{08042026pom4}.}\\
Using Young’s inequality together with~\eqref{03052026caffe1}, we obtain
\begin{align}\label{05052026milano1}
 \int_{\pa F} \nabla_{\pa F} \widetilde{R}_0 \cdot \nabla_{\pa F} \Delta_{\pa F} \xi_{G,F} \leq& \varepsilon \| \nabla_{\pa F} \Delta_{\pa F} \psi_{G,F} \|^2_{L^2(\pa F)}+ C \| \nabla_{\pa F} \xi_{G,F } \|^2_{L^2(\pa F)}\\
 & + C \| \nabla_{\pa F} \widetilde{R}_0 \|_{L^2(\pa F)}^2.\notag
\end{align}
We now derive an estimate for $\| \nabla_{\pa F} \widetilde{R}_0 \|_{L^2(\pa F)}$. To simplify notation, we write $\psi$ instead of $\psi_{G,F}$. Recalling Corollary~\ref{zonzo}, we need to estimate
\begin{equation}\label{dadaisomo050502026}
  \langle \mathcal{A}_1(\psi B_F, \nabla_{\pa F} \psi, \nu_F), \nabla_{\pa F}^2 \psi \rangle, \,\, \langle \mathcal{A}_2(\psi B_F,  \nabla_{\pa F} \psi, \nu_F), \nabla_{\pa F}(\psi B_F) \rangle, \,\, a_0(\psi, \nabla_{\pa F} \psi, B_F, \nu_F). 
\end{equation}
We begin with the first term in~\eqref{dadaisomo050502026}. We have 
\begin{align*}
        &\|  \nabla_{\pa F} \langle \mathcal{A}_1(\psi B_F, \nabla_{\pa F} \psi, \nu_F), \nabla_{\pa F}^2 \psi \rangle \|_{L^2(\pa F)}   \\
        &\leq  C \big(  \| \mathcal{A}_1(\psi B_F, \nabla_{\pa F} \psi, \nu_F) \|_{L^\infty(\pa F)} \| \psi \|_{H^3(\pa F)}+ \|  \mathcal{A}_1(\psi B_F, \nabla_{\pa F} \psi, \nu_F) \|_{H^1(\pa F)} \|  \psi \|_{C^2(\pa F)} \big) \\
         &\leq  Ch^{1/2} \big(  \| \nabla_{\pa F} \Delta_{\pa F} \psi \|_{L^2(\pa F)}+  \| \psi \|_{H^2(\pa F)} \big)+ Ch^{1/4}\| \psi \|_{H^2(\pa F)} \\
     &\leq \varepsilon C \| \nabla_{\pa F} \Delta_{\pa F} \psi \|_{L^2(\pa F)}+ C \| \nabla_{\pa F} \xi_{G,F} \|_{L^2(\pa F)},
\end{align*}
where in the first inequality we applied Lemma~\ref{lem:Leibniz}; in the second inequality we used Lemma~\ref{lem:hilbert-norm} together with the Sobolev embedding Theorem, the regularity of $ \mathcal{A}_1(\psi B_F, \nabla_{\pa F} \psi, \nu_F)   $, and~\eqref{LESTIMEpsiGF}. Similarly, we deduce for the second term in~\eqref{dadaisomo050502026}
\begin{align*}
    \| \nabla_{\pa F} \langle \mathcal{A}_2(\psi B_F,  \nabla_{\pa F} & \psi, \nu_F), \nabla_{\pa F}(\psi B_F) \rangle \|_{L^2(\pa F)}\\
    \leq & \varepsilon C \| \nabla_{\pa F} \Delta_{\pa F} \psi \|_{L^2(\pa F)}+ C \| \nabla_{\pa F} \xi_{G,F} \|_{L^2(\pa F)}.
\end{align*}
 Analogously, for the last term in~\eqref{dadaisomo050502026}, we obtain
\begin{align*}
    \|  a_0(\psi, \nabla_{\pa F} \psi, B_F, \nu_F) \|_{H^1(\pa F)} \leq \varepsilon C \| \nabla_{\pa F} \Delta_{\pa F} \|_{L^2(\pa F)}+ C \| \nabla_{\pa F} \xi_{G,F} \|_{L^2(\pa F)}.
\end{align*}
 Combining~\eqref{05052026milano1} with the estimates above, we conclude that
\begin{equation}\label{05052026cara1}
    \int_{\pa F} \nabla_{\pa F} \widetilde{R}_0 \cdot \nabla_{\pa F} \Delta_{\pa F} \xi_{G,F} \leq \varepsilon \| \nabla_{\pa F} \Delta_{\pa F} \psi_{G,F} \|^2_{L^2(\pa F)}+ C \| \nabla_{\pa F} \xi_{G,F } \|^2_{L^2(\pa F)}.
\end{equation}

Finally, we obtain~\eqref{05052026dueminitui} by combining~\eqref{05052026cara-2},~\eqref{05052026cara-1},~\eqref{05052026cara0},~\eqref{05052026cara111} and~\eqref{05052026cara1}, taking $\varepsilon$ and $h$ sufficiently small.
\end{proof}
We are now in a position to establish the main result of this section.
\begin{lemma}[Iteration]\label{propdiiterazione}
    Let $E,F,G$  be as in the beginning of the section. We set
    \begin{equation}
        \xi_{G,F}= \psi_{G,F}+ \psi_{G,F}^2 \frac{H_F}{2}+ \psi^3_{G,F} \frac{K_F}{3}, \quad  \xi_{F,E}= \psi_{F,E}+ \psi_{F,E}^2 \frac{H_E}{2}+ \psi^3_{F,E} \frac{K_E}{3}.
    \end{equation}
   Then there exist constants $M$ and $h_4$, depending only on $K_0$, $K_{\mathrm{el}}$, and $r_0$, such that
    \begin{multline}\label{ARSENIOLEIMI}
       \int_{\pa F}\vert \nabla_{\pa F} \xi_{G,F} \vert^2 + \frac{h}{2} \mathcal{A}(\nu_F) \nabla_{\pa F} \Delta_{\pa F} \psi_{G,F} \cdot \nabla_{\pa F} \Delta_{\pa F} \psi_{G,F} \\
      \leq (1+Mh)   \int_{\pa E}\vert \nabla_{\pa E} \xi_{F,E} \vert^2 + \frac{h}{8} \mathcal{A}(\nu_E) \nabla_{\pa E} \Delta_{\pa E} \psi_{F,E} \cdot \nabla_{\pa E} \Delta_{\pa E} \psi_{F,E}.
    \end{multline}
\end{lemma}
\begin{proof}
    In what follows, we denote by $C$ a generic constant depending on $K_0,\, K_{el}$ and $r_0$. To prove the thesis, we need to estimate the term on the right-hand side of inequality~\eqref{05052026dueminitui}, namely
\begin{equation}
  - \int_{\pa F} \xi_{G,F} \big(\Delta^2_{\pa F} H^\varphi_F - \Delta^2_{\pa F} Q(E(u_F)) \big).  
\end{equation}
Integrating by parts twice and applying a change of variables in the above expression, we obtain
\begin{equation}\label{06052026tuttiquanti}
\begin{split}
    &- \int_{\pa F} \Delta_{\pa F} \xi_{G,F} \Delta_{\pa F} \big(  H^\varphi_F- Q(E(u_F)) \big)\\
    &=-\int_{\pa E} \big( \Delta_{\pa F} ( H_F^\varphi- Q(E(u_F)))  \big) \circ \Psi_{F,E} \big(\Delta_{\pa F} \xi_{G,F} \big) \circ \Psi_{F,E} J_{\pa E } \Psi_{F,E}, 
    \end{split}
\end{equation}
where $ \Psi_{F,E}: \pa E \rightarrow \pa F$ is defined by $ \Psi_{F,E}(x):=x+ \psi_{F,E}(x)\nu_E(x)$, and $ J_{\pa E} \psi_{F,E}$ denotes the tangential Jacobian (see formula~\eqref{evvrailtempo1}). We now need to relate the Laplace--Beltrami operator on $\pa F$ to the corresponding operator on $\pa E$.\\
\textit{Claim: } If $f \in C^2(\pa F)$. Then, on $\pa E$
\begin{equation}\label{06052026laplace}
\begin{split}
   (\Delta_{\pa F} f ) \circ \Psi_{F,E}=& \Delta_{\pa E} (f \circ \Psi_{F,E})+ a(\nabla^2_{\pa F} f \circ \Psi_{F,E}, \psi_{F,E}, \nabla_{\pa E} \psi_{F,E}, \nu_E, B_E)\\
   &+ b(\nabla_{\pa F}f \circ \Psi_{F,E}, \psi_{F,E}, \nabla_{\pa E}\psi_{F,E}, \nabla_{\pa E}^2 \psi_{F,E}, \nu_E, B_E, \nabla_{\pa E} B_E).
   \end{split}
\end{equation}
Moreover, the following pointwise estimates hold:
\begin{multline}\label{06052026resto1}
    \vert a(\nabla^2_{\pa F} f \circ \Psi_{F,E}, \psi_{F,E}, \nabla_{\pa E} \psi_{F,E}, \nu_E, B_E) \vert \\ \leq C \vert \nabla_{\pa F}^2 f \vert \circ \Psi_{F,E} \big( \vert \psi_{F,E} \vert +\vert \nabla_{\pa E}\psi_{F,E} \vert + \vert \nabla_{\pa E}^2 \psi_{F,E} \vert \big)  
\end{multline}
and
\begin{multline}\label{06052026resto2}
  \vert b(\nabla_{\pa F}f \circ \Psi_{F,E}, \psi_{F,E}, \nabla_{\pa E}\psi_{F,E}, \nabla_{\pa E}^2 \psi_{F,E}, \nu_E, B_E, \nabla_{\pa E} B_E) \vert \\
  \leq C \vert \nabla_{\pa F} f \vert \circ \Psi_{F,E} \big( \vert \psi_{F,E} \vert + \vert \nabla_{\pa E} \psi_{F,E} \vert  \big),
\end{multline}
where the constant $C$ depends on $| B_E |_{C^1(\partial E)}$.

By the chain rule, we obtain
\begin{equation*}
    \nabla_{\pa E} (f \circ \Psi_{F,E})= \nabla_{\pa F} f \circ \Psi_{F,E} \nabla_{\pa E} \Psi_{F,E}.
\end{equation*}
Using formulas~~\eqref{02122025form1} and~\eqref{06052026form2}, we compute on $\partial E$
\begin{align*}
    \nabla_{\pa E} \Psi_{F,E}=& I-\nu_E \otimes \nu_E +A_0(\psi_{F,E}, \nabla_{\pa E} \psi_{F,E}, \nu_E, B_E)\notag\\
    =& I- \nu_F \circ \Psi_{F,E} \otimes \nu_F \circ \Psi_{F,E}\notag\\
    & +   \nu_F \circ \Psi_{F,E} \otimes \nu_F \circ \Psi_{F,E}- \nu_E \otimes \nu_E  + A_0(\psi_{F,E}, \nabla_{\pa E} \psi_{F,E}, \nu_E, B_E)  \notag\\
    =& I- \nu_F \circ \Psi_{F,E} \otimes \nu_F \circ \Psi_{F,E} + A_1(\psi_{F,E}, \nabla_{\pa E} \psi_{F,E}, \nu_E, B_E),
\end{align*}
where $A_0$ and $A_1$ are smooth matrix-valued functions satisfying $A_0(0,0,\cdot,\cdot)=A_1(0,0,\cdot,\cdot)=0$. Hence, we obtain
\begin{equation}\label{eq:ugnablacomp}
    \nabla_{\pa E} (f \circ \Psi_{F,E})= \nabla_{\pa F} f \circ \Psi_{F,E} \big( I+ A_1(\psi_{F,E}, \nabla_{\pa E} \psi_{F,E}, \nu_E,B_E) \big).
\end{equation}
Differentiating the above identity yields
\begin{equation}
    \begin{split}
      \nabla_{\pa E}^2 (f \circ \Psi_{F,E})  &= \nabla_{\pa F}^2 f \circ \Psi_{F,E} \big(  I+ A_1(\psi_{F,E}, \nabla_{\pa E} \psi_{F,E}, \nu_E,B_E) \big)^2\\
      &\quad+{\rm tr} \big( \nabla_{\pa F} f \circ \Psi_{F,E} \otimes \nabla_{\pa E} \big( I+ A_1(\psi_{F,E}, \nabla_{\pa E} \psi_{F,E}, \nu_E,B_E)    \big) \big).
    \end{split}
\end{equation}
Taking the trace of the above identity and computing explicitly $ \nabla_{\pa E} \big( I+ A_1(\psi, \nabla_{\pa E} \psi, \nu_E,B_E)    \big) $ we obtain~\eqref{06052026laplace},~\eqref{06052026resto1}, and~\eqref{06052026resto2}.

By the claim, formula~\eqref{06052026tuttiquanti}, and a further integration by parts, we obtain
\begin{equation}\label{LE-LOUVRE}
\begin{split}
     - &\int_{\pa F} \Delta_{\pa F} \xi_{G,F} \Delta_{\pa F} \big(  H^\varphi_F- Q(E(u_F)) \big)\\
     = & -  \int_{\pa E} \big( \Delta_{\pa F} ( H_F^\varphi- Q(E(u_F)))  \big) \circ \Psi_{F,E} \big(\Delta_{\pa F} \xi_{G,F} \big) \circ \Psi_{F,E} \big( J_{\pa E } \Psi_{F,E}  -\sqrt{J_{\pa E } \Psi_{F,E} } \big) \\
     & -  \int_{\pa E} \big( \Delta_{\pa F} ( H_F^\varphi- Q(E(u_F)))  \big) \circ \Psi_{F,E} \big(\Delta_{\pa F} \xi_{G,F} \big) \circ \Psi_{F,E} \sqrt{ J_{\pa E } \Psi_{F,E} }
     \\
   \leq  &  \int_{\pa E} \nabla_{\pa E} \Delta_{\pa E} \big(  (H_F^\varphi-Q(E(u_F))) \circ\Psi_{F,E} \big) \cdot \nabla_{\pa E} \big( \xi_{G,F} \circ \Psi_{F,E} \big)\sqrt{ J_{\pa E} \Psi_{F,E}} + \| R \|_{L^1(\pa F)}, 
\end{split}
\end{equation}
where $R$ denotes an error term satisfying the following pointwise estimate:
\begin{align*}
    R \leq&\\
    & C \big( \vert \psi_{F,E} \vert^2+ \vert \nabla_{\pa E} \psi_{F,E} \vert^2  \big) \vert \Delta_{\pa F} ( H_F^\varphi- Q(E(u_F)))  \vert \circ \Psi_{F,E} \vert \Delta_{\pa F} \xi_{G,F} \vert \circ \Psi_{F,E} \,\bigg( :=R_1 \bigg) \\
    & + C \vert \Delta_{\pa F} ( H_F^\varphi- Q(E(u_F)))  \vert \circ \Psi_{F,E} \big( \sum_{i=1}^2\vert \nabla_{\pa F}^i \xi_{G,F} \vert    \big)\circ \Psi_{F,E}  \big( \sum_{i=1}^2\vert \nabla_{\pa E}^i \psi_{F,E} \vert    \big) \,\bigg( :=R_2 \bigg)\\
    & + C \vert \Delta_{\pa F} \xi_{G,F} \vert \circ \Psi_{F,E} \big( \sum_{i=1}^2 \vert \nabla_{\pa F}^i ( H_F^\varphi- Q(E(u_F)) ) \vert   \big) \circ \Psi_{F,E} \big( \sum_{i=1}^2\vert \nabla_{\pa E}^i \psi_{F,E} \vert    \big) \,\bigg( :=R_3 \bigg)\\
    & + C \vert \Delta_{\pa E} \big(   (H_F^\varphi-Q(E(u_F))) \circ \Psi_{F,E} \big)\vert \vert \nabla_{\pa E} \big( \xi_{G,F} \circ \Psi_{F,E}  \big)\vert \big( \sum_{i=1}^2 \vert \nabla_{\pa E}^i\psi_{F,E}\vert  \big) \,\bigg( :=R_4 \bigg).
\end{align*}
It remains to estimate $R_i$ for $i=1, \dots,4$. To this end, we fix $\varepsilon>0$, to be chosen later. \\
\textit{Estimate of $\int_{\pa E}R_1+ R_2+R_3$.}

Using~\eqref{LESTIMEpsiFE}, together with the Sobolev embedding theorem and a change of variables, we obtain
\begin{align}
    \int_{\pa E} R_1+ R_2+R_3 \leq& C \| \xi_{G,F} \|^2_{H^2(\pa F)} + C \| \psi_{F,E} \|^2_{W^{2,4}(\pa E)} \|  H_F^\varphi-Q(E(u_F))    \|^2_{W^{2,4}(\pa F)}\notag\\
    \leq  & \varepsilon \| \nabla_{\pa F} \Delta_{\pa F} \psi_{G,F} \|^2_{L^2(\pa F)}+ C \| \nabla_{\pa F} \xi_{G,F} \|^2_{L^2(\pa F)}  + C \| \psi_{F,E} \|^2_{H^3(\pa E)}\notag\\
    \leq & \varepsilon \| \nabla_{\pa F} \Delta_{\pa F} \psi_{G,F} \|^2_{L^2(\pa F)}+ C \| \nabla_{\pa F} \xi_{G,F} \|^2_{L^2(\pa F)}  \label{10052026tuttiquanti}\\
    &  + \varepsilon \| \nabla_{\pa E} \Delta_{\pa E}\psi_{F,E} \|^2_{L^2(\pa E)}+ C \| \nabla_{\pa E} \xi_{F,E} \|^2_{L^2(\pa E)}, \notag
\end{align}
in the second and last inequalities we argue as in~\eqref{03052026caffe1}.\\
\textit{Estimate of $\int_{\pa E} R_4$.}

We recall the Euler-Lagrange equation~\eqref{Cosac'e'}
\begin{equation}\label{08052026dino}
    \frac{\xi_{F,E}}{h} = \Delta_{\pa E} \big(  (H_F^\varphi-Q(E(u_F)))   \circ \Psi_{F,E}\big) \text{ on }\pa E.
\end{equation}
Using this identity together with the Cauchy–Schwarz inequality, we obtain
\begin{align}
        \int_{\pa E} R_4 & \leq C\int_{\pa E} \vert \frac{\xi_{F,E}}{h} \vert \vert \nabla_{\pa E} (\xi_{G,F} \circ \Psi_{F,E})\vert \big( \sum_{i=1}^2 \vert \nabla_{\pa E}^i \psi_{F,E} \vert  \big)  \notag\\
        & \leq C  \| \nabla_{\pa F} \xi_{G,F} \|^2_{L^2(\pa F)}+\frac{\varepsilon}{h^2} \int_{\pa E} \vert \xi_{F,E} \vert^2 \big( \sum_{i=1}^2 \vert \nabla_{\pa E}^i \psi_{F,E} \vert^2  \big) \notag\\
        & \leq C \| \nabla_{\pa F} \xi_{G,F} \|^2_{L^2(\pa F)}+\frac{\varepsilon}{h^2} \| \xi_{F,E} \|^2_{L^4(\pa E)}\| \psi_{F,E} \|^2_{W^{2,4}(\pa E)} \notag\\
        & \leq C \| \nabla_{\pa F} \xi_{G,F} \|^2_{L^2(\pa F)}+\frac{\varepsilon C}{h^2} \|  \xi_{F,E} \|^2_{H^1(\pa E)} \| \psi_{F,E} \|_{H^3(\pa E)}^2 \notag\\
        & \leq C \| \nabla_{\pa F} \xi_{G,F} \|^2_{L^2(\pa F)}+\varepsilon \| \nabla_{\pa E} \Delta_{\pa E} \psi_{F,E} \|^2_{L^2(\pa E)}+C \| \nabla_{\pa E} \xi_{F,E} \|^2_{L^2(\pa E)},
\end{align}
in the second inequality we performed a change of variables, and in the last we applied~\eqref{LESTIMEpsiFE} and argued as in~\eqref{03052026caffe1}.

Finally, by~\eqref{LE-LOUVRE} and~\eqref{08052026dino} combined with the estimates of the error terms, we obtain
\begin{align}
         - &\int_{\pa F} \Delta_{\pa F} \xi_{G,F} \Delta_{\pa F} \big(  H^\varphi_F- Q(E(u_F)) \big)\notag\\
         & \quad \leq \frac{1}{h}\int_{\pa E} \nabla_{\pa E} \xi_{F,E} \cdot \nabla_{\pa E} \big( \xi_{G,F} \circ \Psi_{F,E}  \big)\sqrt{ J_{\pa E} \Psi_{F,E}} +\varepsilon C \| \nabla_{\pa F} \Delta_{\pa F} \psi_{G,F} \|^2_{L^2(\pa F)} \notag\\
       & \qquad+ C \| \nabla_{\pa F} \xi_{G,F} \|^2_{L^2(\pa F)}  + \varepsilon C\| \nabla_{\pa E} \Delta_{\pa E}\psi_{F,E} \|^2_{L^2(\pa E)}+ C \| \nabla_{\pa E} \xi_{F,E} \|^2_{L^2(\pa E)} \notag\\
       & \quad \leq \frac{1}{h} \int_{\pa E} \vert \nabla_{\pa E} \xi_{F,E} \vert \vert \nabla_{\pa F} \xi_{G,F} \vert \circ \Psi_{F,E} \big( \sqrt{J_{\pa E} \Psi_{F,E}}+ C \big(1+ \vert \psi_{F,E} \vert + \vert \nabla_{\pa E} \psi_{F,E} \vert  \big)  \big)\notag\notag\\
       & \qquad +\varepsilon C\| \nabla_{\pa F} \Delta_{\pa F} \psi_{G,F} \|^2_{L^2(\pa F)} + C \| \nabla_{\pa F} \xi_{G,F} \|^2_{L^2(\pa F)} \notag\\
       & \qquad + \varepsilon C\| \nabla_{\pa E} \Delta_{\pa E}\psi_{F,E} \|^2_{L^2(\pa E)}+ C \| \nabla_{\pa E} \xi_{F,E} \|^2_{L^2(\pa E)}  \notag\\
       & \quad \leq \frac{1}{h} \int_{\pa E} \vert \nabla_{\pa E} \xi_{F,E} \vert \vert \nabla_{\pa F} \xi_{G,F} \vert \circ \Psi_{F,E} \sqrt{ J_{\pa E} \Psi_{F,E}}\notag\\
       & \qquad + \frac{C}{h} \| \nabla_{\pa F} \xi_{G,F} \|_{L^2(\pa F)} \bigg( \int_{\pa E} \vert \nabla_{\pa E} \xi_{F,E} \vert^2 (1+\vert \psi_{F,E} \vert^2+ \vert \nabla_{\pa E} \psi_{F,E} \vert^2)  \bigg)^{\frac{1}{2}}
       \notag\\
       & \qquad +\varepsilon C\| \nabla_{\pa F} \Delta_{\pa F} \psi_{G,F} \|^2_{L^2(\pa F)} + C \| \nabla_{\pa F} \xi_{G,F} \|^2_{L^2(\pa F)} \notag\\
       & \qquad + \varepsilon C\| \nabla_{\pa E} \Delta_{\pa E}\psi_{F,E} \|^2_{L^2(\pa E)}+ C \| \nabla_{\pa E} \xi_{F,E} \|^2_{L^2(\pa E)} \notag\\ 
       & \quad \leq  \frac{1}{h} \bigg( C\| \nabla_{\pa E} \xi_{F,E} \|^2_{L^2(\pa E)}+ \frac{1}{2} \| \nabla_{\pa F} \xi_{G,F} \|^2_{L^2(\pa F)}  \bigg) + C \| \psi_{F,E} \|^2_{W^{1,4}(\pa E)} \notag\\
       & \qquad +\varepsilon \| \nabla_{\pa F} \Delta_{\pa F} \psi_{G,F} \|^2_{L^2(\pa F)} + C \| \nabla_{\pa F} \xi_{G,F} \|^2_{L^2(\pa F)} \notag\\
       & \qquad + \varepsilon \| \nabla_{\pa E} \Delta_{\pa E}\psi_{F,E} \|^2_{L^2(\pa E)}+ C \| \nabla_{\pa E} \xi_{F,E} \|^2_{L^2(\pa E)} \\ 
       &\quad \leq \frac{1}{h} \bigg( C\| \nabla_{\pa E} \xi_{F,E} \|^2_{L^2(\pa E)}+ \frac{1}{2} \| \nabla_{\pa F} \xi_{G,F} \|^2_{L^2(\pa F)}  \bigg)\label{10052026fineiter} \\
       & \qquad +\varepsilon \| \nabla_{\pa F} \Delta_{\pa F} \psi_{G,F} \|^2_{L^2(\pa F)} + C \| \nabla_{\pa F} \xi_{G,F} \|^2_{L^2(\pa F)} \notag\\
       & \qquad + \varepsilon  \| \nabla_{\pa E} \Delta_{\pa E}\psi_{F,E} \|^2_{L^2(\pa E)}+ C \| \nabla_{\pa E} \xi_{F,E} \|^2_{L^2(\pa E)},\notag
\end{align}
where in the second inequality we used~\eqref{eq:ugnablacomp}; in the fourth inequality we used Young’s inequality, the estimate $ \vert \nabla_{\pa E} \xi_{F,E}\vert \leq C (\vert \psi_{F,E} \vert + \vert \nabla_{\pa E} \psi_{F,E} \vert)$ as well as~\eqref{LESTIMEpsiFE} and~\eqref{LESTIMEpsiGF}; in the last inequality we applied the estimate $ \| \psi_{F,E} \|_{W^{1,4}(\pa E)}^2 \leq \varepsilon \| \nabla_{\pa E} \Delta_{\pa E}\psi_{F,E} \|^2_{L^2(\pa E)}+ C \| \nabla_{\pa E} \xi_{F,E} \|^2_{L^2(\pa E)} $ that can be shown as~\eqref{03052026caffe1}.

Finally, by~\eqref{05052026dueminitui},~\eqref{10052026fineiter}, taking $\varepsilon>0$ sufficiently small, and using the ellipticity of $\mathcal{A}$, we deduce~\eqref{ARSENIOLEIMI}.
\end{proof}

\section{Proof of the main theorem} \label{sezionefinale}
	In this section, we exploit the iteration estimates established in the previous section to prove that the constrained discrete flat flows, as introduced in Definition~\ref{12092023def1}, can be used to construct classical solutions of~\eqref{MAINEQ}, provided that $K_{el}$ is sufficiently large and the initial datum $E_0 \Subset \Omega$ is sufficiently smooth.

    We first establish the stability of the elastic minimizers with respect to convergence of the underlying domains.
    
    \begin{lemma}\label{lemmafacilespero}
		Let $F_h \Subset \Omega$, for $h>0$, and $F \Subset \Omega$ be uniformly $C^2$-regular sets satisfying the \textsc{UBC} with radius $r$ and $ \vert F_h \vert= \vert F \vert$. Assume that $\chi_{F_h} \to \chi_F$ in $L^1(\Omega)$ and that $ (\pa F_h, \vert \cdot \vert, \mathcal{H}^2 \mrestr \pa F_h)$ converge to  $ (\pa F, \vert \cdot \vert, \mathcal{H}^2 \mrestr \pa F)$ in the measured Gromov-Hausdorff sense. Let $u_{F_h}^{K_{el},h}$ be a minimizer of problem~\eqref{minelvinc} for $h>0$. Then, $u_{F_h}^{K_{el},h} \rightarrow u_F^{K_{el},0}$ as $h \rightarrow 0^+$ in $\mathfrak{C}_{K_{el}}^{3,\frac{1}{4}}(\Omega ,\R^2)$, where $u_F^{K_{el},0} $ is a minimizer of~\eqref{minelvinc} corresponding to $h=0$.
	\end{lemma}
    \begin{proof}
        Throughout the proof, we will implicitly pass to subsequences without further mention.
		By Arzelà-Ascoli Theorem, we have $ u_{F_h}^{K_{el},h} \rightarrow u$ in $C^{3,\beta}(\Omega)$ for $\beta < 1/4$, where $u \in C^{3,\frac{1}{4}}(\Omega)$. Let $ v \in C^5(\Omega)$ be a competitor to the minimality of $u_{F_h}^{K_{el},h}$ for some $h>0$. Then, we have 
		\begin{equation}\label{11052026yoga}
			\begin{split}
				\int_{\Omega \setminus F} Q(E(u)) \, d x& = \lim_{h \rightarrow 0^+} \int_{\Omega \setminus F_h} Q(E(u_{F_h}^{K_{el},h})) \, d x \\
				&\leq \lim_{h \rightarrow 0^+}\int_{\Omega \setminus F_h} Q(E(v))\, d x  = \int_{\Omega \setminus F} Q(E(v))\, d x.
			\end{split}
		\end{equation}
        By the measured Gromov-Hausdorff convergence of $(\pa F_h, \vert \cdot \vert, \mathcal{H}^2 \mrestr \pa F_h)$ to $(\pa F, \vert \cdot \vert, \mathcal{H}^2 \mrestr \pa F)$ together with the fact that both $F_h$ and $ F$ satisfy the \textsc{UBC} with radius $r$ it follows from~\cite{AMBHON2017, Gigli2017} that $$ \| \pa^{\alpha} u_{F}^{K_{el}} \|_{H^{1,2}(\mathcal{H}^2 \mrestr \pa F)} \leq \liminf_{h \rightarrow 0^+}\| \pa^{\alpha} u_{F_h}^{K_{el}} \|_{H^{1,2}(\mathcal{H}^2 \mrestr \pa F_h)}, \text{ for every } |\alpha|=3 . $$   Hence, $ \| \pa^{\alpha} u_{F}^{K_{el}} \|_{H^{1,2}(\mathcal{H}^2 \mrestr \pa F)} \leq K_{el}$. Therefore, combining this estimate with~\eqref{11052026yoga} and applying a standard density argument, this concludes the proof. 
    \end{proof}
    Before proving the main theorem of this section, we state a lemma whose proof can be found in~\cite[Lemma 5.2]{CFJKsd}.
    \begin{lemma}\label{lemma15052026}
        Let $\Omega$ and $E_0$ be as in Definition~\ref{Omegaevrietadiriferimento}, $K>0$, and let $F \Subset \Omega$ be a $(\Lambda,\alpha)$-minimizer of the anisotropic perimeter (see Definition~\ref{15052026deflambdaalpha}) satisfying
        \begin{equation}
            \| B_F \|_{L^\infty(\pa F)} \leq K \quad \text{ and } \quad  \| \Delta_{\pa F} H_F \|_{L^2(\pa F)} \leq K.
        \end{equation}
        There exists $\delta>0$, depending only on $\Omega$, $E_0$, $K$, $\Lambda$, and $\alpha$, such that if $ F \Delta E_0 \subset \mathcal{I}_\delta(\pa E_0)$, then $F$ satisfies the \textsc{UBC} with radius $ \sigma_0 /2$, where $\sigma_0$ is given in~\eqref{Laconstantesigmazerofin}.
    \end{lemma}
    \begin{theorem}\label{1THMMAIN1}
        Let $\Omega$ and $E_0$ be as in Definition~\ref{Omegaevrietadiriferimento},
		and let $K_{el}>0$ be the constant provided by Remark~\ref{02052026falchi}. There exist  $T_s,C_0, \beta_0,\sigma_1>0$ with the following property: for every $\beta < \beta_0$, there exists $\tilde{h}$ such that any discrete constrained flat flow $ \{ E^{h,\beta}_t \}_{t \geq0}$ starting from $E_0$ with $ 0 < h \leq \tilde{h}$ satisfies
		\begin{equation}\label{LaVidaEsUnaLenteja}
			\pa E^{h,\beta}_t= \{x+ f^{h,\beta}(t,x)\nu_{E_0}(x): x \in \pa E_0    \}, \, \| f^{h,\beta} \|_{H^5(\pa E_0)} \leq C_0,\, \| f^{h,\beta} \|_{L^\infty(\pa E_0)} \leq \sigma_1,
		\end{equation}
		for all $t \in [0,T_s]$. Moreover, the functions $f^{h,\beta}$ converge in 
		$L^\infty([0,T_s], H^5(\pa E_0))$ to a function $f^{\beta}$ such that the family $ \{ E^\beta_t \}_{t \in [0,T_s]}$ defined as
		\begin{equation}
			\pa E^\beta_t= \{  x + f^\beta(t,x) \nu_{E_0}(x) : x \in \pa E_0 \}
		\end{equation}
		is a distributional solution of~\eqref{MAINEQ}.	Moreover, there exists $\theta \in (0,1)$ such that $$ f^\beta\in \mathrm{Lip}([0,T_s],L^1(\pa E_0)) \cap C^{0,\theta}([0,T_s], C^{3,\alpha}(\pa E_0)).$$
	\end{theorem}
\begin{proof}
In the proof, we will implicitly pass to subsequences without further mention, except when required for clarity. Fix $K_0 = K_0(K_{el},\sigma_0)$ a sufficiently large constant to be chosen later (recall~\eqref{Laconstantesigmazerofin} for the definition of $\sigma_0$), and fix $\beta \leq \beta_0 \leq \delta$, where $\delta$ is provided by Lemma~\ref{lemma15052026}. 

Let $\{E_{hk}^{h,\beta}\}_{k\in\mathbb{N}}$ be a constrained discrete flat flow starting from $E_0$; see Definition~\ref{12092023def1}. For simplicity of notation, we set $E_k:= E_{hk}^{h,\beta}$ for $k \in \N$. By the choice of $K_{el}$, we can apply Theorem~\ref{Bologna}, which yields \begin{equation}\label{12052025form-1}
			\begin{split}
				&\pa E_1= \{  x +\psi_1(x)\nu_{E_0}(x): x \in \pa E_0 \},\\
				&\| \psi_1 \|_{H^1(\pa E_0)} \leq L_0 h, \quad \| \psi_1 \|_{H^5(\pa E_0)} \leq L_0, \\
				&  \|  H_{E_1} \|_{H^3(\pa E_1)} \leq K_0, \quad  \| \Delta_{\pa E_1}^2 H_{E_1} \|_{L^2(\pa E_1)} \leq K_0 h^{-1/4},
			\end{split} 
		\end{equation}
        and  $u_{E_1}= u^{K_{el}}_{E_1} \mrestr \Omega \setminus E_1 $, where $u_{E_1}$ is given in~\eqref{minelast} and $u^{K_{el}}_{E_1}$ is given in~\eqref{minelvinc}. Here, $L_0=L_0(K_{el})$. Moreover, for any $\alpha \in (0,1)$,  the following estimates also hold
		\begin{equation}\label{12052025form-1/2}
			\begin{split}
      &\| \psi_1 \|_{H^2(\pa E_0)} \leq L_0 h^{\frac{3}{4}},\, \| \psi_1 \|_{H^3(\pa E_0)} \leq L_0 h^{\frac{1}{2}}, \, \| \psi_1 \|_{H^4(\pa E_0)} \leq L_0 h^{\frac{1}{4}}, \, \| \psi \|_{C^{3,\alpha}(\pa E)} \leq \tilde{L}_{0,\alpha} \\
      & \| \psi_1 \|_{C^{0,\alpha}(\pa E_0)} \leq \tilde{L}_{0,\alpha} h^{\frac{3}{4}} \,, \| \psi_1 \|_{C^{1,\alpha}(\pa E_0)} \leq \tilde{L}_{0,\alpha} h^{\frac{1}{2}},\, \| \psi_1 \|_{C^{2,\alpha}(\pa E_0)} \leq \tilde{L}_{0,\alpha} h^{\frac{1}{4}} 
      \end{split}
		\end{equation}
        
        Let $k_0 $ denote the largest natural number such that $\pa E_k \subset \mathcal{I}_{\beta }(\pa E_0)$  for every  $k \leq k_0$, and set $T_0:= k_0 h$.
		We claim that, for every $k \leq k_0$, the following holds:
		\begin{equation}\label{12052025form1}
			\|  H_{E_k} \|_{H^3(\pa E_k)} \leq K_0, \,\,  \| \Delta_{\pa E_k}^2 H_{E_k} \|_{L^2(\pa E_k)} \leq K_0 h^{-1/4}, \text{ and } u_{E_k}= u_{E_k}^{K_{el}}.
		\end{equation}
        We proceed by induction. The base case holds since \eqref{12052025form1} is valid for $k=1$. Assume that the claim holds up to $k-1$. Then, applying Theorem~\eqref{Bologna}, we obtain 
            	\begin{equation}\label{12052025form2}
			\begin{split}
				\pa E_{k}= \{  x +\psi_{k}(x)\nu_{E_{k-1}} (x): x \in \pa E_{k-1} \}, \,\,
				\| \psi_{k} \|_{H^1(\pa E_{k-1})} \leq L_1 h, \,\, \| \psi_{k} \|_{H^5(\pa E_{k-1})} \leq L_1, 
			\end{split} 
		\end{equation}
		where $L_1=L_1(K_0,K_{el})$.  For every $j\geq 1$, we set $\xi_j= \xi_{E_j,E_{j-1}}$; recall~\eqref{xiFEfunz}. Moreover, Theorem~\ref{Bologna} implies that $u_{E_k} = u_{E_k}^{K_{el}} \mrestr \Omega \setminus E_k$. Hence, we can apply Lemma~\ref{propdiiterazione} to obtain
\begin{align*}
   & \int_{\pa E_{j-1} }\vert \nabla_{\pa E_{j-1} } \xi_{ j } \vert^2 + \frac{h}{2} \mathcal{A}(\nu_ {E_{j-1}}) \nabla_{\pa  E_{j-1} } \Delta_{\pa  E_{j-1} } \psi_{ j } \cdot \nabla_{\pa  E_{j-1} } \Delta_{\pa  E_{j-1} } \psi_{ j } \\
       &  \leq (1+Mh)   \int_{\pa E_{j-2} } \Big(  \vert \nabla_{\pa E_{j-2}} \xi_{j-1} \vert^2 + \frac{h}{8} \mathcal{A}(\nu_ {E_{j-2}}) \nabla_{\pa  E_{j-2} } \Delta_{\pa  E_{j-2} } \psi_{ j-1 } \cdot \nabla_{\pa  E_{j-2} } \Delta_{\pa  E_{j-2} } \psi_{ j-1 }  \Big)
\end{align*}
        	for every $ 1 \leq j \leq k$. By iterating the estimate above and using~\eqref{12052025form-1} and~\eqref{12052025form-1/2}, we obtain
            \begin{equation}\label{10052026suseguitemi}
                \begin{split}
                 & \int_{\pa E_{k-1} }\vert \nabla_{\pa E_{k-1} } \xi_{ k } \vert^2 + \frac{h}{4} \sum_{j=1}^k \int_{\pa E_{j-1}}\mathcal{A}(\nu_ {E_{j-1}}) \nabla_{\pa  E_{k-1} } \Delta_{\pa  E_{j-1} } \psi_{ j } \cdot \nabla_{\pa  E_{j-1} } \Delta_{\pa  E_{j-1} } \psi_{ j } \\ 
                 & \leq (1+Mh)^{k-1} \int_{\pa E_{0} }  \Big ( \vert \nabla_{\pa E_{0}} \xi_{1} \vert^2 + \frac{h}{8} \mathcal{A}(\nu_ {E_{0}}) \nabla_{\pa  E_{0} } \Delta_{\pa  E_{0} } \psi_{ 1 } \cdot \nabla_{\pa  E_{0} } \Delta_{\pa  E_{0} } \psi_{ 1 }\Big )\\
                 &  \leq e^{2M h k} L_0^2 h^2  \leq e^{2M h k_0} L_0^2 h^2 \leq 2 L_0^2 h^2,
                \end{split}
            \end{equation}
            if $T_0=h k_0$ is sufficiently small. We recall (see~\eqref{10022026pprinazof2}) that the following holds
		\begin{equation}\label{zizzagna}
			\frac{1}{C}\| \nabla_{\pa E_{j-1}} \psi_{j} \|_{L^2(\pa E_{j-1})}^2 \leq \| \nabla_{\pa E_{j-1}  }\xi_{j}\|^2_{L^2(\pa E_{j-1})} \leq C \| \nabla_{\pa E_{j-1}} \psi_{j} \|_{L^2(\pa E_{j-1})}^2  \quad \forall j \ge 1. 
		\end{equation}
        Up to increasing the constant $L_0$ and using~\eqref{10052026suseguitemi} together with~\eqref{zizzagna}, we obtain
\begin{equation}\label{zizzagna2}
			\| \nabla_{\pa E_k} \psi_k \|^2_{L^2(\pa E_{k-1})}+ h \sum_{j=1}^k \| \nabla_{\pa E_{j-1}}\Delta_{\pa E_{j-1}} \psi_j \|^2_{L^2(\pa E_{j-1})} \leq L_0^2 h^2.
        \end{equation}
        Therefore, we infer \eqref{12052025form1} from Lemma~\ref{lemmahnegativa}, possibly after increasing the constant $K_0$.

        We now show that $T_0=hk_0>0$.
        We can assume that  there exists a point $x_0 \in \pa E_{k_0}$ such that $\mathrm{dist}(x_0,E_0)\geq  \frac{\beta}{2}$. Since $E_{k_0}$ is a $(\Lambda,\alpha)$-minimizer of the anisotropic perimeter for a constant $\Lambda$ independent of $h$, see Proposition~\ref{CaneBernard06}, the density estimates together with the inequality $\mathrm{dist}(x_0,E_0)\geq  \frac{\beta}{2}$ yield 
			$\vert E_{k_0} \Delta  E_0 \vert \geq c \beta^3,$
		for some constant $c$ depending on $\Lambda$ and $\alpha$. 	Therefore, using also~\eqref{zizzagna} and~\eqref{zizzagna2}, we infer
		\begin{align*}
				c \beta^3 &\leq \vert E_{k_0} \Delta  E_0 \vert \leq \sum_{j=1}^{k_0} \vert E_j \Delta E_{j-1} \vert \leq C \sum_{j=1}^{k_0} \| \psi_j \|_{L^1(\pa E_{j-1})}\leq CP(E_{j-1})^{\frac{1}{2}} \sum_{j=1}^{k_0} \| \xi_j \|_{L^2(\pa E_{j-1})} \\
                &\leq  C_\varphi (P_\varphi(E_0)+K_{el}^2\vert \Omega \vert)^{\frac{1}{2}} \sum_{j=1}^{k_0} \| \nabla_{\pa E_{j-1}} \psi_j \|_{L^2(\pa E_{j-1})} \leq  C L_0 T_0,
		\end{align*}
		here we have used that $P_\varphi(E_j) \leq K_{el}^2\vert \Omega \vert +P_\varphi (E_0)$, which follows from the minimizing movements scheme. This proves that $T_0>0$.
        
        There exist constants $C_0,\sigma_1>0$ such that 
		\begin{equation}\label{12052025claimlego}
			E_j \in \mathfrak{H}^5_{C_0,\sigma_1}(E_0) \text{ for all } 0 \leq j \leq k_0 . 
		\end{equation}
       Indeed, by Proposition~\ref{CaneBernard06}, the sets $E_j$ are $(\Lambda,\alpha)$-minimizers of the anisotropic perimeter for some $\Lambda$ independent of $h$. Consequently, Lemma~\ref{epsilonregolarita} implies that each $E_j$ can be represented as a normal graph over $\partial E_0$, with height function $f_j:\partial E_0 \to \mathbb{R}$ satisfying $ \| f_j \|_{C^{1,\gamma}(\pa E_0)} \leq C$, and $\mathrm{dist}(\pa E_j, \pa E_0) \leq \beta \leq \sigma_1 \text{ for every }j.$
		Finally, the first inequality in~\eqref{12052025form1} implies that $\| f_j \|_{ H^5(\pa E_0)} \leq C_0$.

        Let $0 \leq i,k \leq k_0$. Assume $i < k$, then it holds
		\begin{equation}
			\begin{split}
				\| f_i- f_k \|_{L^1(\pa E_0)} & \leq C\sum_{j=i+1}^{k} \vert E_j \Delta E_{j-1} \vert \leq C\sum_{j=i+1}^{k} \| \psi_j \|_{L^1(\pa E_{j-1})}\leq C \sum_{j=i+1}^{k} \| \xi_j \|_{L^2(\pa E_{j-1})}\\
				&\leq C  \sum_{j=i+1}^{k} \| \nabla_{\pa E_{j-1}} \psi_j \|_{L^2(\pa E_{j-1})} \leq CL_0 (k-i-1) h,
			\end{split}
		\end{equation}
		where we have used~\eqref{zizzagna} and~\eqref{zizzagna2}. Hence, there exists a constant $L_{lip}>0$ such that
        \begin{equation}\label{12052025claimlego2}
			\| f_i- f_k \|_{L^1(\pa E_0)} \leq L_{lip} h\vert k-1-i \vert.
		\end{equation}

        	Combining 
		\eqref{12052025claimlego} and~\eqref{12052025claimlego2},  and applying the Arzelà-Ascoli Theorem, we conclude that there exists a subsequence $\{h_m\}_{m \in \N} $ such that
		$f_{h_m}(t) \rightarrow f^\beta(t)$ in $L^1(\pa E_0)$ for a.e. $t \in [0,T_0]$ as $m \rightarrow + \infty$, where
		\begin{equation}\label{tempidiconvenzioni}
			f^\beta \in \mathrm{Lip} ([0,T_0], L^1(\pa E_0)), \quad f^\beta \in L^\infty ([0,T_0], H^5(\pa E_0)).
		\end{equation}
		In the following, we omit the dependence on $m$  for this subsequence. By the Sobolev embedding Theorem, it follows that $f^\beta \in L^\infty([0,T_0], C^{3,\alpha}(\pa E_0))$. 
        Moreover, by Proposition~\ref{prop:interpolation}, we obtain
        \begin{equation}
            \| f^\beta(t)-f^\beta(s) \|_{H^2(\pa E_0)} \leq C \| f^\beta (t)- f^\beta(s)\|_{H^5(\pa E_0)}^{1/2} \| f(t)-f(s) \|^{1/2}_{L^1(\pa E_0)} \leq C \vert t-s \vert^{1/2},
        \end{equation}
        which, together with the Sobolev embedding theorem, yields $f^\beta \in C^{0,1/2}([0,T_0], C^{0,\alpha}(\pa E_0))$ for every $\alpha \in (0,1)$. Finally, using~\eqref{interHOLDER}, we deduce that  $ f^{\beta}  \in C^{0,\theta}([0,T_0], C^{3,\alpha}(\pa E_0)) $ for some $\theta \in (0,1)$. 
        
        Let $ \{ E^\beta_t \}_{t \in [0,T_0]} $ be defined  by
		\begin{equation}\label{27052025form2}
			E^\beta_t \Delta E_0 \subset \mathcal{I}_{\sigma_1}(\pa E_0) \text{ and } \pa E^\beta_t := \{  x+  f^\beta(t,x) \nu_{E_0}(x) :  x \in \pa E_0\}.
		\end{equation}
        Recalling that $u_{E_j}=u_{E_j}^{K_{el},h}\mrestr \Omega \setminus E_j$, where $u_{E_j}^{K_{el},h}$ is a minimizer of problem~\eqref{minelvinc}, we deduce from Lemma~\ref{lemmafacilespero} that
		\begin{equation}
			u_{E_{j_h}}^{K_{el},h} \rightarrow u_{E^\beta_t}^{K_{el},0} \text{ in } \mathfrak{C}_{K_{el}}^{3,\frac{1}{4}}(\Omega ,\R^2) \text{ as } h \rightarrow 0^+,
		\end{equation}
		where $ u^{K_{el},0}_{E^\beta_t}$ is a minimizer of~\eqref{minelvinc} corresponding to $h=0$.
        We show that $\{ E_t^\beta \}_{t \in [0,T_0]}$ is a distributional solution of 
		\begin{equation}\label{EQ1primasoluzione}
			\left\{
			\begin{aligned}
				& V_t= \Delta_{\pa E^\beta_t} \big(   H^\varphi_{E^\beta_t}-Q (E(u_{E^\beta_t}^{K_{el},0}))\big), \text{ on } \pa E_t^\beta  \\
				& E^\beta_0=E_0, \\
				&u_{E^\beta_t}^{K_{el},0}  \text{ solution of~\eqref{minelvinc} for $h=0$.}
			\end{aligned}
			\right.
		\end{equation} 
        
        Let $\Psi_j : \pa E_0 \rightarrow \pa E_j$ be defined by $\Psi_j(x):= x+ f_j (x)\nu_{E_0}(x) .$  
		We consider $N_j : \pa E_0 \rightarrow \R^3$ given by
		$  N_j(x):=-(I +f_j(x)B_{E_0}(x))^{-1} \nabla_{\pa E_0} f_j(x)+ \nu_{E_0}(x).$
		We observe that 
		\begin{equation}
			\vert N_j \vert= \frac{ J_{\pa E_0} \Psi_j}{1+ H_{E_0} f_j+ K_{E_0} f_j^2}
		\end{equation}
        where $J_{\pa E_0} \Psi_j$ denotes the tangential Jacobian of $\Psi_j$.\\
        \textit{Claim:} The following formula holds:
		\begin{equation}\label{terranovacanel2}
			\lim_{h \rightarrow 0^+} \left \|  \frac{\psi_{j+1}}{h} \circ \Psi_j - \frac{f_{j+1}-f_j}{\vert N_j \vert h} \right\|_{L^2(\pa E_0)} =0.
		\end{equation} 
        The estimate~\eqref{zizzagna2} implies
		$$\| \psi_{j+1} \circ \Psi_j \|_{C^1(\pa E_0)} \leq C h^{\frac{1}{2}} \text{ and } \| \psi_{j+1} \circ \Psi_j \|_{L^2(\pa E_0)} \leq C h. $$
		From the bound $ \| f_j \|_{C^{1,\gamma}}(\pa E_0) \leq C$ and the previous estimate, we deduce 
		$$ \vert f_{j+1}(x)-f_j(x) \vert \leq C \vert \psi_{j+1} \circ \Psi_j (x)\vert \, \, \forall x \in \pa E_0 \text{ and }  \|f_{j+1}-f_j \|_{C^1(\pa E_0)} \leq C h^{\frac{1}{2}}.$$
		Let $G : \pa  E_0 \rightarrow \R$ be a function satisfying $ \| G \|_{C^1(\pa E_0)} \leq C h^{\gamma}$ for some $\gamma$.  By parametrizing $E_{j+1}\triangle E_j$ over $\pa E_0$ and by the coarea formula, we obtain
		\begin{align*}
				&\int_{\R^3} G \circ \pi_{\pa E_0}(x) \big(   \chi_{E_{j+1}}(x)- \chi_{E_{j}}(x) \big)\, dx \\
				&= \int_{\pa E_0} G(x) (f_{j+1}(x)-f_{j}(x)) (1+ f_j(x) H_{E_0}(x)+ f_j(x)^2 K_{E_0}(x))\, d \mathcal{H}^2_x+ o(h^2)\\
				&= \int_{\pa E_0} G(x) J_{\pa E_0} \Psi_j(x) \frac{f_{j+1}(x)- f_j(x)}{\vert N_j (x) \vert} \, d \mathcal{H}^2_x+ o(h^2).
		\end{align*}
		Let
		$ \Phi_{j,t} : \pa E^j \rightarrow \R^3$ be given by $\Phi_{j,t}(x):= x+ t \nu_{E_j}(x)$. Then we have
        	\begin{align*}
				&\int_{\R^3} G \circ \pi_{\pa E_0}(x) \big(  \chi_{E_{j+1}}(x)- \chi_{E_j}(x) \big) \, d x \\
				& = \int_{\pa E_j} \int_{0}^{\psi_{j+1}(x)} \big( G \circ \pi_{\pa E_0} ( \Phi_{j,t}(x)) -G \circ \pi_{\pa E_0}(x)+ G \circ \pi_{\pa E_0}(x)\big)J_{\pa E_j} \Phi_{j,t}(x) \, d t \,d \mathcal{H}^2_x \\
				& = \int_{\pa E_0} \psi_{j+1} \circ \Psi_j (x) G (x) J_{\pa E_0} \Psi_j(x)\, d\mathcal{H}_x^2+ o(h^2).
		\end{align*}
Comparing the two equalities above, we conclude that for all $G: \pa E_0 \rightarrow \R$ with $ \|G \|_{C^1(\pa E_0)} \leq Ch^{\gamma}$, it holds true
		\begin{equation}\label{caneterranoval3}
			\int_{\pa E_0}  G(x) J_{\pa E_0} \Psi_{j}(x) \bigg[ \psi_{j+1}\circ \Psi_j(x)-\frac{f_{j+1}(x)- f_j(x)}{\vert N_j (x) \vert}  \bigg]\, d \mathcal{H}^2_x=              o(h^2).
		\end{equation}
        We choose
		\begin{equation*}
			G(x):= \frac{1}{J_{\pa E_0} \Psi_{j}(x)}\bigg[ \psi_{j+1}\circ \Psi_j(x)-\frac{f_{j+1}(x)- f_j(x)}{\vert N_j (x) \vert}  \bigg].
		\end{equation*}  
		A straightforward computation yields  $ \| G \|_{C^1(\pa E_0) } \leq C h^\gamma$ for some $\gamma \in (0,1)$. Using this particular choice of $G$ in~\eqref{caneterranoval3} yields~\eqref{terranovacanel2}. 
        
        	We now return to proving that $E^\beta(t)$ is a solution of~\eqref{EQ1primasoluzione}. So far, we have established:
		\begin{equation}\label{27052025form1}
			\|f_j \|_{H^5(\pa E_0)} \leq C_0, \, \| f_j \|_{L^\infty(\pa E_0)} \leq \sigma_1 ,\,  \bigg\| \frac{f_{j+1}-f_j}{\vert N_j \vert h} \bigg\|_{L^2(\pa E_0)} \leq C \, \text{ for every } j h \leq T_0 .
		\end{equation} 
        Therefore, using~\eqref{27052025form1} together with~\eqref{terranovacanel2} and~\eqref{tempidiconvenzioni}, we conclude that
		\begin{equation}\label{19052025fonzi}
			\exists \, L^2(\pa E_0)-\lim_{h \rightarrow 0^+} \frac{\psi_{j+1}}{h} \circ \Psi_j(\cdot)= \frac{\pa_t f^\beta(t,\cdot)}{\vert N(t,\cdot)\vert },\,\,  \text{ for } t \in [0,T_0],
		\end{equation}
		where 
        $$  \vert N(t,x) \vert= \frac{J_{\pa E_0} \Psi_t(x)}{1+H_{E_0}(x) f^\beta(t,x) +K_{E_0}(x)(f^\beta(t,x))^2 } \text{ and } \Psi_t(x):= x+ f^\beta(t,x)\nu_{E_0}(x)$$
          for $x \in \pa E_0$. Let $l \in C^2_c(\R^3)$. Multiplying the Euler–Lagrange equation~\eqref{PeerGynt} by $l$ and integrating by parts yields
          	\begin{equation}\label{11052026tecnologia}
			\begin{split}
				\int_{\pa E_j} \frac{\psi_{j+1}(x)}{h} & l(x) \, d\mathcal{H}^2_x+  \int_{\pa E_j} \frac{\psi^2_{j+1}(x)}{h} \bigg( \frac{H_{E_j}(x)}{2}+ \psi_{j+1}(x) \frac{K_{E_j}(x)}{3}  \bigg)l(x)\, d \mathcal{H}^2_x\\
				&= \int_{\pa E_j}  - \div_{\pa E_j} \big( \mathcal{A}(\nu_E)\nabla_{\pa E_j}\psi_{j+1}  \big)(x) \Delta_{\pa E_j} l(x) \, d \mathcal{H}^2_x \\
				&\quad- \int_{\pa E_j}  Q(E(u_{E_{j+1}}^{K_{el}}))(x+ \psi_{j+1}(x)\nu_{E_j}(x)) \Delta_{\pa E_j} l(x) \, d \mathcal{H}^2_x
				\\ 
				&\quad+ \int_{\pa E_j} \widetilde{\mathcal{R}}_0(x) \Delta_{\pa E_j} l(x)\, d \mathcal{H}^2_x+ \int_{\pa E_j} H^\varphi_{E_j}(x) \Delta_{\pa E_j}l(x)\, d \mathcal{H}^2_x.
			\end{split}
		\end{equation}
        We observe that~\eqref{12052025form1} and~\eqref{zizzagna2} imply
        \begin{equation}\label{nanannananananananana}
                \lim_{h \rightarrow 0^+} \int_{\pa E_j} \frac{\psi^2_{j+1}(x)}{h} \bigg( \frac{H_{E_j}(x)}{2}+ \psi_{j+1}(x) \frac{K_{E_j}(x)}{3}  \bigg)l(x)\, d \mathcal{H}^2_x   =0.
        \end{equation}
        	From the previous estimates, we also have $ \| \psi_{j+1} \|_{C^2(\pa E_j)} \leq C h^\gamma.$ Hence, recalling~\eqref{nuovoresto}, we obtain
		\begin{equation}\label{centrodigravitapermanente}
			\| \widetilde{\mathcal{R}}_0 \|_{C^0(\pa E_j)} \leq C h^\gamma.
		\end{equation}
       Passing to the limit as $h \rightarrow 0^+$ in~\eqref{11052026tecnologia}, and using~\eqref{nanannananananananana},~\eqref{centrodigravitapermanente} together with a change of variables, we obtain that
        \begin{multline}
         \int_{\pa E_0} \frac{\pa_t f^\beta(t,x)}{\vert N(t,x)\vert } l(\Psi_t(x)) J_{\pa E_0}(\Psi_t(x))\, d \mathcal{H}^2_x \\
         = \int_{\pa E_0}  \Delta_{\pa E_t^\beta} \big( H^\varphi_{E^\beta_t}-Q( E(u_{E_t^\beta}^{K_{el}})  )  \big) l(\Psi_t(x)) J_{\pa E_0}(\Psi_t(x))\, d \mathcal{H}^2_x.
        \end{multline}

        To conclude we show that there exists $0 < T_s \leq T_0$ such that $\{ E_t^\beta \}_{t \in [0,T_s]}$ is solution of~\eqref{MAINEQ} by arguing as in Remarks~\ref{02052026falchi},~\ref{01052026remrkino}, and Lemma~\ref{02052026silvia}. First, we construct an extension $\tilde{u}_{E_t^\beta}$ of $u_{E_t^\beta}$ in the whole $\Omega$, and  prove that there exists $T_s \in (0,T_0] $ such that $ \| \tilde{u}_{E_t^\beta} \|_{C^{3,1/4}(\Omega)} < K_{el}$ for all $t \in [0,T_s]$. Then, we show (possibly after further reducing $T_s$) that $ \| \nabla^3 u_{E_t^\beta} \|_{H^1(\pa F)}< K_{el} $ for all $t \in [0,T_s]$. These two properties imply that the minimizer of problem~\eqref{minelast} coincides with the minimizer of problem~\eqref{minelvinc}, from which the claim follows.  We recall that $K_{el}$ is the constant provided by Remark~\ref{02052026falchi}, namely
        \begin{equation}\label{perlachinaValle}
            u_{E_0}^{K_{el}} \mrestr E_0 = u_{E_0}, \,\,\, \| u_{E_0}^{K_{el}} \|_{C^{3,\frac{1}{4}}(\Omega)}< K_{el}, \,\,\, \| \nabla^3 u_{E_0} \|_{H^1(\pa E_0)} < K_{el}.
        \end{equation}
        \textit{Claim:}   There exists $T_s \in (0,T_0]$ such that the functions \begin{equation}\label{estensioneu_E^beta}
			\tilde{u}_{E^\beta_t}(x)=\left\{
			\begin{aligned}
				& u_{E^\beta_t}(x) \text{ if } x \in \Omega \setminus {E^\beta_t} , \\
				&  \eta\bigg( \frac{ d_{E_0}(x)}{\sigma_{0}}  \bigg)u_{E^\beta_t} \bigg( \pi_{\pa E_0}(x)+ f^\beta(t,x)\nu_{E_0}(\pi_{\pa E_0}(x))   \bigg) \text{ if } x \in E^\beta_t
			\end{aligned}
			\right.
		\end{equation}
        satisfy $ \|\tilde{u}_{E^\beta_t} \|_{C^{3,1/4}(\Omega)} < K_{el} $ for $t \in [0, T_s]$,
       where $\eta \in C^{\infty}_c((-2,2))$ with $\eta \geq 0$ and $\eta \equiv 1$ in $(-1,1)$, and where $\sigma_0$ is defined in~\eqref{Laconstantesigmazerofin}. 

        By the definition of $ \tilde{u}_{E^\beta_t}$ and since $u_{E_t^\beta}$ solves a system of the form~\eqref{eqelliel}, we obtain
        \begin{align*}
                \| \tilde{u}_{E^\beta_t} \|_{C^{3,1/4}(\Omega)} &\leq C \| u_{E_t^\beta} \|_{C^{3,1/4}(\Omega \setminus E_t^\beta)} \leq C \big( \| w_0 \|_{C^{3,1/2}(\pa \Omega)}+ \| f^\beta(t,\cdot)\|_{C^{3,1/4}(\pa E_0)}  \big)\\
                & \leq C  \| w_0 \|_{C^{3,1/2}(\pa \Omega)} + C(K_{el}) \| f^\beta(t,\cdot) \|_{L^1(\pa E_0)}^\theta\\
                & \leq C \| w_0 \|_{C^{3,1/2}(\pa \Omega)} + C(K_{el}) t^\theta,
        \end{align*}
       where $\theta  \in (0,1)$, here  the third inequality follows by interpolation, and in the last inequality we used~\eqref{tempidiconvenzioni}. Therefore, the claim follows by taking $t\le T_s$ sufficiently small.
       \\ 
        \textit{Claim:} There exists $ T_s$ such that $ \| \nabla^3 u_{E_t^\beta} \|_{H^1(\pa F)} < K_{el}$ for all $t \in [0,T_s]$.

        Using a change of variable and the fractional interpolation inequality  combined with formulas ~\eqref{28042026meriggiare} and~\eqref{26042026onepiece1} (for $k=5$), we obtain
        \begin{align*}
                \| \nabla^3 u_{E_t^\beta} \|_{H^1(\pa F)} \leq & C \| \nabla^3 u_{E_t^\beta}(\cdot+ f^\beta(t,\cdot)\nu_{E_0}(\cdot)) \|_{H^1(\pa E_0)} \\
                \leq & C \| \nabla^3 u_{E_t^\beta}(\cdot+ f^\beta(t,\cdot)\nu_{E_0}(\cdot)) - \nabla^3 u_{E_0}\|_{H^1(\pa E_0)}+ C  \| \nabla^3 u_{E_0} \|_{H^1(\pa E_0)} \\
                \leq& C \| f^\beta(t,\cdot) \|_{L^2(\pa E_0)}^{1/3} \| f^\beta(t,\cdot)\|^{2/3}_{H^5(\pa E_0)}+ C \| \nabla^3 u_{E_0} \|_{H^1(\pa E_0)} \\
                \leq & C(K_{el}) t^{1/6}+ C \| \nabla^3 u_{E_0} \|_{H^1(\pa E_0)},  
        \end{align*}
        we recall that the constant arising from the change of variables depends only on the $C^1$ norm of $f^\beta (t,\cdot)$ which is bounded by a universal constant, and in last inequality we have used~\eqref{tempidiconvenzioni}. Hence, by~\eqref{perlachinaValle}, and taking $t$ sufficiently small, the claim follows.

        Therefore, we conclude that the family $\{ E_t^\beta \}_{t \in [0,T_s]}$, parametrized by the diffeomorphisms $\Psi_t(x)= x+ f^\beta(t,x)\nu_{E_0}(x)$, is a strong solution
to the anisotropic surface diffusion equation with elasticity.  More precisely, using the expansion of the anisotropic curvature from Corollary~\ref{zonzo} and the expansion of the Laplace-Beltrami operator as in~\cite{FJM2020}, we find that the function $f^\beta: [0,T_s] \times \pa E_0 \rightarrow \R$ is a strong solution to the
equation
\begin{equation}\label{130520262043zz}
    \pa_t u= -\Delta_{\pa E_0} \div_{\pa E_0 }\big( \mathcal{A}(\nu_{E_0}) \nabla_{\pa E_0} u \big)+ \widetilde{R}(\cdot,u,\overline{\nabla}u, \overline{\nabla}^2 u, \overline{\nabla}^3 u, \overline{\nabla}^4 u  ) \text{ on } \pa E_0 \times [0,T_s]
\end{equation}
with initial datum $u(0,\cdot)=0$ and where $\widetilde{R}$ is related to the Laplace-Beltrami of~\eqref{nuovoresto}. Where by strong solution we mean that $f^\beta \in L^\infty([0,T_s], H^5(\pa E_0)) \cap {\rm Lip}([0,T_s], L^1(\pa E_0))$ and satisfies equation~\eqref{130520262043zz}. By a standard Gr\"onwall argument, one deduces that the strong solution to~\eqref{130520262043zz} with zero initial datum is unique. This shows that the limiting flat flow coincides with the classical solution on the time interval $[0,T_s]$, which concludes the proof.
\end{proof}

	\begin{remark}
		\label{rem:thm2}
		We may quantify the statement of Theorem ~\ref{1THMMAIN1} as follows. Let $E_0\Subset \Omega$ be a open connected set of class $C^{6,1}$ that satisfies the UBC with radius $2\sigma_0$,
		and assume that the heightfunction $ \psi_1$, see~\eqref{12052025form-1}, satisfies 
		\[
		\|\nabla_{\pa E_0}\psi_1 \|_{L^2(\pa E_0)} \leq L_0 h \quad \text{and} \quad \| \nabla_{\pa E_0}\Delta_{\pa E_0}\psi_1 \|_{L^2(\Sigma_0)} \leq  L_0\sqrt{h}.
		\]
		Then, there exist constants  $K_0 = K_0(\sigma_0,L_0,K_{el})$ and $\beta_0=\beta_0(\sigma_0,K_0, K_{el})$ such that, if
		\[
		\|H_{ E_0}^\varphi  \|_{H^3(\pa E_0)} \leq K_0  \quad \text{and} \quad \|  \Delta_{\pa E_0}^2  H_{\pa E_0}^\varphi \|_{L^2(\pa E_0)} \leq  K_0 h^{-\frac14}
		\]
		then  the discrete constrained flat flow $\{E^{h,\beta}_t\}_{t\geq 0}$ with $\beta \leq \beta_0$ satisfies the UBC with radius $r_0(\sigma_0,L_0, K_{el})$, and 
		\[
		\|H_{E^{h,\beta}_t }^\varphi \|_{H^3(\pa E^{h,\beta}_t)} \leq  K_0  \quad \text{and} \quad \| \Delta_{\pa E^{h,\beta}_t}^2 H^\varphi_{E^{h,\beta}_t} \|_{L^2(\pa E^{h,\beta}_t)} \leq K_0 h^{-\frac14}
		\]
		for all $t \in [0,T_s]$, where $T_s = T_s(r_0,K_0.K_{el})$.  
	\end{remark}

\begin{remark}
		\label{rem:uniqueness-strong}
		The arguments in the proof of Theorem~\ref{1THMMAIN1} imply that if a constrained discrete flat flow $\{ E^{h, \beta}_t \}_{t \geq 0}$, starting from $E_0$ satisfies
		\[
		\|H_{E^{h,\beta}_t }^\varphi \|_{H^3(\pa E^{h,\beta}_t)} \leq  K_0  \quad \text{and} \quad \| \Delta_{\pa E^{h,\beta}_t}^2 H^\varphi_{E^{h,\beta}_t} \|_{L^2(\pa E^{h,\beta}_t)} \leq K_0 h^{-\frac14}
		\]
		for all $t \in [0,T_s]$, then the limiting flat flow coincides with the classical solution on the time interval $[0,T_s]$.  
	\end{remark}
\section{Consistency of flat flows}\label{sezioneconvglobaleee}
We are now in a position to show that flat flow solutions, as defined in Subsection~\ref{subsec:flatsol}, coincide with classical solutions whenever the latter exist.
The proof of this theorem is similar to those presented in~\cite{CFJKsd}[Theorem 1.2],~\cite{Kub2025}[Theorem 7.1],~\cite{KLMP2025}[Theorem 7.4]; therefore, we only report the main ideas here.
\begin{theorem}\label{convglobalsolu}
    Let $\Omega$ and $E_0$ be as in Definition~\ref{Omegaevrietadiriferimento}.
	Let $T_e>0$ and $\{E_t\}_{t \in [0,T_e)}$ be a classical solution of~\eqref{MAINEQ} with initial datum $E_0$. Then, for every $T < T_e$, there exist $\beta(T)>0$ and $K_{el}(T)>0$ such that, for all $\beta \in (0,\beta(T)]$, the constrained flat flow  $E^\beta_t$, starting from $E_0$, coincides with $E_t$ in $[0,T]$.  
\end{theorem}
\begin{proof}
    Let $T< T_e$, since $E_t$ is regular in $[0,T]$, there exist positive constants $ K_2,\sigma_2$, and $\tilde K_{el}$ such that 
		\begin{equation}\label{14052026form1}
			E_t \in \mathfrak{H}^5_{K_2,\sigma_2}(E_0),\, \,\| \tilde{u}_{E_t}\|_{C^{3,1/4}(\Omega)} < \tilde K_{el},\,\, \| \nabla^3 u_{E_t} \|_{H^1(\pa E_t)} < \tilde K_{el},
		\end{equation}
        for all $t \in [0,T]$, where $\tilde{u}_{E_t}$ is  defined in~\eqref{estensioneu_E^beta} with $E^\beta_t$ replaced by $E_t$. Let $\beta_0,T_s$ be given by Theorem~\ref{1THMMAIN1} and Remark~\ref{rem:thm2} for 
		$K_{el}=4 \tilde{K}_{el} $, and 
		fix $\beta\leq  \beta_0$. Let $k_0 $ be the first natural number such that $T_0 \in [hk_0, h(k_0+1))$, and let $\{E^{h,\beta}_{hk}\}_{k \in \N}$ be a discrete constrained flat flow starting from $E_0$. By~\eqref{12052025form-1} with $K_0 > K_{el}$, we have
        \begin{equation}
            	\begin{split}
				&\pa E^{h,\beta}_{h}= \{  x +\psi_1(x)\nu_{E_0}(x): x \in \pa E_0 \},\\
				&\| \nabla_{\pa E_0} \psi_1 \|_{L^2(\pa E_0)} \leq L_0 h, \quad \| \psi_1 \|_{H^5(\pa E_0)} \leq L_0, 
				\\
				&  \|  H_{E_h^{h,\beta}}^{\varphi} \|_{H^3(\pa E_h^{h,\beta})} \leq K_0, \quad  \| \Delta_{\pa E^{h,\beta}}^2 H_{E_h^{h,\beta}}^\varphi \|_{L^2(\pa E_h^{h,\beta})} \leq K_0 h^{-1/4}.
			\end{split} 
        \end{equation}
         The conclusion of the theorem follows from Remark~\ref{rem:uniqueness-strong}. 
		together with the estimates
		\begin{equation}\label{19052025ancora1}
			\|H_{E^{h,\beta}_t}^{\varphi} \|_{H^3(\pa E^{h,\beta}_t )} \leq K_0, \quad  \| \Delta_{\pa E^{h,\beta}_t }^2 H_{E^h_t}^\varphi \|_{L^2(\pa E^{h,\beta}_t )} K_0 h^{-1/4},
		\end{equation}
        for every $t \in [0,T]$.
The proof of these inequalities proceeds analogously to that  in~\cite{Kub2025}[Theorem 7.1],~\cite{KLMP2025}[Theorem 7.4], and~\cite{CFJKsd}[Theorem 1.2 formula 5.18]. Therefore, we omit it here.
\end{proof}

\appendix
\section{} \label{appendix}
In this appendix, we prove Lemma~\ref{Lemmacommutatoredifficile}.  Let $\Sigma$ be a Riemannian manifold, and let ${E_1,E_2}$ be a local frame.
We adopt the Einstein summation convention, whereby repeated indices are implicitly summed over. We begin by recalling some definitions and properties of the covariant derivative; for further details, we refer the reader to~\cite[Chapter 4]{LeeBook} and~\cite[Chapters 4 and 5]{LeeBook2018}.

We recall that the second covariant derivative $\overline{\nabla}^2$ of a tensor $F \in T^{(h,k)}(T \Sigma)$ is a tensor of the type $T^{(h,k+2)}(T \Sigma)$ defined by
\begin{equation}
    \overline{\nabla}_{X,Y}^2 F (\dots)= \overline{\nabla}^2 F (\dots, Y,X)= \overline{\nabla}_X \big( \overline{\nabla}_Y F \big)- \overline{ \nabla}_{\overline{\nabla}_X Y}F.
\end{equation}
Taking the metric trace of $\overline{\nabla}^2$ with respect to the last two covariant indices yields the Laplace-Beltrami operator, namely 
\begin{equation}
        \Delta_{\Sigma} F:= {\rm tr}_g \big( \overline{\nabla}^2 F \big)= g^{i j} \overline{\nabla}_{E_i, E_j} F  .
\end{equation}
We also recall that, given a tensor $F \in T^{(h,k)} (T \Sigma)$ with $h,k \neq 0$, the covariant derivative $\overline{\nabla}$ commutes with the trace taken over any pair of indices consisting of one covariant and one contravariant index. Hence,
\begin{equation}
    \overline{\nabla}_X({\rm tr}\,\, F)= {\rm tr} \big( \overline{\nabla}_X F \big) \text{ for all } X \in \mathfrak{X}(\Sigma).
\end{equation}
Furthermore, if $F \in T^{(h,k)}(T \Sigma)$ and $G \in T^{(l,m)}(T \Sigma) $, then 
\begin{equation}
    \overline{\nabla} (F \otimes G)= (\overline{\nabla} F) \otimes G + F \otimes (\overline{\nabla} G). 
\end{equation}
Moreover, for $\omega \in \mathfrak{X}^* (\Sigma)$ and $X,Y\in \mathfrak{X}(\Sigma)$ we have that
\begin{equation}\label{ZZZZZZZZZZZZas}
    \overline{\nabla}_Y \omega(X)= \big( \overline{\nabla}_Y \omega \big) (X)+ \omega \big( \overline{\nabla}_Y X  \big).
\end{equation}
Finally, we recall that we adopt the notation $S \star T$ from~\cite{Hamilton1982, Mantegazza2002} to denote a tensor obtained by contracting
some indexes of tensors $S$ and $T$ using the coefficients of the metric tensor $g_{ij}$.
\begin{lemma}
  Let $\Sigma$ be a Riemannian manifold, and let $X$ be a $C^3$ vector field on $\Sigma$. Then
   \begin{equation} \label{LapDivXDivLapX}
       \Delta_{\Sigma} \div_{\Sigma} X= \div_{\Sigma} \Delta_{\Sigma} X+r_1(B_\Sigma \star \overline{\nabla}X, \overline{\nabla} B_\Sigma \star X, \overline{\nabla}^2X)
    \end{equation} 
    where $r_1$ is a smooth function satisfying $r_1(\cdots)=0$ if $X=0$ and $ X \mapsto r_1(B_\Sigma \star \overline{\nabla}X, \overline{\nabla} B_\Sigma \star X,\overline{\nabla}^2X) $ is a linear function.
    \end{lemma}
\begin{proof}
We first compute  
\begin{align}
    \div_{\Sigma} \Delta_{\Sigma} X&= {\rm tr} \left( \overline{\nabla} g^{ij} \otimes \big( \overline{\nabla}_{E_i} \overline{\nabla}_{E_j} X- \overline{\nabla}_{ \overline{\nabla}_{E_i} E_j  }X  \big)   \right) + {\rm tr} \left( g^{ij} \overline{\nabla} \big( \overline{\nabla}_{E_i} \overline{\nabla}_{E_j} X- \overline{\nabla}_{  \overline{\nabla}_{E_i} E_j} X   \big)   \right)\notag\\
    &={\rm tr} \left( g^{ij} \overline{\nabla} \big( \overline{\nabla}_{E_i} \overline{\nabla}_{E_j} X \big)   \right) + {\rm tr} \left( \overline{\nabla} g^{ij} \otimes \big( \overline{\nabla}_{E_i} \overline{\nabla}_{E_j} X- \overline{\nabla}_{ \overline{\nabla}_{E_i} E_j  }X  \big)   \right) \label{13032026fz1}\\
    &  \quad-  {\rm tr} \left( g^{ij} \overline{\nabla} \big(  \overline{\nabla}_{  \overline{\nabla}_{E_i} E_j} X   \big)   \right)    \notag.
\end{align}  
We observe that, in the formula above, the leading term is
$ {\rm tr} \left( g^{ij} \overline{\nabla} \big( \overline{\nabla}_{E_i} \overline{\nabla}_{E_j} X   \big)   \right), $
 since it is the only term containing third-order covariant derivatives, whereas all the remaining terms are of lower order.

We have
\begin{align}
    \Delta_{\Sigma} \div_{\Sigma} X &=  g^{ij} \overline{\nabla}_{E_i E_j} {\rm tr} \big( \overline{\nabla} X \big)  \notag\\
    & = {\rm tr} \left( g^{ij} \overline{\nabla}_{E_i} \overline{\nabla}_{E_j} \overline{\nabla} X- g^{ij} \overline{\nabla}_{ \overline{\nabla}_{E_i} E_j   } \overline{\nabla}X   \right).\label{1203primadorm-1}
\end{align}
Moreover, the following identity holds:
\begin{equation}\label{1203primdorm2}
 \overline{\nabla}_{E_k} \overline{\nabla} X- \overline{\nabla} \big(\overline{\nabla}_{E_k} X \big)= R(E_k,\cdot) X + \overline{\nabla}_{[E_k, \cdot]} X - \overline{\nabla}_{\overline{\nabla}_{E_k} \cdot} X,
\end{equation}
where $R$ denotes the Riemann curvature tensor; see~\eqref{ILRiemann}.
Indeed, recall that $\overline{\nabla} X$ is a $(1,1)$ tensor. Given $ \omega \in \mathfrak{X}^*(\Sigma)$ and $ Y \in \mathfrak{X}(\Sigma)$, we have
\begin{align*}
    \overline{\nabla}_{E_k} \big( \overline{\nabla} X \big)(\omega,Y)&= \omega \big( \overline{\nabla}^2_{E_k,Y} X    \big)= \omega \big( \overline{\nabla}_{E_k} \big( \overline{\nabla}_Y X \big) - \overline{\nabla}_{\overline{\nabla}_{E_k} Y} X  \big)\\
    &= \omega \big( R(E_k,Y)X+ \overline{\nabla}_Y \big( \overline{\nabla}_{E_k} X  \big)  +\overline{\nabla}_{[E_k,Y]}X - \overline{\nabla}_{\overline{\nabla}_{E_k} Y} X \big)
\end{align*}
where we used the definition of the Riemann curvature tensor; see~\eqref{ILRiemann}. Since $\omega \in \mathfrak{X}^*(\Sigma)$ and $ Y \in \mathfrak{X}(\Sigma
)$ are arbitrary, identity~\eqref{1203primdorm2} follows.
Applying~\eqref{1203primdorm2} twice, we obtain
\begin{align}
     \overline{\nabla}_{E_i} \overline{\nabla}_{E_j} \overline{\nabla}  X&= \overline{\nabla}_{E_i} \overline{\nabla}  \big( \overline{\nabla}_{E_j} X \big)+ \overline{\nabla}_{E_i} R(E_j,\cdot)X+ \overline{\nabla}_{E_i} \overline{\nabla}_{ [E_j,\cdot]  } X - \overline{\nabla}_{E_i}\overline{\nabla}_{\overline{\nabla}_{E_j} \cdot} X \notag\\
     &= \overline{\nabla} \big( \overline{\nabla}_{E_i} \overline{\nabla}_{E_j} X \big)+ R(E_i,\cdot) \overline{\nabla}_{E_j} X+ \overline{\nabla}_{ [E_i,\cdot]} \overline{\nabla}_{E_j}X - \overline{\nabla}_{\overline{\nabla}_{E_i} \cdot} \overline{\nabla}_{E_j}X\label{13032026matt1}\\
     &\quad + \overline{\nabla}_{E_i} R(E_j,\cdot) X+ \overline{\nabla}_{E_i} \overline{\nabla}_{[E_j,\cdot] } X - \overline{\nabla}_{E_i}\overline{\nabla}_{\overline{\nabla}_{E_j} \cdot} X.\notag
\end{align}
Combining~\eqref{13032026matt1} with~\eqref{1203primadorm-1}, we obtain
\begin{align}
        \Delta_{\Sigma} \div_{\Sigma} X&={\rm tr} \bigg( g^{ij} \overline{\nabla} \big( \overline{\nabla}_{E_i} \overline{\nabla}_{E_j} X \big)   \bigg)+ {\rm tr} \bigg( g^{ij} R(E_i,\cdot) \overline{\nabla}_{E_j}X   \bigg) \label{13032026fz2}\\
        & \quad+ {\rm tr} \bigg( g^{ij} \overline{\nabla}_{[E_i,\cdot]}  \overline{\nabla}_{E_j}X   \bigg) - {\rm tr} \bigg( g^{ij} \overline{\nabla}_{\overline{\nabla}_{E_i} \cdot} \overline{\nabla}_{E_j}X   \bigg)\notag\\
        &\quad +{\rm tr} \bigg( g^{ij} \overline{\nabla}_{E_i} R(E_j,\cdot) X   \bigg) + {\rm tr} \bigg( g^{ij} \overline{\nabla}_{E_i}\overline{\nabla}_{[E_j,\cdot]}  X   \bigg)\notag\\
        & \quad  - {\rm tr} \bigg( g^{ij} \overline{\nabla}_{E_i}\overline{\nabla}_{ \overline{\nabla}_{E_j} \cdot  }  X   \bigg)- {\rm tr} \bigg( g^{ij} \overline{\nabla}_{ \overline{\nabla}_{E_i} E_j  }  X   \bigg).\notag
\end{align}

We show~\eqref{LapDivXDivLapX}.
Combining~\eqref{13032026fz1} and~\eqref{13032026fz2}, we obtain
\begin{align}
     \Delta_{\Sigma} \div_{\Sigma} X- \div_{\Sigma} \Delta_{\Sigma} X &=  {\rm tr} \bigg( g^{ij} R(E_i,\cdot) \overline{\nabla}_{E_j}X   \bigg) \label{13032026dif}\\
        & \quad +{\rm tr} \bigg( g^{ij} \overline{\nabla}_{[E_i,\cdot]}  \overline{\nabla}_{E_j}X   \bigg)  - {\rm tr} \bigg( g^{ij} \overline{\nabla}_{\overline{\nabla}_{E_i} \cdot} \overline{\nabla}_{E_j}X   \bigg)\notag\\
        &\quad + {\rm tr} \bigg( g^{ij} \overline{\nabla}_{E_i} R(E_j,\cdot) X   \bigg) + {\rm tr} \bigg( g^{ij} \overline{\nabla}_{E_i}\overline{\nabla}_{[E_j,\cdot]}  X   \bigg)\notag\\
        & \quad  - {\rm tr} \bigg( g^{ij} \overline{\nabla}_{E_i}\overline{\nabla}_{ \overline{\nabla}_{E_j} \cdot  }  X   \bigg)- {\rm tr} \bigg( g^{ij} \overline{\nabla}_{ \overline{\nabla}_{E_i} E_j  }  X   \bigg) \notag\\
        & \quad + {\rm tr} \bigg( \overline{\nabla} g^{ij} \otimes \big( \overline{\nabla}_{E_i} \overline{\nabla}_{E_j} X- \overline{\nabla}_{ \overline{\nabla}_{E_i} E_j  }X  \big)   \bigg) \notag\notag\\
        & \quad-  {\rm tr} \bigg( g^{ij} \overline{\nabla} \big(  \overline{\nabla}_{  \overline{\nabla}_{E_i} E_j} X   \big)   \bigg)=:a. \notag
\end{align}
We observe that the left-hand side of~\eqref{13032026dif} is independent of the choice of local chart and the local frame. Consequently, the right-hand side is also independent of these choices; in particular, it does not depend on $E_i, E_j$. 
Therefore, we deduce that $a=a(R \star \overline{\nabla}X, \overline{\nabla} R \star X, \overline{\nabla}^2X)$, where $a$ is a smooth function satisfying $a(0,0,\cdot,\cdot)=0$. Finally, recalling  that the Riemann curvature tensor can be expressed in terms of the second fundamental form, we conclude that the function $a(R \star \overline{\nabla}X, \overline{\nabla} R \star X, \overline{\nabla}^2X)$ can be rewritten as $r_1(B_\Sigma \star \overline{\nabla}X, \overline{\nabla} B_\Sigma \star X, \overline{\nabla}^2X)$. It is also immediate that $r_1$ is linear in $X$.
\end{proof}

\begin{lemma}
    Let $\Sigma$ be a Riemmanian manifold. Let $X$ be a $C^2$ vector field and let $A$ be a $C^2$ $(1,1)$-tensor field. Then
    \begin{equation}\label{LapAXugLapA_x+ALapx+rest}
        \Delta_{\Sigma}  \big(AX \big)= (\Delta_{\Sigma}A)X+ A \Delta_{\Sigma}X+ b( A \star X \star B_\Sigma, \overline{\nabla }A \star \overline{\nabla} X),
    \end{equation}
   where $b$ is a smooth vector field satisfying $b(\cdots)=0$ whenever $X=0$, and  $X \mapsto b( A \star X \star B_\Sigma, \overline{\nabla }A \star \overline{\nabla} X) $ is a linear function. 
\end{lemma}
\begin{proof}
   We may assume, without loss of generality, that $X$ and $A$ are smooth. Let  $\{\varepsilon^i\}$ be the dual frame of $\{ E_i\}$. We divide the proof into several steps.\\
    \textit{Step 1:} We first compute  $\overline{\nabla}^2_{E_m,E_k} A$, $ (\overline{\nabla}^2_{E_m,E_k} A )X$ and $\big(\Delta_{\Sigma} A \big) X$. We write $A= A_j^i \varepsilon^j \otimes E_i$, where the coefficients $A_j^i$ are smooth functions. Then we have
    \begin{equation}
        \overline{\nabla}_{E_k} \big(  A^i_j \varepsilon^j \otimes E_i \big) = \overline{\nabla}_{E_k} \big(  A^i_j \big) \varepsilon^j \otimes E_i+ A_j^i \overline{\nabla}_{E_k} \varepsilon^j   \otimes E_i+ A_j^i  \varepsilon^j   \otimes \overline{\nabla}_{E_k} E_i  .
    \end{equation}
    Moreover, we get
    \begin{equation*}
    \begin{split}
     \overline{\nabla}_{E_m}   \overline{\nabla}_{E_k} \big(  A^i_j \varepsilon^j \otimes E_i \big) 
     &= \overline{\nabla}_{E_m} \overline{\nabla}_{E_k} A_j^i \varepsilon^j \otimes E_i + \overline{\nabla}_{E_k} A_j^i \overline{\nabla}_{E_m} \varepsilon^j \otimes E_i + \overline{\nabla}_{E_k}A_j^i \varepsilon^j \otimes \overline{\nabla}_{E_m} E_i\\
     & \quad +A^i_j \overline{\nabla}_{E_m} \big( \overline{\nabla}_{E_k} \varepsilon^j \big) \otimes E_i+ A^i_j \overline{\nabla}_{E_k} \varepsilon^j \otimes \overline{\nabla}_{E_m} E_i\\
     & \quad+ A^i_j \overline{\nabla}_{E_m} \varepsilon^j \otimes \overline{\nabla}_{E_k} E_i + A^i_j \varepsilon^j \otimes \overline{\nabla}_{E_m} \overline{\nabla}_{E_k} E_i\\
     & \quad+ \overline{\nabla}_{E_m} A^i_j \overline{\nabla}_{E_k}\varepsilon^j \otimes  E_i + \overline{\nabla}_{E_m} A^i_j \varepsilon^j \otimes \overline{\nabla}_{E_k}  E_i.
     \end{split}
    \end{equation*}
    Next,  we have
    \begin{equation*}
        \overline{\nabla}_{  \overline{\nabla}_{E_m} E_k  } A = \big(\overline{\nabla}_{  \overline{\nabla}_{E_m} E_k  } A_j^i \big) \varepsilon^j \otimes E_i + A_j^i  \big(  \overline{\nabla}_{  \overline{\nabla}_{E_m} E_k  } \varepsilon^j \big) \otimes E_i+ A_j^i \varepsilon^j \otimes \big(  \overline{\nabla}_{  \overline{\nabla}_{E_m} E_k  } E_i \big).
    \end{equation*}
    Combining the above expressions, we obtain the expression of $\overline{\nabla}^2_{E_m,E_k} A$. Consequently, we have
    \begin{equation}
        \begin{split}
            (\overline{\nabla}^2_{E_m, E_k} A) X&= X^j \overline{\nabla}_{E_m} \overline{\nabla}_{E_k} A^i_j  E_i + \overline{\nabla}_{E_k} A^i_j \big( \overline{\nabla}_{E_m} \varepsilon^j \big) (X) E_i+ X^j\overline{\nabla}_{E_k} A^i_j  \overline{\nabla}_{E_m} E_i \\
            & \quad +A^i_j \overline{\nabla}_{E_m} \big( \overline{\nabla}_{E_k} \varepsilon^j \big)(X) E_i+ A^i_j \big( \overline{\nabla}_{E_k} \varepsilon^j \big) (X) \overline{\nabla}_{E_m} E_i\\
            &\quad+ A^i_j \big(\overline{\nabla}_{E_m} \varepsilon^j \big) (X)\overline{\nabla}_{E_k} E_i+ A^i_j X^j \overline{\nabla}_{E_m} \overline{\nabla}_{E_k} E_i \\
            &\quad+ \overline{\nabla}_{E_m} A^i_j \big(\overline{\nabla}_{E_k}\varepsilon^j \big) (X)  E_i + X^j\overline{\nabla}_{E_m} A^i_j \overline{\nabla}_{E_k}  E_i \\
            & \quad - X^j\overline{\nabla}_{ \overline{\nabla}_{E_m} E_k  }A^i_j E_i-A^i_j \big( \overline{\nabla}_{ \overline{\nabla}_{E_m} E_k} \varepsilon^j   \big)(X) E_i - A^i_j X^j \overline{\nabla}_{ \overline{\nabla}_{E_m} E_k} E_i.
            \end{split}
    \end{equation}
    Contracting the above expression with $g^{mk}$ and summing, we obtain $\big(\Delta_{\pa E} A \big) X$.\\
    \textit{Step 2:} We  compute $\overline{\nabla}^2_{E_m,E_k} \big( AX \big)$ and $ \Delta_{\pa E} \big( AX \big)$. 
    Recall that $A^i_j X^j E_i= A^i_j \varepsilon^j(X) E_i$. Applying the Leibniz rule together with~\eqref{ZZZZZZZZZZZZas}, we obtain
    \begin{equation}
    \begin{split}
        \overline{\nabla}_{E_k} \big( A^i_j \varepsilon^j(X) E_i  \big)&= \overline{\nabla}_{E_k} A^i_j X^j E_i + A^i_j \big(\overline{\nabla}_{E_k} \varepsilon^j\big)(X)E_i\\
        & \quad + A^i_j \varepsilon^j \big(  \overline{\nabla}_{E_k} X \big) E_i + A^i_j \varepsilon^j(X) \overline{\nabla}_{E_k} E_i.
    \end{split}
    \end{equation}
   Then, we obtain
    \begin{align*}
            \overline{\nabla}_{E_m} \big(  \overline{\nabla}_{E_k} \big( A^i_j \varepsilon^j(X)E_i  \big)  \big)= & \big( \overline{\nabla}_{E_m} \overline{\nabla}_{E_k} A^i_j \big) X^j E_i + \overline{\nabla}_{E_k} A^i_j \big(\overline{\nabla}_{E_m} \varepsilon_j \big)(X) E_i\\
            & \quad +\overline{\nabla}_{E_k} A^i_j \varepsilon_j \big (\overline{\nabla}_{E_m} X \big)E_i+\overline{\nabla}_{E_k} A^i_j X^j \overline{\nabla}_{E_m} E_i\\
            & \quad + \overline{\nabla}_{E_m} A^i_j \big(\overline{\nabla}_{E_k} \varepsilon^j  \big) (X) E_i + A^i_j \big( \overline{\nabla}_{E_m} \big( \overline{\nabla}_{E_k} \varepsilon^j  \big)  \big) (X) E_i \\
            & \quad+ A^i_j \big(\overline{\nabla}_{E_k} \varepsilon^j \big) ( \overline{\nabla}_{E_m} X) E_i +A^i_j \big( \overline{\nabla}_{E_k} \varepsilon^j \big)(X) \overline{\nabla}_{E_m}E_i\\
            &\quad+\overline{\nabla}_{E_m} A_j^i \varepsilon^j \big( \overline{\nabla}_{E_k} X  \big) E_i+ A^i_j \big( \overline{\nabla}_{E_m} \varepsilon^j \big) (\overline{\nabla}_{E_k} X) E_i\\
            & \quad+A^i_j \varepsilon^j \big( \overline{\nabla}_{E_m} \overline{\nabla}_{E_k} X  \big) E_i  +  A^i_j \varepsilon^j (\overline{\nabla}_{E_k} X) \overline{\nabla}_{E_m} E_i\\
            &\quad +\overline{\nabla}_{E_m} A^i_j  \varepsilon^j (X) \overline{\nabla}_{E_k} E_i + A^i_j \big( \overline{\nabla}_{E_m} \varepsilon^j \big)(X) \overline{\nabla}_{E_k} E_i \\
            & \quad+ A^i_j \varepsilon^j \big(  \overline{\nabla}_{E_m} X\big) \overline{\nabla}_{E_k} E_i+ A^i_j \varepsilon^j (X) \overline{\nabla}_{E_m} \overline{\nabla}_{E_k} E_i.
    \end{align*}
    Moreover, we have that
    \begin{align*}
        \overline{\nabla}_{ \overline{\nabla}_{E_m} E_k  } \big( AX \big)&= \overline{\nabla}_{ \overline{\nabla}_{E_m} E_k  } \big( A^i_j \varepsilon^j(X)E_i \big) \\
        &=  X^j\overline{\nabla}_{ \overline{\nabla}_{E_m} E_k  } A^i_j E_i+ A^i_j X^j \overline{\nabla}_{ \overline{\nabla}_{E_m} E_k  } E_i \\
        & \quad + A^i_j \big(\overline{\nabla}_{ \overline{\nabla}_{E_m} E_k  } \varepsilon^j \big) (X) E_i + A^{i}_j \varepsilon^j \big( \overline{\nabla}_{ \overline{\nabla}_{E_m} E_k  } X \big) E_i.
    \end{align*}
    From the above equations, we obtain the expression for $ \overline{\nabla}^2_{E_m,E_k} \big( AX \big)$. We conclude by using that $$ \Delta_{\Sigma} \big( A X \big) = g^{mk} \overline{\nabla}^2_{E_m,E_k} \big( AX \big).$$
    \textit{Step 3:} We compute $A \Delta_{\Sigma}X $. 
    We have 
    \begin{equation}
        \begin{split}
            \Delta_{\Sigma} X=  g^{mk} \big( \overline{\nabla}_{E_m} \overline{\nabla}_{E_k} X-  \overline{\nabla}_{\overline{\nabla}_{E_m}E_k} X \big).
        \end{split}
    \end{equation}
    Hence, we  get
    \begin{equation}
        \begin{split}
            A \Delta_{\Sigma} X=  g^{mk} \big( A^i_j \varepsilon^j\big(\overline{\nabla}_{E_m} \overline{\nabla}_{E_k} X \big) - A^i_j  \varepsilon^j \big( \overline{\nabla}_{\overline{\nabla}_{E_m} E_k} X \big)  \big)E_i.
        \end{split}
    \end{equation}
    \textit{Step 4:}
    Formula~\eqref{LapAXugLapA_x+ALapx+rest} now follows from Steps 1- 3. Indeed, several cancellations occur in the expression $\Delta_{\Sigma}  \big(AX \big)- (\Delta_{\Sigma}A)X- A \Delta_{\Sigma}X,$
    in particular the second order covariant derivatives of $X$ cancel. Moreover, we note that the above expression is independent of the particular choice of the local chart.   We also note that the remainder term $b$ should, in principle, contain additional terms of the form $A \star X, \, \overline{\nabla}A \star X, \, A\star \overline{\nabla} X $, but since they are lower-order terms, we can omit them in the remainder notation.
\end{proof}
We conclude this appendix by computing the commutator of the Laplace–Beltrami operator and the gradient on a manifold.
\begin{lemma}
    Let $\Sigma$ be a Riemannian manifold and let $f \in C^3(\Sigma)$. Then
    \begin{equation}\label{Ketteratura}
        \Delta_{\Sigma} {\rm grad\,} f = {\rm grad\,} \Delta_{\Sigma} f+ a(\overline{\nabla}f \star B_\Sigma,\overline{\nabla}^2f)
    \end{equation}
    where $a$ is smooth function such that $a(\cdots)=0$ whenever $f=0$ and $ f \rightarrow a(\overline{\nabla}f \star B_\Sigma,\overline{\nabla}^2f)$ is a linear function.
\end{lemma}
\begin{proof}
    We begin by computing $ \Delta_{\Sigma} {\rm grad\,} f$.
    We have 
    \begin{align}
            \Delta_{\Sigma} {\rm grad\,} f =&  g^{mk} \bigg[ \overline{\nabla}_{E_m} \overline{\nabla}_{E_k} \big( g^{ij} \overline{\nabla}_{E_j} f E_i  \big) - \overline{\nabla}_{ \overline{\nabla}_{E_m} E_k  } \big( g^{ij} \overline{\nabla}_{E_j} f E_i  \big)    \bigg]\notag\\
            = & g^{mk}\bigg[   \overline{\nabla}_{E_m} \overline{\nabla}_{E_k} g^{ij} \overline{\nabla}_{E_j} f E_i+ \overline{\nabla}_{E_k} g^{ij} \overline{\nabla}_{E_m} \overline{\nabla}_{E_j}f E_i+ \overline{\nabla}_{E_k} g^{ij} \overline{\nabla}_{E_j} f \overline{\nabla}_{E_m} E_i \label{Letteratura1}\\
            & \quad\quad +\overline{\nabla}_{E_m} g^{ij} \overline{\nabla}_{E_k} \overline{\nabla}_{E_j} f E_i + g^{ij} \overline{\nabla}_{E_m} \overline{\nabla}_{E_k} \overline{\nabla}_{E_j} f E_i + g^{ij} \overline{\nabla}_{E_k} \overline{\nabla}_{E_j} f \overline{\nabla}_{E_m} E_i   \notag\\
            & \quad \quad + \overline{\nabla}_{E_m} g^{ij}\overline{\nabla}_{E_j} f \overline{\nabla}_{E_k} E_i+ g^{ij} \overline{\nabla}_{E_m} \overline{\nabla}_{E_j} f \overline{\nabla}_{E_k} E_i + g^{ij} \overline{\nabla}_{E_j} f \overline{\nabla}_{E_m} \overline{\nabla}_{E_k} E_i  \notag\\
            & \quad \quad - \overline{\nabla}_{ \overline{\nabla}_{E_m} E_k  } g^{ij} \overline{\nabla}_{E_j} f E_i- g^{ij}  \overline{\nabla}_{ \overline{\nabla}_{E_m} E_k  } \overline{\nabla}_{E_j}f E_i- g^{ij} \overline{\nabla}_{E_j} f \overline{\nabla}_{\overline{\nabla}_{E_m} E_k} E_i \bigg].\notag
    \end{align}

    Now, we compute  $ {\rm grad\, }  \Delta_{\Sigma} f$.
    Using again  formula~\eqref{nabladeE=grad}, we obtain
    \begin{align}
               {\rm grad\, }  \Delta_{\Sigma} f =& g^{ij} \overline{\nabla}_{E_j} \bigg[ g^{mk} \overline{\nabla}^2_{E_m,E_k} f \bigg] E_i\notag\\
            =& g^{ij} \bigg[    \overline{\nabla}_{E_j} g^{mk} \big[  \overline{\nabla}^2_{E_m,E_k} f \big]+ g^{mk} \overline{\nabla}_{E_j} \overline{\nabla}_{E_m} \overline{\nabla}_{E_k} f- g^{mk} \overline{\nabla}_{E_j} \overline{\nabla}_{ \overline{\nabla}_{E_m} E_k  } f\bigg] E_i.\label{Letteratura2}
    \end{align}
    
    We observe that, in formulas~\eqref{Letteratura1} and~\eqref{Letteratura2}, the highest-order terms are
    \begin{equation}
        \overline{\nabla}_{E_m} \overline{\nabla}_{E_k} \overline{\nabla}_{E_j} f \text{ and }\overline{\nabla}_{E_j} \overline{\nabla}_{E_m} \overline{\nabla}_{E_k} f.
    \end{equation}
    Since 
    \begin{equation}
        \overline{\nabla}_{E_m} \overline{\nabla}_{E_k} \overline{\nabla}_{E_j} f = \overline{\nabla}_{E_j} \overline{\nabla}_{E_m} \overline{\nabla}_{E_k} f- [ E_m, E_j] \big(   \overline{\nabla}_{E_k} f \big)- \overline{\nabla}_{E_k} \big( [E_k,E_j] f \big),
    \end{equation}
    and since~\eqref{Letteratura1} contains the second derivatives of $E_i$
 (that is curvature terms), we deduce~\eqref{Ketteratura}.
\end{proof}
\section*{Conflict of interest}
    The authors states that there is no conflict of interest.
\section*{Acknowledgments} 
Andrea Kubin was supported by the Academy of Finland grant 314227. 
Anna Kubin research has been supported by the Austrian Science Fund (FWF) through grants 10.55776/F65, 10.55776/P35359, 10.55776/Y1292.

\end{document}